\documentclass[11 pt]{amsart}

\usepackage{amsmath,amssymb,amsthm,amsfonts,setspace,hyperref,dsfont,yhmath,setspace}
\usepackage[nameinlink]{cleveref}
\usepackage{hyperref}\hypersetup{colorlinks=true, citecolor=green, linkcolor=blue, hypertexnames=false}
\usepackage[utf8]{inputenc}

\newcommand{\NN}{\mathbb{N}}
\newcommand{\RR}{\mathbb{R}}

\newcommand{\CC}{\mathbb{C}}

\newcommand{\ZZ}{\mathbb{Z}}
\newcommand{\GG}{\mathbb{G}}

\newcommand{\norm}[1]{\lVert#1\rVert}

\newtheorem{theorem}{Theorem}[section]
\newtheorem{corollary}[theorem]{Corollary}

\newtheorem{lemma}[theorem]{Lemma}
\newtheorem{proposition}[theorem]{Proposition}
\newtheorem{example}{Example}

\newcommand{\comment}[1]{}
\theoremstyle{definition}
\newtheorem{definition}[theorem]{Definition}
\newtheorem{remark}[theorem]{Remark}

\newtheorem{sect}{}

\numberwithin{equation}{section}

\DeclareMathOperator{\diag}{diag}

\makeatletter
\renewcommand*\env@matrix[1][*\c@MaxMatrixCols c]{%
  \hskip -\arraycolsep
  \let\@ifnextchar\new@ifnextchar
  \array{#1}}
\makeatother
\begin{document}
\title[Local rigidity]
{Local rigidity of weak or no hyperbolicity algebraic actions}


\thanks{ $^1$ Based on research supported by NSF grants   DMS-1700837 and DMS-1845416}

\author[]{ Zhenqi Jenny Wang$^1$ }

\address{Department of Mathematics\\
        Michigan State University\\
        East Lansing, MI 48824,USA}
\email{wangzq@math.msu.edu}

\begin{abstract}
In this paper we study rigidity properties of abelian \hyphenation{break-able}actions with weak or no hyperbolicty. We introduce a general strategy for
proving $C^\infty$ local rigidity of algebraic actions. As a consequence, we show $C^\infty$ local rigidity for a broad class of parabolic algebraic actions on homogeneous spaces of semisimple Lie groups.
This is the first time in the literature that (strong) local rigidity for
these actions is addressed.
\end{abstract}

\maketitle

\setcounter{tocdepth}{2}
\tableofcontents

\section{Introduction}
\subsection{Abelian algebraic actions}\label{sec:21} Let $G$ be a
connected Lie group, $A\subseteq G$ a closed abelian subgroup which
is isomorphic to $\ZZ^k\times \RR^\ell$, and $\Gamma$ a (cocompact) torsion free lattice in
$G$. Then $A$ acts by left translation on the compact space
$\mathcal{X}=G/\Gamma$, which is called \emph{an algebraic $A$-action} and is denoted by $\alpha_A$. $\alpha_A$ is higher-rank if $k+\ell\geq2$. The linear part $\rho$ of $\alpha_A$ is induced by the adjoint representation of $A$ on $\text{Lie}(G)$. Define the Lyapunov exponents of $\alpha_A$ as the log's of the absolute values of the eigenvalues of $\rho$. We get linear functionals $\chi:\,A\to\RR$, which  are called \emph{Lyapunov functionals} of $\rho$.

\begin{itemize}
  \item \(\alpha_A\) is \emph{partially hyperbolic} if \(\rho\) has at least one nonzero Lyapunov functional.
    In particular, \(\alpha_A\) is \emph{hyperbolic} if, in addition, all zero Lyapunov functionals of \(\rho\) appear only in the orbit distribution of \(\alpha_A\).

  \smallskip

  \noindent
  \(\alpha_A\) is \emph{weakly partially hyperbolic} if there exists a proper rank-one subgroup $A'\subset A$ such that the (nontrivial) Lyapunov distributions for the full $A$-action coincide with those for the $A'$-action. (In other words, although $A$ may be higher rank, the nonzero Lyapunov exponents can be completely determined by a suitable one-parameter subgroup.)

  \smallskip

  \item \(\alpha_A\) is \emph{parabolic} if all Lyapunov functionals of \(\rho\) are zero, and \(\rho\) is not semisimple over \(\mathbb{C}\).
\end{itemize}
Generally, one can define partially hyperbolic and weakly partially hyperbolic actions without further restrictions; in this paper the notion of \emph{strong hyperbolicity} is reserved for those cases where either the action is higher rank hyperbolic or, when \(G\) has nontrivial semisimple components, the \(A\)-action exhibits a higher-rank behavior on each simple factor; that is,  the restriction of \(\rho\) to each simple factor of the semisimple part of \(G\) has at least two nontrivial, non-proportional Lyapunov exponents.

\subsection{Rigidity of actions and related notions} Let Act$^r(A,\mathcal{X})$ be the space of $A$ actions
by diffeomorphisms of class $C^r$ of a compact manifold $\mathcal{X}$. If $A$ is a connected Lie group, the $C^r$ topology in Act$^r(A,\mathcal{X})$ is induced by
the $C^r$ topology on vector fields which generate the action of the Lie algebra of $A$, see \cite[Section 1.2]{Damjanovic4}.  In this paper, we only consider continuous Lie groups actions of $\RR^k$.

 We say that $\alpha_A$ is (strong) $C^{\infty,\ell, \infty}$ locally
rigid, i.e., for any $C^\infty$ perturbation $A$-action $\tilde{\alpha}$ which is
sufficiently $C^{\ell}$ close to $\alpha_{A}$, there is $h\in \emph{\small Diff}^{\,\infty}(\mathcal{X})$ such that for any $x\in \mathcal{X}$ and $a\in A$ we have
\begin{align}\label{for:195}
 h(\tilde{\alpha}(a,x))=\alpha_{A}(\mathfrak{i}(a),h(x)),
\end{align}
where $\mathfrak{i}$ is an automorphism of $A$.

A weaker notion is \emph{transversally locally rigid}: Suppose that there exists a finite dimensional family $\{\alpha^\lambda_{A}\}_{\lambda\in\RR^d}$ of smooth $A$ actions on $\mathcal{X}$ such
that $\alpha^0_{A}=\alpha_{A}$, and the family is $C^1$ transversally i.e. it is $C^1$ in the parameter $\lambda$. Action $\alpha_{A}$ is transversally locally rigid with respect to the family $\{\alpha^\lambda_{A}\}$ if every
sufficiently small perturbation of the family $\{\alpha^\lambda_{A}\}$
in a neighborhood of $\lambda=0$ intersects the smooth conjugacy class of $\alpha_{A}$, where the smooth conjugacy class of $\alpha_{A}$
consists of all actions $\{h\circ \alpha_{A}\circ h^{-1}:\,h\in\emph{\small Diff}^\infty(\mathcal{X})\}$.

\subsection{History and motivation}
Motivated by the Zimmer program, the study of smooth local rigidity  of higher rank actions has become one of the most active areas of smooth dynamics and rigidity theory.
The main goal of local classification is to completely understand the dynamics of smooth actions that are small perturbations of a given action, which is usually well understood. Starting with the seminal work of Katok and Spatzier on Anosov actions \cite{Spatzier}, significant progress has been made over the past decades in the study of actions with strongly hyperbolic features, some of the highlights are \cite{as},
\cite{Damjanovic2}, \cite{Damjanovic4}, \cite{damjanovic07}, \cite{zwang-1}, \cite{zwang-2}, \cite{vinhage15}, \cite{vw}.

Most of current methods are developed from the strong hyperbolic
theory which heavily relies  on persistent geometric structures.  Unfortunately, these methods are not applicable to a large class of algebraic actions,  especially for weakly partially hyperbolic actions and parabolic actions (see Section \ref{sec:21}).  As a consequence there were many unanswered questions concerning the $C^\infty$ local rigidity of these actions.

$(Q_1)$ Is strong hyperbolicity a necessary condition for local classification of partially hyperbolic actions?

For partially hyperbolic actions, strong hyperbolicity condition was essential for applying the tools and theory developed so far.
The extension to broader actions is not straightforward, due to the lack of robust  geometric structures.

$(Q_2)$ Is the $\RR^2$ upper triangular parabolic action on
\begin{align*}
SL(2,\RR)\times SL(2,\RR)/\Gamma,
\end{align*}
where $\Gamma$ is an irreducible lattice in $SL(2,\RR)\times SL(2,\RR)$, \emph{transversally locally rigid}?

$(Q_3)$ Let $A$ be a maximal abelian subgroup of $SL(n,\RR)$, $n\geq4$. Is $\alpha_A$ \emph{strong locally rigid} on
\begin{align*}
 SL(n,\RR)/\Gamma,
\end{align*}
where $\Gamma$ is an irreducible lattice in $SL(n,\RR)$? Here $\alpha_A$ is parabolic, see \cite{Malcev}.

 $(Q_2)$ and $(Q_3)$ concern the study of parabolic actions. In general, parabolic actions are not structurally stable (one can easily find partially hyperbolic algebraic actions arbitrarily close to them). This lack of structural stability usually leads to weaker rigidity results. Specifically, for the
$\RR^2$ parabolic action in $(Q_2)$, one can only expect a weaker rigidity result, namely transversal local rigidity. In contrast, the action in $(Q_3)$ is geometrically stable, meaning it is structurally stable among algebraic perturbations (see Definition \ref{de:1} and Remark \ref{re:2} for details). This geometric stability justifies pursuing a strong
result in that case.

Local classification for parabolic actions is substantially more difficult than for hyperbolic actions. Parabolic actions do not have any form of hyperbolicity. Since  there are no invariant geometric structures altogether, the geometric considerations cannot even get started. The only progress so far is the recent work of J. Tanis and D. Damjanovi{\'c}, proving  $C^\infty$ weak local rigidity for $\ZZ^2$ parabolic actions on Heisenberg nilmanifolds \cite{DT}. Their method is the KAM scheme, which was first introduced by A. Katok and D. Damjanovi{\'c} in \cite{Damjanovic4} to prove $C^\infty$ local rigidity for higher rank partially hyperbolic automorphisms on
tori. However,  similar applications to algebraic actions of semisimple type pose new challenges.  Even to $SL(2,\RR)\times SL(2,\RR)/\Gamma$, the most basic one, the application is not straightforward.   The following remarks may illustrate the main difficulty.
The KAM method for obtaining local rigidity results reformulates the local conjugation problem as a nonlinear operator,  describing a (twisted) coboundary over the perturbed action. After linearization, we estimate how far the resulting (twisted) coboundary deviates  from being a (twisted) coboundary over the unperturbed algebraic action. In particular, if we can show
that the projection of this (twisted) coboundary onto the space of (twisted) coboundaries for the unperturbed action yields a quadratically small error, then a suitable inverse of the coboundary operator can
be defined,  and one may hope to employ the KAM iterative method to produce a \(C^\infty\) conjugacy. In short, the KAM method involves two essential ingredients: quantifying the error between the algebraic action and its perturbation, and applying the KAM iteration. The quantifying procedure is usually hard to perform, which usually needs the full machinery of the representation theory. This is the major difficulty in \cite{DT} and the main reason that their results are restricted to step 1 nilpotent groups.  It should be stressed that the representation theory for nilpotent groups is in many ways the next easiest for a Lie group (after abelian cases). Compared to nilpotent groups, the representation theory for semisimple  groups is substantially more complicated. In general the unitary dual of many higher rank simple Lie groups is not completely classified.
Even when the classification is known, it is too complicated to apply. As a consequence  there was no progress toward the study of higher rank simple Lie groups, even in the case of  $SL(3,\RR)$  whose unitary dual is  well-understood \cite{Vogen}. \hyphenation{break-able} Another  problem is that the quantifying procedure requires
consideration of simple Lie groups case-by-case. Probably, specific information from representation theory would be needed that may be available  for some Lie groups and not for others. As a result, it seems very hard to perform the quantifying procedure to general  Lie groups.

We stress that even when the quantifying procedure is possible, the KAM iteration may fail to work.  More precisely, ``goodness" of the inverses of coboundary operators is
essential for the convergence of the KAM iteration. Here,  ``goodness"  means that the  $C^r$ norm \hyphenation{norm}of the inverse can be bounded by the $C^{\theta r+\sigma}$
norm of the given data, where $1\leq\theta<2$ (see \cite{Fayad}) and $\sigma$ is a constant. Note that when \(\theta=1\), this property coincides with the standard notion of tameness; however, for $1<\theta<2$, the inverses are not tame in the classical sense, though the KAM scheme can sometimes still be applied (see \cite{Fayad}).
For the $\RR^2$ parabolic action on $SL(2,\RR)\times SL(2,\RR)/\Gamma$, if there are errors involved, the inverses of the coboundary operators seem unlikely to be ``good" since the orders \hyphenation{orders}of the obstructions to solving the cohomological equation are not uniformly bounded, see \cite{Forni} and the beginning of Section  \ref{sec:34}. This is the main reason that $(Q_2)$ remained unanswered. Moreover, For the $\ZZ^2$ parabolic action on $SL(2,\RR)\times SL(2,\RR)/\Gamma$, the inverse of the coboundary equation is not good (even when a solution exists and there is no error
involved), see \cite{tanis2018Cohomology}. This is also the case for the discrete parabolic action on $SL(n,\RR)/\Gamma$ , $n\geq3$, see Theorem 2.2 of \cite{tanis2017cohomological}. The above results suggest that ``goodness" may fail for every semisimple Lie group, which is why treating parabolic actions poses a significant challenge for the KAM method.

\subsection{Results of the paper}\label{sec:5}

The main results of the paper address these long-standing unanswered questions. Specifically, it presents smooth local rigidity for a large class of abelian actions with weak or without hyperbolicty, which greatly strengthens the results of local classification. The new technique introduced in the paper is a fundamental step towards achieving a complete solution to the program of local classification of algebraic actions.

Let $\mathbb{G}$ denote a higher-rank semisimple Lie group with finite center without compact factors satisfying: $\mathbb{G}=\mathbb{G}_1\times \cdots\times \mathbb{G}_k$, where $\mathbb{G}_1=SL(n,\RR)$, $n\geq2$. $\Gamma$ is a cocompact irreducible lattice of $\mathbb{G}$. We say that $u\in \text{Lie}(\mathbb{G})$ is \emph{nilpotent} if $\text{ad}_u$ is
nilpotent. We say that a subgroup of $\GG$ is \emph{unipotent} if its Lie algebra is (linearly) spanned by nilpotent vectors.

For any abelian subgroup $A$ of $\GG$ we use $\alpha_A$ to denote the the action of $A$ by left translations
on $\mathcal{X}=\mathbb{G}/\Gamma$.  Recall that the definition of geometric stability is provided in Definition \ref{de:1} and further elaborated in Remark \ref{re:2}.

The next two theorems are our main results for algebraic abelian actions.
\begin{theorem}\label{th:13} Suppose $\GG\neq \GG_1$. Let $A\subseteq G$ be a closed abelian subgroup of $\GG$ with the following property:
\begin{enumerate}
  \item there exist $\textbf{u},\,\textbf{v}\in\text{Lie}(A)$ such that $\textbf{u}$ is in a root space of $\GG_1$ and $\textbf{v}\in \text{Lie}(\GG_2\times\cdots \times \GG_n)$ is nilpotent;

      \smallskip
  \item $\alpha_A$ is geometrically stable.
\end{enumerate}
Then
 there is $\ell\in \NN$ such that the action $\alpha_A$ is $C^{\infty,\ell,\infty}$ {\em locally
rigid}.

\end{theorem}

We use $\Phi_1$ to denote the set of roots of $SL(n,\RR)$ and $\mathfrak{u}_\phi$ to denote the root space of $\phi$, $\phi\in \Phi_1$.

\begin{theorem}\label{th:11}
 Suppose $\GG=SL(n,\RR)$ $n\geq6$. Let $A\subseteq G$ be a closed abelian subgroup of $\GG$ with the following property:
 \begin{enumerate}
   \item there exist $\phi_i$, $1\leq i\leq3$ such that $\bigoplus\mathfrak{u}_{\phi_i}\subset \text{Lie}(A)$;

   \smallskip
   \item\label{for:350} $\phi_i-\phi_j\notin\Phi_1$ for any $1\leq i,j\leq 3$;

   \smallskip
   \item $\alpha_A$ is  geometrically stable.
 \end{enumerate}


 Then there is $\ell\in \NN$ such that the action $\alpha_A$ is $C^{\infty,\ell, \infty}$ {\em locally
rigid}.

\end{theorem}

The condition of being geometrically stable ensures that $\alpha_A$ is locally rigid among the algebraic perturbations (see Remark \ref{re:2}).
The hypotheses on $\textbf{u}$ and $\textbf{v}$ (resp. on  $\phi_i$) are given for two purposes. Firstly,
  they ensure  that cocycle rigidity holds for the $\RR^2$ (resp. $\RR^3$) action generated by $\textbf{u}$ and $\textbf{v}$ (resp. by $\mathfrak{u}_{\phi_i}$). It is analogous to the condition of the technical theorems of \cite{tanis2017cohomological} and \cite{W1} on the study of cocycle rigidity over abelian parabolic actions. We point out that if $\alpha_A$ is locally rigid then it is cocycle rigid. Secondly, they ensure that the algebraic property \hyperlink{o.0}{P} (see Section \ref{sec:51}), discussed in detail in Sections \ref{sec:43} and \ref{sec:37},  holds. Property \hyperlink{o.0}{P} is the key property that underlies the local rigidity in the paper.

Below, we list some interesting results derived from the above theorems. Applying Theorem \ref{th:11} we obtain the following result:
\begin{corollary}\label{cor:8}
Suppose $\GG=SL(n,\RR)$, $n\geq6$. Let $A$ be a maximal abelian subgroup of $\GG$.
Then there is $\ell_0\in \NN$ such that the action $\alpha_A$ is $C^{\infty,\ell_0,\infty}$ \emph{locally
rigid}.

\end{corollary}
Corollary \ref{cor:8} partially answers $(Q_3)$ for $n\geq6$. We require $n\geq6$ to satisfy assumption \eqref{for:350} in Theorem \ref{th:11}.  We currently do not know if the techniques in this
paper will be effective in  addressing the cases of $n=4,\,5$.

\begin{corollary}\label{cor:2} Suppose $\GG=\underbrace{SL(n,\RR)\times \cdots \times SL(n,\RR)}_{k \text{ copies}}$, $n\geq4$ and $k\geq2$. Let $A_{i}$ be a maximal abelian subgroup of $\GG_i$, $i\geq1$. Set $A=A_1\times A_2\times \cdots \times A_k$.
 There is $\ell_0\in \NN$ such that the action $\alpha_A$ is $C^{\infty,\ell_0,\infty}$ {\em locally
rigid}.

\end{corollary}

\begin{remark}\label{re:7} For $n\geq 4$, any maximal abelian subgroup in $SL(n,\RR)$ is unique up to automorphisms and is unipotent \cite{Malcev}. The condition of being maximal abelian in Corollary \ref{cor:8} and \ref{cor:2} guarantees that $\alpha_A$ is geometrically stable and parabolic.

\end{remark}
\begin{corollary}\label{cor:13}
Suppose $\GG=SL(n,\RR)$, $n\geq7$. Let $A_1$ be a maximal abelian subgroup of $SL(n-1,\RR)$ and $\textbf{x}$ be a diagonal matrix in $\mathfrak{sl}(n,\RR)$ which commutes with $\text{Lie}(A_1)$. Let $\tilde{\textbf{x}}$ be the one-parameter subgroup generated by $\textbf{x}$ and let $A=A_1\times \tilde{\textbf{x}}$.
Then there is $\ell_0\in \NN$ such that the action $\alpha_A$ is $C^{\infty,\ell_0,\infty}$ \emph{locally
rigid}.

\end{corollary}

From Remark \ref{re:7}, we see that $A_1$ is a unipotent subgroup. This means $\alpha_A$ is weakly partially hyperbolic.
Thus we answer $(Q_1)$.

If we remove the geometrically stable condition in Theorem  \ref{th:13}, it is reasonable to expect weak local rigidity. The following result answers $(Q_2)$.
\begin{corollary}\label{cor:11} Suppose $\GG=SL(2,\RR)\times SL(2,k_1)\cdots \times SL(2,k_n)$, where $k_i=\RR$ or $\CC$. Let $A$ be a $\RR^m$, $m\leq 1+n$, action of upper triangular unipotents.
Then there is $\ell\in \NN$ such that the action $\alpha_A$ is transversally  $C^{\infty,\ell,\infty}$ {\em locally
rigid}.
\end{corollary}

\noindent{\bf Acknowledgements.}
I would also like to express my sincere appreciation to the referees for their very careful review and valuable suggestions.
Their feedback significantly contributed to improving the paper and identifying a defective argument in the initial version of Lemma \ref{le:5}.

\section{Proof strategy}\label{sec:49}

To prove Theorems \ref{th:13} and \ref{th:11}, we introduce a general strategy, which is a combination of representation theory and analysis. The strategy differs from most of the previous methods, but has some features in common with the KAM method, such as the need for a detailed analysis of
the first cohomology and coboundary operators. However, a key difference between our approach and the classical KAM method is that our strategy does not require ``goodness" in every direction. We begin with an inverse coboundary operator that is tame only in certain directions. The lack of tameness in the remaining directions is overcome by employing a truncation procedure (specifically, the directional smoothing operators technique, which is the core innovation of this paper) alongside the higher-rank trick. Both tools are enabled by the algebraic property \hyperlink{o.0}{P} of the underlying algebraic action.
As a result, our new method works in greater generality.

In this section, we outline our proof strategy and provide pointers to the remainder of the paper, enabling the reader to understand the overall structure. First, we explain the proof of Theorem \ref{cor:7}, which provides a \(C^\infty\) splitting of a set of almost twisted cocycles. In other words, even when the twisted cocycle condition holds only approximately, the family of almost twisted cocycles can be decomposed into a genuine twisted cocycle component and an error term that is quadratically small. This splitting is crucial for controlling the error terms in the subsequent KAM iteration. Next, we describe how the KAM iteration converges.

\subsection{Construction of inverses of  coboundary operators}\label{sec:51} The key step of the strategy is obtaining a $C^\infty$ inverse of twisted coboundary operators. Our goal is to prove that
\begin{enumerate}
  \item []\hypertarget{o.9}{($\mathcal{A}$)} For any set of almost
(twisted) cocycles over $\text{Lie}(A)$-action, we can find a common $C^\infty$ approximate solution such that both the approximate solution and the errors have nice Sobolev estimates. ``Nice" means one can employ the KAM iterative method to inductively obtain a $C^\infty$ conjugation.
\end{enumerate}

The precise statement is given in Theorem \ref{cor:7}. In the following subsections from \ref{sec:44} to \ref{sec:45}, we summarize the ideas
behind the proof of \hyperlink{o.9}{$\mathcal{A}$} coming from Sections \ref{sec:34} to \ref{sec:22}.

\subsubsection{Splitting for $u$}\label{sec:44} In this part, we summarize the ideas behind the proof of \hyperlink{o.9}{$\mathcal{A}$} coming from
Section \ref{sec:34}. Fix a vector $u\in \text{Lie}(A)$ which is inside a $\RR$ one-dimensional root space $\mathfrak{u}_\phi$. We start by constructing a splitting for the first coboundary operator for $u$.  More precisely, our goal is to construct a splitting for the $u$-almost coboundary $\mathfrak{p}_{u}$,
such that:
\begin{enumerate}

  \item\label{for:351} both $\eta$ and the error part $\mathcal{E}_{u}$ in the following equation
\begin{align}\label{for:282}
 u\eta=\mathfrak{p}_{u}-\mathcal{E}_{u}
\end{align}
have nice Sobolev estimates;

  \smallskip
  \item\label{for:352} if $\mathfrak{p}_{u}$ is a $u$-coboundary, then $\mathcal{E}_{u}=0$.

\end{enumerate}
\emph{Note}. Theorem \ref{cor:7} concerns the splitting for the $u$-twisted almost coboundary
\begin{align*}
(u+\textrm{ad}_u)\eta=\mathfrak{p}_{u}-\mathcal{E}_{u}
\end{align*}
where $\eta$, $\mathcal{E}_{u}$ and $\mathfrak{p}_{u}$ are vector valued maps on $\GG/\Gamma$. In contrast, \eqref{for:282} is a simplified version where
$\eta$, $\mathcal{E}_{u}$ and $\mathfrak{p}_{u}$ are functions on $\GG/\Gamma$. However, this simplification does not affect the presentation of ideas.
\smallskip

Before we introduce the construction, we need some preliminary notations  and examples.
\begin{enumerate}

     \item\label{for:274} Partially tame: we say that $\eta$ is \emph{partially tame} on a subgroup $H$ of $\GG$ (with respect to $\mathfrak{p}_{u}$) if there exists $\sigma > 0$ such that, for any $r \geq 0$, the Sobolev norm of order $r$ of $\eta$ on $H$ can be bounded by the Sobolev norm of order $r+\sigma$ of $\mathfrak{p}_{u}$.

 \smallskip

    \item Nice Sobolev estimates: we say that $\eta$ has \emph{nice Sobolev estimates} if its Sobolev estimates are sufficiently good for the KAM iteration. We emphasize that $\eta$ being tame on $\GG$ is usually a stronger condition than $\eta$ having nice Sobolev estimates.

  \item $G_u$: the $SL(2,\RR)$ subgroup generated by $\mathfrak{u}_\phi$ and $\mathfrak{u}_{-\phi}$.

  \smallskip
  \item $G_u'$: the normalizer of $u$ in $G_u$, i.e., the subgroup generated by $\mathfrak{u}_\phi$ and $[\mathfrak{u}_\phi,\mathfrak{u}_{-\phi}]$.

  \smallskip

  \item $S_{1,u}$: the subgroup generated by $G_u'$ and $C(G_u)$, where $C(G_u)$ is the centralizer $C(G_u)$ of $G_u$ in $\GG$. We call $\text{Lie}(S_{1,u})$ the \emph{tame} subspace of $u$ and call vectors in $\text{Lie}(S_{1,u})$ the  tame directions to $u$. Similarly, we call $\text{Lie}(C(G_u))$ the \emph{super tame} subspace of $u$ and call vectors in $\text{Lie}(C(G_u))$ the super tame directions to $u$.

  \smallskip

  \item\label{for:365} Friendly pair: $u$ and $v$ are a friendly pair if $[u,v]=0$ and there exists a subalgebra $\mathfrak{B}\subseteq \text{Lie}(\GG)$ containing $u$ and $v$ with the following properties:
    \begin{enumerate}
      \item $\mathfrak{B}$ is isomorphic to $\mathfrak{sl}(2,\RR)\times \mathfrak{sl}(2,\RR)$;

      \smallskip
      \item $\text{Lie}(G_u)\subset \mathfrak{B}$.
    \end{enumerate}
We note that if $v$ lies in the super tame subspace of $u$, then $u$ and $v$ form a friendly pair.

  \smallskip

  \item \emph{Complementary} directions to $u$: vectors in root spaces that are not in  $\text{Lie}(S_{1,u})$.

  \smallskip

  \item $(CS)_u$:  the subspace of $\text{Lie}(\GG)$ spanned by complementary directions of $u$. We call $(CS)_u$ the \emph{complementary subspace} of $u$.

  \smallskip

  \item\label{for:2023}  From the definition, it is easy to verify that  we have a decomposition of $\text{Lie}(\GG)$:
\begin{align*}
 \mathfrak{g}=(CS)_u\oplus\text{Lie}(S_{1,u}).
\end{align*}
The decomposition shows that the tame subspace (of $u$) and its complementary subspace span $\mathfrak{g}$. We will show that $\eta$ is partially tame on $\text{Lie}(S_{1,u})$. This is the reason we call $\text{Lie}(S_{1,u})$ the tame subspace. However, the tameness of $\eta$ along the complementary directions is not straightforward.

\smallskip

\item Property (P): we say that \( \text{Lie}(A) \) has property \hypertarget{o.0}{(P)} if we can choose finitely many elements \( u, v_1, v_2, \dots \) in \( \text{Lie}(A) \) such that
    \begin{itemize}
    \item There exists some $v_i$ such that  $v_i$ lies in the super tame subspace of $u$.

      \item The complementary directions of \( u \) are covered by abelian nilpotent subalgebras $\mathfrak{s}_1,\,\mathfrak{s}_2,\cdots$  (Here, we say that a subalgebra is \emph{nilpotent} if, for every $x$ in the subalgebra, the operator $\mathrm{ad}_x$ is nilpotent \emph{on the entire Lie algebra} $\mathrm{Lie}(\mathbb{G})$).

      \item For any $\mathfrak{s}_i$, there exists some $v_j$ such that $\mathfrak{s}_i$ is contained in the centralizer of $v_j$.

      \item For any $i\neq j$, the vectors $v_i$ and $v_j$ form a friendly pair.

    \end{itemize}
    (See Lemma~\ref{le:13}). We will show that property \hyperlink{o.0}{(P)} plays an essential role in the proof of \hyperlink{o.9}{\(\mathcal{A}\)}.



\end{enumerate}
Below are two typical examples. Throughout this paper, it is recommended to keep these examples in mind for the sake of clarity and transparency.
\begin{example}\label{ex:1} Let $\GG=SL(4,\RR)\times SL(4,\RR)$ and set $u=\begin{pmatrix}
0 & 1 & 0 &0\\
0 & 0 & 0 &0\\
0 & 0 & 0 &0\\
0 & 0 & 0 &0
\end{pmatrix}\times	0$. Then:

\begin{align*}
 G_u&=\begin{pmatrix}
a & b & 0 &0\\
c & d & 0 &0\\
0 & 0 & 0 &0\\
0 & 0 & 0 &0
\end{pmatrix}\times I:\,ad-bc=1;\\
G_u'&=\begin{pmatrix}
a & b & 0 &0\\
0 & a^{-1} & 0 &0\\
0 & 0 & 0 &0\\
0 & 0 & 0 &0
\end{pmatrix}\times I:\,a>0, \,b\in\RR;\\
C(G_u)&=\begin{pmatrix}
a & 0 & 0 &0\\
0 & a & 0 &0\\
0 & 0 & c &d\\
0 & 0 & e &g
\end{pmatrix}\times SL(4,\RR):\, a^2(cg-de)=1;\\
S_{1,u}&=\begin{pmatrix}
a & h & 0 &0\\
0 & b & 0 &0\\
0 & 0 & c &d\\
0 & 0 & e &g
\end{pmatrix}\times SL(4,\RR):\,ab(cg-de)=1.
\end{align*}
Let $v_1=0\times \begin{pmatrix}
0 & 1 & 0 &0\\
0 & 0 & 0 &0\\
0 & 0 & 0 &0\\
0 & 0 & 0 &0
\end{pmatrix}$, and let $A\subseteq \GG$ be a closed abelian subgroup of $\GG$ such that $u,\,v_1\in \text{Lie}(A)$. We claim that $\text{Lie}(A)$ has property \hyperlink{o.0}{(P)}. To verify that
$\text{Lie}(A)$ satisfies property \hyperlink{o.0}{(P)} in this scenario (which we refer to as ``Property (P) in case I"), we make the following key observations:

\smallskip
\noindent\hypertarget{o.20}{Property (P) in case I}:
\begin{enumerate}
  \item\label{for:353} $v_1$ lies in the super tame subspace of $u$. Then $u$ and $v_1$ are a friendly pair.

  \smallskip
  \item\label{for:354} The complementary directions to $u$ are contained in the three abelian nilpotent subalgebras:
\begin{align*}
 \mathfrak{V}&=\begin{pmatrix}
0 & 0 & a &b\\
0 & 0 & c &d\\
0 & 0 & 0 &0\\
0 & 0 & 0 &0
\end{pmatrix}\times 0:\,a,\,b,\,c,\,d\in\RR; \\
\mathfrak{U}&=\begin{pmatrix}
0 & 0 & 0 &0\\
0 & 0 & 0 &0\\
a & b & 0 &0\\
c & d & 0 &0
\end{pmatrix}\times 0:\,a,\,b,\,c,\,d\in\RR;\\
\mathfrak{W}&=\begin{pmatrix}
0 & 0 & 0 &0\\
a & 0 & 0 &0\\
0 & 0 & 0 &0\\
0 & 0 & 0 &0
\end{pmatrix}\times 0:\,a\in\RR.
\end{align*}
\item\label{for:358} All three subalgebras, $\mathfrak{V}, \mathfrak{U}, \mathfrak{W}$ are contained in the centralizer of
$v_1. $

\end{enumerate}
\hyperlink{o.20}{Property (P) in case I} implies \hyperlink{o.0}{(P)} holds for $\text{Lie}(A)$, with this choice of the pair $(u,v_1)$.



\end{example}
\begin{example}\label{ex:2} Let $\GG=SL(6,\RR)$ and set $u=\begin{pmatrix}\begin{matrix}
0 & 1 \\
0 & 0
\end{matrix} & \vline & 0\\
\hline 0 &\vline &0 \end{pmatrix}$. Then
\begin{align*}
 G_u&=\begin{pmatrix}SL(2,\RR) & \vline & 0\\
\hline 0 &\vline & I_4 \end{pmatrix};\\
G_u'&=\begin{pmatrix}\begin{matrix}
a & b \\
0 & a^{-1}
\end{matrix} & \vline & 0\\
\hline 0 &\vline & I_4 \end{pmatrix}:\,a>0, \,b\in\RR;\\
C(G_u)&=\begin{pmatrix}\begin{matrix}
a & 0 \\
0 & b
\end{matrix} & \vline & 0\\
\hline 0 &\vline & GL(4,\RR) \end{pmatrix}\subseteq SL(6,\RR);\\
 S_{1,u}&=\begin{pmatrix}\begin{matrix}
a & c \\
0 & b
\end{matrix} & \vline & 0\\
\hline 0 &\vline & GL(4,\RR) \end{pmatrix}\subseteq SL(6,\RR).
\end{align*}
Let $v_1=\begin{pmatrix}0 & \vline & 0& \vline &0\\
\hline 0 &\vline & \begin{matrix}
0 & 1 \\
0 & 0
\end{matrix} &\vline&0\\
\hline 0 &\vline & 0&\vline &0
\end{pmatrix}$,  $v_2=\begin{pmatrix}0 & \vline & 0& \vline &0\\
\hline 0 &\vline & 0 &\vline&0\\
\hline 0 &\vline & 0&\vline &\begin{matrix}
0 & 1 \\
0 & 0
\end{matrix}
\end{pmatrix}$ and let $A\subseteq \GG$ be a closed abelian subgroup of $\GG$ such that $u,\,v_1,\,v_2\in \text{Lie}(A)$. We claim that $\text{Lie}(A)$ has property \hyperlink{o.0}{(P)}. To verify that
$\text{Lie}(A)$ satisfies property \hyperlink{o.0}{(P)} in this scenario (which we refer to as ``Property (P) in case II"), we make the following key observations:

\smallskip
\noindent\hypertarget{o.21}{Property (P) in case II}:
\begin{enumerate}
  \item\label{for:355} $v_1$ and $v_2$ both lie in the super tame subspace of $u$, then both ($u$, $v_1$) and ($u$, $v_2$) form friendly pairs.
  \item\label{for:356} The complementary directions to $u$ are contained in the five abelian nilpotent subalgebras:
\begin{align*}
 \mathfrak{V}_1&=\begin{pmatrix}0 & \vline & 0& \vline &\begin{matrix}
a & b \\
c & d
\end{matrix}\\
\hline 0 &\vline & 0 &\vline&0\\
\hline 0 &\vline & 0&\vline &0
\end{pmatrix}:\,a,\,b,\,c,\,d\in\RR; \\
\mathfrak{V}_2&=\begin{pmatrix}0 & \vline & \begin{matrix}
a & b \\
c & d
\end{matrix}& \vline &0\\
\hline 0 &\vline & 0 &\vline&0\\
\hline 0 &\vline & 0&\vline &0
\end{pmatrix}:\,a,\,b,\,c,\,d\in\RR;
\end{align*}
\begin{align*}
\mathfrak{U}_1&=\begin{pmatrix}0 & \vline & 0& \vline &0\\
\hline 0 &\vline & 0 &\vline&0\\
\hline \begin{matrix}
a & b \\
c & d
\end{matrix} &\vline & 0&\vline &0
\end{pmatrix}:\,a,\,b,\,c,\,d\in\RR; \\
\mathfrak{U}_2&=\begin{pmatrix}0 & \vline & 0& \vline &0\\
\hline \begin{matrix}
a & b \\
c & d
\end{matrix} &\vline & 0 &\vline&0\\
\hline 0 &\vline & 0&\vline &0
\end{pmatrix}:\,a,\,b,\,c,\,d\in\RR; \\
 \mathfrak{W}&=\begin{pmatrix}\begin{matrix}
0 & 0 \\
a & 0
\end{matrix} & \vline & 0\\
\hline 0 &\vline & I_4 \end{pmatrix}:\,a\in\RR.
\end{align*}
\item\label{for:359} Moreover,
\begin{gather*}
\mathfrak{V}_1, \mathfrak{U}_1, \mathfrak{W}
    \subseteq\;
    \text{(centralizer of }v_1), \\ \mathfrak{V}_2, \mathfrak{U}_2, \mathfrak{W}
    \subseteq\;
    \text{(centralizer of }v_2).
\end{gather*}
  \item\label{for:357} $v_1$ and $v_2$ lie in each other's  super tame subspace, then $v_1$ and $v_2$ is a  friendly pair.
\end{enumerate}
\hyperlink{o.21}{Property (P) in case II} implies that
\hyperlink{o.0}{(P)} holds for $\text{Lie}(A)$, with the choice of the triple $u,\,v_1,\,v_2$.

\end{example}

\medskip
The construction of
$\eta$ uses representation theory of $G_u$. As $G_u$ is isomorphic to $SL(2,\RR)$, we bypass the complexity of higher rank representation theory. However, Sobolev estimates of $\eta$ are not tame
even on $G_u$ (see the beginning of Section \ref{sec:34}). Luckily, by using the normalizer trick  (see Lemma \ref{le:5}) we show that $\eta$ is tame on $G_u'$ (recall \eqref{for:274} of Section \ref{sec:44}).
Further, the centralizer trick (see Lemma \ref{le:10}) allows us to extend the tameness of $\eta$ to $C(G_u)$. In particular,
$\eta$ satisfies the following properties:
\begin{enumerate}
  \item\label{for:244} $\eta$ is partially smooth and tame on $S_{1,u}$. This is the reason we call vectors in $\text{Lie}(S_{1,u})$ the tame directions of $u$;

  \smallskip
  \item\label{for:245} all possible non-smooth directions of
 $\eta$ are inside $(CS)_u$,  the complementary subspace of $u$.
\end{enumerate}
We will instead construct a new approximate solution $\Theta$ from $\eta$, such that $\Theta$
is smooth along the complementary directions of $u$ while
preserving the  smoothness of $\eta$ on $S_{1,u}$. As a result, $\Theta$ is a $C^\infty$ approximation.
Before we present the construction of $\Theta$, we introduce two important tools that will be used for this purpose.

\subsubsection{Higher rank trick}\label{sec:46} In this part we summarize the results of Section \ref{sec:14}. Recall properties \eqref{for:244} and \eqref{for:245} for $\eta$.  We will use the higher rank trick
to prove the following result:
    \begin{proposition}\label{po:6}
Suppose $v\in\text{Lie}(A)$ is nilpotent and we have an almost cocycle equation
\begin{align*}
 u\mathfrak{p}_v+\mathfrak{p}_v-(v\mathfrak{p}_u+\mathfrak{p}_u)=\mathfrak{w}_{u,v}.
\end{align*}
If $u$ and $v$ are a friendly pair, then we can write
\begin{align*}
 \mathfrak{p}_{u}=u\eta+\mathcal{E}_{u}\quad\text{and}\quad \mathfrak{p}_{v}=v\eta+\mathcal{E}_{v},
\end{align*}
where Sobolev estimates of both $\mathcal{E}_{u}$ and $\mathcal{E}_{v}$ are comparable to those of $\mathfrak{w}_{u,v}$.
\end{proposition}
The precise statement is given in Proposition \ref{po:1}.  We note that tame Sobolev norm estimates are obtained only in certain directions. In a typical KAM iteration,  the term $\mathfrak{w}_{u,v}$ is \emph{quadratically small} relative to $\mathfrak{p}_{v}$ and $\mathfrak{p}_{u}$. Proposition \ref{po:6} shows that $\eta$ is also an approximate solution for the $v$-(twisted) almost
coboundary with a \emph{nice error} (i.e., the error is quadratically small). This observation then allows us to construct $\Theta$ from $\eta$ by considering a different (twisted) almost cocycle.

\subsubsection{Directional smoothing operator}\label{sec:42} We show a general construction of  smoothing operators in Section \ref{sec:4}. Let $S$ be an abelian closed unipotent subgroup of $\GG$. Fix a set of basis $\mathfrak{u}=\{\mathfrak{u}_1,\cdots,\mathfrak{u}_m\}$ of $\text{Lie}(S)$. For any subset $X\subseteq \RR^m$, $I_X$ denotes the characteristic function of $X$. We can define a linear map $\pi_{\mathfrak{u}}$ from $L^\infty(\RR^m)$ to the set of bounded linear operators on $L^2(\GG/\Gamma)$ such that
the assignment $X\to \pi_{\mathfrak{u}}(I_X)$ is a \emph{projection-value measure}, where $X\subseteq \RR^m$ is a Borel set. The precise definition is given in
Section \ref{sec:3}. We denote $f(\text{$\frac{t}{a}$})$ by $(f\circ a^{-1})(t)$, $a>0$. It is harmless to think $S$ is $\RR^m$ and $\pi_{\mathfrak{u}}(f\circ a^{-1})$ is the truncation using Fourier transform (see Section \ref{sec:30}).

Here are some key properties of $\pi_{\mathfrak{u}}(f\circ a^{-1})$ that will be used later:
\begin{enumerate}

  \item\label{for:231} (\eqref{ob:3} of Lemma \ref{le:2}) if $v\in \text{Lie}(\GG)$ commutes with $\text{Lie}(S)$, then
  \begin{align*}
   v\pi_{\mathfrak{u}}(f)=\pi_{\mathfrak{u}}(f) v;
  \end{align*}

\item\label{for:233} (\eqref{for:63} of Lemma \ref{le:7}) if $f$ is Schwartz and $\xi$ is an $L^2$ function, then $\pi_{\mathfrak{u}}(f\circ a^{-1})\xi$ is an $S$-smooth funcion. This is why $\pi_{\mathfrak{u}}(f\circ a^{-1})$ is called a directional smoothing operator (along $S$);

\smallskip

\item\label{for:232} (Corollary \ref{cor:1}) suppose $f$ is Schwartz.  then
\begin{align*}
 \pi_{\mathfrak{u}}(f\circ a^{-1})\big(C^\infty(\GG/\Gamma)\big)\subseteq C^\infty(\GG/\Gamma).
\end{align*}
Further, if $\xi\in C^\infty(\GG/\Gamma)$ has nice Sobolev estimates, then both
\begin{align*}
 \pi_{\mathfrak{u}}(f\circ a^{-1})\xi\quad\text{and the error}\quad \xi-\pi_{\mathfrak{u}}(f\circ a^{-1})\xi
\end{align*}
have nice Sobolev estimates. This means applying the directional smoothing operators to nice $C^\infty$ functions will not hurt the convergence in the KAM iteration;



\smallskip
\item\label{for:229} (Lemma \ref{cor:3}) suppose $Q$ is a subgroup of $\GG$ and $H=Q\ltimes S$.
If an $L^2$ function $\xi$ is $Q$-smooth, then $\pi_{\mathfrak{u}}(f\circ a^{-1})\xi$ is $H$-smooth. Further, if the
Sobolev estimates of $\xi$ are nice on $Q$, then
Sobolev estimates of $\pi_{\mathfrak{u}}(f\circ a^{-1})\xi$ are nice on $H$.

Simply speaking, the operator $\pi_{\mathfrak{u}}(f\circ a^{-1})$ has the following good properties:

\smallskip
\begin{enumerate}
      \item [(i)] it provides $H$-smoothness for those vectors only losing smoothness along $S$;

          \smallskip
      \item [(ii)] it will not hurt the KAM iteration if applied to vectors with nice Sobolev estimates on $Q$.
    \end{enumerate}



\end{enumerate}
\subsubsection{Construction of $\Theta$, part I}\label{sec:45} Here we summarize the ideas behind the proof of \hyperlink{o.9}{$\mathcal{A}$} coming from Section \ref{sec:19}.
Firstly, we consider $u=\textbf{u}$ and $v_1=\textbf{v}$ where $\textbf{u}$ and $\textbf{v}$ come from $\text{Lie}(A)$  in Theorem \ref{th:13}.
We will show that the directional smoothing operators collaborate effectively with
the higher rank trick in constructing $\Theta$.

To facilitate a better understanding of the proof, it is harmless for readers to assume that $\GG$, $u$ and $v_1$ are as described in Example \ref{ex:1}.  We recall that $\eta$ is an approximate solution for the $\textbf{u}$-almost coboundary $\mathfrak{p}_\textbf{u}$ (see \eqref{for:282} where $u=\textbf{u}$).
Since $\textbf{u}$ and $\textbf{v}$ are a friendly pair  (see \eqref{for:353} of \hyperlink{o.20}{Property (P) in case I}), we can invoke the higher rank trick (see Proposition \ref{po:6}) to conclude that
$\eta$ is also an approximate solution for the $\textbf{v}$-almost coboundary $\mathfrak{p}_\textbf{v}$ with the error $\mathcal{E}_{\textbf{v}}$:
\begin{align}\label{for:230}
 \textbf{v}\eta=\mathfrak{p}_{\textbf{v}}-\mathcal{E}_{\textbf{v}},
\end{align}
where $\mathcal{E}_{\textbf{v}}$ has nice Sobolev estimates.

We recall that $\eta$ is partially smooth and tame on $S_{1,\textbf{u}}$. However, it may lose smoothness along all complementary directions to $\textbf{u}$ (see \eqref{for:244} and \eqref{for:245} of Section \ref{sec:44}).

To construct $\Theta$, we first recall  notations  in Example \ref{ex:1}.
$\mathfrak{U}$, $\mathfrak{V}$, and $\mathfrak{W}$ determine three directional smoothing operators $\pi_{\mathfrak{U}}(f_1\circ a^{-1})$,
$\pi_{\mathfrak{V}}(f_2\circ a^{-1})$ and $\pi_{\mathfrak{W}}(f_3\circ a^{-1})$. For simplicity, we denote them by
$\pi_{\mathfrak{U}}$, $\pi_{\mathfrak{V}}$ and $\pi_{\mathfrak{W}}$. We let
\begin{align*}
 \Theta=\pi_{\mathfrak{U}}\pi_{\mathfrak{W}}\pi_{\mathfrak{V}}\eta.
\end{align*}
Using  property \eqref{for:229} of Section \ref{sec:42}, we see that $\Theta$ is a $C^\infty$ function. This is because $\mathfrak{U}$, $\mathfrak{V}$, and $\mathfrak{W}$ span the complementary subspace $(CS)_u$ (see \eqref{for:354} of \hyperlink{o.20}{Property (P) in case I}), which consists of all possible non-smooth directions of $\eta$.
Moreover, $\Theta$ has nice Sobolev estimates (the detailed proof is given in \eqref{for:26} of Corollary \ref{cor:9}).

\smallskip
\emph{Important Remark on Ordering}:

A particular order of applying  $\pi_{\mathfrak{U}}$, $\pi_{\mathfrak{W}}$ and $\pi_{\mathfrak{V}}$ operators is required to ensure $\Theta$ is \emph{globally smooth}  (see \ref{sect:1} of Section \ref{sec:52} for a detailed explanation)

\smallskip
 We now have a new approximation for the $\textbf{v}$-almost coboundary $\mathfrak{p}_\textbf{v}$:
\begin{align}
 \textbf{v}\Theta=\mathfrak{p}_{\textbf{v}}-\mathcal{R}_{\textbf{v}}.
\end{align}
Finally, let us see why the new error $\mathcal{R}_{\textbf{v}}$ has nice Sobolev estimates.

\smallskip

\emph{Note}. If we let $\Theta'=\mathfrak{s}_b\eta$, where $\mathfrak{s}_b$ is the standard smoothing operator (see Section \ref{sec:26}),
then $\Theta'$ is also a $C^\infty$ function with nice Sobolev estimates. However, the key challenge is ensuring that the error term in
\begin{align*}
  \textbf{v}\Theta'=\mathfrak{p}_{\textbf{v}}-\mathcal{R}_{\textbf{v}}'
\end{align*}
has nice Sobolev estimates. In other words, controlling the Sobolev estimates of $\mathcal{R}_{\textbf{v}}'$ to the desired level is non-trivial. Therefore, we employ directional smoothing operators. These operators are designed to commute with the corresponding vector fields and they smooth only in directions where $\eta$ lacks regularity, leaving the already smooth directions untouched. This essentially ensures that both the approximate solution $\Theta$ and the error $\mathcal{R}_{\textbf{v}}$
have nice Sobolev estimates.

\smallskip

We apply the operator
\begin{align}\label{for:289}
 \mathcal{P}=\pi_{\mathfrak{U}}\pi_{\mathfrak{W}}\pi_{\mathfrak{V}}
\end{align}
to both sides of equation \eqref{for:230}:
\begin{align*}
 &\pi_{\mathfrak{U}}\pi_{\mathfrak{W}}\pi_{\mathfrak{V}}(\textbf{v}\eta)=\mathcal{P}\mathfrak{p}_{\textbf{v}}-\mathcal{P}\mathcal{E}_{\textbf{v}}.
\end{align*}
Since $\mathfrak{U}$, $\mathfrak{V}$, and $\mathfrak{W}$ are inside the centralizer of $\textbf{v}=v_1$ (see \eqref{for:358} of \hyperlink{o.20}{Property (P) in case I}), from property \eqref{for:231} of
Section \ref{sec:42}, we see that
\begin{align*}
 &\pi_{\mathfrak{U}}\pi_{\mathfrak{W}}\pi_{\mathfrak{V}}\textbf{v}
 =\textbf{v}\pi_{\mathfrak{U}}\pi_{\mathfrak{W}}\pi_{\mathfrak{V}}.
\end{align*}
The detailed proof is given in Lemma \ref{for:106}. Then we have
\begin{align*}
 &\textbf{v}\Theta=\mathfrak{p}_{\textbf{v}}
 -\underbrace{\big((\mathfrak{p}_{\textbf{v}}-\mathcal{P}\mathfrak{p}_{\textbf{v}})+\mathcal{P}\mathcal{E}_{\textbf{v}}\big)}_{\mathcal{R}_{\textbf{v}}}.
\end{align*}
From \eqref{for:232} of Section \ref{sec:42}, we see that $\mathfrak{p}_{\textbf{v}}-\mathcal{P}\mathfrak{p}_{\textbf{v}}$ has nice Sobolev estimates.
 As $\mathcal{E}_{\textbf{v}}$ is nice on $S_{1,\textbf{u}}$, \eqref{for:229} of Section \ref{sec:42} shows that $\mathcal{P}\mathcal{E}_{\textbf{v}}$
 is also nice. As a result, $\mathcal{R}_{\textbf{v}}$ has nice Sobolev estimates.


Once $\Theta$ is constructed for the $\textbf{v}$-almost coboundary $\mathfrak{p}_{\textbf{v}}$, the usual higher rank trick implies that it is, in fact,  an approximate solution for all almost coboundaries whose errors satisfy similarly ``nice" Sobolev estimates. Thus,  we have completed the proof of \hyperlink{o.9}{$\mathcal{A}$}.





\subsubsection{Construction of $\Theta$, part II}\label{sec:48} Here, we provide a summary of the proof in Section \ref{sec:20}.
Recall that $\mathfrak{u}_{\phi_i}$, $1\leq i\leq 3$ are inside $\text{Lie}(A)$ (see Theorem \ref{th:11}). Choose $0\neq\textbf{u}_i\in \mathfrak{u}_{\phi_i}$, $1\leq i\leq 3$.  Let $u=\textbf{u}_3$ (it is harmless to assume that $\textbf{u}_1=v_1$ and $\textbf{u}_2=v_2$ where $v_1$, $v_2$ and $u$ are as described in Example \ref{ex:2}).

We recall that $\eta$ is an approximate solution for the $\textbf{u}_3$-almost coboundary $\mathfrak{p}_{\textbf{u}_3}$ (see \eqref{for:282} where $u=\textbf{u}_3$). Since $\textbf{u}_3$ and $\textbf{u}_1$ is a friendly pari (see \eqref{for:355} of \hyperlink{o.21}{Property (P) in case II}), we can invoke the higher rank trick (see Proposition \ref{po:6}) to conclude that  $\eta$ is also an approximate solution for the $\textbf{u}_1$-almost coboundary $\mathfrak{p}_{\textbf{u}_1}$ with the error $\mathcal{E}_{\textbf{u}_1}$:
\begin{align}\label{for:234}
 \textbf{u}_1\eta=\mathfrak{p}_{\textbf{u}_1}-\mathcal{E}_{\textbf{u}_1},
\end{align}
where $\mathcal{E}_{\textbf{u}_1}$ has nice  Sobolev estimates on a subgroup of $S_{1,\textbf{u}_3}$.

We recall that $\eta$ is partially smooth and tame on $S_{1,\textbf{u}_3}$. To construct $\Theta$, we also recall notations in Example \ref{ex:2}. $\mathfrak{U}_1$, $\mathfrak{U}_2$, $\mathfrak{V}_1$, $\mathfrak{V}_2$ and $\mathfrak{W}$ determine five directional smoothing operators $\pi_{\mathfrak{U}_i}$,
$\pi_{\mathfrak{V}_i}$, $i=1,2$ and $\pi_{\mathfrak{W}}$. We let
\begin{align*}
 \Theta=\pi_{\mathfrak{U}^1}\pi_{\mathfrak{U}^2}\pi_{\mathfrak{W}}\pi_{\mathfrak{V}^2}\pi_{\mathfrak{V}^1}\eta.
\end{align*}
Using property \eqref{for:229} of Section \ref{sec:42}, we see that $\Theta$ is a $C^\infty$ function. This is because $\mathfrak{U}_1$, $\mathfrak{U}_2$, $\mathfrak{V}_1$, $\mathfrak{V}_2$ and $\mathfrak{W}$ span the complementary subspace $(CS)_u$ (see \eqref{for:356} of \hyperlink{o.21}{Property (P) in case II}), which consists of all possible non-smooth directions of $\eta$. Moreover, $\Theta$ has nice Sobolev estimates (the detailed proof is in \eqref{for:28} of Corollary \ref{cor:9}).

\smallskip
\emph{Important Remark on Ordering}:

A particular order of applying  $\pi_{\mathfrak{U}^i}$, $\pi_{\mathfrak{W}}$ and $\pi_{\mathfrak{V}^i}$ operators, $i=1,2$ is required to ensure $\Theta$ is \emph{globally smooth}  (see \ref{sect:1} of Section \ref{sec:52} for a detailed explanation). Consequently, we have to alternate from
$\textbf{u}_1$-almost coboundaries
  to
$\textbf{u}_2$-almost coboundaries and then back again, each time invoking the higher rank trick to control the resulting error (Note that $\textbf{u}_1$ and $\textbf{u}_2$ is a friendly pair, see \eqref{for:357} of \hyperlink{o.21}{Property (P) in case II}). This interlacing of the two almost coboundaries ensures that, after each application of  the directional smoothing operators, the new error terms introduced  still have nice Sobolev estimates for the next step. Ultimately, this yields a globally smooth approximate solution $\Theta$ whose error also satisfies nice Sobolev estimates.

\smallskip
After constructing $\Theta$, we obtain a new approximation:
\begin{align}\label{for:2021}
\textbf{u}_1\Theta=\mathfrak{p}_{\textbf{u}_1}-\mathcal{R}_{\textbf{u}_1}.
\end{align}
The remaining  discussion focuses on showing that the new error $\mathcal{R}_{\textbf{u}_1}$ has nice Sobolev estimates. A straightforward approach is to apply the operator
\begin{align*}
 \mathcal{P}=\pi_{\mathfrak{U}^1}\pi_{\mathfrak{U}^2}\pi_{\mathfrak{W}}\pi_{\mathfrak{V}^2}\pi_{\mathfrak{V}^1}
\end{align*}
to both sides of equation \eqref{for:2021}, as we did in the previous case. However, recalling \eqref{for:359} of \hyperlink{o.21}{Property (P) in case II} we observe  a key difference from the previous case: the complementary directions to  $\textbf{u}_3$ are not contained exclusively within the centralizer of
$\textbf{u}_1$ or that of
$\textbf{u}_2$.  Instead, they lie in the union of these two centralizers.

Since the complementary directions to $\textbf{u}_3$ do not fall entirely within a single centralizer, we have to employ the higher rank trick and apply the directional smoothing operators multiple times to get the desired
almost coboundary, which is almost solved by $\Theta$ with nice Sobolev estimates.

Firstly, we apply $\pi_{\mathfrak{V}^1}$ to each side of \eqref{for:234}. Since $\mathfrak{V}^1$ is  inside the centralizer of $\textbf{u}_1=v_1$ (see \eqref{for:359} of \hyperlink{o.21}{Property (P) in case II}), we have $\pi_{\mathfrak{V}^1}\textbf{u}_1=\textbf{u}_1\pi_{\mathfrak{V}^1}$. Thus we have
\begin{align}\label{for:235}
 \textbf{u}_1(\pi_{\mathfrak{V}^1}\eta)&=\pi_{\mathfrak{V}^1}\mathfrak{p}_{\textbf{u}_1}-\pi_{\mathfrak{V}^1}\mathcal{E}_{\textbf{u}_1}\notag\\
 &=\mathfrak{p}_{\textbf{u}_1}
 -\underbrace{\big((\mathfrak{p}_{\textbf{u}_1}-\pi_{\mathfrak{V}^1}\mathfrak{p}_{\textbf{u}_1})+\pi_{\mathfrak{V}^1}\mathcal{E}_{\textbf{u}_1}\big)}_{\mathfrak{R}_{\textbf{u}_1,1}}.
\end{align}
From \eqref{for:232} and \eqref{for:229} of Section \ref{sec:42}, we see that $\mathfrak{R}_{\textbf{u}_1,1}$ has nice Sobolev estimates.

\eqref{for:235} means $\pi_{\mathfrak{V}^1}\eta$ solves the $\textbf{u}_1$-(twisted) almost coboundary $\mathfrak{p}_{\textbf{u}_1}$ approximately with nice Sobolev estimates.
Since $\textbf{u}_2=v_2$ lies in the super tame subspace of $\textbf{u}_1=v_1$ (see \eqref{for:357} of \hyperlink{o.21}{Property (P) in case II}), we can invoke the higher rank trick
to show that $\pi_{\mathfrak{V}^1}\eta$ also solves the $\textbf{u}_2$-almost coboundary $\mathfrak{p}_{\textbf{u}_2}$ approximately with nice Sobolev estimates. Namely, we have
\begin{align}\label{for:236}
 \textbf{u}_2(\pi_{\mathfrak{V}^1}\eta)=\mathfrak{p}_{\textbf{u}_2}-\mathfrak{R}_{\textbf{u}_2,1}
\end{align}
where $\mathfrak{R}_{\textbf{u}_2,1}$ has nice Sobolev estimates.

Secondly, we apply $\pi_{\mathfrak{U}^2}\pi_{\mathfrak{W}}\pi_{\mathfrak{V}^2}$ to each side of \eqref{for:236}.
Since $\mathfrak{V}_2$, $\mathfrak{U}_2$ and $\mathfrak{W}$ are inside the centralizer of $\textbf{u}_2=v_2$ (see \eqref{for:359} of \hyperlink{o.21}{Property (P) in case II}), we have
\begin{align*}
 \pi_{\mathfrak{U}^2}\pi_{\mathfrak{W}}\pi_{\mathfrak{V}^2}\textbf{u}_2=\textbf{u}_2\pi_{\mathfrak{U}^2}\pi_{\mathfrak{W}}\pi_{\mathfrak{V}^2}.
\end{align*}
Thus we have
\begin{align}\label{for:237}
 &\textbf{u}_2(\pi_{\mathfrak{U}^2}\pi_{\mathfrak{W}}\pi_{\mathfrak{V}^2}\pi_{\mathfrak{V}^1}\eta)\notag\\
 &=\pi_{\mathfrak{U}^2}\pi_{\mathfrak{W}}\pi_{\mathfrak{V}^2}\mathfrak{p}_{\textbf{u}_2}-\pi_{\mathfrak{U}^2}\pi_{\mathfrak{W}}\pi_{\mathfrak{V}^2}\mathfrak{R}_{\textbf{u}_2,1}\notag\\
 &=\mathfrak{p}_{\textbf{u}_2}-\underbrace{\big((\mathfrak{p}_{\textbf{u}_2}-\pi_{\mathfrak{U}^2}\pi_{\mathfrak{W}}\pi_{\mathfrak{V}^2}\mathfrak{p}_{\textbf{u}_2})
 +\pi_{\mathfrak{U}^2}\pi_{\mathfrak{W}}\pi_{\mathfrak{V}^2}\mathfrak{R}_{\textbf{u}_2,1}\big)}_{\mathfrak{R}_{\textbf{u}_2,2}}.
\end{align}
Again by \eqref{for:232} and \eqref{for:229} of Section \ref{sec:42}, we see that  $\mathfrak{R}_{\textbf{u}_2,2}$ has nice Sobolev estimates.

\eqref{for:237} means $\pi_{\mathfrak{U}^2}\pi_{\mathfrak{W}}\pi_{\mathfrak{V}^2}\pi_{\mathfrak{V}^1}\eta$ solves the $\textbf{u}_2$-almost coboundary $\mathfrak{p}_{\textbf{u}_2}$ approximately with nice Sobolev estimates.
Since $\textbf{u}_1=v_1$ lies in the super tame subspace of $\textbf{u}_2=v_2$ (see \eqref{for:357} of \hyperlink{o.21}{Property (P) in case II}), we can invoke the higher rank trick
to show that $\pi_{\mathfrak{U}^2}\pi_{\mathfrak{W}}\pi_{\mathfrak{V}^2}\pi_{\mathfrak{V}^1}\eta$ also solves the $\textbf{u}_1$-almost coboundary $\mathfrak{p}_{\textbf{u}_1}$ approximately with nice Sobolev estimates. Namely, we have
\begin{align}\label{for:238}
 \textbf{u}_1(\pi_{\mathfrak{U}^2}\pi_{\mathfrak{W}}\pi_{\mathfrak{V}^2}\pi_{\mathfrak{V}^1}\eta)
 =\mathfrak{p}_{\textbf{u}_1}-\mathfrak{R}_{\textbf{u}_1,2}
\end{align}
where $\mathfrak{R}_{\textbf{u}_1,2}$ has nice Sobolev estimates.

Finally, we apply $\pi_{\mathfrak{U}^1}$ to each side of \eqref{for:238}. Since $\mathfrak{U}_1$ is  inside the centralizer of $\textbf{u}_1=v_1$ (see \eqref{for:359} of \hyperlink{o.21}{Property (P) in case II}) we have
$\pi_{\mathfrak{U}^1}\textbf{u}_1=\textbf{u}_1\pi_{\mathfrak{U}^1}$. This implies
\begin{align*}
 \textbf{u}_1\Theta&= \textbf{u}_1(\pi_{\mathfrak{U}^1}\pi_{\mathfrak{U}^2}\pi_{\mathfrak{W}}\pi_{\mathfrak{V}^2}\pi_{\mathfrak{V}^1}\eta)
 =\pi_{\mathfrak{U}^1}\mathfrak{p}_{\textbf{u}_1}-\pi_{\mathfrak{U}^1}\mathfrak{R}_{\textbf{u}_1,2}\notag\\
 &=\mathfrak{p}_{\textbf{u}_1}
 -\underbrace{\big((\mathfrak{p}_{\textbf{u}_1}-\pi_{\mathfrak{U}^1}\mathfrak{p}_{\textbf{u}_1})+\pi_{\mathfrak{U}^1}\mathfrak{R}_{\textbf{u}_1,2}\big)}_{\mathcal{R}_{\textbf{u}_1}}
\end{align*}
where $\mathcal{R}_{\textbf{u}_1}$ has nice Sobolev estimates.

Once $\Theta$ is constructed for the $\textbf{u}_1$-(twisted) coboundary $\mathfrak{p}_{\textbf{u}_1}$, the usual higher rank trick implies that it is, in fact,  an approximate solution for all almost coboundaries whose errors satisfy similarly ``nice" Sobolev estimates. Thus,  we have completed the proof of \hyperlink{o.9}{$\mathcal{A}$}.

\subsection{Convergence in the iteration } In this part, we summarize the results of Section \ref{sec:8}.
Fix a set of generators $E=\{E_1,\cdots,E_d\}$  of $\text{Lie}(A)$.
 We can identify $E_i$ with an element of $\text{Vect}^\infty(\GG/\Gamma)$ such that  $E=\{E_1,\cdots,E_d\}$ generate $\alpha_A$.
 A smooth $A$-perturbation $\tilde{\alpha}_A$ of the action $\alpha_A$ is generated by commuting vector
fields $\tilde{E}=E+\mathfrak{p}=\{E_1+\mathfrak{p}_1,\cdots,E_d+\mathfrak{p}_d\}$, where $\mathfrak{p}_i\in \text{Vect}^\infty(\GG/\Gamma)$, $1\leq i\leq d$.

Now let $h$ be a diffeomorphism of $\GG/\Gamma$, close to the identity. Let $\tilde{E}^{(1)}=E+\mathfrak{p}^{(1)}$ be the commuting vector
fields that generate $h\circ \tilde{\alpha}_A\circ h^{-1}$. To show the convergence of
the iteration, we need the following proposition telling us that by making a good choice of $h$, $\mathfrak{p}^{(1)}$ is much smaller than $\mathfrak{p}$.
\begin{proposition} For any $a,\,b>1$, there is a diffeomorphism $h$ of $\GG/\Gamma$ such that the following estimates hold:
\begin{enumerate}
\item\label{for:249}  for any $s\geq\ell>\varrho$ ($\varrho$ is a constant only dependent on $A$ and $\GG/\Gamma$) we have:
       \begin{align*}
\norm{\mathfrak{p}^{(1)}}&_{C^0}\leq Ca^{2\varrho}b^{2\varrho}\norm{\mathfrak{p}}^2_{C^{\varrho+1}}
+C_{\ell}a^{2\varrho}(\norm{\mathfrak{p}}_{C^{\varrho+1}})^{2(1-\frac{\varrho}{\ell})}(\norm{\mathfrak{p}}_{C^{\ell+\varrho}})^{\frac{\varrho}{\ell}}\notag\\
 &+C_{\ell}a^{2\varrho}(a^{\text{\tiny$-s$}}
 \norm{\mathfrak{p}}_{C^s})^{1-\frac{\varrho}{\ell}}(\norm{\mathfrak{p}}_{C^{\ell+\varrho}})^{\frac{\varrho}{\ell}}\notag\\
 &+C_{\ell}a^{2\varrho}(a^{-s}\norm{\mathfrak{p}}_{C^s})^{(\text{\tiny$1-\frac{\varrho}{\ell}$})^2}
 (\norm{\mathfrak{p}}_{C^{\ell+\varrho}})^{\frac{\varrho}{\ell}(2-\frac{\varrho}{\ell})}\notag\\
 &+C_\ell b^{\text{\tiny$-\ell+1$}}a^{\ell+\varrho}\norm{\mathfrak{p}}_{C^\varrho}+C_\ell b^{\text{\tiny$-\ell+1$}}\norm{\mathfrak{p}}_{C^{\ell+\varrho}};
\end{align*}

 \item\label{for:2}   for any $r\geq\varrho$ we have:
 \begin{align*}
  \norm{\mathfrak{p}^{(1)}}_{C^r}\leq C_{r}(a^{r}b^\varrho \norm{\mathfrak{p}}_{C^{\varrho}}+b^\varrho\norm{\mathfrak{p}}_{C^r}+1).
 \end{align*}
 \end{enumerate}

\end{proposition}
The precise statement is given in Proposition \ref{po:5}. In \eqref{for:249}, $s$ and $\ell+\varrho$ Sobolev orders of $\mathfrak{p}$ are used simultaneously to bound $\norm{\mathfrak{p}^{(1)}}_{C^0}$, which are more complex than classical KAM estimates. In \eqref{for:2} the term $a^{r}b^\varrho \norm{\mathfrak{p}}_{C^{\varrho}}$ is not found in previous KAM works.
  If the term $a^{r}b^\varrho \norm{\mathfrak{p}}_{C^{\varrho}}$ could be ignored, then \eqref{for:2} would be
\begin{align}\label{for:248}
  \norm{\mathfrak{p}^{(1)}}_{C^r}\leq C_r(b^\varrho\norm{\mathfrak{p}}_{C^r}+1),\qquad \forall\,r\geq\varrho.
 \end{align}
If we let $s=\ell+\varrho$ in \eqref{for:249}, then $\norm{\mathfrak{p}^{(1)}}_{C^0}$ can be simplified as (we note that $\varrho>2$):
        \begin{align}\label{for:247}
\norm{\mathfrak{p}^{(1)}}&_{C^0}\leq Ca^{2\varrho}b^{2\varrho}\norm{\mathfrak{p}}^2_{C^{\varrho+1}}
+C_{\ell}a^{2\varrho}(\norm{\mathfrak{p}}_{C^{\varrho+1}})^{2(1-\frac{\varrho}{\ell})}(\norm{\mathfrak{p}}_{C^{\ell+\varrho}})^{\frac{\varrho}{\ell}}\notag\\
 &+C_{\ell}a^{-\ell+4\varrho}\norm{\mathfrak{p}}_{C^{\ell+\varrho}}\notag\\
 &+C_{\ell}a^{-\ell+4\varrho}\norm{\mathfrak{p}}_{C^{\ell+\varrho}}\notag\\
 &+C_\ell b^{\text{\tiny$-\ell+1$}}a^{\ell+\varrho}\norm{\mathfrak{p}}_{C^\varrho}+C_\ell b^{\text{\tiny$-\ell+1$}}\norm{\mathfrak{p}}_{C^{\ell+\varrho}}.
\end{align}
The proof of convergence would be extremely standard for the KAM iteration if one could have \eqref{for:248} and \eqref{for:247} (see \cite{Fayad}).
However, the presence of the term $a^{r}b^\varrho \norm{\mathfrak{p}}_{C^{\varrho}}$ is a direct consequence of our method and cannot be ignored, as it arises from the use of directional smoothing operators.

Since the term $a^{r\mu_1}b^\varrho \norm{\mathfrak{p}}_{C^{\varrho}}$ with $\mu_1=1$ appears in estimating  $\norm{\mathfrak{p}^{(1)}}_{C^r}$ (see \eqref{for:2}) and
the term $a^{-\ell\mu_2+4\varrho}\norm{\mathfrak{p}}_{C^{\ell+\varrho}}$ with $\mu_2=1$ appears in estimating $\norm{\mathfrak{p}^{(1)}}_{C^0}$ (see \eqref{for:247}), the KAM iteration may diverge in $C^r$ topology for large $r$. The general KAM scheme needs $\frac{\mu_2}{\mu_1}>2$ to guarantee the convergence in $C^r$ topology for each $r$ (in fact $\frac{\mu_2}{\mu_1}>1$ is sufficient)  (see \cite{Fayad}).
 In order to get around this problem, we introduce the two-orders trick. We fix a well chosen $b$ and compare the increasing speed of $a^{r}b^\varrho \norm{\mathfrak{p}}_{C^{\varrho}}$  and  $\norm{\mathfrak{p}}_{C^{r}}$ as $r$ increases.  Based on this comparison, we choose
$s$ and $a$ accordingly.   If the latter
grows faster, it means the term $a^{r}b^\varrho \norm{\mathfrak{p}}_{C^{\varrho}}$ is controllable. In this case we let $s=\ell+\varrho$ and $a=b^{\frac{1}{2}}$.
 If the former grows faster, we are forced to choose smaller $a$ (specifically, $a<b^{\frac{1}{2}}$) to reduce the growth rate of $a^{r}b^\varrho \norm{\mathfrak{p}}_{C^{\varrho}}$ so that it is comparable to  $\norm{\mathfrak{p}}_{C^{r}}$. As a result,  $s$ has to be chosen sufficiently large ($s\gg \ell$) to ensure the smallness of $a^{-s}\norm{\mathfrak{p}}_{C^s}$ in \eqref{for:249}.

 Due to the
directional smoothing construction in Section \ref{sec:4}, the constants in
 \eqref{for:249} do not depend
on $s$ (see Remark \ref{re:14}). If they were not, increasing $s$ to make
$a^{-s}\norm{\mathfrak{p}}_{C^s}$ small would simultaneously cause the constants to become large. This could potentially negate any benefit gained from choosing a large
$s$, as the overall estimate for
$\norm{\mathfrak{p}^{(1)}}_{C^0}$ might not decrease. Since the constants are independent of
$s$, we are free to choose
$s$ as large as needed without worrying about inflating the constants in the estimate.
 Thus,  we can show that $\norm{\mathfrak{p}^{(1)}}_{C^0}$ is small in this case, leading to convergence in any $C^r$ topology.

\subsection{Scope of the strategy } This strategy has three advantages. Firstly, its application only needs representation theory of rank one subgroups, which substantially reduced the difficulty of the
quantifying procedure. Secondly, tameness is not a prerequisite for the application, including tameness of the solution of (twisted)-cohomological equation and tameness of the inverse of (twisted)-coboundary operators. Thirdly, the smoothing techniques introduced in the paper can be applied to general Lie groups. Therefore, it allows for treating algebraic actions of considerably greater generality.

Although in the current paper we only consider $\GG_1=SL(n,\RR)$,  a very important representative case for the sake of transparency of ideas and exposition,
the general criteria of the proof can be applied with appropriate modifications to a broader range of actions satisfying property \hyperlink{o.0}{(P)}.

\section{Notations and preparatory steps}\label{sec:38}

\subsection{Notation throughout this paper}\label{sec:47} We will use notations from this section throughout subsequent
sections. So the reader should consult this section if an unfamiliar symbol
appears.

In what follows, $C$ will denote any constant that depends only
  on the given  group $\GG$, the manifolds $\mathcal{X}$ and the action $A$. $C_{x,y,z,\cdots}$ will denote any constant that in addition to the
above depends also on parameters $x, y, z,\cdots$.
\begin{enumerate}
  \item\label{for:214} $\mathbb{G}$ denotes a higher-rank semisimple Lie group with finite center without compact factors satisfying: $\mathbb{G}=\mathbb{G}_1\times \cdots\times \mathbb{G}_k$, where $\mathbb{G}_1=SL(n,\RR)$, $n\geq2$.  $\Gamma$ is a cocompact irreducible lattice of $\mathbb{G}$. For any subgroup $A$ of $\GG$ we use $\alpha_A$ to denote the the action of $A$ by left translations
on $\mathcal{X}=\mathbb{G}/\Gamma$.  Let $\text{Vect}^\infty(\mathcal{X})$ be the space of $C^\infty$ vector fields on $\mathcal{X}$.

\smallskip

\item  We use $\mathfrak{g}$ (resp. $\mathfrak{g}_1$) to denote the Lie algebra of $\GG$ (resp. $\GG_1$). Fix an inner product on $\mathfrak{g}$. Let $\mathfrak{g}^1$ be the set of unit vectors in $\mathfrak{g}$.  We use $\Phi$ (resp. $\Phi_1$) to denote the set of roots of $\mathbb{G}$ (resp. the set of restricted roots of  $\mathbb{G}_1$) and $\mathfrak{u}_\phi$ to denote the root space of $\phi\in \Phi$.

\smallskip
  \item For any subgroups $h_1,h_2,\cdots$ of $\GG$, we use $\{h_1,h_2,\cdots\}$ to denote the subgroup generated by $h_1,h_2,\cdots$.  For any subset $B\subseteq \mathfrak{g}$ we use $\exp(B)$ to denote the connected subgroup of $\GG$ with its Lie algebra generated by $B$.

  \smallskip
  \item\label{for:218} $L^2_0(\GG/\Gamma)$ denotes the subspace of $L^2(\GG/\Gamma)$ orthogonal to constants.
We use $(\pi, \mathcal{O})$ to denote the regular representation of $L^2_0(\GG/\Gamma)$.

\smallskip
 \item\label{for:272} Let $S$ be a Lie group and $(\pi, \mathcal{H})$ be a unitary representation of $S$. Suppose
 $P$ is a subgroup of $S$. We say that $\eta\in \mathcal{H}$ is partially tame on  $P$  (with respect to $\xi\in \mathcal{H}$) if there is $\gamma>0$ such that the $s$ order Sobolev norm of $\eta$ on $P$ can be bounded by the $s+\gamma$ order Sobolev
norm of $\xi$ for any $s\geq0$.

\smallskip

\item We say that $u\in \mathfrak{g}$ is \emph{nilpotent} if $\text{ad}_u$ is
nilpotent. For any nilpotent $u\in \mathfrak{g}$:
\begin{enumerate}
  \item\label{for:2024} there is $u'\in \mathfrak{g}$ such that $\{u, u', X_u=[u,u']\}$ is a $\mathfrak{sl}(2,\RR)$-triple (see Jacobson-Morosov theorem).
  We use $\mathfrak{g}_{u}$ to denote the subalgebra of $\mathfrak{g}$ generated by $\{u, u', X_u=[u,u']\}$;

\smallskip
  \item $G_{u}$ is the connected subgroup in $\GG$ with Lie algebra $\mathfrak{g}_{u}$;

  \smallskip
  \item $G'_{u}$ is the subgroup of $G_{u}$ with Lie algebra spanned by $u$ and $X_u$;

  \item $\mathcal{C}(\mathfrak{g}_{u})$ is the centralizer of $\mathfrak{g}_{u}$ in $\mathfrak{g}$, namely,
\begin{align*}
  \mathcal{C}(\mathfrak{g}_{u})=\{v\in \mathfrak{g}: [v,u]=0, \,[v,u']=0\};
\end{align*}

\item  $C(G_{u})$ denotes  the connected subgroup in $\GG$ whose Lie algebra is $\mathcal{C}(\mathfrak{g}_{u})$;

\smallskip

\item\label{for:273} Define $S_{0,u}=\{G_u,\,C(G_{u})\}$ and $S_{1,u}=\{G_u',\,C(G_{u})\}$. We also write $S_{0}$ and $S_{1}$ if there is no confusion.
\begin{itemize}
   \item $\text{Lie}(S_{1,u})$ is called the \emph{tame} subspace of $u$ and vectors in  $\text{Lie}(S_{1,u})$ are called the \emph{tame} directions to $u$.
   \item Vectors in root spaces of $\mathbb{G}$ that are not in  $\text{Lie}(S_{1,u})$ are called
     the \emph{complementary} directions to $u$. Denote by $(CS)_u$ the subspace of $\mathfrak{g}$ spanned by complementary directions of $u$; this is the \emph{complementary subspace} of $u$.
   \item $\mathcal{C}(\mathfrak{g}_{u})$ is called the \emph{super tame} subspace of $u$.
 \end{itemize}

\smallskip
 \item\label{for:278} we note that for general nilpotent $u$, $u'$ from \eqref{for:2024} is not unique. So if needed, we add $u'$ to lower indices to avoid confusion. For example, we write $\mathfrak{g}_{u,u'}$ to emphasize the dependence of the $\mathfrak{sl}(2,\RR)$ triple on $u'$. Then $G_{u,u'}$, $G'_{u,u'}$ and $C(G_{u,u'})$ etc. are defined accordingly;

\smallskip
\emph{Note.} If $u\in \mathfrak{u}_\phi\cap\mathfrak{g}^1$, $\phi\in \Phi$, then $u'$ is unique if we require $u'\in \mathfrak{u}_{-\phi}$. In this case,
we always choose the unique $u'$ inside $\mathfrak{u}_{-\phi}$.

  \smallskip

  \item \label{de:5} if $0\neq u\in \mathfrak{u}_\phi$, $\phi\in \Phi_1$, for any $v\in\mathfrak{u}_\psi$, where $\psi\in\Phi$, if $v$ is a complementary directions to $u$, then either $[\mathfrak{u}_\phi,\mathfrak{u}_\psi]\neq0$ or
$[\mathfrak{u}_{-\phi},\mathfrak{u}_\psi]\neq0$;

\smallskip

\item\label{for:275} we say that  $0\neq u,\,v\in \text{Lie}(\GG)$ are \emph{a friendly pair} if
\begin{enumerate}
  \item[(a)] there is $\phi\in \Phi$ such that $u\in \mathfrak{u}_\phi$;

  \item[(b)] $v$ is nilpotent and $[u,v]=[u',v]=0$. This means that $v$ commutes with the $\mathfrak{sl}(2,\RR)$ triple: $\{u, u', X_u=[u,u']\}$.
\end{enumerate}

 \end{enumerate}
\smallskip
 \item\label{for:215} Set $\sigma=\text{\small$\frac{3}{2}$}\dim\mathfrak{g}$, $\sigma_0=(8+\sigma)\sigma$  and
  $\sigma_1=2\sigma+3$.

\smallskip

\item\label{for:285} Let $\mathfrak{u}_{i,j}\in \mathfrak{g}_1=\mathfrak{sl}(n,\RR)$ be the elementary $n\times n$ matrix with only one nonzero entry
equal to one, namely, that in the row $i$ and the column $j$. Then $\mathfrak{u}_{i,j}$, $i\neq j$ is in the root space of $L_i-L_j$. Let $U=\mathfrak{u}_{1,2}$.

\smallskip

\item\label{for:287}  Set
\begin{gather*}
 \mathfrak{V}=\{\mathfrak{u}_{1,j},\,\mathfrak{u}_{2,j}:j\geq 3\}, \quad \mathfrak{U}=\{\mathfrak{u}_{j,1},\mathfrak{u}_{j,2}:j\geq3\},\quad \mathfrak{W}=\{\mathfrak{u}_{2,1}\}
 \end{gather*}
and set
\begin{align*}
 \mathfrak{C}^1&=\{v\in \mathfrak{C}: [v,\mathfrak{u}_{3,4}]=0\text{ and }[v,\mathfrak{u}_{4,3}]=0\},\,\,\text{and}\\
 \mathfrak{C}^2&=\{v\in \mathfrak{C}: [v,\mathfrak{u}_{5,6}]=0\text{ and }[v,\mathfrak{u}_{6,5}]=0\}.
 \end{align*}
where $\mathfrak{C}$ stands for $\mathfrak{U}$ or $\mathfrak{V}$.

\smallskip

\item\label{for:288} Let $\mathcal{J}$ denote the subgroup of $\GG_1$ with its Lie algebra generated by $\mathfrak{u}_{3,4}$, $\mathfrak{u}_{4,3}$ and $\mathfrak{u}_{5,6}$, $\mathfrak{u}_{6,5}$.
It is clear that $\mathcal{J}$ is isomorphic to $SL(2,\RR)\times SL(2,\RR)$.

 \smallskip

 \item\label{for:286} Let $A$ be as described as in Theorem \ref{th:13} or Theorem \ref{th:11}.  We fix $E$, a set of generators of $\text{Lie}(A)$.   Since the Weyl group acts simply transitively on Weyl chambers, it is harmless to assume that (we recall $U=\mathfrak{u}_{1,2}$, see \eqref{for:285}):

\smallskip
\begin{enumerate}
  \item\label{for:290} if $\GG\neq \GG_1$, $\{\textbf{u}, \textbf{v}\}\subset E$ (see Theorem \ref{th:13}) and $\textbf{u}=U$.
Set $E_0=\{\textbf{v}\}$;

\smallskip
  \item\label{for:291} if $\GG=\GG_1$, $\phi_1=L_3-L_4$, $\phi_2=L_5-L_6$ and $\phi_3=L_1-L_2$. Set
  $\textbf{u}_1=\mathfrak{u}_{3,4}$, $\textbf{u}_2=\mathfrak{u}_{5,6}$ and $\textbf{u}_3=U$. By Remark \ref{re:7}, it is harmless to assume a basis $E$ of the maximal abelian subalgebra for type $A_n$, $n\geq 3$ is:
\begin{align}\label{for:22}
 E=\{\mathfrak{u}_{i,j}:i\in 2\NN-1,j\in 2\NN\},
\end{align}
and $|E|=\lfloor\text{\small$\frac{(n+1)^2}{4}$}\rfloor$. We set $E_0=\{\textbf{u}_1,\,\textbf{u}_2\}$. It is clear that $\{\textbf{u}_1, \textbf{u}_2,\textbf{u}_3\}\subset E$.
\end{enumerate}

\end{enumerate}

\subsection{Basic examples}\label{sec:50} In this part, before we provide additional examples,  we encourage readers to recall Examples \ref{ex:1} and \ref{ex:2} from Section \ref{sec:49}, which will help them gain a better understanding of the concepts in Section \ref{sec:47}.
\begin{example}\label{ex:4}  Then $G_U$ (resp. $G_U'$) (we recall $U=\mathfrak{u}_{1,2}$, see \eqref{for:285}) is the subgroup with its Lie algebra spanned by $\{\mathfrak{u}_{1,2}, \mathfrak{u}_{2,1},\mathfrak{u}_{1,1}-\mathfrak{u}_{2,2}\}$ (resp. $\{\mathfrak{u}_{1,2}, \mathfrak{u}_{1,1}-\mathfrak{u}_{2,2}\}$). $S_{0,U}$ and $S_{1,U}$ are subgroups of $\GG$ with the following forms:
\begin{gather*}
  S_{0,U}=\begin{pmatrix}M_{2,2} & \vline & 0\\
\hline 0 &\vline &M_{n-2,n-2} \end{pmatrix}\times \mathbb{G}_2\times \cdots\times \mathbb{G}_k, \\
S_{1,U}=\begin{pmatrix}
\begin{matrix}a&b\\
0&c
\end{matrix} &  \vline & 0\\
 \hline  0& \vline & M_{n-2,n-2} \end{pmatrix}\times \mathbb{G}_2\times \cdots\times \mathbb{G}_k,
\end{gather*}
where $M_{m,k}$ denotes the set of $m\times k$ matrices and $a,b,c\in\RR$. More precisely,  if we let
\begin{align*}
 G_0&=\{g=(g_{i,j})\in SL(n,\RR):g_{2,j}=g_{1,j}=g_{j,1}=g_{j,2}=0,\,j\geq3\}; \,\text{and}\\
 G_1&=\{g=(g_{i,j})\in SL(n,\RR):g_{2,1}=g_{2,j}=g_{1,j}=g_{j,1}=g_{j,2}=0,\,j\geq3\},
\end{align*}
then $S_{0,U}=G_0\times \mathbb{G}_2\times \cdots\times \mathbb{G}_k$;
$S_{1,U}=G_1\times \mathbb{G}_2\times \cdots\times \mathbb{G}_k$.
\end{example}
\begin{example}\label{ex:5} $\exp(\mathfrak{V})$, $\exp(\mathfrak{U})$ and $\exp(\mathfrak{W})$ are subgroups of $\GG_1$ with following forms:
\begin{gather*}
  \exp(\mathfrak{V})=\begin{pmatrix}I_{2} & \vline & M_{2,n-2}\\
\hline 0 &\vline &I_{n-2} \end{pmatrix}, \quad\exp(\mathfrak{U})=\begin{pmatrix}I_{2} & \vline & 0\\
\hline M_{n-2,2} &\vline &I_{n-2} \end{pmatrix}\\
\exp(\mathfrak{W})=\begin{pmatrix}\begin{matrix}1&0 \\
d &1
\end{matrix} & \vline & 0\\
\hline 0 &\vline &I_{n-2} \end{pmatrix},
\end{gather*}
where $I_m$ denotes the $m\times m$ identity matrix and $d\in\RR$.

$\mathcal{J}$ is a subgroups of $\GG_1$ with the following form:
\begin{gather*}
 \mathcal{J}=\begin{pmatrix}I_2 & 0 & 0 &\vline & 0\\
 0 & SL(2,\RR) & 0 &\vline & 0\\
0 & 0 &SL(2,\RR) & \vline & 0\\
 \hline 0& 0 & 0 & \vline &I_{n-6} \end{pmatrix}.
\end{gather*}
\end{example}

\subsection{Basic algebraic properties for $U$}\label{sec:33}
The following result illustrates the algebraic properties of $U$ and $E_0$:
\begin{lemma}\label{le:13} (Property \hyperlink{o.0}{(P)}) The following properties hold:
\begin{enumerate}
  \item For any $u\in E_0$, we have $U\subset \mathcal{C}(\mathfrak{g}_{u})$.
  \item The complementary directions of $U$ are covered  by $\bigcup_{u\in E_0}\mathcal{C}(\mathfrak{g}_{u})$.
  \item When $E_0=\{\textbf{u}_1,\,\textbf{u}_2\}$, the elements $\textbf{u}_1$ and $\textbf{u}_2$ lie in each other's super tame subspace.
\end{enumerate}

\end{lemma}
We omit the proof as it is straightforward from the definition.

\subsection{Basic algebraic properties for $\textbf{v}$ and $\textbf{u}$}\label{sec:43} We assume $\GG\neq\GG_1$. We recall notations in \eqref{for:290} of \eqref{for:286} of Section \ref{sec:47}. In this case, Lemma \ref{le:13}  implies the following key algebraic properties for $\textbf{u}=U$ and $\textbf{v}$:

\noindent\hypertarget{o.1}{\emph{Property (P) in case I}}:
\begin{enumerate}
  \item\label{for:280} $\textbf{v}$ lies in the super tame subspace of $\textbf{u}=U$. It is clear that $\textbf{u}$ and $\textbf{v}$ is a friendly pair (see \eqref{for:275} of Section \ref{sec:47})

  \smallskip
  \item\label{for:360} The complementary directions to $\textbf{u}=U$ are contained in the three abelian  nilpotent subalgebras: $\mathfrak{V}, \mathfrak{U}$ and $\mathfrak{W}$.

\smallskip
\item\label{for:362} $\mathfrak{V}, \mathfrak{U}, \mathfrak{W}
    \subseteq\;
    \text{(centralizer of }\textbf{v})$.

\end{enumerate}


\subsection{Basic algebraic properties for $\textbf{u}_i$, $1\leq i\leq 3$}\label{sec:37} We assume $\GG=\GG_1$. We recall notations in \eqref{for:291} of \eqref{for:286} of Section \ref{sec:47}. In this case, from Lemma \ref{le:13}, we see that the key algebraic properties for $\textbf{u}_i$, $1\leq i\leq 3$ are:

\smallskip
\noindent\hypertarget{o.3}{\emph{Property (P) in case II}}:
\begin{enumerate}
  \item\label{for:284} $\textbf{u}_1$ and $\textbf{u}_2$ both lie in the super tame subspace of $\textbf{u}_3=U$. It is clear that $\textbf{u}_1$ and $\textbf{u}_3$ is a friendly pair (see \eqref{for:275} of Section \ref{sec:47}).
  \item\label{for:361} The complementary directions to $\textbf{u}_3=U$ are contained in the five abelian nilpotent subalgebras: $\mathfrak{U}_1$, $\mathfrak{U}_2$, $\mathfrak{V}_1$, $\mathfrak{V}_2$ and $\mathfrak{W}$.

\item\label{for:363} Moreover,
\begin{gather*}
\mathfrak{V}_1, \mathfrak{U}_1, \mathfrak{W}
    \subseteq\;
    \text{(centralizer of }\textbf{u}_1), \\ \mathfrak{V}_2, \mathfrak{U}_2, \mathfrak{W}
    \subseteq\;
    \text{(centralizer of }\textbf{u}_2).
\end{gather*}
  \item\label{for:364} $\textbf{u}_1$ and $\textbf{u}_2$ lie in each other's  super tame subspace. This implies that
  \begin{align*}
   G_{\textbf{u}_2}\subset C(G_{\textbf{u}_1})\quad\text{and}\quad G_{\textbf{u}_1}\subset C(G_{\textbf{u}_2}).
  \end{align*}

\end{enumerate}

\section{Preliminaries on unitary representation theory}\label{sec:15}

\subsection{Sobolev spaces and elliptic regularity theorem}\label{sec:17} Let $\pi$ be a unitary representation of a Lie group $G$ with Lie algebra $\mathfrak{G}$ on a
Hilbert space $\mathcal{H}=\mathcal{H}(\pi)$. Fix an inner product $|\cdot|$ on $\mathfrak{G}=\text{Lie}(G)$. Let $\mathfrak{G}^1$ be the set of unit vectors in $\mathfrak{G}$.
\begin{definition}\label{de;1}
For $k\in\NN$, $\mathcal{H}^k(\pi)$ consists of all $v\in\mathcal{H}(\pi)$ such that the
$\mathcal{H}$-valued function $g\rightarrow \pi(g)v$ is of class $C^k$ ($\mathcal{H}^0=\mathcal{H}$). For $X\in\mathfrak{G}$, $d\pi(X)$ denotes the infinitesimal generator of the
one-parameter group of operators $t\rightarrow \pi(\exp tX)$, which acts on $\mathcal{H}$ as an essentially skew-adjoint operator. For any $v\in\mathcal{H}$, we also write $Xv:=d\pi(X)v$.
\end{definition}
We shall call $\mathcal{H}^k=\mathcal{H}^k(\pi)$ the space of $k$-times differentiable vectors for $\pi$ or the \emph{Sobolev space} of order $k$. The
following basic properties of these spaces can be found, e.g., in \cite{nelson1959analytic} and \cite{goodman1970one}:
\begin{enumerate}
  \item $\mathcal{H}^k=\bigcap_{m\leq k}D(d\pi(Y_{j_1})\cdots d\pi(Y_{j_m}))$, where $\{Y_j\}$ is a basis for $\mathfrak{G}$, and $D(T)$
denotes the domain of an operator on $\mathcal{H}$.

  \item $\mathcal{H}^k$ is a Hilbert space, relative to the inner product
  \begin{align*}
    \langle v_1,\,v_2\rangle_{G,k}:&=\sum_{1\leq m\leq k}\langle Y_{j_1}\cdots Y_{j_m}v_1,\,Y_{j_1}\cdots Y_{j_m}v_2\rangle+\langle v_1,\,v_2\rangle
  \end{align*}
  \item The spaces $\mathcal{H}^k$ coincide with the completion of the
subspace $\mathcal{H}^\infty\subset\mathcal{H}$ of \emph{infinitely differentiable} vectors with respect to the norm
\begin{align*}
    \norm{v}_{G,k}=\bigl\{\norm{v}^2+\sum_{1\leq m\leq k}\norm{Y_{j_1}\cdots Y_{j_m}v}^2\bigl\}^{\frac{1}{2}}.
  \end{align*}
induced by the inner product in $(2)$. The subspace $\mathcal{H}^\infty$
coincides with the intersection of the spaces $\mathcal{H}^k$ for all $k\geq 0$.

\item $\mathcal{H}^{-k}$, defined as the Hilbert space duals of
the spaces $\mathcal{H}^{k}$, are subspaces of the space $\mathcal{E}(\mathcal{H})$ of distributions, defined as the
dual space of $\mathcal{H}^\infty$.
  \end{enumerate}
We write $\norm{v}_{k}:=\norm{v}_{G,k}$ and $ \langle v_1,\,v_2\rangle_{k}:= \langle v_1,\,v_2\rangle_{G,k}$ if there is no confusion. Otherwise,
we use subscripts to emphasize that the regularity is measured with respect to $G$. If we want to consider the restricted representation on a subgroup $S$ of $G$
we use $\mathcal{H}^k_{S}$ to denote the Sobolev space of order $k$ with respect to $S$.

For any $u_1,u_2,\cdots\in \mathcal{H}^k$ set
\begin{align}\label{for:36}
 \norm{u_1,u_2,\cdots}_k=\max\{\norm{u_1}_k,\norm{u_2}_k,\cdots\}
\end{align}
For any set $\mathcal{C}\subset\RR^n$, $\norm{\cdot}_{(C^r,\mathcal{C})}$ stands for $C^r$ norm for functions having continuous derivatives up to order $r$ on $\mathcal{C}$. We also write $\norm{\cdot}_{C^r}$ if there is no confusion.

We list the well-known elliptic regularity theorem which will be frequently
used in this paper (see \cite[Chapter I, Corollary 6.5 and 6.6]{Robinson}):
\begin{theorem}\label{th:4}
Fix a basis $\{Y_j\}$ for $\mathfrak{G}$ and set $L_{2m}=\sum Y_j^{2m}$, $m\in\NN$. Then
\begin{align*}
    \norm{v}_{2m}\leq C_m(\norm{L_{2m}v}+\norm{v}),\qquad \forall\, m\in\NN
\end{align*}
where $C_m$ is a constant only dependent on $m$ and $\{Y_j\}$.
\end{theorem}
Suppose $\Gamma$ is an
irreducible torsion-free cocompact lattice in $G$. Denote by $\mathcal{O}$ the regular representation of $G$ on $\mathcal{H}(\mathcal{O})=L^2(G/\Gamma)$. Then we have the following subelliptic regularity theorem (see \cite{Spatzier2}):
\begin{theorem}\label{th:5}
Fix $\{Y_j\}$ in $\mathfrak{G}$ such that commutators of $Y_j$ of length at most $r$ span $\mathfrak{G}$. Also set $L_{2m}=\sum Y_j^{2m}$, $m\in\NN$. Suppose $f\in\mathcal{H}(\mathcal{O})$. If $L_{2m}f\in \mathcal{H}(\mathcal{O})$ for any $m\in\NN$, then $f\in \mathcal{H}^\infty(\mathcal{O})$ and satisfies
\begin{align}\label{for;1}
\norm{f}_{\frac{2m}{r}-1}\leq
C_m(\norm{L_{2m}f}+\norm{f}),\qquad \forall\, m\in\NN
\end{align}
where $C_m$ is a constant only dependent on $m$ and $\{Y_j\}$.
\end{theorem}

\subsection{Extended representations and linear operators}\label{sec:2} The adjoint representation of $\mathfrak{G}$ is isomorphic to a subset of $\dim(\mathfrak{G})\times \dim(\mathfrak{G})$ matrices. Let $\mathfrak{G}(\mathcal{H})$
denote the set of $(\dim(\mathfrak{G})\times 1)$ matrices with entries from $\mathcal{H}$. Then the adjoint representation of $\mathfrak{G}$ has a natural action on $\mathfrak{G}(\mathcal{H})$. Similarly, any linear map $\mathcal{T}$ on $\mathfrak{G}$  has a natural action on $\mathfrak{G}(\mathcal{H})$.

For any $\xi\in \mathfrak{G}(\mathcal{H})$, we can write $\xi=(\xi_1,\cdots,\xi_{\dim(\mathfrak{g})})$. Then the unitary representation $\pi$ has a natural extension $\bar{\pi}$ on $\mathfrak{G}(\mathcal{H})$ by acting on each coordinate:
\begin{align*}
  \bar{\pi}(\xi)=\big(\pi(\xi_1),\cdots,\pi(\xi_{\dim(\mathfrak{G})})\big).
\end{align*}
Similarly, any linear operator $\mathcal{F}: \mathcal{H}\to \mathcal{H}$ has a natural extension $\bar{\mathcal{F}}$ on $\mathfrak{G}(\mathcal{H})$:
\begin{align*}
 \bar{\mathcal{F}}(\xi)=\big(\mathcal{F}(\xi_1),\cdots,\mathcal{F}(\xi_{\dim(\mathfrak{G})})\big).
\end{align*}
It is clear that for any linear map $\mathcal{T}$ on $\mathfrak{G}$ we have
\begin{align}\label{for:2004}
 \bar{\mathcal{F}}\circ \mathcal{T}=\mathcal{T}\circ \bar{\mathcal{F}}.
\end{align}
We will still write $\pi$
or $\mathcal{F}$ instead of $\bar{\pi}$ or $\bar{\mathcal{F}}$ if there is no confusion. We say that $\xi\in \mathfrak{G}(\mathcal{H})^s$, if $\xi_i\in \mathcal{H}^s$, $1\leq i\leq \dim(\mathfrak{G})$. Set
\begin{align*}
 \norm{\xi}_s=\norm{\xi_1,\cdots,\xi_{\dim(\mathfrak{g})}}_s
\end{align*}
For any subgroup $S$ of $G$, the Hilbert space $\mathfrak{G}(\mathcal{H})^s_S$ and the norm $\norm{\cdot }_{S,s}$ are defined similarly.

\subsection{Direct decompositions of Sobolev space}\label{sec:9}
For any Lie group $G$ of type $I$, there is a decomposition of $\pi$ into a direct integral
\begin{align*}
  \pi=\int_Z\pi_zd\mu(z)
\end{align*}
of irreducible unitary representations for some measure space $(Z,\mu)$ (we refer to
\cite[Chapter 2.3]{Zimmer} or \cite{margulis1991discrete} for more detailed account for the direct integral theory). All the operators in the enveloping algebra are decomposable with respect to the direct integral decomposition. Hence there exists for all $s\in\RR$ an induced direct
decomposition of the Sobolev spaces
\begin{align*}
 \mathcal{H}^s=\int_Z\mathcal{H}_z^sd\mu(z)
\end{align*}
with respect to the measure $d\mu(z)$.

The existence of the direct integral decompositions allows us to reduce our analysis of the
cohomological equation to irreducible unitary representations. This point of view is
essential for our purposes.

\subsection{Useful results} We review several important results which will serve as ready
references later. Suppose $G$ denotes a semisimple Lie group of non-compact type with finite center and
$\Gamma$ is an irreducible lattice of $G$. The following result is quoted from \cite{mieczkowski2007first}, which is derived from \cite{Clozel}, \cite{Kleinbock1} and \cite{Shalom}.
\begin{theorem}\label{th:10} Suppose $G=P_1\times \cdots\times P_k$ where $P_i$, $1\leq i\leq k$ is a simple factor of $G$. Then the restriction of $L_0^2(G/\Gamma)$, the subspace of $L^2(G/\Gamma)$ orthogonal to constants, to each
$P_i$, $1\leq i\leq k$ has a spectral gap (outside a fixed neighborhood of the trivial
representation of $P_i$ in the Fell topology).
\end{theorem}
The next result provides global estimates for the solution of the extended regular representations. We leave the proof of Theorem \ref{th:2} to Appendix \ref{sec:12}.
\begin{theorem}\label{th:2}
 Suppose $v\in \mathfrak{G}^1$ is nilpotent. Suppose $\Gamma$ is a cocompact irreducible lattice and $\mathcal{H}=L^2_0(G/\Gamma)$. Then there are constants $\lambda, \lambda_1>0$ dependent only on $G$ and $\Gamma$ such that if $\mathfrak{u},\omega\in \mathfrak{G}(\mathcal{H})^\infty$ satisfy  the cohomological equation
 \begin{align}\label{for:20}
  (v+\text{ad}_v)\mathfrak{u}=\omega,
 \end{align}
 then we have
\begin{align}\label{for:293}
 \norm{\mathfrak{u}}_t&\leq C_t\norm{\omega}_{\lambda t+\lambda_1},\qquad t\geq0.
 \end{align}

\end{theorem}
\begin{remark}\label{re:6} The Sobolev estimates of $\mathfrak{u}$ are obtained by using Theorem \ref{th:5}, which results that $\lambda>2$ in \eqref{for:293}. Tameness of the solution of the coboundary
equation (over parabolic flows) is not in literature for cases other than $SL(n,\RR)$, $SO_o(m,m)$, $E_{6(6)}$, $E_{7(7)}$ and $E_{8(8)}$ (see \cite{W1}).
\end{remark}

\section{Conjugacy problem and linearization}\label{sec:25}  In this part  we deduce linearized conjugacy equation over $\alpha_A$.
We follow the procedure outlined in a general form in \cite{DK-parabolic}. The results in this part are valid for general abelian algebraic actions. Let $\mathcal{X}=\GG/\Gamma$, where $\Gamma$ is an irreducible cocompact lattice in $\GG$.

 Let $\text{Vect}^\infty(\mathcal{X})$ be the space of $C^\infty$ vector fields on $\mathcal{X}$. Suppose $E=\{E_1,\cdots,E_d\}$ is a set of generators of $\text{Lie}(A)$.
 We can identify $E_i$ with an element of $\text{Vect}^\infty(\mathcal{X})$ such that  $E=\{E_1,\cdots,E_d\}$ generate $\alpha_A$.
 A smooth $A$-perturbation $\tilde{\alpha}_A$ of the action $\alpha_A$ is generated by commuting vector
fields $\tilde{E}=E+\mathfrak{p}=\{E_1+\mathfrak{p}_1,\cdots,E_d+\mathfrak{p}_d\}$, where $\mathfrak{p}_i\in \text{Vect}^\infty(\mathcal{X})$, $1\leq i\leq d$.

\smallskip
\emph{Note.}  We also write $\mathfrak{p}_{E_i}$ instead of $\mathfrak{p}_i$ if we want to emphasize the vector $E_i$.

\smallskip
For any linear map $T=(T_{i,j})_{d\times d}$ on $\RR^{d}$ we have a new basis of $\text{Lie}(A)$: $TE:=\{(TE)_1,\cdots,(TE)_d\}$, where $(TE)_i=\sum_{j=1}^dT_{i,j}E_j$; and a generating vector fields $T\tilde{E}:=\{(T\tilde{E})_1,\cdots,(T\tilde{E})_d\}$, where
$T\tilde{E}_i=\sum_{j=1}^d T_{i,j}(E_j+\mathfrak{p}_j)$. Then $T$ incudes a coordinate change for $\tilde{E}$.

A diffeomorphism $h:\mathcal{X}\to \mathcal{X}$ induces a map $h_*$ on $\text{Vect}^\infty(\mathcal{X})$, the space of $C^\infty$ vector fields on $\mathcal{X}$:
 \begin{align*}
  (h_* Y)(x)=(Dh)_{h^{-1}(x)}Y\circ h^{-1}(x),\qquad x\in \mathcal{X}.
 \end{align*}
 Define operators $\mathcal{L}$ and $\mathcal{M}$ in the following way:
\begin{align}\label{for:126}
 \text{Vect}^\infty(\mathcal{X})&\overset{\text{\tiny$\mathcal{L}$}}{\rightarrow}
\text{Vect}^\infty(\mathcal{X})^d \overset{\text{\tiny$\mathcal{M}$}}{\rightarrow}
\text{Vect}^\infty(\mathcal{X})^{d\times d},\quad\text{where}\notag\\
&\mathfrak{h}\overset{\text{\tiny$\mathcal{L}$}}{\rightarrow}h_*E=
(h_*E_1,\cdots,h_*E_d),\notag\\
&(Y_1,\cdots, Y_d)\overset{\text{\tiny$\mathcal{M}$}}{\rightarrow} ([Y_i,Y_j])_{d\times d},
\end{align}
if $h=\exp(\mathfrak{h})$. Obviously, $\mathcal{M}\circ \mathcal{L}=0$. Denote by $\mathcal{L}\to\mathcal{M}$ the nonlinear sequence of operators defined as above.
Linearizing the sequence $\mathcal{L}\to\mathcal{M}$ at $\mathfrak{h}=0$ and at $E=(E_1,\cdots,E_d)\in \text{Vect}^\infty(\mathcal{X})^d$ the linearized sequence is given as follows:
\begin{gather*}
 \text{Vect}^\infty(\mathcal{X})\overset{\text{\tiny$L$}}{\rightarrow}
\text{Vect}^\infty(\mathcal{X})^d \overset{\text{\tiny$M$}}{\rightarrow}
\text{Vect}^\infty(\mathcal{X})^{d\times d}\\
\mathfrak{h}\overset{\text{\tiny$L$}}{\rightarrow}
(\mathcal{L}_{E_1}\mathfrak{h},\cdots,\mathcal{L}_{E_d}\mathfrak{h})\quad\text{and}\quad
\mathfrak{p}\overset{\text{\tiny$M$}}{\rightarrow} (\mathcal{L}_{E_i}\mathfrak{p}_j-\mathcal{L}_{E_j}\mathfrak{p}_i)_{d\times d}.
\end{gather*}
It is clear that $M\circ L=0$.

For any $Y_1,\,Y_2\in \text{Vect}^\infty(\mathcal{X})$ we have
\begin{align}\label{for:113}
 \norm{[Y_1,Y_2]}_{C^t}\leq C_t(\norm{Y_1}_{C^t}\norm{Y_2}_{C^{t+1}}+\norm{Y_1}_{C^{t+1}}\norm{Y_2}_{C^t}),\quad t\geq0.
\end{align}
For any $Y=(Y_1,\cdots,Y_{\dim \mathfrak{g}})\in \text{Vect}^\infty(\mathcal{X})$ let
\begin{align}\label{for:44}
 \text{Ave}(Y)=\big(\int_\mathcal{X} Y_1(x)dx,\cdots, \int_\mathcal{X} Y_{\dim\mathfrak{g}}(x)dx\big),
\end{align}
where $dx$ is the Haar measure. It is clear that $\text{Ave}(Y)\in \mathfrak{g}$. As a direct consequence of \eqref{for:113} we have
\begin{lemma}\label{le:21} If $\tilde{E}=E+\mathfrak{p}\in \text{Vect}^\infty(\mathcal{X})^d$ satisfying $[\tilde{E}_i,\tilde{E}_j]=0$, then for $t\geq0$ we have
\begin{align*}
 \norm{M(\mathfrak{p})}_{C^{t}}&\leq C_t\norm{\mathfrak{p}}_{C^{0}}\norm{\mathfrak{p}}_{C^{t+1}}\,\text{ and }\\
 \|M(\text{Ave}(\mathfrak{p}))\|&\leq C\norm{\mathfrak{p}}_{C^{0}}\norm{\mathfrak{p}}_{C^{1}},
\end{align*}
\end{lemma}
\subsection{Structural stability of $E$}\label{sec:27} For $c>0$ and a set of vectors $E'=(E_1',\cdots,E_d')$ where $E_i'\in \mathfrak{g}$, we say that $E'$ is a $c$-perturbation of $E$ if $\sum_{i=1}^d\norm{E_i-E_i'}<c$.

\begin{definition}\label{de:1}
Let \( \alpha_A \) be the action of an abelian subgroup \( A \) of \(\GG\) by left translations on \(\mathcal{X} = \GG/\Gamma\). Suppose $E=\{E_1,\cdots,E_d\}$ is a set of generators of $\text{Lie}(A)$. We say that:
\begin{enumerate}
  \item \(\alpha_A\) is \emph{structural stability among algebraic perturbations} if any abelian algebraic action \(\alpha_{A'}\) that is sufficiently close to \(\alpha_A\) is conjugate to \(\alpha_A\) up to a time change.

      \smallskip
  \item \(\alpha_A\) is \emph{geometrically stable} if there is $\delta>0$ such that for any
$c$-perturbation $E'$ of $E$, if $c+\norm{\mathcal{M}(E')}<\delta$ (see \eqref{for:126}), there is a coordinate change $\mathcal{T}$ of $A$ and $g\in \GG$ with
\begin{align*}
 \norm{\mathcal{T}-I}+\norm{g-I}\leq Cc,
\end{align*}
such that
\begin{align}\label{for:216}
 \norm{\mathcal{T}E'-\text{Ad}_{g}E}<C(\norm{\mathcal{M}(E')}+c^2).
\end{align}
\end{enumerate}

\begin{remark}\label{re:2} Condition \eqref{for:216} implies that for any small algebraic perturbation \(E'\) of \(E\), if \(\norm{\mathcal{M}(E')}\) is quadratically small, then after a coordinate change and an inner automorphism of \(\GG\), the new algebraic perturbation \(\text{Ad}_{g^{-1}}(\mathcal{T}E')\) is quadratically closer to \(E\) than \(E'\).
 In particular, for any abelian algebraic action \(\alpha_{A'}\) sufficiently close to \(\alpha_A\), one can choose a generating set \(E'\) for \(\operatorname{Lie}(A')\) with \(\mathcal{M}(E') = 0\), so that \eqref{for:216} yields
\[
\|\mathcal{T}E' - \operatorname{Ad}_g E\| < C\|E' - E\|^2.
\]
This quadratic closeness allows one to apply a KAM iteration argument to obtain a conjugacy up to a time change of $E'$ (given by  an inner automorphism of \(\GG\)). Consequently, geometric stability implies structural stability among algebraic perturbations.
\end{remark}

\end{definition}
\begin{proposition}\label{po:2} Suppose $\alpha_A$ is as described in Corollary\ref{cor:8}, \ref{cor:2} and \ref{cor:13}. Then $\alpha_A$ is geometrically stable.
\end{proposition}
We postpone the proof to Appendix \ref{sec:7}.
\subsection{Smoothing operators and some norm inequalities}\label{sec:26} There exists a collection of smoothing operators $\mathfrak{s}_b:\text{Vect}^\infty(\mathcal{X})\to \text{Vect}^\infty(\mathcal{X})$, $b>0$, such that for any $s, s_1,s_2\geq0$, the following holds:
\begin{align}
 \norm{\mathfrak{s}_bY}_{C^{s+s_1}}&\leq C_{s,s_1}b^{s_1}\norm{Y}_{C^{s}},\quad \text{and}\label{for:111}\\
 \norm{(I-\mathfrak{s}_b)Y}_{C^{s-s_2}}&\leq C_{s,s_2}b^{-s_2}\norm{Y}_{C^{s}},\quad \text{if }s\geq s_2, \label{for:117}
\end{align}
see \cite{Hamilton}.

The next result follows directly from Sobolev embedding theorem on compact manifolds.  For any $Y\in \text{Vect}^\infty(\mathcal{X})$ and $s\geq0$ the following hold:
\begin{align}\label{for:179}
  \norm{Y}_{s}\leq C_s\norm{Y}_{C^{s}},\quad \norm{Y}_{C^{s}}\leq C_{s}\norm{Y}_{s+\beta},
\end{align}
where $\beta>0$ is a constant  dependent only on $\mathcal{X}$.

\section{Construction of approximations on $S_{1,u}$}\label{sec:34}
Throughout this section, $(\pi,\mathcal{H})$ denotes a unitary representation of $SL(2,\RR)$ with a spectral gap.
In an ideal scenario, one would like to construct a linear map $\mathcal{E}:\mathcal{H}^\infty\to \mathcal{H}^\infty$ that provides a splitting for the first coboundary operator over the horocycle flow and satisfies the following properties:
\begin{enumerate}
  \item\label{for:250} \emph{Tame solvability}:
   the equation $u\theta=\omega+\mathcal{E}(\omega)$ admits a solution $\theta\in \mathcal{H}^\infty$ with \emph{tame estimates}. In particular, there exists $\sigma>0$ such that
  \begin{align*}
   \norm{\theta}_s\leq C_s\norm{\omega}_{s+\sigma},\qquad \forall\,s\geq0;
  \end{align*}
  \item\label{for:251} \emph{Vanishing on coboundaries}: if $\omega$ is a $u$-coboundary, then $\mathcal{E}(\omega)=0$.

\end{enumerate}
If such a map $\mathcal{E}$ existed, we would call $\mathcal{E}(\omega)$ the \emph{error} of $\omega$ solving the $u$-cohomology. Moreover, defining the linear map
$\mathfrak{D}:\,\omega\to \theta$ would yield the \emph{inverse operator} of the $u$-coboundary operator, since $\mathfrak{D}\circ u=I$. The construction of a splitting satisfying properties \eqref{for:250} and \eqref{for:251} intended as a preparatory step for applying the KAM scheme, as done in previous works. However, constructing such an $\mathcal{E}$ turns out to be problematic. In the following sections, we will elaborate on the challenges faced in trying to construct a splitting with these exact properties and how we addressed these difficulties.

To illustrate the difficulty, we introduce some notations and recall relevant results.

\textbf{Irreducible representations $(\pi_\nu,\mathcal{H}_\nu)$ of $SL(2,\RR)$} (Section \ref{sec:1}): The nontrivial ones are classified by
\begin{itemize}
  \item $\nu=\textrm{i}\RR$, principal series;
  \item $0 < \nu< 1$, complementary series;
  \item $\nu=\pm (n-1)$, $n\geq1$, discrete series.
\end{itemize}
Let $r=1-\nu^2$. The Casimir operator $\Box$
acts as a constant $r$ on $\mathcal{H}_\nu$.

(\textbf{Theorem \ref{th:3}}) Let $U=\begin{pmatrix}
  0 & 1 \\
  0 & 0
\end{pmatrix}$ and $s_\nu=\text{\small$\frac{3}{2}$}+\frac{1}{2}(|\Re(\nu)|+1)$. In $\pi_\nu$ with a spectral gap $r_0$, there is a linear map $\text{\small$\mathcal{D}_{\nu}$}$ defined on $\mathcal{H}^{s_\nu}_\nu$ satisfying the following properties:
\begin{enumerate}
  \item [(a)] for any $\omega\in \mathcal{H}^{\infty}_\nu$
  \begin{align*}
\norm{\text{\small$\mathcal{D}_{\nu}$}(\omega)}_t\leq C_{t,r_0}\norm{\omega}_{t+s_\nu};
\end{align*}
  \item [(b)] for any $\omega\in \mathcal{H}^{s}_\nu$, $s\geq s_\nu$, the equation $U\theta=\omega+\text{\small$\mathcal{D}_{\nu}$}(\omega)$
      has a solution $\theta\in \mathcal{H}^{s-s_\nu}_\nu$ with estimates: for any  $0\leq t\leq s-s_\nu$
\begin{align*}
\norm{\theta}_t\leq C_{t,r_0} \norm{\omega}_{t+s_\nu};
\end{align*}

  \item [(c)] $\text{\small$\mathcal{D}_{\nu}$}(U\omega)=0$ if $\omega\in \mathcal{H}^{s_\nu}_\nu$.
\end{enumerate}

\smallskip
\textbf{Difficulty}: $(b)$ and $(c)$ of Theorem \ref{th:3} show that in each irreducible representation of $SL(2,\RR)$, $\text{\small$\mathcal{D}_{\nu}$}$ gives a desired splitting.
Suppose $(\pi,\mathcal{H})$ has a spectral gap $r_0$ and $\pi$ contains a sequence of  discrete series with $|\nu|\to\infty$ (which covers almost all $SL(2,\mathbb{R})$ representations of interest so far).  To construct a splitting satisfying \eqref{for:250} and \eqref{for:251} in $\pi$, we intuitively define $\mathcal{E}$ formally as follows:
\begin{align*}
 \mathcal{E}(\omega)=\int_{\oplus} \mathcal{D}_{\nu}(\omega_r)d\mu(r)\qquad \omega\in \mathcal{H}^\infty
\end{align*}
(see \eqref{for:1} for the direct integral decomposition). $(a)$ of Theorem \ref{th:3}  shows that
\begin{align*}
\norm{\text{\small$\mathcal{D}_{\nu}$}(\omega_r)}\leq C_{r_0}\norm{\omega_r}_{s_\nu}.
\end{align*}
We note that $s_\nu\to \infty$ for discrete series if $|\nu|\to\infty$. This implies that  $\mathcal{E}(\omega)$ may not be a bonafide vector in $\mathcal{H}$.  This is the main difficulty
in applying KAM to horocycle
flows: the Sobolev order of the obstructions to solving the coboundary equation tends to infinity. This is quite different from the case of geodesic flow, whose order is uniformly bounded \cite{mieczkowski2006}.

 In order to get around this problem, we define
\begin{align*}
 \mathcal{E}_\iota(\omega)=\mathcal{E}(\omega-D^\iota\omega),\qquad \iota\in\NN
\end{align*}
(see \eqref{for:253}), where $\omega-D^\iota\omega$ removes the contributions from the discrete series with $|\nu| \geq \iota$. In simple terms,
 $\mathcal{E}_\iota(\omega)$ contains the error parts of all irreducible components, except for those discrete series with $|\nu|\geq \iota$.
For each fixed $\iota$, $\mathcal{E}_\iota(\omega)$ is well defined and is a smooth vector. However, we cannot generally expect that the equation
\begin{align}\label{for:256}
 U\theta=\omega+\mathcal{E}_\iota(\omega)
\end{align}
has a globally smooth solution $\theta$.

\smallskip
\textbf{Key observation}: By using the normalizer trick (see Lemma \ref{le:5}) we can show that
if $\iota\geq5$, then the solution $\theta$ to equation \eqref{for:256} is smooth along $U$ and $X=\begin{pmatrix}
  1 & 0 \\
  0 & -1
\end{pmatrix}$. We explain this result in more detail in Remark \ref{re:3}. This finding plays a crucial role in the next section (see Proposition \ref{po:1}), where we construct a splitting whose non-smooth directions lie inside unions of nilpotent subalgebras. This construction then serves as the foundation for ultimately producing a globally smooth splitting in Section \ref{sec:22}.

\subsection{Notations and main results} In this section, we provide an overview of the results that will be proven in the remainder of  Section \ref{sec:34}.
\begin{enumerate}



  \item\label{for:253}  (Section \ref{sec:13}) For any $\iota\in\NN$ we define two linear operators $\mathcal{D}^\iota$ and $\mathcal{E}_\iota$. $\mathcal{D}^\iota:\mathcal{H}\to \mathcal{H}$ is a projection to the subspace spanned by discrete series components with
  $|\nu|\geq \iota$.

  $\mathcal{E}_\iota: \mathcal{H}^{s_\iota}\to \mathcal{H}$ is defined as follows: if $\omega\in \mathcal{H}^{s_\iota}$, then
 \begin{align*}
  \mathcal{E}_\iota(\omega)=\int_{\oplus}g_rd\mu(r)
 \end{align*}
 where
\begin{align*}
  g_r=\left\{\begin{aligned}&\text{\small$\mathcal{D}_{\nu}$}(\omega_r),&\quad &\text{if }\nu\in \textrm{i}\RR\cup (0,1)\cup\{0,\pm1,\cdots,\pm(\iota-1)\};\\
 &0, &\quad &\text{if }\nu\in\ZZ,\text{ and }|\nu|\geq \iota.
\end{aligned}
 \right.
\end{align*}

\smallskip

For any $\omega\in \mathcal{H}$, $\mathcal{D}^\iota(\omega)$ contains all the components of $\omega$ in the discrete series with $|\nu|\geq \iota$; and $\omega-\mathcal{D}^\iota(\omega)$ contains all the components of $\omega$ in the principal/complementary series, as well as the components in discrete series with $|\nu|< \iota$.
Consequently,  we have a direct sum decomposition:
\begin{align*}
 \mathcal{H}=\text{Im}(\mathcal{D}^\iota)\oplus\ker(\mathcal{D}^\iota).
\end{align*}
This means for any $\omega\in \mathcal{H}$, we can write $\omega=\omega_1+\omega_2$, where $\omega_1=\mathcal{D}^\iota(\omega)\in \text{Im}(\mathcal{D}^\iota)$ and $\omega_2=\omega-\mathcal{D}^\iota(\omega)\in \ker(\mathcal{D}^\iota)$.

Thus equation \eqref{for:256} decomposes into two equations:
\begin{align}
 U\theta_1&=\omega_1\qquad\text{and}\label{for:269}\\
 U\theta_2&=\omega_2+\mathcal{E}_\iota(\omega)=\omega_2+\mathcal{E}_\iota(\omega_2).\label{for:268}
\end{align}

\smallskip

\item \label{for:254} (\eqref{for:16} of Lemma \ref{le:8}) If $\iota\geq 3$ and if $\mathcal{D}^\iota(\omega)=0$, the equation $U\theta=\omega+\mathcal{E}_\iota(\omega)$
      has a solution $\theta\in \mathcal{H}^{\infty}$ satisfying $\mathcal{D}^\iota(\theta)=0$ with estimates
      \begin{align*}
      \norm{\theta}_{t}\leq C_{t} \norm{\omega}_{t+2+\frac{\iota}{2}},\qquad t\geq0.
      \end{align*}

    \noindent (\eqref{for:19} of Lemma \ref{le:8}) Suppose $\iota\geq5$. If $\mathcal{D}^\iota(\omega)=\omega$, then the equation $U\theta=\omega$ has a solution $\theta\in \mathcal{H}$ satisfying $\mathcal{D}^\iota(\theta)=\theta$ with estimates:
      \begin{align*}
\norm{Y^j\theta}_{t}\leq C_{j,t} \norm{\omega}_{t+j+\frac{3}{2}},\qquad j\geq 0
\end{align*}
if $0\leq t\leq \frac{\iota}{2}-\frac{3}{2}$, where $Y$ stands for $X$ or $U$.

\smallskip

\begin{remark}\label{re:3} The first result shows that if $\iota\geq3$ equation  \eqref{for:268} has a solution $\theta_2$ which is smooth on the whole $SL(2,\RR)$.
The second result shows that if $\iota\geq5$ equation \eqref{for:269}  has a solution $\theta_1$ which is partially smooth on $G_U'$ (we recall that $G_U'$ is generated by $X$ and $U$).
Hence if $\iota\geq5$ equation \eqref{for:256} has a solution $\theta$ which
is partially smooth on $G_U'$. Moreover, $\theta$ has partially tame estimates (with respect to $\omega$) on $G_U'$ (see \eqref{for:272} of Section \ref{sec:47}).
\end{remark}

\smallskip
\item Suppose $(\pi, \mathcal{H})$ is a unitary representation of $\GG$ whose restriction to each simple factor of $\GG$ has a spectral gap.
Fix $\phi\in \Phi$ and $u\in \mathfrak{u}_\phi\cap\mathfrak{g}^1$.
 By the centralizer trick (see Lemma \ref{le:10}), we can extend the smoothness of
$\theta$ in equation \eqref{for:256} to the centralizer of $G_u$.

\smallskip
\noindent(\eqref{for:60} of Lemma \ref{le:1}) If $\omega\in \mathcal{H}_{S_{0}}^\infty$ and $\mathcal{D}^\iota(\omega)=0$, $\iota\geq 3$, the equation $u\theta=\omega+\mathcal{E}_\iota(\omega)$ has a solution $\theta\in \mathcal{H}_{S_0}^{\infty}$ satisfying $\mathcal{D}^\iota(\theta)=0$ with estimates
  \begin{align*}
   \norm{\theta}_{S_0,t}\leq C_{t} \norm{\omega}_{S_0,t+\frac{11}{2}+\frac{\iota}{2}},\qquad \forall\,t\geq0.
  \end{align*}
\noindent(\eqref{for:11} of Lemma \ref{le:1}) If $\omega\in \mathcal{H}_{S_{0}}^\infty$ and $\mathcal{D}^\iota(\omega)=\omega$ and $\iota\geq 5$, then equation $u\theta=\omega$ has a solution $\theta\in \mathcal{H}$ satisfying $\mathcal{D}^\iota(\theta)=\theta$ with estimates:
    \begin{align*}
\norm{Y^j\theta}_{G_u,t}\leq C_{j,t} \norm{\omega}_{S_0,t+j+\frac{3}{2}},\qquad \forall\,j\geq 0
\end{align*}
if $0\leq t\leq \frac{\iota}{2}-\text{\small$\frac{3}{2}$}$, where $Y$ stands for $X_u$, $u$ or $Y\in \mathcal{C}(\mathfrak{g}_u)$.

\begin{remark}
The above results imply that if $\iota\geq 5$ and $\omega\in \mathcal{H}_{S_{0,u}}^\infty$, then equation \eqref{for:256}
has a solution $\theta\in \mathcal{H}_{S_{1,u}}^\infty$ with partially tame estimates (with respect to $\omega$). This is the reason we call $\text{Lie}(S_{1,u})$ the \emph{tame} directions to $u$ (see \eqref{for:273} of Section \ref{sec:47}).

\end{remark}

\noindent (Corollary \ref{cor:6}) Let $H$ be a subgroup  of $C(G_u)$.  Suppose $\Omega,\,\Theta\in \mathfrak{g}(\mathcal{H})_{\{H,G_u\}}^s$, $s\geq \text{\small$\frac{5}{2}$}\dim\mathfrak{g}$ satisfy the equation
\begin{align*}
 (u+\text{ad}_u)\Theta=\Omega.
\end{align*}
Then for any $t\leq s-\text{\small$\frac{5}{2}$}\dim\mathfrak{g}$, we have
\begin{align*}
\norm{\Omega}_{\{H,G_u\},t}\leq C_{t} \norm{\Theta}_{\{H,G_u\},t+\text{\small$\frac{5}{2}$}\dim\mathfrak{g}}.
\end{align*}
\end{enumerate}

\subsection{Unitary dual of $SL(2,\RR)$}\label{sec:1} We recall the conclusions in \cite{tan} and \cite{flaminio2016effective}. We choose as generators for $\mathfrak{sl}(2,\RR)$ the elements
\begin{align}\label{for:4}
X=\begin{pmatrix}
  1 & 0 \\
  0 & -1
\end{pmatrix},\quad U=\begin{pmatrix}
  0 & 1 \\
  0 & 0
\end{pmatrix},\quad V=\begin{pmatrix}
  0 & 0 \\
  1 & 0
\end{pmatrix}.
  \end{align}
The \emph{Casimir} operator is then given by
\begin{align*}
\Box:= -X^2-2(UV+VU),
\end{align*}
which generates the center of the enveloping algebra of $\mathfrak{sl}(2,\RR)$. The Casimir operator $\Box$
acts as a constant $r\in\RR$ on each irreducible unitary representation space  and its value classifies them into four classes.
For \emph{Casimir parameter} $r$ of $SL(2,\RR)$, let $\nu=\sqrt{1-r}$ be a representation parameter. Then all the irreducible unitary representations of $SL(2,\RR)$
must be equivalent to one the following:
\begin{itemize}
  \item principal series representations $\pi_\nu^{\pm}$, $r\geq 1$ so that
$\nu=\textrm{i}\RR$,
\smallskip
  \item complementary series representations $\pi^0_\nu$, $0 <r< 1$, so that $0 < \nu< 1$,
  \smallskip
  \item discrete series representations $\pi^0_\nu$ and $\pi^0_{-\nu}$, $r=-n^2+2n$, $n\geq 1$, so that $\nu=n-1$,
  \smallskip
  \item the trivial representation, $r=0$.
\end{itemize}
Any unitary representation $(\pi,\mathcal{H})$ of $SL(2,\RR)$ is decomposed into a direct integral (see \cite{Forni} and \cite{mautner1950unitary})
\begin{align}\label{for:1}
\mathcal{H}=\int_{\oplus}\mathcal{H}_r d\mu(r)\quad\text{and}\quad \omega=\int_{\oplus}\omega_rd\mu(r) \quad \forall\,\omega\in \mathcal{H}
\end{align}
with respect to a positive Stieltjes measure $d\mu(r)$ over the spectrum $\sigma(\Box)$. The
Casimir operator acts as the constant $u\in \sigma(\Box)$ on every Hilbert space $\mathcal{H}_r$. The
representations induced on $\mathcal{H}_r$ do not need to be irreducible. In fact, $\mathcal{H}_r$ is in general
the direct sum of an (at most countable) number of unitary representations equal
to the spectral multiplicity of $r\in \sigma(\Box)$. We say that \emph{$\pi$ has a spectral gap (of $r_0$)} if $r_0>0$ and $\mu((0,r_0])=0$.

\subsection{Coboundary for the horocycle
flow of $SL(2,\RR)$}\label{sec:18} For the classical horocycle flow defined by the $\mathfrak{sl}(2,\RR)$-matrix $U=\begin{pmatrix}
  0 & 1 \\
  0 & 0
\end{pmatrix}$, Flaminio and Forni made a detailed study in \cite{Forni}.

For any non-trivial irreducible representation $(\pi_\nu,\mathcal{H}_\nu)$ of $SL(2,\RR)$, set $s_\nu=\text{\small$\frac{3}{2}$}+\frac{1}{2}(|\Re(\nu)|+1)$. Let $(\mathcal{H}_\nu)_U^{-k}=\{\mathcal{D}\in (\mathcal{H}_\nu)^{-k}: \mathcal{L}_U\mathcal{D}=0 \}$, $0<k\leq \infty$.

Below we
summarize some conclusions adapted to the needs of the current paper.
\begin{theorem}\label{th:3} In any non-trivial irreducible representation $(\pi_\nu,\mathcal{H}_\nu)$ of $SL(2,\RR)$ with a spectral gap $r_0$.  There exists a linear map $\text{\small$\mathcal{D}_{\nu}$}:\mathcal{H}^{s_\nu}_\nu\to \mathcal{H}_\nu$ such that for any $\omega\in \mathcal{H}^{s}_\nu$, $s\geq0$ we have:
\begin{enumerate}
\item \label{for:8} if $s\geq s_\nu$, then for any  $0\leq t\leq s-s_\nu$ we have
\begin{align*}
\norm{\text{\small$\mathcal{D}_{\nu}$}(\omega)}_t\leq C_{t,r_0}\norm{\omega}_{t+s_\nu};
\end{align*}

\item \label{for:13}  if $s\geq s_\nu$,  the equation $U\theta=\omega+\text{\small$\mathcal{D}_{\nu}$}(\omega)$
      has a solution $\theta\in \mathcal{H}^{s-s_\nu}_\nu$ with estimates: for any  $0\leq t\leq s-s_\nu$
\begin{align*}
\norm{\theta}_t\leq C_{t,r_0} \norm{\omega}_{t+s_\nu};
\end{align*}

  \item \label{for:29}  Suppose $s>1$ and $\mathcal{D}(\omega)=0$ for all $\mathcal{D}\in (\mathcal{H}_\nu)_U^{-s}$. Then
 the equation $U\theta=\omega$ has a solution $\theta\in \mathcal{H}_\nu^t$ with Sobolev estimates
 \begin{align*}
  \norm{\theta}_t\leq C_{t,s,r_0}\norm{\omega}_s
 \end{align*}
 for any $0\leq t<s-1$. Further, if the equation $U\theta=\omega$ has a solution $\theta\in \mathcal{H}^{s_\nu}_\nu$ then $\text{\small$\mathcal{D}_{\nu}$}(\omega)=0$;

\smallskip
  \item\label{for:23} if $\pi_\nu$ is a discrete series and $s\geq2$ and $|\nu|\geq 3$, then the equation $U\theta=\omega$ has a solution $\theta\in \mathcal{H}^{\min\{\frac{1}{2}|\nu|-\frac{3}{2},s-\frac{3}{2}\}}_\nu$ with estimates
      \begin{align*}
\norm{\theta}_t\leq C_t \norm{\omega}_{t+\frac{3}{2}}
\end{align*}
for any $0\leq t\leq \min\{\text{\small$\frac{1}{2}$}|\nu|-\text{\small$\frac{3}{2}$},s-\text{\small$\frac{3}{2}$}\}$;

\smallskip
\item\label{for:186} if the equation $U\theta=\omega$ has a solution $\theta\in \mathcal{H}^t_\nu$, then
\begin{align*}
 \norm{\theta}_t\leq C_{r_0,t}\norm{\omega}_{t+\frac{3}{2}}
\end{align*}
for any $0\leq t\leq s-\text{\small$\frac{3}{2}$}$.

\end{enumerate}

\end{theorem}
\emph{Note}. \eqref{for:8} and \eqref{for:13} show that $D_\nu$ is a splitting for the $U$-coboundary in $\pi_\nu$.
\begin{proof} \eqref{for:29} is from Theorem $1.2$ of \cite{Forni}; \eqref{for:23} follows from Theorem $1.1$, Theorem $1.2$ of \cite{Forni};
\eqref{for:186} is from Theorem $1.1$, Theorem $1.2$ and Theorem $1.3$ of \cite{Forni}.

\smallskip
\eqref{for:8}: The space $(\mathcal{H}_\nu)_U^{-\infty}$ of $U$-invariant distributions is described in Theorem $1.1$ of \cite{Forni} as follows: it is finite-dimensional,
spanned by distributions $D_{\nu,1},\cdots, D_{\nu,m}$, $m\leq2$ with estimates
\begin{align}\label{for:203}
  |D_{\nu,j}(\omega)|\leq C_{r_0,\epsilon}\norm{\omega}_{\frac{1}{2}(|\Re(\nu)|+1)+\epsilon},\qquad   1\leq j\leq m.
\end{align}
Fix $\xi_{\nu,1},\cdots, \xi_{\nu,m}\in \mathcal{H}_\nu$ such that
\begin{align}\label{for:204}
 D_{\nu,j}(\xi_{\nu,i})=\delta_{i,j}\quad\text{and}\quad \norm{\xi_{\nu,j}}_{t}\leq 2(|\nu|+1)^t
\end{align}
for all $j\leq m$ and any $t\geq0$. Define
\begin{align*}
 \text{\small$\mathcal{D}_{\nu}$}(\omega)=-\sum_{j=1}^m D_{\nu,j}(\omega)\xi_{\nu,i},\qquad \omega\in \mathcal{H}^{s_\nu}_\nu.
\end{align*}
From the construction of $\text{\small$\mathcal{D}_{\nu}$}$, we see that

It follows from \eqref{for:203} and \eqref{for:204} that
\begin{align}\label{for:205}
\norm{\text{\small$\mathcal{D}_{\nu}$}(\omega)}_t&\leq C_{r_0}(|\nu|+1)^t\norm{\omega}_{\frac{1}{2}(|\Re(\nu)|+1)+\frac{1}{4}}\notag\\
&=C_{r_0}\big\|(|\nu|+1)^t\omega\big\|_{\frac{1}{2}(|\Re(\nu)|+1)+\frac{1}{4}}\notag\\
&\leq C_{r_0,t}\norm{\omega}_{\frac{1}{2}(|\Re(\nu)|+1)+\frac{1}{4}+t}.
\end{align}
for any  $0\leq t\leq s-(\frac{1}{2}(|\Re(\nu)|+1)+\frac{1}{4})$. This implies \eqref{for:8}.

\smallskip

\eqref{for:13}: From the construction of $\text{\small$\mathcal{D}_{\nu}$}$ in \eqref{for:8}, we see that
\begin{align*}
 D\big(\omega+\text{\small$\mathcal{D}_{\nu}$}(\omega)\big)= D(\omega)+D\big(\text{\small$\mathcal{D}_{\nu}$}(\omega)\big)=0
\end{align*}
for any $\omega\in \mathcal{H}^{s_\nu}_\nu$ and any $D\in (\mathcal{H}_\nu)_U^{-\infty}$.

It follows from \eqref{for:29} that the equation $U\theta=\omega+\text{\small$\mathcal{D}_{\nu}$}(\omega)$
      has a solution $\theta\in \mathcal{H}_\nu$ with estimates
      \begin{align*}
  \norm{\theta}_t\leq C_{t,r_0}\norm{\omega+\text{\small$\mathcal{D}_{\nu}$}(\omega)}_{t+\frac{5}{4}}\overset{\text{(a)}}{\leq}
  C_{t,r_0,1}\norm{\omega}_{t+\frac{1}{2}(|\Re(\nu)|+1)+\frac{3}{2}}
 \end{align*}
 for any  $0\leq t\leq s-(\frac{1}{2}(|\Re(\nu)|+1)+\frac{3}{2})=s-s_\nu$. Here in $(a)$ we use \eqref{for:205}. Thus we get the result.
\end{proof}

\begin{remark}\label{re:13} It is well-known that principal series and discrete series are tempered. Tempered
representations are those outside a fixed neighborhood of the trivial representation
in the Fell topology. Then the spectral gap condition only aims at complementary series.

Theorem \ref{th:3} is still valid to irreducible unitary representations of Lie groups whose
Lie algebra is $\mathfrak{sl}(2,\RR)$. All of these are unitarily equivalent to
irreducible representations of $SL(2,\RR)$ itself \cite{tan}.
\end{remark}

\eqref{for:29} and \eqref{for:23} of the above theorem show that unlike the cases of principal/complementary series, for the discrete series (even when $\omega$ is $C^\infty$) the existence of a low regularity solution of the coboundary $U\theta=\omega$ can not guarantee the existence of a high regularity solution. To overcome this difficulty we use the ``normalizer trick", which shows that we can expect high regularity along $X$ and $U$ directions. This technique is similar to the one used in \cite{ramirez2009cocycles},
\cite{tanis2018Cohomology}, \cite{tanis2017cohomological} and \cite{W1} to study the coboundary equation.

\begin{lemma}\label{le:5}(normalizer trick) Suppose $\pi_\nu$ is a discrete series, $|\nu|\geq 5$. Also suppose $m\geq0$ and $s\geq \frac{5}{2}$. If $X^j\omega$ and $U^j\omega$ are in $\mathcal{H}_\nu^s$ for any $0\leq j\leq m$, then the equation $U\theta=\omega$ has a solution $\theta\in \mathcal{H}_\nu$ with estimates: for any $0\leq j\leq m$
      \begin{align}
\norm{Y^j\theta}_{t}\leq C_{j,t} \max_{0\leq i\leq j}\{\norm{Y^i\omega}_{t+\frac{3}{2}}\}
\end{align}
if $0\leq t\leq \min\{\text{\small$\frac{1}{2}$}|\nu|-\text{\small$\frac{3}{2}$},s-\text{\small$\frac{3}{2}$}\}$, where $Y$ stands for $X$ or $U$.
\end{lemma}
\begin{proof}
Instead of proving the lemma first, we will prove the following statement: $(*)$ for any $0\leq j\leq m$, there is a polynomial $p_{j}$ of degree $j$ such that  $Y^j\theta\in \mathcal{H}_\nu$ and satisfies the equation
\begin{align}\label{for:262}
 U(Y^j\theta)=p_j(Y)\omega.
\end{align}
We prove by induction. It follows from \eqref{for:23} of Theorem \ref{th:3} that the statement $(*)$ holds for $j=0$. Suppose it holds for $j\leq k$, $k\leq m-1$. Then we have
\begin{align}\label{for:93}
 U(Y^{k}\theta)=p_k(Y)\omega.
\end{align}
By assumption, $p_k(Y)\omega\in \mathcal{H}^{s}_\nu$.  Applying
\eqref{for:23} of Theorem \ref{th:3} to \eqref{for:93}, we see that $Y^k\theta\in \mathcal{H}^{\min\{\frac{1}{2}|\nu|-\frac{3}{2},s-\frac{3}{2}\}}_\nu$. We note that $\min\{\frac{1}{2}|\nu|-\frac{3}{2},s-\frac{3}{2}\}\geq1$ by assumption, which means $Y^{k+1}\theta\in \mathcal{H}_\nu$.

We note that
\begin{align}\label{for:261}
 [Y,U]=aU,\qquad a=2\text{ or }0.
\end{align}
Then inductively we can show that for any $k\geq 1$
\begin{align}\label{for:5}
Y^kU=UY^k+q_{k-1}(Y)U
\end{align}
where $q_{k-1}$ is a polynomial of degree $k-1$. Set $q_{-1}=0$.

It follows from \eqref{for:93}  that
\begin{gather*}
 YU(Y^k\theta)=Yp_k(Y)\omega\\
 \overset{\text{(1)}}{\Rightarrow} (UY+aU)(Y^k\theta)=Yp_k(Y)\omega\\
 \overset{\text{(2)}}{\Rightarrow} U(Y^{k+1}\theta)=Yp_k(Y)\omega-a\big(Y^kU-q_{k-1}(Y)U\big)\theta\\
 \overset{\text{(3)}}{\Rightarrow} U(Y^{k+1}\theta)=Yp_k(Y)\omega-a\big(Y^k-q_{k-1}(Y)\big)\omega.
\end{gather*}
Here in $(1)$ we use \eqref{for:261}; in $(2)$ we use \eqref{for:5} and in $(3)$ we recall $U\theta=\omega$. Let
\begin{align*}
  p_{k+1}(x)=xp_k(x)-a\big(x^k-q_{k-1}(x)\big).
\end{align*}
It is clear that $p_{k+1}$ is a polynomial of degree $k+1$ and satisfies the equation
\begin{align*}
 U(Y^{k+1}\theta)=p_{k+1}(Y)\omega.
\end{align*}
Then we proved the case of $k+1$ and thus finish the proof.

Finally, using the statement $(*)$, from equation \eqref{for:262} it follows from \eqref{for:23} of Theorem \ref{th:3} that
\begin{align*}
\norm{Y^j\theta}_t&\leq C_t\norm{p_{j}(Y)\omega}_{t+\frac{3}{2}}\leq C_{j,t}\max_{0\leq i\leq j}\{\norm{Y^i\omega}_{t+\frac{3}{2}}\},
\end{align*}
if $t\leq \min\{\text{\small$\frac{1}{2}$}|\nu|-\text{\small$\frac{3}{2}$},s-\text{\small$\frac{3}{2}$}\}$. Then we finish the proof.

\end{proof}

\subsection{Constructions in unitary representation of $SL(2,\RR)$ with a spectral gap}\label{sec:13}
Suppose  $(\pi, \mathcal{H})$ is a  unitary representation of $SL(2,\RR)$ with a spectral gap $r_0$. By general arguments in Section \ref{sec:1} we have a direct decomposition of $\mathcal{H}$: $\mathcal{H}=\int_{\oplus}\mathcal{H}_rd\mu(r)$, where $\mu((0,r_0])=0$; and $\omega=\int_{\oplus}\omega_rd\mu(r)$ for any $\omega\in \mathcal{H}$. For any $\iota\in\NN$ define
\begin{align}\label{de:3}
 \mathcal{D}^\iota(\omega)=\int_{\oplus}h_rd\mu(r)
\end{align}
where
\begin{align*}
  h_r=\left\{\begin{aligned}&0,&\quad &\text{if }\nu\in \textrm{i}\RR\cup (0,1)\cup\{0,\pm1,\cdots,\pm(\iota-1)\};\\
 &\omega_r, &\quad &\text{if }\nu\in\ZZ,\text{ and }|\nu|\geq \iota.
\end{aligned}
 \right.
\end{align*}
Then $\mathcal{D}^\iota:\mathcal{H}\to \mathcal{H}$ is a linear operator.

 We define another linear operator $\mathcal{E}_\iota: \mathcal{H}^{s_\iota}\to \mathcal{H}$ as follows: if $\omega\in \mathcal{H}^{s_\iota}$, then
 \begin{align}\label{de:4}
  \mathcal{E}_\iota(\omega)=\int_{\oplus}g_rd\mu(r)
 \end{align}
 where
\begin{align*}
  g_r=\left\{\begin{aligned}&\text{\small$\mathcal{D}_{\nu}$}(\omega_r),&\quad &\text{if }\nu\in \textrm{i}\RR\cup (0,1)\cup\{0,\pm1,\cdots,\pm(\iota-1)\};\\
 &0, &\quad &\text{if }\nu\in\ZZ,\text{ and }|\nu|\geq \iota.
\end{aligned}
 \right.
\end{align*}
We also write $(\mathcal{E}_{\iota})_U$ or $(\mathcal{D}^{\iota})_U$ to emphasize the dependence on $U$.

\begin{lemma}\label{le:8} Suppose  $(\pi, \mathcal{H})$ is a  unitary representation of $SL(2,\RR)$ with a spectral gap $r_0$. Suppose $\omega\in \mathcal{H}^s$, $s\geq0$ then:
\begin{enumerate}
\item\label{for:77} for any $\iota\in\NN$ and any $0\leq t\leq s$
\begin{align*}
\norm{\mathcal{D}^\iota(\omega)}_t\leq \norm{\omega}_{t};
\end{align*}

  \item\label{for:54} if $\iota\geq 3$ and $s\geq\frac{\iota}{2}+2$, then
  \begin{align*}
   \norm{\mathcal{E}_\iota(\omega)}_t\leq C_{t,r_0}\norm{\omega}_{t+2+\frac{\iota}{2}}
  \end{align*}
if $0\leq t\leq s-2-\text{\small$\frac{\iota}{2}$}$;

\smallskip
  \item\label{for:16} if $\iota\geq 3$ and $s\geq\frac{\iota}{2}+2$, and if $\mathcal{D}^\iota(\omega)=0$, the equation $U\theta=\omega+\mathcal{E}_\iota(\omega)$
      has a solution $\theta\in \mathcal{H}^{s-2-\frac{\iota}{2}}$ with estimates
      \begin{align*}
      \norm{\theta}_{t}\leq C_{t} \norm{\omega}_{t+2+\frac{\iota}{2}}
      \end{align*}
      if $0\leq t\leq s-2-\text{\small$\frac{\iota}{2}$}$;

\smallskip
  \item\label{for:142} if the equation $U\theta=\omega$ has a solution $\theta\in \mathcal{H}^{\frac{\iota}{2}+2}$ then $\mathcal{E}_\iota(\omega)=0$;
  \smallskip
  \item \label{for:3}if the equation $U\theta=\omega$ has a solution $\theta\in \mathcal{H}^r$, $s\geq r+\text{\small$\frac{3}{2}$}$ then for any  $0\leq t\leq r$
  \begin{align*}
    \norm{\theta}_t\leq C_{r_0,t}\norm{\omega}_{t+\frac{3}{2}};
  \end{align*}

\item \label{for:19} suppose $\iota\geq5$, $s\geq \frac{5}{2}$ and $m\geq0$. If $\mathcal{D}^\iota(\omega)=\omega$ and if $X^j\omega\in \mathcal{H}^s$ and $U^j\omega\in \mathcal{H}^s$ for any $0\leq j\leq m$, then the equation $U\theta=\omega$ has a solution $\theta\in \mathcal{H}$ satisfying $\mathcal{D}^\iota(\theta)=\theta$ with estimates: for any $0\leq j\leq m$
      \begin{align*}
\norm{Y^j\theta}_{t}\leq C_{j,t} \max_{0\leq i\leq j}\{\norm{Y^i\omega}_{t+\frac{3}{2}}\},
\end{align*}
if $0\leq t\leq \min\{\text{\small$\frac{\iota}{2}$}-\text{\small$\frac{3}{2}$},s-\text{\small$\frac{3}{2}$}\}$, where $Y$ stands for $X$ or $U$.
\end{enumerate}
\end{lemma}
\begin{proof} \eqref{for:77}--\eqref{for:3} follow from Theorem \ref{th:3} and arguments in Section \ref{sec:9}. \eqref{for:19} is from  Lemma \ref{le:5}  and arguments in Section \ref{sec:9}.
\end{proof}

\subsection{Constructions in unitary representation of $\GG$}\label{sec:16}
In this section we use $(\pi, \mathcal{H})$ to denote a unitary representation of $\GG$ whose restriction to each simple factor of $\GG$ has a spectral gap.

Fix $\phi\in \Phi$ and $u\in \mathfrak{u}_\phi\cap\mathfrak{g}^1$. By a result of Shalom (see \cite[Theorem $C$]{Shalom}), $\pi|_{G_u}$ has a spectral gap $r_0$. Then for any $\omega\in \mathcal{H}$, $(\mathcal{D}^\iota)_u(\omega)$ and $(\mathcal{E}_\iota)_u(\omega)$ are well defined. The following is a technical result called the ``centralizer trick".

\begin{lemma}(centralizer trick)\label{le:10}
Suppose $H$ is a connected subgroup of $C(G_u)$.
 If $\theta\in \mathcal{H}_{G_u}^{s_1}$ and $\omega\in \mathcal{H}_{\{G_u,H\}}^s$ with $0\leq s_1\leq s-\frac{3}{2}$,
then:
\begin{enumerate}

\item\label{for:141} for any $v\in\text{Lie}(C(G_u))$ and any $\iota\in\NN$, if $\mathcal{E}_\iota(\theta)\in\mathcal{H}$, then
\begin{align*}
   \mathcal{E}_\iota(v^j\theta)=v^j\mathcal{E}_\iota(\theta),\quad \mathcal{D}^\iota(v^j\theta)=v^j\mathcal{D}^\iota(\theta)\quad\text{ as distributions }
 \end{align*}
for any $j\geq0$;

  \item\label{for:47} if $\theta$ and $\omega$ satisfy
the equation $u\theta=\omega$ and $s\geq \frac{5}{2}$, then $\theta\in \mathcal{H}_{\{H\}}^{s-\frac{5}{2}}$ with the estimate
  \begin{align*}
 \norm{\theta}_{H,t}&\leq C_{t,r_0}\norm{\omega}_{\{G_u,H\},t+\frac{5}{2}}
 \end{align*}
 for any $0\leq t\leq s-\frac{5}{2}$;

 \smallskip
  \item\label{for:185} if $\theta$ and $\omega$ satisfy
the equation $u\theta=\omega$ and $s_1\geq1$,  then $\theta\in \mathcal{H}_{\{G_u,H\}}^{s_1-1}$ with the estimate
   \begin{align*}
 \norm{\theta}_{\{G_u,H\}, t}&\leq C_{t,r_0}\norm{\omega}_{\{G_u,H\},t+\frac{5}{2}}
 \end{align*}
 for any $0\leq t\leq s_1-1$.
\end{enumerate}

\end{lemma}
\begin{proof}
 For any vector $v\in\text{Lie}(C(G_u))$, denote by $\tilde{v}$ the one-parameter subgroup  with its algebra generated by $v$. Let $S=\{G_u,\,\tilde{v}\}$. Then $S$ is isomorphic to $(G_u\times \tilde{v})/K$, where $K=\{(k,k^{-1} ): \, k\in G_u\cap \tilde{v}\}$. Then $\pi|_{S}$ can be view as a representation of $G_u\times \tilde{v}$
which is trivial on $K$.

By Section \ref{sec:9}, we have a decomposition
\begin{align*}
  \pi|_{S}=\int_{Z} \sigma_{z} d\mu(z)
\end{align*}
for some measure $(Z,\mu)$, where $\sigma_{z}$ is an irreducible representation of $G_u\times \tilde{v}$ such that $\sigma_{z}|_K$ is trivial on $K$ and $\sigma_{z}|_{G_u}$ has a spectral gap of $r_0$.
More precisely,  $\sigma_{z}=(\rho_{z}\otimes \chi_{z},\,\mathcal{H}_{z})$, where $(\rho_{z},\,\mathcal{H}_{z})$ is an irreducible representation of $G_u$ with a spectral gap $r_0$ and $\chi_{z}$ is a unitary character of $\tilde{v}$.

We then decompose $\theta$ and $\omega$ as
\begin{align*}
 \theta=\int_{Z} \theta_{z}d\mu(z)\quad\text{and}\quad \omega=\int_{Z} \omega_{z}d\mu(z)
\end{align*}
where $\theta_{z}\in \mathcal{H}^{s_1}_{z}$, $\omega_{z}\in \mathcal{H}^{s}_{z}$ for almost all $z$  (with respect to $\mu$).

We note that $v$ acts as a constant $\lambda_{z}\in \CC$ on each $\mathcal{H}_{z}$.

\smallskip
\eqref{for:141}: For any vector $v\in\text{Lie}(H)$, from the discussion at the beginning of the proof, we see that
\begin{align*}
 \big(\mathcal{E}_\iota(v^j\theta)\big)_z&=\lambda_{z}^j\big(\mathcal{E}_\iota(\theta)\big)_z=\big(v^j\mathcal{E}_\iota(\theta)\big)_z\quad\text{and}\\
 \big(\mathcal{D}_\iota(v^j\theta)\big)_z&=\lambda_{z}^j\big(\mathcal{D}_\iota(\theta)\big)_z=\big(v^j\mathcal{D}_\iota(\theta)\big)_z
\end{align*}
for almost all $z\in Z$. This implies the result.

\smallskip
\eqref{for:47}: For any vector $v\in\text{Lie}(H)$, from the discussion at the beginning of the proof, the equation $u\theta=\omega$ can be decomposed as
\begin{align*}
 u\theta_{z}=\omega_{z},\qquad \text{a.e. } z\in Z.
\end{align*}
Let $\Lambda=(I-v^2)^{\frac{1}{2}}$. We note that $\Lambda$ acts as a constant $\tau_{z}\in \RR^+$ on each $\mathcal{H}_{z}$. Hence we have
\begin{align*}
 u(\Lambda^{t}\theta_{z})=\Lambda^{t}\omega_{z},\qquad \text{a.e. } z\in Z,\,\,\forall\,t\geq0.
\end{align*}
By assumption $\theta_{z}\in \mathcal{H}_{z}$ for almost all $z$. Thus $\Lambda^{t}\theta_{z}\in \mathcal{H}_{z}$ for almost all $z$ and any $t\geq0$.  It follows from \eqref{for:186} of Theorem \ref{th:3}  that
\begin{align*}
 \norm{\Lambda^{t}\theta_{z}}\leq C_{r_0}\norm{\Lambda^{t}\omega_{z}}_{G_u,\frac{3}{2}}\leq C_{t,r_0}\norm{\omega_z}_{\{G_u,H\},t+\frac{3}{2}},\qquad \text{a.e. } z\in Z,
\end{align*}
for any $0\leq t\leq s-\frac{3}{2}$, which gives
\begin{align}\label{for:31}
 \norm{\Lambda^{t}\theta}\leq C_{t,r_0}\norm{\omega}_{\{G_u,H\},t+\frac{3}{2}}
\end{align}
for any $0\leq t\leq s-\frac{3}{2}$. This shows that $\theta\in \mathcal{H}_{\tilde{v}}^{s-\frac{3}{2}}$.  Then \eqref{for:47} follows from \eqref{for:31} and Theorem \ref{th:4}.

\smallskip

\eqref{for:185}:
 It follows from \eqref{for:186} of Theorem \ref{th:3} that
\begin{align*}
 \norm{\theta}_{G_u,t}\leq C_{t,r_0}\norm{\omega}_{G_u,t+\frac{3}{2}}
\end{align*}
for any $0\leq t\leq s_1$.  Then \eqref{for:185} follows from the above estimate, \eqref{for:31} and Theorem \ref{th:4}.

\end{proof}

In the following lemma, by the ``centralizer trick" we extend the smoothness of $\mathcal{D}^\iota(\omega)$, $\mathcal{E}_\iota(\omega)$, as well as the solution $\theta$ to all directions commuting with the $G_u$. Moreover, we  show that they are partially tame (with respect to $\omega$) along all directions commuting with $G_u$.

\begin{lemma}\label{le:1}
Suppose $\omega\in \mathcal{H}_{S_{0}}^s$, $s\geq 1$,  then:
\begin{enumerate}

\item\label{for:50} for any $\iota\in\NN$ and $0\leq t\leq s-1$
\begin{align*}
 \norm{\mathcal{D}^\iota(\omega)}_{S_{0},t}\leq \norm{\omega}_{S_0,t+1};
\end{align*}

  \item\label{for:55} if $\iota\geq 3$ and $s\geq 3+\frac{\iota}{2}$, then for any subgroup $L$ of $S_0$ containing $G_u$, we have
  \begin{align*}
   \norm{\mathcal{E}_\iota(\omega)}_{L,t}\leq C_{t}\norm{\omega}_{L,t+3+\frac{\iota}{2}}
  \end{align*}
  if $0\leq t\leq s-3-\text{\small$\frac{\iota}{2}$}$;

  \smallskip
  \item\label{for:60} if $\mathcal{D}^\iota(\omega)=0$, $\iota\geq 3$ and $s\geq \frac{11}{2}+\frac{\iota}{2}$, the equation $u\theta=\omega+\mathcal{E}_\iota(\omega)$ has a solution $\theta\in \mathcal{H}_{S_0}^{s-\frac{11}{2}-\frac{\iota}{2}}$ satisfying $\mathcal{D}^\iota(\theta)=0$ with estimates
  \begin{align*}
   \norm{\theta}_{S_0,t}\leq C_{t} \norm{\omega}_{S_0,t+\frac{11}{2}+\frac{\iota}{2}}
  \end{align*}
if $0\leq t\leq s-\frac{11}{2}-\text{\small$\frac{\iota}{2}$}$;

\smallskip
 \item \label{for:11}if $\mathcal{D}^\iota(\omega)=\omega$ and $\iota\geq 5$, $s\geq \frac{5}{2}$, then equation $u\theta=\omega$ has a solution $\theta\in \mathcal{H}$ satisfying $\mathcal{D}^\iota(\theta)=\theta$ with estimates:
    \begin{align*}
\norm{Y^j\theta}_{G_u,t}\leq C_{j,t} \max_{0\leq i\leq j}\{\norm{Y^i\omega}_{G_u,t+\frac{3}{2}}\}
\end{align*}
for any $j\leq s-\frac{5}{2}$,  if $0\leq t\leq \min\{\frac{\iota}{2}-\text{\small$\frac{3}{2}$},s-\text{\small$\frac{3}{2}$}-j\}$, where $Y$ stands for $X_u$, $u$ or $Y\in \mathcal{C}(\mathfrak{g}_u)$;

\smallskip
\item\label{for:140} for any subgroup $H$ of $C(G_u)$, if $\omega,\,\theta\in \mathcal{H}_{\{H,G_u\}}^s$, $s\geq \text{\small$\frac{5}{2}$}$ satisfy the equation $u\theta=\omega$, then
\begin{align*}
\norm{\theta}_{\{H,G_u\},t}\leq C_{t} \norm{\omega}_{\{H,G_u\},t+\frac{5}{2}}
\end{align*}
if $0\leq t\leq s-\text{\small$\frac{5}{2}$}$.
\end{enumerate}
\end{lemma}
\begin{proof}
 \eqref{for:50}:  For any $v\in \mathcal{C}(\mathfrak{g}_u)$ we have
 \begin{align}
  \norm{v^j\mathcal{D}^\iota(\omega)}&\overset{\text{(a)}}{=}\norm{\mathcal{D}^\iota(v^j\omega)}\overset{\text{(b)}}{\leq} \norm{v^j\omega}\leq\norm{\omega}_{S_0,j}\label{for:2001}
 \end{align}
if $j\leq s$. Here in $(a)$ we use \eqref{for:141} of Lemma \ref{le:10}; $(b)$ we use \eqref{for:77} of Lemma \ref{le:8}.

 Recall that $S_{0}=\{G_u,\,C(G_u)\}$. Then \eqref{for:50} is a direct consequence of \eqref{for:2001}, \eqref{for:77} of Lemma \ref{le:8} and Theorem \ref{th:4}.

 \smallskip

 \eqref{for:55}: \eqref{for:54} of Lemma \ref{le:8} shows that
 \begin{align}\label{for:263}
   \norm{\mathcal{E}_\iota(\omega)}_t\leq C_{t,r_0}\norm{\omega}_{t+2+\frac{\iota}{2}}
  \end{align}
if $0\leq t\leq s-2-\text{\small$\frac{\iota}{2}$}$.
 For any $w\in \text{Lie}(L)\cap\mathcal{C}(\mathfrak{g}_u)$ we have
\begin{align}
 \norm{w^i\mathcal{E}_\iota(\omega)}&\overset{\text{(a)}}{=}\norm{\mathcal{E}_\iota(w^i\omega)}\overset{\text{(b)}}{\leq} \norm{w^i\omega}_{G_u,2+\frac{\iota}{2}}\leq\norm{\omega}_{L,2+i+\frac{\iota}{2}}, \label{for:2002}
\end{align}
if $i\leq s-\frac{\iota}{2}-2$. Here in $(a)$ we
use \eqref{for:141} of Lemma \ref{le:10} as \eqref{for:263} shows that $\mathcal{E}_\iota(\omega)\in \mathcal{H}$; in $(b)$ we use \eqref{for:54} of Lemma \ref{le:8}.

We note that $L=\{L\cap C(G_u),  G_u \}$. Then \eqref{for:55} follows from \eqref{for:2002}, \eqref{for:263} and Theorem \ref{th:4}.

\smallskip
\eqref{for:60}: By \eqref{for:16} of Lemma \ref{le:8}, the equation
\begin{align*}
 u\theta=\omega+\mathcal{E}_\iota(\omega)
\end{align*}
has a solution $\theta\in \mathcal{H}_{G_u}^{s-2-\frac{\iota}{2}}$.
  From \eqref{for:55}, we see that $\omega+\mathcal{E}_\iota(\omega)\in \mathcal{H}_{S_0}^{s-3-\frac{\iota}{2}}$. Then it follows from \eqref{for:185} of Lemma \ref{le:10}, we see that $\theta\in \mathcal{H}_{S_0}^{s-\frac{11}{2}-\frac{\iota}{2}}$ with estimates
\begin{align*}
\norm{\theta}_{S_0,t}&\overset{\text{(a)}}{\leq} C_{t} \norm{\omega+\mathcal{E}_\iota(\omega)}_{s_0,t+\frac{5}{2}}\leq C_{t} \norm{\omega}_{S_0,t+\frac{5}{2}}+C_{t} \norm{\mathcal{E}_\iota(\omega)}_{S_0,t+\frac{5}{2}}\\
&\overset{\text{(b)}}{\leq} C_{t,1}\norm{\omega}_{S_0,t+\frac{11}{2}+\frac{\iota}{2}}
\end{align*}
if $0\leq t\leq s-\frac{11}{2}-\text{\small$\frac{\iota}{2}$}$. Here in $(a)$ we use \eqref{for:185} of Lemma \ref{le:10}; in $(b)$ we use \eqref{for:55}.

\smallskip
\eqref{for:11}: By \eqref{for:19} of Lemma \ref{le:8} the equation
\begin{align*}
 u\theta=\omega
\end{align*}
has a solution $\theta\in \mathcal{H}$ satisfying $\mathcal{D}^\iota(\theta)=\theta$.
It follows from \eqref{for:47} of Lemma \ref{le:10} that $\theta\in \mathcal{H}_{C(G_u)}^{s-\frac{5}{2}}$. This implies that  $v^j\theta\in \mathcal{H}$ for any $j\leq s-\frac{5}{2}$ and $v\in \mathcal{C}(\mathfrak{g}_u)$.
We also note that
\begin{align*}
 u(v^j\theta)=v^j\omega,\qquad \forall \,v\in \mathcal{C}(\mathfrak{g}_u).
\end{align*}
It follows from  \eqref{for:19} of Lemma \ref{le:8} that
\begin{align*}
  \norm{v^j\theta}_{G_u,t}\leq C_t\norm{v^j\omega}_{G_u,t+\frac{3}{2}},\qquad j\leq s-\text{\tiny$\frac{5}{2}$},\,\,v\in \mathcal{C}(\mathfrak{g}_u)
\end{align*}
if $0\leq t\leq \min\{\frac{\iota}{2}-\text{\small$\frac{3}{2}$},s-\text{\small$\frac{3}{2}$}-j\}$. Then we get the estimate if $Y=v$. The estimates for $Y=X_u$ or $Y=u$ follow from \eqref{for:19} of Lemma \ref{le:8}.

\smallskip

\eqref{for:140}: A direct consequence of \eqref{for:185} of Lemma \ref{le:10}.

\end{proof}

The next result is the extended representation version of \eqref{for:140}. The proof is left for Appendix \ref{sec:31}.
\begin{corollary}\label{cor:6} Let $H$ be a subgroup  of $C(G_u)$.  Suppose $\Omega,\,\Theta\in \mathfrak{g}(\mathcal{H})_{\{H,G_u\}}^s$, $s\geq \text{\small$\frac{5}{2}$}\dim\mathfrak{g}$ satisfy the equation
\begin{align}\label{for:17}
 (u+\text{ad}_u)\Theta=\Omega.
\end{align}
Then for any $t\leq s-\text{\small$\frac{5}{2}$}\dim\mathfrak{g}$, we have
\begin{align*}
\norm{\Omega}_{\{H,G_u\},t}\leq C_{t} \norm{\Theta}_{\{H,G_u\},t+\text{\small$\frac{5}{2}$}\dim\mathfrak{g}}.
\end{align*}

\end{corollary}
 We call the equation $u\theta=\omega$ the \emph{reduced version} of \eqref{for:17} in $\mathcal{H}$.
To solve \eqref{for:17} in $\mathfrak{g}(\mathcal{H})$ we start from
the reduced version in $\mathcal{H}$. Then the results for the extended version are obtained
by applying induction on each Jordan block of $\text{ad}_u$. This is a standard scheme in treating equations in $\mathfrak{g}(\mathcal{H})$.

\section{Almost cocycle equation in extended representation}\label{sec:14}

In this section we obtain a splitting that is used to decompose the almost twisted cocycle into two parts: one part that is close to a twisted cocycle and another part that is small in size.
Results of this section can be viewed as the first step in obtaining
a globally smooth splitting in Section \ref{sec:22}, which is essential to apply the KAM scheme.

\subsection{Notations}\label{sec:24} Throughout this section, $(\pi, \mathcal{H})$ denotes a unitary representation of $\GG$ whose restriction to each simple factor of $\GG$ has a spectral gap.

\begin{enumerate}

  \item\label{for:200}   Fix $\phi\in \Phi$ and $u\in \mathfrak{u}_\phi\cap\mathfrak{g}^1$. We recall that
  \begin{align}\label{for:227}
   S_{0,u}=\{G_u,\,C(G_u)\}\quad\text{and}\quad S_{1,u}=\{G_u',\,C(G_u)\}.
  \end{align}
   For simplicity, we will write $S_{0}$ and $S_{1}$ in place of $S_{0,u}$ and $S_{1,u}$, respectively.

      Suppose
$v\in \mathfrak{g}^1$ is nilpotent. Also suppose $u$ and $v$ is a friendly pair (see \eqref{for:275} of Section \ref{sec:47}). This means that there exists $v'\in \mathcal{C}(\mathfrak{g}_{u})$ such that
 $\{v, v', X_v=[v,v']\}$ is a $\mathfrak{sl}(2,\RR)$ triple. It is clear that $G_v=G_{v,v'}\in C(G_u)$ and $G_u\in C(G_{v})$ (see \eqref{for:278} of Section \ref{sec:47}). Let
\begin{align}\label{for:267}
 L=\{C(G_v),G_v\}\cap S_0.
\end{align}
Then
 \begin{align}\label{for:270}
 \{G_u,\,G_v\}\subseteq L\subseteq S_0=\{C(G_u),G_u\}
\end{align}
\noindent \emph{Note.} The fact that
$u$ and $v$ are contained in a subalgebra of $\text{Lie}(L)$ which is isomorphic to $\mathfrak{sl}(2,\RR)\times \mathfrak{sl}(2,\RR)$ is essential to carry out the ``higher rank trick" (see the proofs of Lemma \ref{le:3} and \ref{le:9}) on the twisted cocycles of $u$ and $v$.

  \smallskip
  \item\label{for:199}  Set $\sigma=\text{\small$\frac{3}{2}$}\dim\mathfrak{g}$ (see Corollary \ref{le:6}), $\sigma_0=(8+\sigma)\sigma$ (see Proposition \ref{po:1}) and
  $\sigma_1=2\sigma+3$ (see Corollary \ref{le:6}).
\end{enumerate}

\begin{example}\label{ex:3} For $\GG=SL(n,\RR)$, $n\geq4$, if $\phi=L_1-L_2$, $u=\mathfrak{u}_{1,2}$ and $v=\mathfrak{u}_{3,4}$ (see Section \eqref{for:285} of \ref{sec:47}) then
$\text{Lie}(G_{v,v'})$ is spanned by $\mathfrak{u}_{3,4}$, $\mathfrak{u}_{4,3}$ and $\mathfrak{u}_{3,3}-\mathfrak{u}_{4,4}$; and
$S_0$, $S_1$, $S_v$ and $L$ are the sets of matrices in $SL(n,\RR)$ with the following forms respectively:
\begin{gather*}
  S_0=\begin{pmatrix}M_{2,2} & \vline & 0\\
\hline 0 &\vline &M_{n-2,n-2} \end{pmatrix}, \quad S_1=\begin{pmatrix} a & b & \vline & 0\\
0 & c & \vline & 0\\
 \hline  0& 0 &\vline & M_{n-2,n-2} \end{pmatrix}\\
 S_{0,v}=\begin{pmatrix} M_{2,2} & 0 & \vline & M_{2,n-4}\\
0 & M_{2,2} & \vline & 0\\
 \hline  M_{n-4,2}& 0 &\vline & M_{n-4,n-4} \end{pmatrix},\quad L=\begin{pmatrix} M_{2,2} & 0 & \vline & 0\\
0 & M_{2,2} & \vline & 0\\
 \hline  0& 0 &\vline & M_{n-4,n-4} \end{pmatrix}
\end{gather*}
where $M_{m,k}$ denotes the set of $m\times k$ matrices and $a,b,c\in\RR$. More precisely,
\begin{align*}
 S_0&=\{(g_{i,j})\in SL(n,\RR): g_{1,j}=g_{2,j}=g_{j,1}=g_{j,2}=0,\quad j\geq3\},\\
 S_1&=\{(g_{i,j})\in SL(n,\RR): g_{2,1}=g_{1,j}=g_{2,j}=g_{j,1}=g_{j,2}=0,\quad j\geq3\},\\
 S_{0,v}&=\{(g_{i,j})\in SL(n,\RR): g_{3,j}=g_{4,j}=g_{j,3}=g_{j,4}=0,\quad j\neq 3,4\}.
\end{align*}

\end{example}

\begin{example} For $\GG=SL(n,\RR)\times SL(n,\RR)$, $n\geq4$, if $\phi=L_1-L_2$, $u=\mathfrak{u}_{1,2}$ in the first copy of $SL(n,\RR)$ and $v=\mathfrak{u}_{3,4}$ in the second copy,  then
$\text{Lie}(G_{v,v'})$ is spanned by $\mathfrak{u}_{3,4}$, $\mathfrak{u}_{4,3}$ and $\mathfrak{u}_{3,3}-\mathfrak{u}_{4,4}$ in the second copy; and
$S_0$, $S_1$, $S_{0,v}$ and $L$ are the sets of matrices in $SL(n,\RR)$ with the following forms respectively:
\begin{gather*}
  S_0=\begin{pmatrix}M_{2,2} & \vline & 0\\
\hline 0 &\vline &M_{n-2,n-2} \end{pmatrix}\times SL(n,\RR), \\
S_1=\begin{pmatrix} a & b & \vline & 0\\
0 & c & \vline & 0\\
 \hline  0& 0 &\vline & M_{n-2,n-2} \end{pmatrix}\times SL(n,\RR)
 \end{gather*}
 \begin{gather*}
 S_{0,v}=SL(n,\RR)\times \begin{pmatrix} M_{2,2} & 0 & \vline & M_{2,n-4}\\
0 & M_{2,2} & \vline & 0\\
 \hline  M_{n-4,2}& 0 &\vline & M_{n-4,n-4} \end{pmatrix},\\
 L=\begin{pmatrix}M_{2,2} & \vline & 0\\
\hline 0 &\vline &M_{n-2,n-2} \end{pmatrix}\times\begin{pmatrix} M_{2,2} & 0 & \vline & M_{2,n-4}\\
0 & M_{2,2} & \vline & 0\\
 \hline  M_{n-4,2}& 0 &\vline & M_{n-4,n-4} \end{pmatrix}.
\end{gather*}

\end{example}

\subsection{Main result} The following result gives partially tame estimates (see \eqref{for:272} of Section \ref{sec:47}) for the Sobolev norms of an approximate solution to the twisted almost cocycle
equation of friendly $u$ and $v$ as described in \eqref{for:200} of Section \ref{sec:24}.

Before stating Proposition 7.1, we recall several key definitions given in \eqref{for:227} and \eqref{for:267}.

\begin{proposition}\label{po:1} Suppose $u$ and $v$ is a friendly pair (see \eqref{for:275} of Section \ref{sec:47}).  For any $\Omega,\,\Psi,\,\mathfrak{w}\in \mathfrak{g}(\mathcal{H})_{S_{0,u}}^\infty$ satisfying the equation
\begin{align}\label{for:12}
(u+\textrm{ad}_u)\Omega-(v+\textrm{ad}_{v})\Psi=\mathfrak{w},
\end{align}
there exists $\eta\in \mathfrak{g}(\mathcal{H})_{S_1}^{\infty}$ with estimates
\begin{align}\label{for:73}
 \norm{\eta}_{S_1,t}\leq C_{t}\norm{\Psi}_{S_0,t+\sigma_0}
\end{align}
for any $t\geq0$, such that
\begin{align*}
 \Psi&=(u+\textrm{ad}_u)\eta+\mathcal{R}_1\qquad\text{and}\\
 \Omega&=(v+\textrm{ad}_v)\eta+\mathcal{R}_2
\end{align*}
where $\mathcal{R}_1,\,\mathcal{R}_2\in \mathfrak{g}(\mathcal{H})_{L}^{\infty}$ with estimates
\begin{align}\label{for:281}
\norm{\mathcal{R}_1,\,\mathcal{R}_2}_{L\cap S_1,t}&\leq C_{t,v}\norm{\mathfrak{w}}_{S_0,t+\sigma_0}
\end{align}
for any $t\geq0$ ($\sigma_0$ and $L$ are defined in Section \ref{sec:24}).
\end{proposition}
\emph{Note}. Proposition \ref{po:1} implies that  if $\mathfrak{w}=0$, then $\mathcal{R}_1=\mathcal{R}_2=0$. This implies twisted cocycle rigidity over $u$ and $v$,
which means that under these conditions the twisted cohomological equations
\begin{align*}
 \Psi=(u+\textrm{ad}_u)\eta\quad\text{and}\quad \Omega=(v+\textrm{ad}_v)\eta
\end{align*}
have a common  solution $\eta$ simultaneously without any residual errors.

The assumption that $u$ and $v$ are contained in a subalgebra isomorphic to $\mathfrak{sl}(2,\RR)\times \mathfrak{sl}(2,\RR)$ is a necessary condition (see counterexamples in \cite{tanis2017cohomological} and \cite{W1}).

\begin{remark} We emphasize that $\eta$, $\mathcal{R}_1$ and $\mathcal{R}_2$  may be only $L^2$
vectors, even though they possess partial smoothness.
\end{remark}

\subsubsection{Proof strategy} The proof of Proposition \ref{po:1} relies on the constructions in Section \ref{sec:16} and the  ``higher rank trick". The proof consists of three parts:

1. We consider equation \eqref{for:128}, the reduced version of \eqref{for:12} in $\mathcal{H}$  at first. We split \eqref{for:128} into two parts, one inside $(\mathcal{D}^l)_u(\mathcal{H})$ (see Lemma \ref{le:3}), the other inside $\ker((\mathcal{D}^\iota)_u)$ (see Lemma \ref{le:9}).

2. We solve the almost coboundary equation for $u$ in $(\mathcal{D}^\iota)_u(\mathcal{H})$ and $\ker((\mathcal{D}^\iota)_u)$ respectively. By using the  ``higher rank trick" we show that the almost twisted coboundary equation for $v$ are solved simultaneously; moreover, the errors are partially tame (with respect to $\psi$).

3.  We extend the  results to the extended representation in Section \ref{sec:28}, which leads to the proof of the proposition.

\subsection{Almost cocycle equation in $\mathcal{H}$} In this part, we study the almost cocycle equation
\begin{align}\label{for:128}
 v\omega-u\xi=\psi
\end{align}
where $\omega,\,\xi,\,\psi\in (\mathcal{D}^\iota)_u(\mathcal{H})$ or $\ker((\mathcal{D}^\iota)_u)$. For simplicity, we will write
$\mathcal{D}^{\iota}$ instead of $(\mathcal{D}^{\iota})_u$,
with the understanding that the operator is associated with \(u\).

\begin{lemma}\label{le:3} Suppose $\omega,\,\xi,\,\psi\in \mathcal{H}_{S_0}^{\infty}$ and $\iota\geq5$. If $\mathcal{D}^\iota(\omega)=\omega$, $\mathcal{D}^\iota(\xi)=\xi$, $\mathcal{D}^\iota(\psi)=\psi$ and satisfy equation \ref{for:128}, then there exists $\theta\in \mathcal{H}_{S_1}^{\infty}$ satisfying $\mathcal{D}^\iota(\theta)=\theta$ with estimates
\begin{align}\label{for:206}
 \norm{\theta}_{S_1,t}\leq C_{t}\norm{\omega}_{S_0,t+\text{\tiny$\frac{3}{2}$}}
\end{align}
for any $t\geq0$, such that
\begin{align*}
 \omega=u\theta,\quad\text{and}\quad \xi=v\theta+\mathcal{R}
\end{align*}
where $\mathcal{R}\in \mathcal{H}_{S_1}^{\infty}$ with estimates
\begin{align}\label{for:207}
 \norm{\mathcal{R}}_{S_1,t}\leq C_{t}\norm{\psi}_{S_0,t+\text{\tiny$\frac{3}{2}$}},\quad t\geq \text{\small$0$}.
\end{align}

\end{lemma}
\begin{proof} Since $\mathcal{D}^\iota(\omega)=\omega$, by \eqref{for:11} of Lemma \ref{le:1} we see that the equation
\begin{align}\label{for:127}
 u\theta=\omega
\end{align}
has a solution $\theta\in \mathcal{H}$ satisfying $\mathcal{D}^\iota(\theta)=\theta$ with estimates
      \begin{align*}
\norm{Y^j\theta}\leq C_{j,t} \max_{0\leq i\leq j}\{\norm{Y^i\omega}_{G_u,\frac{3}{2}}\}
\end{align*}
for any $j\geq0$, where $Y$ stands for $X_u,\,u$ or $Y\in \mathcal{C}(\mathfrak{g}_u)$. Then \eqref{for:206} follows from the above estimates and Theorem \ref{th:4}.

Let $\mathcal{R}=\xi-v\theta$. Next, we use the \emph{higher rank trick} to show that $\mathcal{R}$ is comparable to $\psi$. By substituting from \eqref{for:127} the expression for $\omega$ into \eqref{for:128}, we have
\begin{align*}
 vu\theta-u\xi=\psi\overset{\text{(1)}}{\Longrightarrow}uv\theta-u\xi=\psi\overset{\text{(2)}}{\Longrightarrow} u\mathcal{R}=-\psi.
\end{align*}
Here in $(1)$ we use $[v,u]=0$; in $(2)$ we set $\mathcal{R}=\xi-v\theta$.

Since $\mathcal{D}^\iota(\psi)=\psi$,  by applying  \eqref{for:11} of Lemma \ref{le:1} to the equation
\begin{align}\label{for:276}
 u\mathcal{R}=-\psi
\end{align}
we have
\begin{align*}
\norm{Y^j\mathcal{R}}\leq C_{j,t} \max_{0\leq i\leq j}\{\norm{Y^i\psi}_{G_u,\frac{3}{2}}\},
\end{align*}
for any $j\geq0$, where $Y$ stands for $X_u,\,u$ or $Y\in \mathcal{C}(\mathfrak{g}_u)$. Then \eqref{for:207} follows  from the above estimates and Theorem \ref{th:4}.
\end{proof}
\begin{remark}  Since \eqref{for:127} always has a solution $\theta$, the error $\mathcal{R}$ from solving the $v$-almost equation $\xi=v\theta+\mathcal{R}$ is estimated by solving the $u$-coboundary of $-\psi$ (see \eqref{for:276}). As a result,  $\mathcal{R}$ has partially tame estimates (with respect to $\psi$) on $S_1$ (see \eqref{for:207}).
\end{remark}

\begin{lemma}\label{le:9} Suppose $\omega,\,\xi,\,\psi\in \mathcal{H}_{S_0}^{\infty}$ and $\iota\geq3$. If $\mathcal{D}^\iota(\omega)=0$, $\mathcal{D}^\iota(\xi)=0$, $\mathcal{D}^\iota(\psi)=0$ and satisfy
\begin{align}\label{for:2003}
 v\omega-u\xi=\psi,
\end{align}
then there exists $\theta\in \mathcal{H}_{S_0}^{\infty}$ satisfying $\mathcal{D}^\iota(\theta)=0$ with estimates
\begin{align*}
 \norm{\theta}_{S_0,t}\leq C_{t}\norm{\omega}_{S_0,t+6+\text{\small$\frac{l}{2}$}}
\end{align*}
for any $t\geq0$, such that
\begin{align*}
 \omega&=u\theta+\mathcal{R}_1,\quad\text{and}\quad \xi=v\theta+\mathcal{R}_2
\end{align*}
where $\mathcal{R}_1,\,\mathcal{R}_2\in \mathcal{H}_{L}^{\infty}$ with estimates
\begin{align*}
 \norm{\mathcal{R}_1,\,\mathcal{R}_2}_{L,t}\leq C_{t}\norm{\psi}_{L, t+6+\text{\small$\frac{l}{2}$}},\quad t\geq0.
\end{align*}

\end{lemma}
\begin{remark} From the proof we see that $\mathcal{R}_1$ is estimated by solving the $v$-coboundary (see \eqref{for:35}) and
$\mathcal{R}_2$ is estimated by solving the $u$-coboundary (see \eqref{for:277}). \eqref{for:140} of Lemma \ref{le:1} shows that
both $\mathcal{R}_1$ and $\mathcal{R}_2$ are partially tame on
\begin{align*}
L=\{C(G_v),G_v\}\cap \{C(G_u),G_u\}=S_{0,v}\cap S_{0,u}
\end{align*}
 (see \eqref{for:267} of Section \ref{sec:24}).
\end{remark}
\begin{proof} By \eqref{for:60} of Lemma \ref{le:1} we see that the equation
\begin{align}\label{for:34}
  u\theta=\omega+\mathcal{E}_\iota(\omega)
\end{align}
has a solution $\theta\in \mathcal{H}_{S_0}^{\infty}$ satisfying $\mathcal{D}^\iota(\theta)=0$ with estimates
\begin{align*}
 \norm{\theta}_{S_0,t}\leq C_{t} \norm{\omega}_{S_0,t+\frac{11}{2}+\frac{\iota}{2}},\qquad t\geq0.
\end{align*}
Set $\mathcal{R}_1=\mathcal{E}_\iota(\omega)$. The above inequality implies that $\mathcal{R}_1$ is comparable to $\omega$ on $S_0$. Next, we use the \emph{higher rank trick} to show that $\mathcal{R}_1$ is also comparable to $\psi$.

 From \eqref{for:2003} we see that
\begin{align*}
 \mathcal{E}_\iota(v\omega)-\mathcal{E}_\iota(u\xi)=\mathcal{E}_\iota(\psi)\overset{\text{(1)}}{\Longrightarrow}v(\mathcal{E}_\iota(\omega))=\mathcal{E}_\iota(\psi).
\end{align*}
Here in $(1)$ we use \eqref{for:141} of Lemma \ref{le:10} and \eqref{for:142} of Lemma \ref{le:8}.

Next, we use the equation
\begin{align}\label{for:35}
 v\mathcal{R}_1=\mathcal{E}_\iota(\psi)
\end{align}
to estimate $\norm{\mathcal{R}_1}_{L,t}$. Recall \eqref{for:267} and \eqref{for:270}  of Section \ref{sec:24}:
\begin{align*}
 G_v\subseteq L\subseteq \{C(G_v),G_v\}.
\end{align*}
We also note that
\begin{align}\label{for:201}
 \mathcal{R}_1=\mathcal{E}_\iota(\omega),\,\mathcal{E}_\iota(\psi)\overset{\text{(1)}}{\in} \mathcal{H}_{S_0}^{\infty}\overset{\text{(2)}}{\subseteq} \mathcal{H}_{L}^{\infty}.
\end{align}
Here in $(1)$ we use \eqref{for:55} of Lemma \ref{le:1}; in $(2)$ we note that $L\subseteq S_0$. Then it follows from \eqref{for:140} of Lemma \ref{le:1}
that
\begin{align*}
 \norm{\mathcal{R}_1}_{L,t}\leq C_t\norm{\mathcal{E}_\iota(\psi)}_{L,t+\frac{5}{2}}
 \overset{\text{(3)}}{\leq}C_{t} \norm{\psi}_{L,t+\frac{11}{2}+\frac{\iota}{2}}
\end{align*}
 for any $t\geq0$. Here in $(3)$ we use
\eqref{for:55} of Lemma \ref{le:1}.

Set $\mathcal{R}_2=\xi-v\theta$.  Next, we use the \emph{higher rank trick} again to show that $\mathcal{R}_2$ is also comparable to $\psi$. By substituting from \eqref{for:34} the expression for $\omega$ into \eqref{for:2003}, we have
\begin{align*}
 &v(u\theta-\mathcal{E}_\iota(\omega))-u\xi=\psi\overset{\text{(4)}}{\Longrightarrow}u(v\theta-\xi)=\psi+\mathcal{E}_\iota(\psi).
 \end{align*}
 Here in $(4)$ we use $[v,u]=0$ (see \eqref{for:200} of Section \ref{sec:24}).

 Next, we use the equation
 \begin{align}\label{for:277}
  -u\mathcal{R}_2=\psi+\mathcal{E}_\iota(\psi)
 \end{align}
to estimate $\norm{\mathcal{R}_2}_{L,t}$. Recall that $\mathcal{R}_2=\xi-v\theta\in \mathcal{H}_{S_0}^{\infty}$ and $\psi+\mathcal{E}_\iota(\psi)\in \mathcal{H}_{S_0}^{\infty}$ (see \eqref{for:201}). We also recall \eqref{for:270} of Section \ref{sec:24}:
 \begin{align*}
 G_u\subseteq L\subseteq \{C(G_u),G_u\}.
\end{align*}
It follows from \eqref{for:140} of Lemma \ref{le:1} that
 \begin{align*}
   \norm{\mathcal{R}_2}_{L,t}=\norm{v\theta-\xi}_{L,t}\leq C_t\norm{\psi+\mathcal{E}_\iota(\psi)}_{L,t+\frac{5}{2}}
 \overset{\text{(5)}}{\leq}C_{t} \norm{\psi}_{L,t+\frac{11}{2}+\frac{\iota}{2}}
 \end{align*}
 if $t\geq0$. Here in  $(5)$ we use
\eqref{for:55} of Lemma \ref{le:1}.

Then the results follow from the above estimates.
\end{proof}

\subsection{Almost cocycle equation in extended representation}\label{sec:28}
In this part, we list the extended representation versions of Lemma \ref{le:3} and \ref{le:9}. It is natural to extend the linear operators $\mathcal{D}^\iota$ and $\mathcal{E}_\iota$ to $\mathfrak{g}(\mathcal{H})$ (see Section \ref{sec:2}) by acting on coordinate vectors.
The proofs follow a standard argument, which are left for Appendix \ref{sec:35} and \ref{sec:36} respectively.
\begin{corollary}\label{le:6} For any $\Omega,\,\Psi,\,\mathfrak{w}\in \mathfrak{g}(\mathcal{H})_{S_0}^{\infty}$, if $\mathcal{D}^\iota(\Omega)=\Omega$, $\mathcal{D}^\iota(\Psi)=\Psi$, $\mathcal{D}^\iota(\mathfrak{w})=\mathfrak{w}$ where $\iota\geq 2\sigma+3$ and satisfy the equation
\begin{align}\label{for:18}
(v+\textrm{ad}_v)\Omega-(u+\textrm{ad}_u)\Psi=\mathfrak{w},
\end{align}
then there exists $\eta\in \mathfrak{g}(\mathcal{H})_{S_1}^{\infty}$ with estimates
\begin{align}\label{for:88}
 \norm{\eta}_{S_1,t}\leq C_{t}\norm{\Omega}_{S_0,t+\sigma+\text{\tiny$\frac{3}{2}$}}
\end{align}
for any $t\geq0$, such that
\begin{align}\label{for:83}
 \Omega&=(u+\textrm{ad}_u)\eta\qquad\text{and}\notag\\
 \Psi&=(v+\textrm{ad}_{v})\eta+\mathcal{R},
\end{align}
where $\mathcal{R}\in \mathfrak{g}(\mathcal{H})_{S_1}^{\infty}$ with estimates
\begin{align}\label{for:89}
 \norm{\mathcal{R}}_{S_1,t}\leq C_{t}\norm{\mathfrak{w}}_{S_0,t+\sigma+\text{\tiny$\frac{3}{2}$}},\qquad t\geq0,
\end{align}
where $\sigma$ is defined in \eqref{for:199} of Section \ref{sec:24}.
\end{corollary}

\begin{corollary}\label{le:11} Suppose $\iota\geq3$, and any $\Omega,\,\Psi,\,\mathfrak{w}\in \mathfrak{g}(\mathcal{H})_{S_0}^{\infty}$, if $\mathcal{D}^\iota(\Omega)=0$, $\mathcal{D}^\iota(\Psi)=0$, $\mathcal{D}^\iota(\mathfrak{w})=0$, and satisfy the equation
\begin{align}\label{for:6}
(v+\textrm{ad}_v)\Omega-(u+\textrm{ad}_u)\Psi=\mathfrak{w},
\end{align}
then there exists $\eta\in \mathfrak{g}(\mathcal{H})_{S_0}^{\infty}$ satisfying $\mathcal{D}^\iota(\eta)=0$ with estimates
\begin{align*}
 \norm{\eta}_{S_0,t}\leq C_{t}\norm{\Omega}_{S_0,t+(6+\text{\tiny$\frac{\iota}{2}$})\sigma}
\end{align*}
for any $t\geq0$ ($\sigma$ is defined Corollary \ref{le:6}), such that
\begin{align*}
 \Omega&=(u+\textrm{ad}_u)\eta+\mathcal{R}_1,\qquad\text{and}\\
 \Psi&=(v+\textrm{ad}_{v})\eta+\mathcal{R}_2
\end{align*}
where $\mathcal{R}_1,\,\mathcal{R}_2\in \mathfrak{g}(\mathcal{H})_{L}^{\infty}$ with estimates
\begin{align*}
 \norm{\mathcal{R}_1,\,\mathcal{R}_2}_{L,t}\leq C_{t}\norm{\mathfrak{w}}_{L,t+(6+\text{\tiny$\frac{\iota}{2}$})\sigma},\qquad t\geq0.
\end{align*}

\end{corollary}
\subsection{Proof of Proposition \ref{po:1}} For any $\mathfrak{p}\in \mathfrak{g}(\mathcal{H})$ we have a decomposition $\mathfrak{p}=\mathfrak{p}^0+\mathfrak{p}^1$ where $\mathfrak{p}^0=\mathcal{D}^\iota(\mathfrak{p})$ and $\mathfrak{p}^1=\mathfrak{p}-\mathcal{D}^\iota(\mathfrak{p})$, where
$\iota=2\sigma+3$ (see  \eqref{for:199} of Section \ref{sec:24}). Recall that the operator $\mathcal{D}^\iota$ is associated with $u$ and that $u$ and $v$ form a friendly pair.
In particular, we have $v\in\text{Lie}(C(G_u))$. This fact allows us to apply \eqref{for:141} of Lemma \ref{le:10} and \eqref{for:2004} of Section \ref{sec:2} to obtain
\begin{align*}
  \mathcal{D}^\iota z=z\mathcal{D}^\iota,\qquad \mathcal{D}^\iota\circ \textrm{ad}_z=\textrm{ad}_z\circ \mathcal{D}^\iota
\end{align*}
where $z$ stands for $u$ or $v$.

Consequently, we have
\begin{align*}
 \mathcal{D}^\iota\big((z+\textrm{ad}_z)\mathfrak{p}\big)=(z+\textrm{ad}_z)(\mathcal{D}^\iota(\mathfrak{p})).
\end{align*}
The above discussion shows that \eqref{for:12} has a corresponding decomposition:
\begin{align}\label{for:52}
 (u+\textrm{ad}_u)\Omega^\delta-(v+\textrm{ad}_v)\Psi^\delta&=\mathfrak{w}^\delta,\quad \delta=0,\,1.
 \end{align}
 From \eqref{for:50} of Lemma \ref{le:1} we have: for $\delta=0,\,1$,
\begin{align}\label{for:53}
  \norm{\mathfrak{z}^\delta}_{S_0,t}\leq \norm{\mathfrak{z}}_{S_0,t+1},\qquad t\geq0,
\end{align}
where $\mathfrak{z}$ stands for $\Omega$, $\Psi$, $\mathfrak{w}$.

From \eqref{for:52} for $\delta=0$,  by Corollary \ref{le:6}  there exists $\eta^0\in \mathfrak{g}(\mathcal{H})_{S_1}^{\infty}$ with estimates
\begin{align*}
 \norm{\eta^0}_{S_1,t}\leq C_{t}\norm{\Psi^0}_{S_0,t+\sigma+\text{\tiny$\frac{3}{2}$}}\overset{\text{(1)}}{\leq}
 C_{t}\norm{\Psi}_{S_0,t+\sigma+\text{\tiny$\frac{5}{2}$}}
\end{align*}
for any $t\geq0$, such that
\begin{align*}
 \Psi^0=(u+\textrm{ad}_u)\eta^0\quad \text{and}\quad \Omega^0=(v+\textrm{ad}_{v})\eta^0+\mathcal{R}^0
\end{align*}
with estimates
\begin{align*}
 \norm{\mathcal{R}^0}_{S_1,t}\leq C_{t}\norm{\mathfrak{w}^0}_{S_0,t+\sigma+\text{\tiny$\frac{3}{2}$}}\overset{\text{(1)}}{\leq}
 C_{t}\norm{\mathfrak{w}}_{S_0,t+\sigma+\text{\tiny$\frac{5}{2}$}}
\end{align*}
for any $t\geq0$. Here in $(1)$ we use \eqref{for:53}.

By Corollary \ref{le:11}  there exists $\eta^1\in \mathfrak{g}(\mathcal{O})_{S_0}^{\infty}$ with estimates
\begin{align*}
 \norm{\eta^1}_{S_0,t}\leq C_{t}\norm{\Psi^1}_{S_0,t+(6+\text{\tiny$\frac{\iota}{2}$})\sigma}\overset{\text{(2)}}{\leq} C_{t}\norm{\Psi}_{S_0,t+(6+\text{\tiny$\frac{\iota}{2}$})\sigma+1}
\end{align*}
for any $t\geq0$, such that
\begin{align*}
 \Psi^1=(u+\textrm{ad}_u)\eta^1+\mathcal{R}^1_1\quad\text{and}\quad \Omega^1=(v+\textrm{ad}_{v})\eta^1+\mathcal{R}^1_2
\end{align*}
with estimates
\begin{align*}
 \norm{\mathcal{R}^1_1,\,\mathcal{R}^1_2}_{L,t}\leq C_{t}\norm{\mathfrak{w}^1}_{L,t+(6+\text{\tiny$\frac{\iota}{2}$})\sigma}\overset{\text{(2)}}{\leq} C_{t}\norm{\mathfrak{w}}_{S_0,t+(6+\text{\tiny$\frac{\iota}{2}$})\sigma+1}
\end{align*}
for any $t\geq0$. Here in $(2)$ we use \eqref{for:53}.  Set
\begin{align*}
 \eta=\eta^0+\eta^1,\quad \mathcal{R}_1=\mathcal{R}^0_1+\mathcal{R}^1_1,\quad \mathcal{R}_2=\mathcal{R}^0_2+\mathcal{R}^1_2.
\end{align*}
Also set $\sigma_0=(8+\sigma)\sigma$. Then the result is a direct consequence of the above analysis.

\section{Directional smoothing operators}\label{sec:4}
In this part we show a general construction of  smoothing operators. This part plays a crucial role in the construction of the approximation in Section \ref{sec:22}. In Section \ref{sec:30} we give the motivation to  construct ``directional" smoothing operators.  In Section \ref{sec:11} we present an equivalent construction by using group algebra and obtain Sobolev norms of these operators. Applications of these operators
are discussed in Section \ref{sec:10}.

\subsection{Motivation}\label{sec:30} We denote by $W^{2,q}(\RR^m)$ the Sobolev space of $L^2$
functions with $L^2$ weak partial derivatives up to order $q$. Fix a bump function $f$. We define smoothing operators $\pi(f\circ a^{-1})$, $a>0$ on $W^{2,0}(\RR^m)=L^2(\RR^m)$ as follows:
\begin{align}\label{for:62}
 \pi(f\circ a^{-1})(g)(x)=\text{\tiny$\frac{1}{(\sqrt{2\pi})^m}$}\int_{\RR^m} f(\text{\tiny$\frac{\chi}{a}$})\hat{g}(\chi)e^{\textrm{i}\chi\cdot x} d\chi.
\end{align}
where $\hat{g}(\chi)=\text{\tiny$\frac{1}{(\sqrt{2\pi})^m}$}\int_{\RR^m} g(x)e^{-\textrm{i}\chi\cdot x} dx$.

Assuming that the Sobolev space
$W^{2,q}$ is defined in the Fourier domain with the weight $(1+\norm{\chi}^2)^\frac{q}{2}$ we can express the norm of a function
$g$ as:
\begin{align}\label{for:2019}
 \norm{g}_{W^{2,q}}=\Big\|(\sum_{i=1}^m \chi_i^2+1)^{\frac{q}{2}}\hat{g}(\chi)\Big\|_{W^{2,0}},\qquad \chi=(\chi_1,\cdots,\chi_m)
\end{align}
It is easy to check that the following property holds:
\begin{enumerate}
  \item\label{for:45} $ \pi(f_1\circ a^{-1})\pi(f_2\circ a^{-1})=\pi\big((f_1f_2)\circ a^{-1} \big)$;
  \smallskip
  \item $\langle \pi(f\circ a^{-1})(g),\,g_1\rangle=\langle g,\,\pi(\bar{f}\circ a^{-1})(g_1)\rangle$, where $\bar{f}$ is the complex conjugate of $f$;

  \smallskip
  \item\label{for:2006} we have
\begin{align*}
  &\text{\tiny$\frac{\partial^{n_1+n_2+\cdots+n_m}}{\partial x_1^{n_1}\partial x_2^{n_2}\cdots \partial x_m^{n_m}}$}\big(\pi(f\circ a^{-1})(g)(x)\big)\notag\\
  &=\text{\tiny$\frac{1}{(\sqrt{2\pi})^m}$}\int_{\RR^m} f(\text{\tiny$\frac{\chi}{a}$})(\chi_1\textrm{i})^{n_1}(\chi_2\textrm{i})^{n_2}\cdots (\chi_m\textrm{i})^{n_m}\hat{g}(\chi)e^{\textrm{i}\chi\cdot x} d\chi\notag\\
  &=a^{n_1+n_2+\cdots+n_m}\pi(f^*\circ a^{-1})(g)(x)
\end{align*}
where $f^*(x_1,x_2,\cdots,x_m)=f(x)(x_1 \textrm{i})^{n_1}\dots (x_m \textrm{i})^{n_m}$.

It shows that $\pi(f\circ a^{-1})(g)\in W^{2,\infty}$; and the following estimates hold
\begin{align*}
 \norm{\pi(f\circ a^{-1})(g)}_{W^{2,p}}&\leq C_{p} a^{p}\norm{g}_{W^{2,0}}, \qquad \forall\,p\geq 0;
\end{align*}

\item\label{for:40}
If $1-f(\chi)=0$ whenever $\norm{\chi}\leq1$,  then for any $q\geq0$
\begin{align*}
\big\|\norm{\chi}^{-q}(1-f(\text{\tiny$\frac{\chi}{a}$}))\big\|_{L^\infty}\leq a^{-q}(\|f\|_{C^0}+1).
\end{align*}
Recalling \eqref{for:2019}, this bound allows us to estimate the smoothing error  for any $g\in W^{2,q}$:
\begin{align*}
&\norm{g-\pi(f\circ a^{-1})(g)}_{W^{2,0}}=\|(1-f(\text{\tiny$\frac{\chi}{a}$}))\hat{g}(\chi)\|_{W^{2,0}}\\
&=\Big\| \big(\norm{\chi}^{-q}(1-f(\text{\tiny$\frac{\chi}{a}$}))\big)\cdot \big(\norm{\chi}^{q}\hat{g}(\chi)\big)\Big\|_{W^{2,0}}\\
&\leq \big\|\norm{\chi}^{-q}(1-f(\text{\tiny$\frac{\chi}{a}$}))\big\|_{L^\infty}\cdot\big\| \norm{\chi}^{q}\hat{g}(\chi)\big\|_{W^{2,0}}\\
&\leq a^{-q}(\|f\|_{C^0}+1) \cdot\Big\| (\sum_{i=1}^m \chi_i^2+1)^{\frac{q}{2}}\hat{g}(\chi)\Big\|_{W^{2,0}}\\
&= a^{-q}(\|f\|_{C^0}+1)\cdot \norm{g}_{W^{2,q}}\\
&= C_{f} a^{-q}\norm{g}_{W^{2,q}},
\end{align*}
where $C_{f}=(\|f\|_{C^0}+1)$ is a constant is a constant that depends only on the bump function
$f$. Importantly, $C_{f}$ does not depend on the parameters
$q$ and $a$.
\begin{remark}\label{re:14} For general smoothing operators, the constants involved in  estimating the error from smoothing typically depend on the order $q$, as seen in \eqref{for:117} of Section \ref{sec:26}. However, the above estimates show that by carefully choosing and designing specific smoothing operators, the constants in the error estimates can be made independent of $q$.   This is a core observation in constructing the ``directional" smoothing
operators in the subsequent part.
\end{remark}
\end{enumerate}
In this section, we will generalize the construction of smoothing operators by truncation to a locally compact abelian group $S$.

\subsection{Notations} \label{sec:39} Throughout this section, we fix a Lie group $H$ and an abelian closed subgroup $S$ of $H$ which is isomorphic to $\RR^m$. Let $(\pi,\mathcal{H})$ be a unitary representation of $H$.

\begin{enumerate}
  \item\label{for:279}  Set $\mathfrak{h}=\text{Lie}(H)$.  Fix a set of basis $\mathfrak{u}=\{\mathfrak{u}_1,\cdots,\mathfrak{u}_m\}$ of $\text{Lie}(S)$. We recall a vector $u\in \mathfrak{h}$ is \emph{nilpotent} if $\text{ad}_u$ is
nilpotent. We say that a subgroup of $H$ is \emph{unipotent} if its Lie algebra is (linearly) spanned by nilpotent vectors.

\smallskip
  \item\label{for:239} For any $a>0$ and a function $f:\RR^n\to\CC$, we denote $f(\text{$\frac{t}{a}$})$ by $(f\circ a^{-1})(t)$.

  \smallskip
  \item\label{de:2} For $f\in C^\infty(\RR^n)$, we say:
   \begin{itemize}

     \item $f$ is \emph{standard} if $0\leq f\leq 1$,  $f(t)=1$ for $\norm{t}\leq 1$ and $f(t)=0$ for $\norm{t}\geq 2$, where $\norm{t=(t_1,\cdots,t_n)}=\max_{1\leq i\leq n}|t_i|$;

         \smallskip
     \item $f$ is \emph{related to $\mathfrak{u}=\{\mathfrak{u}_1,\cdots,\mathfrak{u}_m\}$} if $f$ is standard and $n=m$;

     \smallskip

     \item for a subset $\mathfrak{o}=\{\mathfrak{u}_{i_1,},\mathfrak{u}_{i_2,},\cdots,\mathfrak{u}_{i_j}\}$ of $\mathfrak{u}$, we see that $f$ is \emph{free} on $\mathfrak{o}$
     if $\partial_{t_{i_l}}f=0$, $1\leq l\leq j$.
   \end{itemize}

  \smallskip

  \item\label{for:240} Let $\ZZ_{0,+}=\{0\}\cup \NN$. For $j=(j_1,\cdots, j_m)\in\ZZ_{0,+}^m$, let $|j|=\sum_i j_i$. Define
  \begin{gather*}
    f^{[j]}:=\partial^{\text{\tiny$j_1$}}_{\text{\tiny$t_1$}}\cdots \partial^{\text{\tiny$j_m$}}_{\text{\tiny$t_m$}} f, \qquad
    (\text{ad}_{\mathfrak{u}})^j:=\text{ad}^{\text{\tiny$j_1$}}_\mathfrak{u_1}\cdots \text{ad}^{\text{\tiny$j_m$}}_\mathfrak{u_m},\\
    t^j:=t_1^{j_1}\cdots t_m^{j_m},
  \end{gather*}
  where $t=(t_1,\cdots,t_m)\in \RR^m$;

\smallskip
  \item\label{for:241} We set
\begin{align*}
 \tilde{\mathcal{S}}(\RR^m)=\{f\in \text{\small$C^{\infty}(\RR^m)$}: f^{[j]}\in \text{\small$L^\infty(\RR^m)$}, \forall\,j\in \ZZ_{0,+}^m\}
\end{align*}
and we define the norm as
\begin{align*}
 \norm{f}_{\tilde{\mathcal{S}}(\RR^m),n}=\max_{j\in \ZZ_{0,+}^m,\,|j|=n}\{ \norm{f^{[j]}}_{\text{\small$L^\infty(\RR^m)$}}\},\quad \forall\,f\in \tilde{\mathcal{S}}(\RR^m).
\end{align*}
\end{enumerate}

\subsection{Main results}\label{sec:52}
In Section \ref{sec:3}, for any $f\in L^\infty(\RR^m)$ we define a linear operator $\pi_{\mathfrak{u}}(f)$ on $\mathcal{H}$ satisfying the following properties:

\begin{sect}
For any $f_1,\,f_2\in L^\infty(\RR^m)$
\begin{align}\label{for:196}
 \pi_{\mathfrak{u}}(f_1)\pi_{\mathfrak{u}}(f_2)=\pi_{\mathfrak{u}}(f_1f_2);
\end{align}
and
\begin{align}\label{for:198}
 \bigl\langle \pi_{\mathfrak{u}}(f)\xi,\eta\bigl\rangle=\bigl\langle \xi,\pi_{\mathfrak{u}}(\bar{f})\eta\bigl\rangle,\qquad \xi,\,\eta\in \mathcal{H},
\end{align}
where $\bar{f}$ is the complex conjugate of $f$.

If $X\subseteq \RR^m$ is a Borel set and $I_X$ denotes the characteristic function of $X$, from \eqref{for:196} and \eqref{for:198} we see that $\pi_{\mathfrak{u}}(I_X)$ is idempotent and self-adjoint, i.e., an orthogonal projection onto a subspace of $\mathcal{H}$. Thus the assignment $X\to \pi_{\mathfrak{u}}(I_X)$ is a \emph{projection-value measure};
\end{sect}

\begin{sect}

 (\eqref{ob:3} of Lemma \ref{le:2}) if $v$ commutes with $\text{Lie}(S)$, then $v\pi_{\mathfrak{u}}(f)=\pi_{\mathfrak{u}}(f) v$.

\end{sect}
\begin{sect}
 (Lemma \ref{le:7}) Suppose $a>0$. Then:
\begin{enumerate}
  \item If $\xi\in \mathcal{H}$ and $f$ is Schwartz,  then $\pi_{\mathfrak{u}}(f\circ a^{-1})\xi\in \mathcal{H}_{S}^\infty$ with estimates
\begin{align*}
 \norm{\pi_{\mathfrak{u}}(f\circ a^{-1})\xi}_{S,l}\leq C_{f,l}a^l\norm{\xi},\qquad \forall\,l\geq0.
\end{align*}
\item If $\xi\in \mathcal{H}_{S}^s$ and $f$ is related to $\mathfrak{u}$,  then
\begin{align*}
 \norm{\xi-\pi_{\mathfrak{u}}(f\circ a^{-1})\xi}\leq C_{f} a^{-s}\norm{\xi}_{S,s}, \qquad \forall\,s\geq0.
\end{align*}

\end{enumerate}
\eqref{for:63} of Lemma \ref{le:7}  shows that the $\pi_{\mathfrak{u}}(f\circ a^{-1})$ operators provide smoothness along $S$-directions. This is the reason to call them ``directional" smoothing operators. We emphasize that the constant $C_f$ in \eqref{for:145} of Lemma \ref{le:7}  is independent of $s$. This fact will be used  for  subsequent parts.

\end{sect}

\begin{sect}
 (Corollary \ref{cor:1}) Let $S$ be unipotent in $H$. Suppose $\xi\in \mathcal{H}^s$, $s\geq 0$, $a\geq1$ and $0\leq\ell\leq s$. Then:
\begin{enumerate}
\item\label{for:264} if $f\in \tilde{\mathcal{S}}(\RR^m)$, then
$\pi_{\mathfrak{u}}(f\circ a^{-1})\xi\in \mathcal{H}^s$ with estimates
\begin{align*}
  \norm{\pi_{\mathfrak{u}}(f\circ a^{-1})\xi}_r\leq C_{f,r}\norm{\xi}_r, \qquad \forall\, 0\leq r\leq s;
\end{align*}
\item \label{for:146} if $f$ is related to $\mathfrak{u}$, then for any $0\leq r\leq \ell$
  \begin{align*}
   \norm{\xi-\pi_{\mathfrak{u}}(f\circ a^{-1})\xi}_r\leq C_{\ell,f}\big(a^{-s}\norm{\xi}_s\big)^{\text{\tiny$1-\frac{r}{\ell}$}}\norm{\xi}^{\text{\tiny$\frac{r}{\ell}$}}_{\ell}.
  \end{align*}
\end{enumerate}
\eqref{for:264} of Corollary \ref{cor:1} shows that if $S$ is unipotent, then $\pi_{\mathfrak{u}}(f\circ a^{-1})$ is a global smoothing operator and the estimates
are similar
to those of the standard smoothing operators (see Section \ref{sec:26}). We point out that  the constant in \eqref{for:146}  is independent of $s$.

 \eqref{for:146} provides the estimate for the error coming from the smoothing. Let $s=\ell$. Then \eqref{for:145} becomes
\begin{align}\label{for:242}
   \norm{\xi-\pi_{\mathfrak{u}}(f\circ a^{-1})\xi}_r\leq C_{\ell,f}a^{-(\ell-r)}\norm{\xi}_{\ell}.
  \end{align}
\eqref{for:264} and \eqref{for:242} show that if $a\geq1$, the estimates of the directional smoothing operators
are similar to those of the standard smoothing operators (see Section \ref{sec:26}).

In \eqref{for:146} we use Sobolev orders of $s$ and $\ell$ simultaneously to estimate the error from smoothing. This results in the estimates more complex
than
the classic ones that use only the order $\ell$. The reason for doing so will be explained  in Section \ref{sec:8}.

\end{sect}

\begin{sect}\label{sect:1}
 Let $(\pi,\mathcal{H})$ be a unitary representation of $\GG$. We recall notations in \eqref{for:287} of Section \ref{sec:47}.
Choose $f_1$ related to $\mathfrak{V}$ (see \eqref{de:2} of Section \ref{sec:39}), $f_2$ related to $\mathfrak{W}$ and $f_3$ related to $\mathfrak{V}_1$. Then
$f_1$ is also related to $\mathfrak{U}$, and $f_3$ is also related to $\mathfrak{U}_i$ and $\mathfrak{V}_i$, $i=1,2$.

(Corollary \ref{cor:9}) Suppose $a>1$, $\xi\in \mathcal{H}^s_{S_{1,U}}$, $s\geq 0$. Set
\begin{gather*}
  \xi'=\pi_{\mathfrak{U}}(f_1\circ a^{-1})\pi_{\mathfrak{W}}(f_2\circ a^{-1})\pi_{\mathfrak{V}}(f_1\circ a^{-1})\xi;\\
  \xi''=\pi_{\mathfrak{U}^1}(f_3\circ a^{-1})\pi_{\mathfrak{U}^2}(f_3\circ a^{-1})\pi_{\mathfrak{W}}(f_2\circ a^{-1})\pi_{\mathfrak{V}^2}(f_3\circ a^{-1})\pi_{\mathfrak{V}^1}(f_3\circ a^{-1})\xi.
\end{gather*}
$i=1,2$.   Then:  $\xi_1,\,\xi_2\in\mathcal{H}^s$ with estimates
\begin{align*}
 \norm{\xi'}_{l}&\leq C_{l,f_1,f_2}(\norm{\xi}_{S_{1},l}+a^l\norm{\xi}); \\
 \norm{\xi''}_{l}&\leq C_{l,f_2,f_{3}}(\norm{\xi}_{S_{1},l}+a^l\norm{\xi})
\end{align*}
for any $0\leq l\leq s$.

\smallskip
Since $\xi$ is only partially smooth on $S_{1,U}$ (recall that $U=\mathfrak{u}_{1,2}$, see \eqref{for:285} of Section \ref{sec:47}),  the non-smooth directions for $\xi$ are inside three (resp. five) abelian unipotent subgroups: $\exp(\mathfrak{V})$ (resp.
$\exp(\mathfrak{V}^1)$, $\exp(\mathfrak{V}^2)$),
$\exp(\mathfrak{U})$ (resp. $\exp(\mathfrak{U}^1)$, $\exp(\mathfrak{U}^2)$) and $\exp(\mathfrak{W})$. Corollary \ref{cor:9} shows that
if we apply the directional smoothing operators to $\xi$ successively along all non-smooth directions, we obtain
globally smooth vectors.

It is important to note that Corollary \ref{cor:9} requires a particular sequence of applying $\pi_{\mathfrak{U}}$, $\pi_{\mathfrak{W}}$ and $\pi_{\mathfrak{V}}$
(resp. $\pi_{\mathfrak{U}^i}$, $\pi_{\mathfrak{W}}$ and $\pi_{\mathfrak{V}^i}$) operators to obtain \emph{globally smooth} vectors.
The reason is that  these subalgebras $\mathfrak{V}$, $\mathfrak{U}$ and $\mathfrak{W}$
(and thus the exponentials $\exp(\mathfrak{V})$, $\exp(\mathfrak{U})$ and $\exp(\mathfrak{W})$
do NOT commute with one another. Consequently, smoothing first along
$\mathfrak{U}$ and then along
$\mathfrak{V}$ is not the same as smoothing first along $\mathfrak{V}$ and then along
$\mathfrak{U}$. Indeed, applying $\pi_{\mathfrak{U}}(f_1\circ a^{-1})$  makes $\xi$ smooth along $\mathfrak{U}$. But once we then apply  $\pi_{\mathfrak{V}}(f_1\circ a^{-1})$, the latter can ``twist" $\mathfrak{U}$  through the non-trivial adjoint action
\begin{align*}
 \text{Ad}_{\exp(\mathfrak{V})}(\mathfrak{U})=e^{\text{ad}_{\exp(\mathfrak{V})}}(\mathfrak{U})
\end{align*}
on $\text{Lie}(\GG)$. As a result, $\pi_{\mathfrak{V}}(f_1\circ a^{-1})\pi_{\mathfrak{U}}(f_1\circ a^{-1})\xi$ generally loses smoothness along $\mathfrak{U}$.
In simpler terms, since
$\mathfrak{V}$ does not commute with
$\mathfrak{U}$, the directions in
$\mathfrak{U}$ that were just ``smoothed out" can get "re-angled" (or ``re-distributed") when you move in the
$\mathfrak{V}$-directions, thereby potentially \emph{undoing} some of the smoothing effect unless it is carefully controlled (e.g., by following the specified order of applying these operators). Hence, the order in which these directional smoothing operators are
applied is crucial to ensure that the resulting vector becomes smooth in all directions, thus producing a \emph{globally smooth} vector from one that is initially only partially smooth.

\smallskip

\emph{Note}. The estimates of $\xi'$ and $\xi''$ are no longer tame with respect to $\xi$. The new term $a^l\norm{\xi}$, which has not been presented in traditional KAM estimates,  appears as a direct consequence of the application of the directional smoothing operators.

\end{sect}

\begin{sect}
 We assume notations in \ref{sect:1}.  We recall  $f_2$ is related to $\mathfrak{W}$. Choose $f_3$ related to $\mathfrak{V}^1$. Then
$f_3$ is also related to $\mathfrak{V}^2$, $\mathfrak{U}^1$ and $\mathfrak{U}^2$.

(Corollary \ref{cor:10})  Suppose $a>1$, $\eta\in \mathcal{H}^s_{\mathcal{J}}$, $s\geq 0$ (see Section \eqref{sec:37}). Set
\begin{gather*}
 \eta_{i}=\pi_{\mathfrak{U}^{i}}(f_3\circ a^{-1})\eta,\quad \eta'_{i}=\pi_{\mathfrak{V}^{i}}(f_3\circ a^{-1})\eta, \quad \eta''=\pi_{\mathfrak{W}}(f_2\circ a^{-1})\eta,
\end{gather*}
$i=1,2$.   Then $\eta_{i},\,\eta_{i}',\,\eta''\in\mathcal{H}^s_{\mathcal{J}}$, $i=1,2$ with estimates
  \begin{align*}
       \max_{i=1,2}\{\norm{\eta_i}_{\mathcal{J},l}, \norm{\eta'_i}_{\mathcal{J},l},\norm{\eta''}_{\mathcal{J},l}\}\leq C_{l,f_2,f_{3}}\norm{\eta}_{\mathcal{J},l},\qquad \,\,0\leq l\leq s.
      \end{align*}
Corollary \ref{cor:10} show that if $\eta$ is partially smooth on $\mathcal{J}$, then after applying  directional smoothing operators to $\eta$,
the new vector still ``inherits" smoothness on $\mathcal{J}$.

\end{sect}

\subsection{Abelian subgroups in a Lie group}\label{sec:3}
Fix a set of basis $\mathfrak{u}=\{\mathfrak{u}_1,\cdots,\mathfrak{u}_m\}$ of $\text{Lie}(S)$.   For any $t=(t_1,\cdots,t_m)\in \RR^m$, set
\begin{align}\label{for:64}
 \exp(t)=\exp(t_1\mathfrak{u}_1+\cdots+t_m\mathfrak{u}_m)\quad\text{ and }\quad\pi(t)=\pi(\exp(t)).
\end{align}
For $\xi,\,\eta\in \mathcal{H}$, consider
the corresponding matrix coefficients of $\pi\mid_S$: $\phi_{\xi,\eta}(t)=\langle \pi(t)\xi,\,\eta\rangle$, $t\in \RR^m$. There exists a regular Borel measure $\mu$ on $\widehat{\RR^m}$, called \emph{the associated measure of $\pi$ (with respect to $\widehat{\RR^m}$)}, such that $\xi=\int_{\widehat{\RR^m}}\xi_\chi d\mu(\chi)$, and
\begin{align}\label{for:84}
    \phi_{\xi,\eta}(t)&=\int_{\widehat{\RR^m}}\chi(t)\langle \xi_\chi,\,\eta_\chi\rangle d\mu(\chi).
\end{align}
Here $\chi(t)=e^{\textrm{i}\chi\cdot t}$ (we identify $\RR^m$ and $\widehat{\RR^m}$).

Similar to \eqref{for:62}, for any $f\in L^\infty(\RR^m, d\mu)$  we define an operator $\pi_{\mathfrak{u}}(f)$ on $\mathcal{H}$ as follows:
\begin{align}\label{for:2005}
 \pi_{\mathfrak{u}}(f)(\xi):=\int_{\widehat{\RR^m}}f(\chi)\xi_\chi d\mu(\chi),\qquad \forall\,\xi\in \mathcal{H}.
\end{align}
Here the meaning of $f(\chi)$ is clear by identifying $\RR^m$ and $\widehat{\RR^m}$.

Similar to properties \eqref{for:45} to \eqref{for:2006} in Section \ref{sec:30}, we have
\begin{enumerate}
  \item for any $f_1,\,f_2\in L^\infty(\RR^m)$
\begin{align*}
 \pi_{\mathfrak{u}}(f_1)\pi_{\mathfrak{u}}(f_2)=\pi_{\mathfrak{u}}(f_1f_2);
\end{align*}
and
\begin{align}\label{for:21}
 \norm{\pi_{\mathfrak{u}}(f)}\leq \norm{f}_\infty,\qquad \forall\,f\in L^\infty(\RR^m);
\end{align}
  \item for any $\xi,\,\eta\in \mathcal{H}$
  \begin{align*}
 \bigl\langle \pi_{\mathfrak{u}}(f)\xi,\eta\bigl\rangle=\bigl\langle \xi,\pi_{\mathfrak{u}}(\bar{f})\eta\bigl\rangle=\int_{\widehat{\RR^m}}f(\chi)\langle \xi_\chi,\,\eta_\chi \rangle d\mu(\chi),
\end{align*}
where $\bar{f}$ is the complex conjugate of $f$;

\smallskip

  \item for any $a>0$, any $f\in L^\infty(\RR^m)$ and $\xi\in \mathcal{H}$, we have
\begin{align}\label{for:41}
  \mathfrak{u}_1^{k_1}\cdots \mathfrak{u}_m^{k_m}\big(\pi_{\mathfrak{u}}(f\circ a^{-1})\xi\big)=a^{k_1+k_2+\cdots k_m}\pi_{\mathfrak{u}}(f_{k_1,\cdots,k_m}\circ a^{-1})\xi,
\end{align}
if $f_{k_1,\cdots,k_m}\in L^\infty(\RR^m)$,  where $f_{k_1,\cdots,k_m}(t)=f(t)(t_1 \textrm{i})^{k_1}\dots (t_m \textrm{i})^{k_m}$.
\end{enumerate}
Similar to \eqref{for:2006} and \eqref{for:40} of Section \ref{sec:30} the following estimates hold for $\pi_{\mathfrak{u}}(f\circ a^{-1})$:
\begin{lemma}\label{le:7} Suppose $a>0$. Then:
\begin{enumerate}
  \item\label{for:63} If $\xi\in \mathcal{H}$ and $f$ is Schwartz,  then $\pi_{\mathfrak{u}}(f\circ a^{-1})\xi\in \mathcal{H}_{S}^\infty$ (see \eqref{for:239} of Section \ref{sec:39}) with estimates
\begin{align*}
 \norm{\pi_{\mathfrak{u}}(f\circ a^{-1})\xi}_{S,l}\leq C_{f,l}a^l\norm{\xi},\qquad \forall\,l\geq0.
\end{align*}
\item\label{for:145} If $\xi\in \mathcal{H}_{S}^s$ and $f$ is related to $\mathfrak{u}$ (see \eqref{de:2} of Section \ref{sec:39}),  then
\begin{align*}
 \norm{\xi-\pi_{\mathfrak{u}}(f\circ a^{-1})\xi}\leq C_{f} a^{-s}\norm{\xi}_{S,s},\qquad \forall\,s\geq0.
\end{align*}

\end{enumerate}

\end{lemma}
\emph{Note.} Similar to \eqref{for:40} of Section \ref{sec:30}, the constant in \eqref{for:145} is independent of $s$.

\begin{proof}\eqref{for:63}: For any $k_1,k_2,\cdots,k_m\geq0$, if $\sum_{i=1}^m k_i=l$ we have
\begin{align*}
 \big\|\mathfrak{u}_1^{k_1}&\cdots \mathfrak{u}_m^{k_m}(\pi_{\mathfrak{u}}(f\circ a^{-1})\xi)\big\|\overset{\text{(1)}}{=}
 a^l\big\|\pi_{\mathfrak{u}}(f_{k_1,\cdots,k_m}\circ a^{-1})\xi\big\|\notag\\
 &\overset{\text{(2)}}{\leq} a^l\norm{f_{k_1,\cdots,k_m}}_\infty\norm{\xi}
\end{align*}
Here in $(1)$ we use \eqref{for:41}; in $(2)$ we use \eqref{for:21}. This implies the result.

\smallskip
\eqref{for:145}: For set $X\subset \RR^m$ we use $I_X$ to denote the characteristic function of $X$. From \eqref{for:84} for any $r>0$ we see that
\begin{align*}
 \sum_{i=1}^m\norm{\mathfrak{u}_i^{s}\xi}^2&=\sum_{i=1}^m\int_{\widehat{\RR^m}}|\chi_i|^{2s}\norm{\xi_\chi}^2 d\mu(\chi)
 \geq \sum_{i=1}^m\int_{\widehat{\RR^m}}|\chi_i|^{2s}I_{\norm{\chi}\geq r}^2\norm{\xi_\chi}^2 d\mu(\chi)\\
 &\overset{\text{(0)}}{\geq} r^{2s}\int_{\widehat{\RR^m}}I_{\norm{\chi}\geq r}^2\norm{\xi_\chi}^2 d\mu(\chi)\overset{\text{(1)}}{=}
 r^{2s}\big\|\pi_{\mathfrak{u}}(I_{\norm{t}\geq r})\xi\big\|^2.
\end{align*}
Here in $(0)$ we use the fact that $\sum_{i=1}^m|\chi_i|^{2s}\geq r^{2s}$ if $\norm{\chi}\geq r$, where $\norm{\cdot}$ is defined in \eqref{de:2} of Section \ref{sec:39};
in $(1)$ we use \eqref{for:2005}.

This shows that for any $r>0$,
\begin{align}\label{for:97}
\norm{\xi}_{S,s}\geq|r^s|\norm{\pi_{\mathfrak{u}}(I_{\norm{x}\geq r})\xi}.
\end{align}
Hence  we have
\begin{align*}
  \norm{\xi-\pi_{\mathfrak{u}}(f\circ a^{-1})\xi}&\overset{\text{(1)}}{=}\norm{\pi_{\mathfrak{u}}(1-f\circ a^{-1})\xi}\overset{\text{(2)}}{\leq } (\norm{f}_{C^0}+1)\norm{\pi_{\mathfrak{u}}(I_{\norm{x}\geq a})\xi}\\
  &\overset{\text{(3)}}{\leq }  C_{f} a^{-s}\norm{\xi}_{S,s}.
\end{align*}
Here in $(1)$ from \eqref{for:2005} we see that $\xi=\pi_{\mathfrak{u}}(1)\xi$; in $(2)$ we use the fact that $1-f\circ a^{-1}=0$ if $\norm{x}\leq a$; in $(3)$ we use \eqref{for:97}.
\end{proof}

\subsection{Global smoothness of $\pi_{\mathfrak{u}}(f\circ a^{-1})$}\label{sec:11}
In this part we will show that if $S$ is unipotent the $\pi_{\mathfrak{u}}(f\circ a^{-1})$ operators preserve global smoothness, i.e., $\pi_{\mathfrak{u}}(f\circ a^{-1})(\mathcal{H}^s)\subseteq \mathcal{H}^s$. Before that, we give an alternative definition of $\pi_{\mathfrak{u}}(f)$.

Let $dt$ denote the Lebesgue measure on $\RR^m$.  The representation $\pi|_{S}$ can be extended to the Banach algebra $L^1(\RR^m,\,dt)$: for any $f\in L^1(\RR^m)$ and $\xi\in \mathcal{H}$
\begin{align*}
 \pi(f)(\xi):=\int_{\RR^m}f(t)\pi(t)(\xi)dt.
\end{align*}
Clearly $\norm{\pi(f)(\xi)}\leq \norm{f}_{L^1}\norm{\xi}$.

Next, we use group algebra to define $\pi_{\mathfrak{u}}(f)$. We use $\mathcal{S}(\RR^m)$ to denote the Schwartz space of $\RR^m$. The representation $\pi\mid_S$ extends to a $^\ast$-representation on $\mathcal{S}(\RR^m)$: for any $f\in \mathcal{S}(\RR^m)$, $\pi_{\mathfrak{u}}(f)$ is the operator on $\mathcal{H}$ for which
\begin{align*}
 \pi_{\mathfrak{u}}(f):=\text{\tiny$\frac{1}{(\sqrt{2\pi})^m}$}\int_{\RR^m}\hat{f}(t)\pi(t)(\xi)dt.
\end{align*}
\begin{lemma} We can extend $\pi_{\mathfrak{u}}$ to a homomorphism of $L^\infty(\RR^m,\,d\mu)$ to bounded operators on $\mathcal{H}$. Moreover, for any $f\in L^\infty(\RR^m)$, $\pi_{\mathfrak{u}}(f)$ coincides with the definition in \eqref{for:2005}.
\end{lemma}
\begin{proof} Suppose $f\in \mathcal{S}(\RR^m)$. For any $\xi,\,\eta\in \mathcal{H}$ we have
\begin{align*}
\bigl\langle \pi_\mathfrak{u}(f)\xi,\eta\bigl\rangle=\text{\tiny$\frac{1}{(\sqrt{2\pi})^m}$}\int_{\RR^m} \hat{f}(t)\phi_{\xi,\eta}(t) dt.
\end{align*}
Computations show that
\begin{align}
&\bigl\langle \pi_{\mathfrak{u}}(f)\xi,\eta\bigl\rangle=\text{\tiny$\frac{1}{(\sqrt{2\pi})^m}$}\int_{\RR^m}\hat{f}(t)\langle \pi(t)\xi,\eta\rangle dt\label{for:14}\\
&=\text{\tiny$\frac{1}{(\sqrt{2\pi})^m}$}\int_{\RR^m}\int_{\widehat{\RR^m}}\hat{f}(t)\chi(t)\langle \xi_\chi,\,\eta_\chi \rangle d\mu(\chi) dt\notag\\
&=\int_{\widehat{\RR^m}}\langle \xi_\chi,\,\eta_\chi \rangle\big(\text{\tiny$\frac{1}{(\sqrt{2\pi})^m}$}\int_{\RR^m}\hat{f}(t)\chi(t) dt\big)d\mu(\chi)\notag\\
&=\int_{\widehat{\RR^m}}f(\chi)\langle \xi_\chi,\,\eta_\chi \rangle d\mu(\chi).\notag
\end{align}
Since
\begin{align*}
 \norm{\pi_{\mathfrak{u}}(f)}\leq \norm{f}_\infty,\qquad \forall\,f\in \mathcal{S}(\RR^m)
\end{align*}
we can extend $\pi_{\mathfrak{u}}$ from $\mathcal{S}(\RR^m)$ to $L^\infty(\RR^m)$ by taking strong limits of operators and pointwise monotone increasing limits of non-negative functions (see \cite{lang2012sl2} for a detailed treatment). Hence $\pi_{\mathfrak{u}}$ is a homomorphism of $L^\infty(\RR^m)$ to bounded operators on $\mathcal{H}$. Moreover, for any $f\in L^\infty(\RR^m)$ we see that
\begin{align*}
 \pi_{\mathfrak{u}}(f)(\xi)=\int_{\widehat{\RR^m}}f(\chi)\xi_\chi d\mu(\chi),\qquad \forall\,\xi\in \mathcal{H}.
\end{align*}
This coincides with the definition in \eqref{for:2005}.
\end{proof}

The following lemma tells us that if $S$ is unipotent, then $\pi_{\mathfrak{u}}(f\circ a^{-1})$ preserves global smoothness.
\begin{lemma}\label{le:2} Suppose $S$ is unipotent and $f\in \tilde{\mathcal{S}}(\RR^m)$. We have
\begin{enumerate}
  \item\label{for:2007}    for any $v\in \mathfrak{h}$
\begin{align*}
  v\pi_{\mathfrak{u}}(f)=\sum_{j\in\ZZ_{0,+}^{\dim\mathfrak{h}-1} }c_{j}
  \pi_{\mathfrak{u}}(f^{[j]})
  (\text{ad}_\mathfrak{u})^{j}(v)
\end{align*}
(see \eqref{for:240} of Section \ref{sec:39});

\smallskip
  \item\label{for:2008} for any vectors $v_i\in \mathfrak{h}$, $1\leq i\leq l$ we have
\begin{align}\label{for:147}
  &v_l\cdots v_2v_1\pi_{\mathfrak{u}}(f)\notag\\
  &=\sum_{j_i\in\ZZ_{0,+}^{\dim\mathfrak{h}-1},\,1\leq i\leq l }c_{j_1,\cdots,j_l}\pi_{\mathfrak{u}}(f^{[\sum_{i=1}^lj_i]})
  \big((\text{ad}_\mathfrak{u})^{j_l}v_l)\cdots (\text{ad}_\mathfrak{u})^{j_1}v_1)\big);
\end{align}

 \item\label{ob:3} if $v\in \mathfrak{h}$ commutes with $\text{Lie}(S)$, then $v\pi_{\mathfrak{u}}(f)=\pi_{\mathfrak{u}}(f) v$;

\smallskip
\item\label{for:68} suppose $\xi\in \mathcal{H}^s$, $s\geq 0$.  Then $\pi_{\mathfrak{u}}(f)\xi\in \mathcal{H}^s$ with estimates
\begin{align*}
  \norm{\pi_{\mathfrak{u}}(f)\xi}_l\leq C_{l}\norm{f}_{\tilde{\mathcal{S}}(\RR^m),ml \dim\mathfrak{h}}\norm{\xi}_l, \qquad \forall\, 0\leq l\leq s.
\end{align*}
\end{enumerate}

\end{lemma}
\begin{proof} \eqref{for:2007}: Since $\mathfrak{u_i}$, $1\leq i\leq m$ are nilpotent, $(\text{ad}_\mathfrak{u_i})^{\dim\mathfrak{h}}=0$, $1\leq i\leq m$. For any $1\leq i\leq m$ and $r\in\RR$ set
\begin{align*}
 B_{r,i}=\sum_{j=0}^{\dim\mathfrak{h}-1}\text{\small$\frac{(-1)^jr^j}{j!}$}\text{ad}^{j}_\mathfrak{u_i}.
\end{align*}
Then we can write
\begin{align*}
  B_{\text{\tiny$t_1$},1}\cdots B_{\text{\tiny$t_m$},m}=\sum_{j\in\ZZ_{0,+}^{\dim\mathfrak{h}-1} }d_{j}
  t^j(\text{ad}_\mathfrak{u})^j
\end{align*}
(see \eqref{for:240} of Section \ref{sec:39}), where $d_{j}$ are constants only dependent on $j$ and $t=(t_1,\cdots,t_m)$. We recall notation \eqref{for:64}. We note that
\begin{align}\label{for:56}
 v\pi(t)=\pi(t)\text{Ad}_{\exp(-t)}(v)=\pi(t)(B_{\text{\tiny$t_1$},1}\cdots B_{\text{\tiny$t_m$},m})v,\quad \forall\,v\in \mathfrak{h}.
\end{align}
We suppose $f\in \mathcal{S}(\RR^m)$, $\vartheta\in \mathcal{H}^1$ and $\eta\in \mathcal{H}$. Then by \eqref{for:14} we have
\begin{align*}
&\bigl\langle v(\pi_{\mathfrak{u}}(f)\vartheta),\eta\bigl\rangle=\text{\tiny$\frac{1}{(\sqrt{2\pi})^m}$}\int_{\RR^m}\bigl\langle \hat{f}(t)v(\pi(t)\vartheta),\eta\rangle dt\notag\\
&\overset{\text{(1)}}{=}\text{\tiny$\frac{1}{(\sqrt{2\pi})^m}$}\int_{\RR^m}\bigl\langle \hat{f}(t)\pi(t)(B_{\text{\tiny$t_1$},1}\cdots B_{\text{\tiny$t_m$},m})v\vartheta,\,\eta\bigl\rangle dt\\
=\text{\tiny$\frac{1}{(\sqrt{2\pi})^m}$}&\sum_{j\in\ZZ_{0,+}^{\dim\mathfrak{h}-1} }d_{j} \int_{\RR^m}\bigl\langle \hat{f}(t)t^j\pi(t)
\big((\text{ad}_\mathfrak{u})^j(v)\vartheta\big),\,\eta\bigl\rangle dt\\
&\overset{\text{(2)}}{=}\sum_{j\in\ZZ_{0,+}^{\dim\mathfrak{h}-1} }c_{j} \bigl\langle\pi_{\mathfrak{u}}(f^{[j]})
  \big((\text{ad}_\mathfrak{u})^j(v)\vartheta\big),\eta\bigl\rangle
\end{align*}
Here in $(1)$ we use \eqref{for:56}; in $(2)$ we use \eqref{for:14}.

This shows that for any $v\in \mathfrak{h}$ and any $f\in \mathcal{S}(\RR^m)$ we have
\begin{align}\label{for:98}
  v\pi_{\mathfrak{u}}(f)=\sum_{j\in\ZZ_{0,+}^{\dim\mathfrak{h}-1} }c_{j}\pi_{\mathfrak{u}}(f^{[j]})
  (\text{ad}_\mathfrak{u})^j(v).
\end{align}
By arguments below \eqref{for:21}, we can extend \eqref{for:98} from $\mathcal{S}(\RR^m)$ to $\tilde{\mathcal{S}}(\RR^m)$. Then we complete the proof of  \eqref{for:2007}.

\smallskip
\eqref{for:2008}: Keeping using \eqref{for:2007}  we get the result.

\smallskip

\eqref{ob:3}: A direct consequence of \eqref{for:2007}.

\smallskip

\eqref{for:68}: By using \eqref{for:2008}, for any vectors $v_i\in \mathfrak{h}$, $1\leq i\leq l$, $l\leq s$ we have
\begin{align*}
&\norm{v_l\cdots v_2v_1(\pi_{\mathfrak{u}}(f)\xi)}\\
  &\leq\sum_{j_i\in\ZZ_{0,+}^{\dim\mathfrak{h}-1}}|c_{j_1,\cdots,j_l}|
  \Big\|\pi_{\mathfrak{u}}(f^{[\sum_{i=1}^lj_i]})
  \big((\text{ad}_\mathfrak{u})^{j_l}v_l)\cdots (\text{ad}_\mathfrak{u})^{j_1}v_1)\xi\big)\Big\|\\
  &\overset{\text{(1)}}{\leq}\sum_{j_i\in\ZZ_{0,+}^{\dim\mathfrak{h}-1}}|c_{j_1,\cdots,j_l}|\cdot\norm{f^{|\sum_{i=1}^lj_i|}}_{\text{\small$L^\infty(\RR^m)$}}\\
  &\cdot\big\|(\text{ad}_\mathfrak{u})^{j_l}v_l)\cdots (\text{ad}_\mathfrak{u})^{j_1}v_1)\xi\big\|\\
  &\leq C_{l}\norm{f}_{\tilde{\mathcal{S}}(\RR^m),ml \dim\mathfrak{h}}\norm{\xi}_l.
\end{align*}
(see \eqref{for:241} of Section \ref{sec:39}). Here in $(1)$ we use \eqref{for:21}. This  implies \eqref{for:68}. Hence we finish the proof.
\end{proof}
We see that the proof of Lemma \ref{le:2} heavily relies on the unipotency of $S$. From now on we always assume $S$ to be unipotent in $H$.
The following corollary of Lemma \ref{le:2} gives us the crucial estimates for
the directional smoothing operators.
\begin{corollary}\label{cor:1} Suppose $\xi\in \mathcal{H}^s$, $s\geq 0$, $a\geq1$ and $0\leq\ell\leq s$. Then:
\begin{enumerate}
\item\label{for:7} if $f\in \tilde{\mathcal{S}}(\RR^m)$, then
$\pi_{\mathfrak{u}}(f\circ a^{-1})\xi\in \mathcal{H}^s$ with estimates
\begin{align*}
  \norm{\pi_{\mathfrak{u}}(f\circ a^{-1})\xi}_r\leq C_{f,r}\norm{\xi}_r, \qquad \forall\, 0\leq r\leq s;
\end{align*}
\item \label{for:146} if $f$ is related to $\mathfrak{u}$, then for any $0\leq r\leq \ell$
  \begin{align*}
   \norm{\xi-\pi_{\mathfrak{u}}(f\circ a^{-1})\xi}_r\leq C_{\ell,f}\big(a^{-s}\norm{\xi}_s\big)^{\text{\tiny$1-\frac{r}{\ell}$}}\norm{\xi}^{\text{\tiny$\frac{r}{\ell}$}}_{\ell}.
  \end{align*}
\end{enumerate}
\end{corollary}
\emph{Note}. The constant in \eqref{for:146}  is independent of $s$.

\begin{proof}
\eqref{for:7}: From  \eqref{for:68} of Lemma \ref{le:2}, we have
\begin{align*}
  \norm{\pi_{\mathfrak{u}}(f\circ a^{-1})&\xi}_r\leq C_{r}\norm{f\circ a^{-1}}_{\tilde{\mathcal{S}}(\RR^m),mr \dim\mathfrak{h}}\norm{\xi}_r\\
  &\overset{\text{(1)}}{\leq}C_{r,1}\norm{f}_{\tilde{\mathcal{S}}(\RR^m),mr \dim\mathfrak{h}}\norm{\xi}_r\leq C_{f,r}\norm{\xi}_r.
\end{align*}
for any $0\leq r\leq s$. Here in $(1)$ we use the fact $a\geq1$.

\smallskip
\eqref{for:146}: We have
\begin{align*}
 &\norm{\xi-\pi_{\mathfrak{u}}(f\circ a^{-1})\xi}_r\notag\\
 &\overset{\text{(1)}}{\leq} C_{\ell} \norm{\xi-\pi_{\mathfrak{u}}(f\circ a^{-1})\xi}^{\text{\tiny$1-\frac{r}{\ell}$}}
 \norm{\xi-\pi_{\mathfrak{u}}(f\circ a^{-1})\xi}^{\text{\tiny$\frac{r}{\ell}$}}_{\ell}\notag\\
 &\overset{\text{(2)}}{\leq} C_{\ell}( C_fa^{-s}\norm{\xi}_s)^{\text{\tiny$1-\frac{r}{\ell}$}}(\norm{\xi}_{\ell}+\norm{\pi_{\mathfrak{u}}(f\circ a^{-1})\xi}_{\ell})^{\text{\tiny$\frac{r}{\ell}$}}\notag\\
 &\overset{\text{(3)}}{\leq} C_{\ell,f}(a^{-s}\norm{\xi}_s)^{\text{\tiny$1-\frac{r}{\ell}$}}(\norm{\xi}_{\ell}+C_{f,\ell}\norm{\xi}_s)^{\text{\tiny$\frac{r}{\ell}$}}\notag\\
 &\leq C_{\ell,f,1}\big(a^{-s}\norm{\xi}_s\big)^{\text{\tiny$1-\frac{r}{\ell}$}}\norm{\xi}^{\text{\tiny$\frac{r}{\ell}$}}_{\ell}
 \end{align*}
  Here in $(1)$ we use  interpolation inequalities (see \cite{Hamilton}); in $(2)$ we use \eqref{for:145} of Lemma \ref{le:7}; in $(3)$ we use
  \eqref{for:7}.  Hence we finish the proof.
\end{proof}

\subsection{Construction of smooth vectors}\label{sec:10} In this part  we show a  general construction of smooth vectors using
the directional  smoothing operators.  Recall that $S$ is unipotent and is isomorphic to $\RR^m$.

The following lemma shows that if a vector only loses smoothness along $S$-directions, then after applying the directional  smoothing operator along $S$, we will have a
globally smooth vector.
\begin{lemma}\label{cor:3}
Suppose $Q$ is a subgroup of $H$ such that $\mathfrak{h}=\text{Lie}(S)\oplus\text{Lie}(Q)$. Choose $f\in \mathcal{S}(\RR^m)$.  Then for any $\xi\in \mathcal{H}^s_{Q}$, $s\geq0$ and any $a\geq1$, the vector
\begin{align*}
 \xi'=\pi_{\mathfrak{u}}(f\circ a^{-1})\xi
\end{align*}
is in $\mathcal{H}^s$ with estimates
 \begin{align}\label{for:100}
   \norm{\xi'}_{l}\leq C_{f,l}(\norm{\xi}_{Q,l}+a^l\norm{\xi})
 \end{align}
for any $0\leq l\leq s$; and
\begin{align}\label{for:101}
   \norm{\xi'}_{Q,l}\leq C_{f,l}\norm{\xi}_{Q,l},\qquad \forall\, 0\leq l\leq s.
 \end{align}

\end{lemma}
\begin{proof} Let $\mathfrak{q}=\{w_1,\cdots, w_{\dim(\text{Lie}(Q))}\}$ be a basis of $\text{Lie}(Q)$. We recall that $\mathfrak{u}=\{\mathfrak{u}_1,\cdots,\mathfrak{u}_m\}$ is a basis of  $\text{Lie}(S)$. Denote by $\mathcal{U}(\text{Lie}(Q))$ the universal enveloping algebra of $\text{Lie}(Q)$,
with its usual filtration $\{\mathcal{U}_{n}(\text{Lie}(Q))\}_{n\geq0}$. Set
\begin{align*}
 \mathfrak{f}(f)=\{f^{[j]}\cdot (t\textrm{i})^k:\,j,\,k\in\ZZ_{0,+}^m\}.
\end{align*}
\emph{Step 1:} We show that: for any $q\geq1$ and any $v_i\in \mathfrak{q}$, $1\leq i\leq q$, we have
\begin{align*}
 &v_q\cdots v_2v_1\pi_{\mathfrak{u}}(f\circ a^{-1})=\sum_{j} c_j a^{\delta_j} \pi_{\mathfrak{u}}(f_j\circ a^{-1})u_j
 \end{align*}
where $c_j\in\RR$,  $\delta_j\leq0$, $f_j\in \mathfrak{f}(f)$ and $u_j\in \mathcal{U}_{q}(\text{Lie}(Q))$ for each $j$.

We prove by induction. By \eqref{for:2007} of Lemma \ref{le:2} we have
\begin{align}\label{for:69}
  v_1\pi_{\mathfrak{u}}(f\circ a^{-1})=\pi_{\mathfrak{u}}(f\circ a^{-1})v_1+\sum_{j}d_ja^{l_j}\pi_{\mathfrak{u}}(g_j\circ a^{-1})
  u_j'.
\end{align}
where $d_j\in\RR$, $l_j\leq  -1$, $g_j\in \mathfrak{f}(f)$ and $u_j'\in \mathfrak{h}$ for each $j$.

Since $\mathfrak{h}=\text{Lie}(S)\oplus\text{Lie}(Q)$, for each $j$ we can write
\begin{align}\label{for:70}
 u_j'=\sum_{i=1}^mb_{j,i}\mathfrak{u}_i+u_{j,1}
\end{align}
where $u_{j,1}\in \text{Lie}(Q)$. From \eqref{for:69} and \eqref{for:70} we have
\begin{align*}
  v_1\pi_{\mathfrak{u}}(f\circ a^{-1})&=\pi_{\mathfrak{u}}(f\circ a^{-1})v_1+\sum_{j}d_ja^{l_j}\pi_{\mathfrak{u}}(g_j\circ a^{-1})
  u_{j,1}\\
  &+\sum_{i=1}^m\sum_{j}b_{j,i}d_ja^{l_j}\pi_{\mathfrak{u}}(g_j\circ a^{-1})\mathfrak{u}_i\\
  &\overset{(1)}{=}\pi_{\mathfrak{u}}(f\circ a^{-1})v_1+\sum_{j}d_ja^{l_j}\pi_{\mathfrak{u}}(g_j\circ a^{-1})
  u_{j,1}\\
  &+\sum_{i=1}^m\sum_{j}b_{j,i}d_ja^{l_j+1}\pi_{\mathfrak{u}}(g_{j,i}\circ a^{-1}).
\end{align*}
Here in $(1)$ we use \eqref{for:41}, where $g_{j,i}(t)=g_j(t)(t_i \textrm{i})\in \mathfrak{f}(f)$.

Hence we finish the proof for the case of $q=1$. Assume it holds for $q=p$. Suppose $v_i\in \mathfrak{q}$, $1\leq i\leq p+1$. By the induction assumption, we have
\begin{align*}
 v_p\cdots v_2v_1\pi_{\mathfrak{u}}(f\circ a^{-1})=\sum_{j} c_j a^{\delta_j} \pi_{\mathfrak{u}}(f_j\circ a^{-1})u_j
 \end{align*}
where $c_j\in\RR$,  $\delta_j\leq0$, $f_j\in \mathfrak{f}(f)$ and $u_j\in \mathcal{U}_{p}(\text{Lie}(Q))$ for each $j$. Hence,
\begin{align}\label{for:71}
v_{p+1} v_p\cdots v_2v_1\pi_{\mathfrak{u}}(f\circ a^{-1})=\sum_{j} c_j a^{\delta_j} v_{p+1}\pi_{\mathfrak{u}}(f_j\circ a^{-1})u_j.
 \end{align}
Since $v_{p+1}\in \mathfrak{q}$, by the induction assumption for $q=1$, for each $j$ we have
\begin{align*}
 v_{p+1}\pi_{\mathfrak{u}}(f_j\circ a^{-1})=\sum_{i} d_{j,i} a^{\delta_{j,i}} \pi_{\mathfrak{u}}(f_{j,i}\circ a^{-1})u_{j,i}
\end{align*}
where $d_{j,i}\in\RR$, $\delta_{j,i}\leq 0$, $f_{j,i}\in \mathfrak{f}(f_j)$ and $u_{j,i}\in \text{Lie}(Q)$ for each $i$.

Since $f_j\in \mathfrak{f}(f)$ for each $j$, for each $f_{j,i}$ we can write
\begin{align*}
 f_{j,i}=\sum_k l_k f_{j,i,k}
\end{align*}
where $l_k\in\RR$ and $f_{j,i,k}\in \mathfrak{f}(f)$. Hence we have
\begin{align}\label{for:72}
 v_{p+1}\pi_{\mathfrak{u}}(f_j\circ a^{-1})=\sum_k\sum_{i} l_kd_{j,i} a^{\delta_{j,i}} \pi_{\mathfrak{u}}(f_{j,i,k}\circ a^{-1})u_{j,i},\quad \forall\,j.
\end{align}
It follows from \eqref{for:71} and \eqref{for:72} that
\begin{align*}
 v_{p+1}v_p&\cdots v_2v_1\pi_{\mathfrak{u}}(f\circ a^{-1})=\sum_k\sum_{i}\sum_{j} l_kc_j d_{j,i}a^{\delta_j+\delta_{j,i}}\pi_{\mathfrak{u}}(f_{j,i,k}\circ a^{-1})u_{j,i}u_j.
\end{align*}
We note that $\delta_j+\delta_{j,i}\leq 0$ and $u_{j,i}u_j\in \mathcal{U}_{p+1}(\text{Lie}(Q))$. Then we finish the proof for $q=p+1$. Hence we get the result.

\smallskip
\emph{Step 2:} We show that \eqref{for:101} holds.

For any $v_i\in \mathfrak{q}$, $1\leq i\leq l$, by using Step 1 we have
\begin{align*}
 &v_l\cdots v_2v_1\pi_{\mathfrak{u}}(f\circ a^{-1})=\sum_{j} c_j a^{\delta_j} \pi_{\mathfrak{u}}(f_j\circ a^{-1})u_j
 \end{align*}
where $c_j\in\RR$,  $\delta_j\leq0$, $f_j\in \mathfrak{f}(f)$ and $u_j\in \mathcal{U}_{l}(\text{Lie}(Q))$ for each $j$. It follows that
\begin{align*}
 \norm{v_l\cdots v_2v_1\xi'}\leq \sum_{j} |c_j|\norm{\pi_{\mathfrak{u}}(f_j\circ a^{-1})(u_j\xi)}\overset{(*)}{\leq} \sum_{j}C_{f,j} |c_j|\norm{u_j\xi}\leq C_{f,l}\norm{\xi}_{Q,l}.
\end{align*}
Here in $(*)$ we use \eqref{for:7} of Corollary \ref{cor:1}. This implies \eqref{for:101}.

\smallskip
\emph{Step 3:} We show that \eqref{for:100} holds.

From \eqref{for:63} of Lemma \ref{le:7} we have
\begin{align}\label{for:75}
 \norm{\xi'}_{S,l}\leq C_{f,l}a^l\norm{\xi},\qquad \forall\,l\geq0.
\end{align}
Since $\mathfrak{h}=\text{Lie}(S)\oplus\text{Lie}(Q)$, \eqref{for:75}, \eqref{for:101} and Theorem \ref{th:4} imply that $\xi'\in\mathcal{H}^s$ with estimates
 \begin{align*}
   \norm{\xi'}_{l}\leq C_{l}\norm{\xi'}_{S,l}+C_\ell\norm{\xi'}_{Q,l}\leq C_{f,l}(\norm{\xi}_{Q,l}+a^l\norm{\xi})
 \end{align*}
for any $0\leq l\leq s$. Hence we finish the proof.
\end{proof}
We will need the following corollary of Lemma \ref{cor:3}, if $\mathfrak{u}$ splits into two subsets.
\begin{corollary}\label{cor:4} Suppose $Q$ is a subgroup of $H$ such that $\mathfrak{h}=\text{Lie}(S)\oplus\text{Lie}(Q)$.  Also suppose $\mathfrak{o}_i$, $i=1,2$ are subsets of $\mathfrak{u}$ such that $\mathfrak{u}=\mathfrak{o}_1\cup \mathfrak{o}_2$ (we recall that $\mathfrak{u}=\{\mathfrak{u}_1,\cdots,\mathfrak{u}_m\}$ is a  basis of $\text{Lie}(S)$).
Choose $f_1$ related to $\mathfrak{o}_1$ and $f_2$ related to $\mathfrak{o}_2$ (see  \eqref{de:2} of Section \ref{sec:39}).
Then for any $\xi\in \mathcal{H}^s_{Q}$, $s\geq0$ and any $a>1$, the vector
\begin{align*}
 \xi'=\pi_{\mathfrak{o}_1}(f_1\circ a^{-1})\pi_{\mathfrak{o}_2}(f_2\circ a^{-1})\xi
\end{align*}
is in $\mathcal{H}^s$ with estimates
 \begin{align*}
   \norm{\xi'}_{l}\leq C_{f_1f_2,l}(\norm{\xi}_{Q,l}+a^l\norm{\xi})
 \end{align*}
for any $0\leq l\leq s$; and
\begin{align*}
   \norm{\xi'}_{Q,l}\leq C_{f_1f_2,l}\norm{\xi}_{Q,l},\qquad \forall\, 0\leq l\leq s.
 \end{align*}

\end{corollary}
\begin{proof} $f_1$ (resp. $f_2$) can be naturally extended to a function which is free on $\mathfrak{o}_2$ (resp. $\mathfrak{o}_1$) (see \eqref{de:2} of Section \ref{sec:39}).   We note that
\begin{align*}
 \pi_{\mathfrak{o}_2}(f_2\circ a^{-1})\pi_{\mathfrak{o}_1}(f_1\circ a^{-1})=\pi_{\mathfrak{u}}\big((f_2f_1)\circ a^{-1}\big).
\end{align*}
It is clear that $f_2f_1\in \mathcal{S}(\RR^m)$. Then the result follows from Lemma \ref{cor:3}.

\end{proof}
\subsection{Applications of directional smoothing operators} Throughout this subsection $(\pi,\mathcal{H})$ denotes a unitary representation of $\GG$. Recall Examples \ref{ex:4} and \ref{ex:5} in Section \ref{sec:50}. Let
\begin{align*}
 \mathbb{S}_{0}=\{S_{0,U},\,\exp(\mathfrak{V})\}\quad\text{and}\quad \mathbb{S}_{1}=\{S_{1,U},\exp(\mathfrak{V})\}.
\end{align*}
Then they are subgroups of $\GG$ with following forms:
\begin{gather*}
\mathbb{S}_{0}=\begin{pmatrix}M_{2,2} & \vline & M_{2,n-2}\\
\hline 0 &\vline &M_{n-2,n-2} \end{pmatrix}\times \mathbb{G}_2\times \cdots\times \mathbb{G}_k, \\
\mathbb{S}_{1}=\begin{pmatrix}
\begin{matrix}a&b\\
0&c
\end{matrix} &  \vline & M_{2,n-2}\\
 \hline  0& \vline & M_{n-2,n-2} \end{pmatrix}\times \mathbb{G}_2\times \cdots\times \mathbb{G}_k,
\end{gather*}
where $M_{m,k}$ denotes the set of $m\times k$ matrices and $a,b,c\in\RR$.

The following  corollary tells us how to construct a globally smooth vector from a vector which is only partially smooth on $S_{1,U}$.
  Choose $f_1$ related to $\mathfrak{V}$ (see \eqref{de:2} of Section \ref{sec:39}), $f_2$ related to $\mathfrak{W}$ and $f_3$ related to $\mathfrak{V}_1$. Then
$f_1$ is also related to $\mathfrak{U}$, and $f_3$ is also related to $\mathfrak{U}_i$ and $\mathfrak{V}_i$, $i=1,2$.

\begin{corollary}\label{cor:9} Suppose $a>1$, $\xi\in \mathcal{H}^s_{S_{1,U}}$, $s\geq 0$.
\begin{enumerate}
  \item\label{for:26} Let
  \begin{align*}
    \xi'=\pi_{\mathfrak{U}}(f_1\circ a^{-1})\pi_{\mathfrak{W}}(f_2\circ a^{-1})\pi_{\mathfrak{V}}(f_1\circ a^{-1})\xi.
  \end{align*}
  Then $\xi'\in\mathcal{H}^s$ with estimates
  \begin{align*}
   \norm{\xi'}_{l}&\leq C_{l,f_1,f_2}(\norm{\xi}_{S_{1},l}+a^l\norm{\xi}),\qquad 0\leq l\leq s.
  \end{align*}
  \item\label{for:28} Let
  \begin{align*}
   \xi''=\pi_{\mathfrak{U}^1}(f_3\circ a^{-1})\pi_{\mathfrak{U}^2}(f_3\circ a^{-1})\pi_{\mathfrak{W}}(f_2\circ a^{-1})\pi_{\mathfrak{V}^2}(f_3\circ a^{-1})\pi_{\mathfrak{V}^1}(f_3\circ a^{-1})\xi.
  \end{align*}
  Then $\xi''\in\mathcal{H}^s$ with estimates
\begin{align}
 \norm{\xi''}_{l}&\leq C_{l,f_2,f_{3}}(\norm{\xi}_{S_{1},l}+a^l\norm{\xi}),\qquad 0\leq l\leq s.
\end{align}
\end{enumerate}

\end{corollary}

\begin{proof} \eqref{for:26}: Let
\begin{align*}
  \xi_1=\pi_{\mathfrak{V}}(f_1\circ a^{-1})\xi\quad\text{and}\quad\xi_2=\pi_{\mathfrak{W}}(f_2\circ a^{-1})\xi_1.
\end{align*}
Firstly, we show that $\xi_1\in \mathcal{H}_{\mathbb{S}_{1}}^s$ with estimates
\begin{align}\label{for:79}
 \norm{\xi_1}_{\mathbb{S}_{1},l}&\leq C_{l,f_1}(\norm{\xi}_{S_{1},l}+a^l\norm{\xi}),\qquad 0\leq l\leq s.
\end{align}
To to so, let $H=\mathbb{S}_{1}$, $Q=S_{1,U}$ and $S=\exp(\mathfrak{V})$. Then the result follows from \eqref{for:100} of Lemma \ref{cor:3}.

Secondly, we show that $\xi_2\in \mathcal{H}_{\mathbb{S}_{0}}^s$ with estimates
\begin{align}\label{for:80}
 \norm{\xi_2}_{\mathbb{S}_{0},l}&\leq C_{l,f_1,f_2}(\norm{\xi}_{S_{1},l}+a^l\norm{\xi}),\qquad 0\leq l\leq s.
\end{align}
To to so, let $H=\mathbb{S}_{0}$, $Q=\mathbb{S}_{1}$ and $S=\exp(\mathfrak{W})$. By \eqref{for:100} of Lemma \ref{cor:3} we have
\begin{align*}
 \norm{\xi_2}_{\mathbb{S}_{0},l}&\leq C_{l,f_2}(\norm{\xi_1}_{\mathbb{S}_{1},l}+a^l\norm{\xi_1})\overset{(1)}{\leq} C_{l,f_2,f_1}(\norm{\xi}_{S_{1},l}+a^l\norm{\xi})
\end{align*}
for any $0\leq l\leq s$. Here in $(1)$ we use \eqref{for:79}. Hence we get \eqref{for:80}.

We note that
\begin{align*}
 \xi'=\pi_{\mathfrak{U}}(f_1\circ a^{-1})\xi_2.
\end{align*}
Finally, let $H=\GG$, $Q=\mathbb{S}_{0}$ and $S=\exp(\mathfrak{U})$.
By \eqref{for:100} of Lemma \ref{cor:3} we see that $\xi'\in \mathcal{H}^s$ with estimates
\begin{align*}
 \norm{\xi'}_{l}&\leq C_{l,f_1}(\norm{\xi_2}_{\mathbb{S}_{0},l}+a^l\norm{\xi_2})\overset{(1)}{\leq} C_{l,f_2,f_1}(\norm{\xi}_{S_{1},l}+a^l\norm{\xi})
\end{align*}
for any $0\leq l\leq s$. Here in $(1)$ we use \eqref{for:80}. Hence we get \eqref{for:26}.

\medskip
\eqref{for:28}: Let
\begin{align*}
  \xi_3=\pi_{\mathfrak{V}^2}(f_3\circ a^{-1})\pi_{\mathfrak{V}^1}(f_3\circ a^{-1})\xi\quad\text{and}\quad\xi_4=\pi_{\mathfrak{W}}(f_2\circ a^{-1})\xi_3.
\end{align*}
Let $H=\mathbb{S}_{1}$, $Q=S_{1,U}$, $S=\exp(\mathfrak{V})$, $\mathfrak{o}_i=\mathfrak{V}^i$, $i=1,2$. By Corollary \ref{cor:4}, we see that $\xi_3\in \mathcal{H}_{\mathbb{S}_{1}}^s$ with estimates
\begin{align}\label{for:59}
 \norm{\xi_3}_{\mathbb{S}_{1},l}&\leq C_{l,f_3}(\norm{\xi}_{S_{1},l}+a^l\norm{\xi}),\qquad 0\leq l\leq s.
\end{align}
Let $H=\mathbb{S}_{0}$, $Q=\mathbb{S}_{1}$ and $S=\exp(\mathfrak{W})$. By \eqref{for:100} of Lemma \ref{cor:3} we have
\begin{align}\label{for:87}
 \norm{\xi_4}_{\mathbb{S}_{0},l}&\leq C_{l,f_2}(\norm{\xi_3}_{\mathbb{S}_{1},l}+a^l\norm{\xi_3})\overset{(1)}{\leq} C_{l,f_2,f_3}(\norm{\xi}_{S_{1},l}+a^l\norm{\xi})
\end{align}
for any $0\leq l\leq s$. Here in $(1)$ we use \eqref{for:59}.

We note that
\begin{align*}
 \xi''=\pi_{\mathfrak{U}^1}(f_3\circ a^{-1})\pi_{\mathfrak{U}^2}(f_3\circ a^{-1})\xi_4.
\end{align*}
Let $H=\GG$, $Q=\mathbb{S}_{0,U}$, $S=\exp(\mathfrak{U})$, $\mathfrak{o}_i=\mathfrak{U}^i$, $i=1,2$. By Corollary \ref{cor:4}, we see that
$\xi''\in \mathcal{H}^s$ with estimates
\begin{align*}
 \norm{\xi''}_{l}&\leq C_{l,f_3,f_2}(\norm{\xi_4}_{\mathbb{S}_{0},l}+a^l\norm{\xi_4})\overset{(1)}{\leq} C_{l,f_2,f_1}(\norm{\xi}_{S_{1},l}+a^l\norm{\xi})
\end{align*}
for any $0\leq l\leq s$. Here in $(1)$ we use \eqref{for:87}. Hence we get \eqref{for:28}.
\end{proof}

The following corollary says that if a vector is partially smooth on $\mathcal{J}$ (see Section \eqref{sec:37}), then after applying directional smoothing operators, the new vectors
are still partially smooth on $\mathcal{J}$.

We recall  $f_2$ is related to $\mathfrak{W}$. Choose $f_3$ related to $\mathfrak{V}^1$. Then
$f_3$ is also related to $\mathfrak{V}^2$, $\mathfrak{U}^1$ and $\mathfrak{U}^2$.
\begin{corollary}\label{cor:10} Suppose $a>1$, $\eta\in \mathcal{H}^s_{\mathcal{J}}$, $s\geq 0$. Set
\begin{gather*}
 \eta_{i}=\pi_{\mathfrak{U}^{i}}(f_3\circ a^{-1})\eta,\quad \eta'_{i}=\pi_{\mathfrak{V}^{i}}(f_3\circ a^{-1})\eta, \quad \eta''=\pi_{\mathfrak{W}}(f_2\circ a^{-1})\eta,
\end{gather*}
$i=1,2$.   Then $\eta_{i},\,\eta_{i}',\,\eta''\in\mathcal{H}^s_{\mathcal{J}}$, $i=1,2$ with estimates
  \begin{align}\label{for:90}
       \max_{i=1,2}\{\norm{\eta_i}_{\mathcal{J},l}, \norm{\eta'_i}_{\mathcal{J},l},\norm{\eta''}_{\mathcal{J},l}\}\leq C_{l,f_2,f_{3}}\norm{\eta}_{\mathcal{J},l},\qquad \,\,0\leq l\leq s.
      \end{align}

\end{corollary}

\begin{proof}

To prove that \eqref{for:90} holds for $\eta_{i},\,\eta_{i}'$, $i=1,2$, let $H=\{Q,\,S\}$ where $Q=\mathcal{J}$, $S=\exp(\mathfrak{C}^{i})$,  $i=1,\,2$ where $\mathfrak{C}$ stands for $\mathfrak{U}$ or $\mathfrak{V}$. We note that $H=\mathcal{J}\ltimes \exp(\mathfrak{C}^{i})$. Then by using \eqref{for:101} of Lemma \ref{cor:3} we get the result.

To prove that \eqref{for:90} holds for $\eta''$, let $H=\{Q,\,S\}$, where $Q=\mathcal{J}$, $S=\exp(\mathfrak{W})$.  We note that $H=\mathcal{J}\times \exp(\mathfrak{W})$. Then by using \eqref{for:101} of  Lemma \ref{cor:3} we get the result.

\end{proof}

\section{Construction of global splittings}\label{sec:22}  $E$ is a generating set of $\text{Lie}(A)$ and $\{E_0,U\}\subset E$ (see \eqref{for:286} of Section \ref{sec:47} for detailed descriptions of $E_0$ and other notations appeared in the proof). Set
\begin{align}\label{for:99}
 \varrho=\max\{\lambda\beta+\lambda_1, \text{\small$2$}\text{\tiny$(\frac{5}{2}$}\dim\mathfrak{g}+1)+\sigma_0\}
\end{align}
(see \eqref{for:179} for the reference of $\beta$, Theorem \ref{th:2} for $\lambda_1$ and \eqref{for:215} of Section \ref{sec:47} for $\sigma_0$).

\subsection{Main estimate} We recall notations in \eqref{for:214} of Section \ref{sec:47} and Section \ref{sec:25}. The following theorem is the central part of the global splitting argument.
\begin{theorem}\label{cor:7}
Suppose $\mathfrak{p}_v\in \text{Vect}^\infty(\mathcal{X})$, $v\in E$ satisfying
$\text{Ave}(\mathfrak{p}_v)=0$.  Set
\begin{align*}
 \mathcal{L}_{v}\mathfrak{p}_{u}-\mathcal{L}_{u}\mathfrak{p}_{v}= \mathfrak{w}_{v,u}
\end{align*}
 for any $v,\,u\in E$. For any $a>1$, there exist $\Theta,\,\mathcal{R}_v\in \text{Vect}^\infty(\mathcal{X})$, $v\in E$ satisfying $\text{Ave}(\Theta)=0$ and $\text{Ave}(\mathcal{R}_v)=0$ (see \eqref{for:44}) such that
\begin{align*}
 \mathfrak{p}_v=\mathcal{L}_{v}\Theta+\mathcal{R}_v
\end{align*}
with estimates: for all $v\in E$
\begin{align*}
 \max\{\norm{\Theta}_{C^r},\norm{\mathcal{R}_v}_{C^r}\}&\leq C_{r}(a^{r+\varrho}\norm{\mathfrak{p}}_{C^\varrho}+\norm{\mathfrak{p}}_{C^{r+\varrho}}),\quad \forall\,r\geq0;
\end{align*}
and
\begin{align*}
\norm{\mathcal{R}_v}_{C^0}&\leq C\norm{\mathfrak{w}}_{C^\varrho}
+C_{\ell}a^{2\varrho}\norm{\mathfrak{w}}_{C^\varrho}^{1-\frac{\varrho}{\ell}}(\norm{\mathfrak{p}}_{C^{\ell+\varrho}})^{\frac{\varrho}{\ell}}\notag\\
 &+C_{\ell}a^{2\varrho}(a^{\text{\tiny$-s$}}
 \norm{\mathfrak{p}}_{C^s})^{1-\frac{\varrho}{\ell}}(\norm{\mathfrak{p}}_{C^{\ell+\varrho}})^{\frac{\varrho}{\ell}}\notag\\
 &+C_{\ell}a^{2\varrho}(a^{-s}\norm{\mathfrak{p}}_{C^s})^{(\text{\tiny$1-\frac{\varrho}{\ell}$})^2}
 (\norm{\mathfrak{p}}_{C^{\ell+\varrho}})^{\frac{\varrho}{\ell}(2-\frac{\varrho}{\ell})}
\end{align*}
for any $s\geq \ell>\varrho$,  where we denote by
\begin{align*}
 \norm{\mathfrak{p}}_{C^r}=\max_{v\in E}\{\norm{\mathfrak{p}_{v}}_{C^r}\},\,\,\norm{\mathfrak{w}}_{C^r}=\max_{u,\,v\in E}\{\norm{\mathfrak{w}_{v,u}}_{C^r}\},\quad\forall\,r\geq0.
\end{align*}

\end{theorem}

To prove Theorem \ref{cor:7}, it suffices to prove the following Theorem \ref{th:8}. In fact, using the Sobolev embedding theorem, we translate estimates for Theorem \ref{th:8} to those in $C^m$ norms in Theorem \ref{cor:7}, which are then used in our iteration process. We recall the definition of $(\pi,\mathfrak{g}(\mathcal{O}))$, the extended representation of $(\pi,\mathcal{O})$ (\eqref{for:218} of Section \ref{sec:47}) in Section \ref{sec:2}.

\begin{theorem}\label{th:8}
Suppose $\mathfrak{p}_v\in \mathfrak{g}(\mathcal{O})^{\infty}$, $v\in E$. Set
\begin{align}\label{for:95}
 (v+\textrm{ad}_{v})\mathfrak{p}_{u}-(u+\textrm{ad}_{u})\mathfrak{p}_{v}= \mathfrak{w}_{v,u}
\end{align}
for $v,\,u\in E$. Then for any $a>1$ there exist $\Theta\in \mathfrak{g}(\mathcal{O})^{\infty}$ and $\mathcal{R}_v\in \mathfrak{g}(\mathcal{O})^{\infty}$ for each $v\in E$ such that
\begin{align*}
 \mathfrak{p}_v=(v+\textrm{ad}_{v})\Theta+\mathcal{R}_v
\end{align*}
with estimates: for any $v\in E$
\begin{align}\label{for:156}
 \max\{\norm{\Theta}_{l+\beta},\,\norm{\mathcal{R}_v}_{l+\beta}\}&\leq C_{l}(\norm{\mathfrak{p}}_{l+\varrho}+a^{l+\varrho} \norm{\mathfrak{p}}_{\varrho}),\qquad l\geq 0;
\end{align}
and
\begin{align}\label{for:157}
\norm{\mathcal{R}_v}_{\beta}&\leq C\norm{\mathfrak{w}}_{\varrho}
+C_{\ell}a^{2\varrho}\norm{\mathfrak{w}}_{\varrho}^{1-\frac{\varrho}{\ell}}(\norm{\mathfrak{p}}_{\ell+\varrho})^{\frac{\varrho}{\ell}}\notag\\
 &+C_{\ell}a^{2\varrho}(a^{\text{\tiny$-s$}}
 \norm{\mathfrak{p}}_{s})^{1-\frac{\varrho}{\ell}}(\norm{\mathfrak{p}}_{\ell+\varrho})^{\frac{\varrho}{\ell}}\notag\\
 &+C_{\ell}a^{2\varrho}(a^{-s}\norm{\mathfrak{p}}_s)^{(\text{\tiny$1-\frac{\varrho}{\ell}$})^2}
 (\norm{\mathfrak{p}}_{\ell+\varrho})^{\frac{\varrho}{\ell}(2-\frac{\varrho}{\ell})}
\end{align}
for any $s\geq\ell>\varrho$ (see \eqref{for:99} of Section \ref{sec:47}),   where we denote by
\begin{align*}
 \norm{\mathfrak{p}}_r=\max_{v\in E}\{\norm{\mathfrak{p}_{v}}_r\}, \,\, \norm{\mathfrak{w}}_r=\max_{u,\,v\in E}\{\norm{\mathfrak{w}_{v,u}}_r\}\quad\text{ for any }r\geq0.
\end{align*}
\emph{Note}. The constants in \eqref{for:157} are independent of $s$.
\end{theorem}
\begin{remark}
 In \eqref{for:157}, we use $\norm{\mathcal{R}_v}_{\beta}$ to subsequently bound bound  $\norm{\mathcal{R}_v}_{C^0}$ via the Sobolev embedding theorem (see \eqref{for:179}). This is why
we estimate $\norm{\mathcal{R}_v}_{\beta}$ instead of
$\norm{\mathcal{R}_v}$. Similarity, we estimate $\norm{\Theta}_{l+\beta}$ and $\norm{\mathcal{R}_v}_{l+\beta}$ instead of $\norm{\Theta}_{l}$ and $\norm{\mathcal{R}_v}_{l}$.

The estimates for both the solution $\Theta$ and the error $\mathcal{R}$ differ from the classical KAM ones.
For the former, there is a new term $a^{l+\varrho} \norm{\mathfrak{p}}_{\varrho}$ (see \eqref{for:156}).
For the latter,  we use two Sobolev orders $\ell+\varrho$ and $s$ of $\mathfrak{p}$ simultaneously to estimate $\norm{\mathcal{R}_v}_{\beta}$ (see \eqref{for:157}).
In fact, \eqref{for:157} would look much simpler if we let $s=\ell+\varrho$ (see \eqref{for:247}).

In Section \ref{sec:8} we will see that the new term $a^{l+\varrho} \norm{\mathfrak{p}}_{\varrho}$, which arises from applying the directional smoothing operators,  poses new
challenges in showing convergence of the KAM iteration. To overcome this difficulty, we introduce the two-orders trick to estimate $\mathcal{R}$.
\end{remark}

\subsection{Proof of Theorem \ref{th:8} when $\GG\neq \GG_1$}\label{sec:19}
In this case, $E_0=\{\textbf{v}\}$ and $\textbf{u}=U$ (see \eqref{for:290} of \eqref{for:286} of Section \ref{sec:47}). The following properties and notations will be used in the proof:
\begin{enumerate}

  \item We recall the definitions of the sets $\mathfrak{U}$, $\mathfrak{V}$, and $\mathfrak{W}$ in \eqref{for:287} of Section \ref{sec:47} and recall Examples \ref{ex:1} and  \ref{ex:5}.
We also recall the key algebraic \hyperlink{o.1}{Property (P) in case I} in Section \ref{sec:43}.

\smallskip
  \item\label{for:283} Fix $f_1$ related to $\mathfrak{V}$ (see \eqref{de:2} of Section \ref{sec:39}), $f_2$ related to $\mathfrak{W}$.
Then
$f_1$ is also related to $\mathfrak{U} $.
\end{enumerate}
We recommend that readers recall Example \ref{ex:1} for a clearer understanding of notations and refer to Section \ref{sec:45} for a better grasp of the underlying ideas.

\smallskip
\noindent\textbf{\emph{Step $1$}: Existence of $S_{1,U}$ splitting for $\textbf{u}$.} From \eqref{for:95} we have
\begin{align}\label{for:46}
 (\textbf{u}+\textrm{ad}_{\textbf{u}})\mathfrak{p}_{\textbf{v}}-(\textbf{v}+\textrm{ad}_{\textbf{v}})\mathfrak{p}_{\textbf{u}}= \mathfrak{w}_{\textbf{u},\textbf{v}}.
\end{align}
By Theorem \ref{th:10}, the extended regular representation $(\pi, \mathfrak{g}(\mathcal{O}))$ has a spectral gap restricted to each simple factor of $\GG$.
This together with \eqref{for:280} of \hyperlink{o.1}{Property (P) in case I} allow us to  apply Proposition \ref{po:1} to \eqref{for:46} by letting $u=U=\textbf{u}$, $v=\textbf{v}$.  Then  there exists $\eta\in \mathfrak{g}(\mathcal{O})_{S_{1,\textbf{u}}}^{\infty}$ with estimates
\begin{align}\label{for:61}
 \norm{\eta}_{S_1,l}\leq C_{l}\norm{\mathfrak{p}_{U},\,\mathfrak{p}_{\textbf{v}}}_{l+\sigma_0},\quad \forall\,l\geq0
\end{align}
such that
\begin{align}
 \mathcal{E}_{\textbf{u}}&=\mathfrak{p}_{\textbf{u}}-(\textbf{u}+\textrm{ad}_\textbf{u})\eta\qquad\text{and}\notag\\
 \mathcal{E}_{\textbf{v}}&=\mathfrak{p}_{\textbf{v}}-(\textbf{v}+\textrm{ad}_\textbf{v})\eta\label{for:105}
\end{align}
with the estimate
\begin{align}\label{for:10}
 \norm{\mathcal{E}_{\textbf{v}},\,\mathcal{E}_{\textbf{u}}}\overset{(1)}{\leq} C\norm{ \mathfrak{w}_{\textbf{u},\textbf{v}}}_{\sigma_0}.
\end{align}
Here in $(1)$ we use \eqref{for:281} by letting $t=0$.

\smallskip

\textbf{\emph{Step $2$}: Construction of $\Theta$ and $\mathcal{R}_v$ and obtaining estimates. }  Set
\begin{align}\label{for:133}
 \mathcal{P}=\pi_{\mathfrak{U}}(f_1\circ a^{-1})\pi_{\mathfrak{W}}(f_2\circ a^{-1})\pi_{\mathfrak{V}}(f_1\circ a^{-1})
\end{align}
(see \eqref{for:283}).  Let us define $\Theta$ as
\begin{align*}
 \Theta&=\mathcal{P}\eta.
\end{align*}
Then we define $\mathcal{R}_v$ as
\begin{align}\label{for:102}
 \mathcal{R}_v&=\mathfrak{p}_{v}-(v+\textrm{ad}_{v})\Theta,\qquad \forall\,v\in E.
\end{align}
The following estimates follow from the definition
of $\Theta$ and \eqref{for:26} of Corollary \ref{cor:9}
\begin{align}\label{for:66}
  \norm{\Theta}_{l}&\leq C_{l}(\norm{\eta}_{S_{1},l}+a^l\norm{\eta})\overset{(1)}{\leq}
  C_{l}( \norm{\mathfrak{p}}_{l+\sigma_0}+a^l \norm{\mathfrak{p}}_{\sigma_0})
\end{align}
for any $l\geq0$. Here in $(1)$ we use \eqref{for:61}.

Then we have
\begin{align}\label{for:86}
 \norm{\mathcal{R}_v}_{l}\leq \norm{\mathfrak{p}_{v}}_{l}+\norm{\Theta}_{l+1}\overset{\text{\tiny$(1)$}}{\leq} C_{l}( \norm{\mathfrak{p}}_{l+\sigma_0+1}+a^{l+1} \norm{\mathfrak{p}}_{\sigma_0}).
\end{align}
for any $l\geq0$. Here in $(1)$ we use \eqref{for:66}.

It follows from  \eqref{for:66} and \eqref{for:86} that
\begin{align*}
 &\max\{\norm{\Theta}_{l+\beta},\,\norm{\mathcal{R}_v}_{l+\beta}\}\\
 &\leq C_{l}
 \max\big\{\norm{\mathfrak{p}}_{l+\beta+\sigma_0}+a^{l+\beta} \norm{\mathfrak{p}}_{\sigma_0},\,
 \norm{\mathfrak{p}}_{l+\beta+\sigma_0+1}+a^{l+\beta+1} \norm{\mathfrak{p}}_{\sigma_0}\big\}\\
 &\overset{\text{\tiny$(1)$}}{\leq}C_{l}(\norm{\mathfrak{p}}_{l+\varrho}+a^{l+\varrho} \norm{\mathfrak{p}}_{\varrho})
\end{align*}
for any $l\geq0$. Here in $(1)$ we use the definition of $\varrho$ in  \eqref{for:99}. Hence we get \eqref{for:156}.

\smallskip
\textbf{\emph{Step $3$}: Estimates for $\norm{\mathcal{R}_\textbf{v}}$.} Before we estimate $\mathcal{R}_\textbf{v}$,  we need the following lemma:
 \begin{lemma}\label{for:106}
 $\mathcal{P}\textbf{v}=\textbf{v}\mathcal{P}$.
 \end{lemma}
\begin{proof} From \eqref{for:362} of \hyperlink{o.1}{Property (P) in case I}, keeping using \eqref{ob:3} of Lemma \ref{le:2}, we have
\begin{align*}
 \textbf{v}\mathcal{P}&=\textbf{v}\pi_{\mathfrak{U}}(f_1\circ a^{-1})\pi_{\mathfrak{W}}(f_2\circ a^{-1})\pi_{\mathfrak{V}}(f_1\circ a^{-1})\\
 &=\pi_{\mathfrak{U}}(f_1\circ a^{-1})\textbf{v}\pi_{\mathfrak{W}}(f_2\circ a^{-1})\pi_{\mathfrak{V}}(f_1\circ a^{-1})\\
 &=\pi_{\mathfrak{U}}(f_1\circ a^{-1})\pi_{\mathfrak{W}}(f_2\circ a^{-1})\textbf{v}\pi_{\mathfrak{V}}(f_1\circ a^{-1})\\
 &=\pi_{\mathfrak{U}}(f_1\circ a^{-1})\pi_{\mathfrak{W}}(f_2\circ a^{-1})\pi_{\mathfrak{V}}(f_1\circ a^{-1})\textbf{v}\\
 &=\mathcal{P}\textbf{v}.
\end{align*}
Then we finish the proof.
\end{proof}
By applying the operator $\mathcal{P}$ to both sides of \eqref{for:105} we have
\begin{align}\label{for:57}
 \mathcal{P}(\mathcal{E}_{\textbf{v}})&=\mathcal{P}(\mathfrak{p}_{\textbf{v}})-\mathcal{P}\big((\textbf{v}+\textrm{ad}_\textbf{v})\eta\big)\overset{\text{\tiny$(1)$}}{=}
 \mathcal{P}(\mathfrak{p}_{\textbf{v}})-(\textbf{v}+\textrm{ad}_\textbf{v})\mathcal{P}(\eta)\notag\\
 &=\mathcal{P}(\mathfrak{p}_{\textbf{v}})-(\textbf{v}+\textrm{ad}_\textbf{v})\Theta.
\end{align}
Here in $(1)$ we use Lemma \ref{for:106} and \eqref{for:2004} of Section \ref{sec:2}. It follows that
\begin{align*}
 \mathcal{R}_\textbf{v}&=\mathfrak{p}_{\textbf{v}}-(\textbf{v}+\textrm{ad}_\textbf{v})\Theta=\big(\mathfrak{p}_{\textbf{v}}-\mathcal{P}(\mathfrak{p}_{\textbf{v}})\big)
 +\big(\mathcal{P}(\mathfrak{p}_{\textbf{v}})-(\textbf{v}+\textrm{ad}_\textbf{v})\Theta\big)\\
 &\overset{\text{\tiny$(1)$}}{=}\big(\mathfrak{p}_{\textbf{v}}-\mathcal{P}(\mathfrak{p}_{\textbf{v}})\big)+\mathcal{P}(\mathcal{E}_{\textbf{v}}).
\end{align*}
Here in $(1)$ we use \eqref{for:57}.

Then we have
\begin{align}\label{for:197}
 \norm{\mathcal{R}_\textbf{v}}&\leq \norm{\mathfrak{p}_{\textbf{v}}-\mathcal{P}(\mathfrak{p}_{\textbf{v}})}+\norm{\mathcal{P}(\mathcal{E}_{\textbf{v}})}
 \overset{\text{\tiny$(1)$}}{\leq} \norm{\mathfrak{p}_{\textbf{v}}-\mathcal{P}(\mathfrak{p}_{\textbf{v}})}+C\norm{\mathcal{E}_{\textbf{v}}}\notag\\
 &\overset{\text{\tiny$(2)$}}{\leq}\norm{\mathfrak{p}_{\textbf{v}}-\mathcal{P}(\mathfrak{p}_{\textbf{v}})}+C\norm{ \mathfrak{w}_{\textbf{v},U}}_{\sigma_0}.
\end{align}
Here in $(1)$ we use \eqref{for:7} of Corollary \ref{cor:1}; in $(2)$ we use \eqref{for:10}.

Now we need to bound $\norm{\mathfrak{p}_{\textbf{v}}-\mathcal{P}(\mathfrak{p}_{\textbf{v}})}$. We note that
\begin{align}\label{for:43}
 \mathfrak{p}_{\textbf{v}}-\mathcal{P}(\mathfrak{p}_{\textbf{v}})&=\big(\mathfrak{p}_{\textbf{v}}-\pi_{\mathfrak{U}}(f_1\circ a^{-1})\mathfrak{p}_{\textbf{v}}\big)\notag\\
 &+\pi_{\mathfrak{U}}(f_1\circ a^{-1})\big(\mathfrak{p}_{\textbf{v}}-\pi_{\mathfrak{W}}(f_2\circ a^{-1})\mathfrak{p}_{\textbf{v}}\big)\notag\\
 &+\pi_{\mathfrak{U}}(f_1\circ a^{-1})\pi_{\mathfrak{W}}(f_2\circ a^{-1})\big(\mathfrak{p}_{\textbf{v}}-\pi_{\mathfrak{V}}(f_1\circ a^{-1})\mathfrak{p}_{\textbf{v}}\big).
\end{align}
Hence we have
\begin{align}\label{for:9}
 \|\mathfrak{p}_{\textbf{v}}-\mathcal{P}(\mathfrak{p}_{\textbf{v}})\|&\leq\|\mathfrak{p}_{\textbf{v}}-\pi_{\mathfrak{U}}(f_1\circ a^{-1})\mathfrak{p}_{\textbf{v}}\|\notag\\
 &+\big\|\pi_{\mathfrak{U}}(f_1\circ a^{-1})\big(\mathfrak{p}_{\textbf{v}}-\pi_{\mathfrak{W}}(f_2\circ a^{-1})\mathfrak{p}_{\textbf{v}}\big)\big\|\notag\\
 &+\Big\|\pi_{\mathfrak{U}}(f_1\circ a^{-1})\pi_{\mathfrak{W}}(f_2\circ a^{-1})\big(\mathfrak{p}_{\textbf{v}}-\pi_{\mathfrak{V}}(f_1\circ a^{-1})\mathfrak{p}_{\textbf{v}}\big)\Big\|\notag\\
 &\overset{\text{\tiny$(1)$}}{\leq}\|\mathfrak{p}_{\textbf{v}}-\pi_{\mathfrak{U}}(f_1\circ a^{-1})\mathfrak{p}_{\textbf{v}}\|+
 C\norm{\mathfrak{p}_{\textbf{v}}-\pi_{\mathfrak{W}}(f_2\circ a^{-1})\mathfrak{p}_{\textbf{v}}}\notag\\
 &+C\norm{\mathfrak{p}_{\textbf{v}}-\pi_{\mathfrak{V}}(f_1\circ a^{-1})\mathfrak{p}_{\textbf{v}}}\overset{\text{\tiny$(2)$}}{\leq} Ca^{\text{\tiny$-s$}}\norm{\mathfrak{p}_{\textbf{v}}}_{s}.
\end{align}
Here in $(1)$ we use \eqref{for:7} of Corollary \ref{cor:1}; in $(2)$ we use \eqref{for:145} of   Lemma  \ref{le:7}.

It follows from \eqref{for:197} and \eqref{for:9} that
\begin{align}\label{for:65}
  \norm{\mathcal{R}_\textbf{v}}&\leq C\norm{ \mathfrak{w}_{\textbf{v},U}}_{\sigma_0}+Ca^{\text{\tiny$-s$}}\norm{\mathfrak{p}}_s.
  \end{align}

\smallskip

\textbf{\emph{Step $4$}: Estimates for $\norm{\mathcal{R}_v}_{\beta}$, $v\in E$.} We use the higher rank trick in this part. We recall the twisted cocycle equation over $v$ and $\textbf{v}$ from \eqref{for:95}:
\begin{align*}
 (v+\textrm{ad}_{v})\mathfrak{p}_{\textbf{v}}-(\textbf{v}+\textrm{ad}_{\textbf{v}})\mathfrak{p}_{v}= \mathfrak{w}_{v,\textbf{v}}.
\end{align*}
We substitute  the expressions for $\mathfrak{p}_{\textbf{v}}$ and $\mathfrak{p}_{v}$
from \eqref{for:102} respectively into the above equation. Then we get
\begin{align*}
 (v+\textrm{ad}_{v})\big(\mathcal{R}_\textbf{v}+(\textbf{v}+\textrm{ad}_\textbf{v})\Theta\big)
 -(\textbf{v}+\textrm{ad}_{\textbf{v}})\big(\mathcal{R}_v+(v+\textrm{ad}_v)\Theta\big)= \mathfrak{w}_{v,\textbf{v}}
\end{align*}
for each $v\in E$. By using $[\textbf{v},v]=0$ we have
\begin{align}\label{for:132}
 (\textbf{v}+\textrm{ad}_{\textbf{v}})\mathcal{R}_v= -\mathfrak{w}_{v,\textbf{v}}+(v+\textrm{ad}_{v})\mathcal{R}_{\textbf{v}},\qquad \forall\,v\in E.
\end{align}
It follows from Theorem \ref{th:2}   that
\begin{align}\label{for:76}
 \norm{\mathcal{R}_v}_\beta&\leq C\norm{-\mathfrak{w}_{v,\textbf{v}}+(v+\textrm{ad}_{v})\mathcal{R}_{\textbf{v}}}_{\lambda\beta+\lambda_1}\notag\\
 &\overset{\text{\tiny$(0)$}}{\leq} C\norm{\mathfrak{w}}_{\varrho}+\norm{\mathcal{R}_{\textbf{v}}}_{\varrho}\notag\\
 &\overset{\text{\tiny$(1)$}}{\leq} C\norm{\mathfrak{w}}_{\varrho}+C_{\ell}
\norm{\mathcal{R}_{\textbf{v}}}^{1-\frac{\varrho}{\ell}}
(\norm{\mathcal{R}_{\textbf{v}}}_{\ell})^{\frac{\varrho}{\ell}}\notag\\
 &\overset{\text{\tiny$(2)$}}{\leq} C\norm{\mathfrak{w}}_{\varrho}+C_{\ell}(\norm{\mathfrak{w}}_{\sigma_0}+a^{\text{\tiny$-s$}}
 \norm{\mathfrak{p}}_{s})^{1-\frac{\varrho}{\ell}}( \norm{\mathfrak{p}}_{\ell+\sigma_0+1}+a^{\ell+1} \norm{\mathfrak{p}}_{\sigma_0})^{\frac{\varrho}{\ell}}\notag\\
 &\overset{\text{\tiny$(0)$}}{\leq} C\norm{\mathfrak{w}}_{\varrho}+C_{\ell}(\norm{\mathfrak{w}}_{\varrho}+a^{\text{\tiny$-s$}}
 \norm{\mathfrak{p}}_{s})^{1-\frac{\varrho}{\ell}}( \norm{\mathfrak{p}}_{\ell+\varrho}+a^{\ell+1} \norm{\mathfrak{p}}_{\varrho})^{\frac{\varrho}{\ell}}\notag\\
 &\leq C\norm{\mathfrak{w}}_{\varrho}+C_{\ell,1}a^{2\varrho}(\norm{\mathfrak{w}}_{\varrho}+a^{\text{\tiny$-s$}}
 \norm{\mathfrak{p}}_{s})^{1-\frac{\varrho}{\ell}}(\norm{\mathfrak{p}}_{\ell+\varrho})^{\frac{\varrho}{\ell}}\notag\\
 &\overset{\text{\tiny$(3)$}}{\leq} C\norm{\mathfrak{w}}_{\varrho}+C_{\ell,1}a^{2\varrho}\norm{\mathfrak{w}}_{\varrho}^{1-\frac{\varrho}{\ell}}(\norm{\mathfrak{p}}_{\ell+\varrho})^{\frac{\varrho}{\ell}}
 +C_{\ell,1}a^{2\varrho}(a^{\text{\tiny$-s$}}
 \norm{\mathfrak{p}}_{s})^{1-\frac{\varrho}{\ell}}(\norm{\mathfrak{p}}_{\ell+\varrho})^{\frac{\varrho}{\ell}}
\end{align}
for any $\ell\geq \varrho$.  Here in $(0)$ we use the definition of $\varrho$ in  \eqref{for:99}; in $(1)$ we use interpolation inequalities; and in $(2)$ we use \eqref{for:65} to estimate
$\norm{\mathcal{R}_{\textbf{v}}}$ and use \eqref{for:86} to estimate $\norm{\mathcal{R}_{\textbf{v}}}_{\ell}$; in $(3)$ we recall the inequality:
\begin{align}\label{for:217}
 (x+y)^c\leq x^c+y^c,\qquad \forall\, x,\,y>0,\,\,\,0<c<1.
\end{align}
\eqref{for:157} is from \eqref{for:76} (an extra term is from \eqref{for:33} of Section \ref{sec:20} as \eqref{for:157} is the maximum of \eqref{for:76} and \eqref{for:33}).  Then we finish the proof.

\subsection{Proof of Theorem \ref{th:8} when $\GG=\GG_1$}\label{sec:20}

In this case, $E_0=\{\textbf{u}_1,\,\textbf{u}_2\}$ and $\textbf{u}_3=U$ (see \eqref{for:291} of \eqref{for:286} of Section \ref{sec:47}).
 The following properties and notations will be used in the proof:
\begin{enumerate}
  \item We recall the definitions of the sets $\mathfrak{U}^i$, $\mathfrak{V}^i$, $i=1,\,2$ and $\mathfrak{W}$ in Section \ref{sec:43} and we call Example \ref{ex:2}.
We also recall the key algebraic \hyperlink{o.3}{Property (P) in case II} in Section \ref{sec:37}.

\smallskip

  \item\label{for:292} We recall Example \ref{ex:3} in Section \ref{sec:24}. We note that
      \begin{gather*}
 \mathcal{J}\subset L=\{C(G_{\textbf{u}_1}),G_{\textbf{u}_1}\}\cap S_{0,\textbf{u}_3}= S_{0,\textbf{u}_1}\cap S_{0,\textbf{u}_3}\,\text{ and }\\
 \mathcal{J}\subset S_{1,\textbf{u}_3}
 \end{gather*}
 ($\mathcal{J}$ is defined in \eqref{for:288} of \ref{sec:47} and described in Example \ref{ex:5} of Section \ref{sec:50}). Thus we have $\mathcal{J}\subset L\cap S_{1,\textbf{u}_3}$.
  \smallskip

  \item Fix $f_2$ related to $\mathfrak{W}$ (see \eqref{de:2} of Section \ref{sec:39}) and $f_3$ related to $\mathfrak{V}^1$. Then
$f_3$ is also related to $\mathfrak{V}^2$, $\mathfrak{U}^1$ and $\mathfrak{U}^2$.
\end{enumerate}
We recommend that readers recall Example \ref{ex:2} for a clearer understanding of notations and refer to Section \ref{sec:48} for a better grasp of the underlying ideas.

\smallskip
\textbf{\emph{Step $1$}: Existence of $S_{1,\textbf{u}_3}$ splitting for $\textbf{u}_1\in E_0$.} From \eqref{for:95} we have
\begin{align}\label{for:162}
(\textbf{u}_1+\textrm{ad}_{\textbf{u}_1})\mathfrak{p}_{\textbf{u}_3}-(\textbf{u}_3+\textrm{ad}_{\textbf{u}_3})\mathfrak{p}_{\textbf{u}_1}= \mathfrak{w}_{\textbf{u}_1,\textbf{u}_3}.
\end{align}
By Theorem \ref{th:10}, the extended regular representation $(\pi, \mathfrak{g}(\mathcal{O}))$ has a spectral gap.
This together with \eqref{for:284} of \hyperlink{o.3}{Property (P) in case II} allow us to  apply Proposition \ref{po:1} to \eqref{for:162} by letting $u=U=\textbf{u}_3$, $v=\textbf{u}_1$.
Then there exists $\eta\in \mathfrak{g}(\mathcal{O})_{S_{1,\textbf{u}_3}}^{\infty}$ with estimates
\begin{align}\label{for:58}
 \norm{\eta}_{S_1,l}\leq C_{l}\norm{\mathfrak{p}_{\textbf{u}_3},\,\mathfrak{p}_{\textbf{u}_1}}_{l+\sigma_0},\quad \forall\,l\geq0
\end{align}
such that
\begin{align}
\mathfrak{p}_{\textbf{u}_3}&=\mathcal{E}_{\textbf{u}_3}+(\textbf{u}_3+\textrm{ad}_{\textbf{u}_3})\eta, \notag\\
  \mathfrak{p}_{\textbf{u}_1}&=\mathcal{E}_{\textbf{u}_1}+(\textbf{u}_1+\textrm{ad}_{\textbf{u}_1})\eta \label{for:104}
\end{align}
with estimates
\begin{align}\label{for:152}
 \norm{\mathcal{E}_{\textbf{u}_3},\,\mathcal{E}_{\textbf{u}_1}}_{\mathcal{J},l}\overset{(1)}{\leq} \norm{\mathcal{E}_{\textbf{u}_3},\,\mathcal{E}_{\textbf{u}_1}}_{L\cap S_{1,\textbf{u}_3},l}\leq C_l\norm{ \mathfrak{w}_{\textbf{u}_1,\textbf{u}_3}}_{l+\sigma_0},\quad \forall\,l\geq0.
\end{align}
Here in $(1)$ we use \eqref{for:292}.

\smallskip
\textbf{\emph{Step $2$}: Construction of $\Theta$ and $\mathcal{R}_v$, $v\in E$.} Set
\begin{align}\label{for:154}
 \mathcal{P}=\pi_{\mathfrak{U}^1}(f_3\circ a^{-1})\pi_{\mathfrak{U}^2}(f_3\circ a^{-1})\pi_{\mathfrak{W}}(f_2\circ a^{-1})\pi_{\mathfrak{V}^2}(f_3\circ a^{-1})\pi_{\mathfrak{V}^1}(f_3\circ a^{-1}).
\end{align}
Let us define $\Theta$ as
\begin{align*}
 \Theta&=\mathcal{P}\eta.
\end{align*}
Then we define $\mathcal{R}_v$ as
\begin{align}\label{for:103}
 \mathcal{R}_v&=\mathfrak{p}_{v}-(v+\textrm{ad}_{v})\Theta,\qquad \forall\,v\in E.
\end{align}
 The following estimates follow from the definition
of $\Theta$ and \eqref{for:28} of Corollary \ref{cor:9}
\begin{align}\label{for:110}
  \norm{\Theta}_{l}&\leq C_{l}(\norm{\eta}_{S_{1},l}+a^l\norm{\eta})\overset{\text{\tiny$(1)$}}{\leq}
  C_{l}( \norm{\mathfrak{p}}_{l+\sigma_0}+a^l \norm{\mathfrak{p}}_{\sigma_0})
\end{align}
for any $l\geq0$. Here in $(1)$ we use \eqref{for:58}.

Then we have
\begin{align}\label{for:112}
 \norm{\mathcal{R}_v}_{l}\leq \norm{\mathfrak{p}_{v}}_{l}+\norm{\Theta}_{l+1}\overset{\text{\tiny$(1)$}}{\leq} C_{l}( \norm{\mathfrak{p}}_{l+\sigma_0+1}+a^{l+1} \norm{\mathfrak{p}}_{\sigma_0})
\end{align}
for any $l\geq0$. Here in $(1)$ we use \eqref{for:110}.

It follows from  \eqref{for:110} and \eqref{for:112} that
\begin{align*}
 &\max\{\norm{\Theta}_{l+\beta},\,\norm{\mathcal{R}_v}_{l+\beta}\}\\
 &\leq C_{l}
 \max\big\{\norm{\mathfrak{p}}_{l+\beta+\sigma_0}+a^{l+\beta} \norm{\mathfrak{p}}_{\sigma_0},\,
 \norm{\mathfrak{p}}_{l+\beta+\sigma_0+1}+a^{l+\beta+1} \norm{\mathfrak{p}}_{\sigma_0}\big\}\\
 &\overset{\text{\tiny$(1)$}}{\leq}C_{l}(\norm{\mathfrak{p}}_{l+\varrho}+a^{l+\varrho} \norm{\mathfrak{p}}_{\varrho})
\end{align*}
for any $l\geq0$. Here in $(1)$ we use the definition of $\varrho$ in  \eqref{for:99}. Hence we get \eqref{for:156}.

\smallskip

\textbf{\emph{Step $3$}: Construction of $\eta_i$ and $\mathfrak{R}_{\textbf{u}_i,j}$.} Let us define $\mathcal{P}_i$, $i=1,2,3$ as
\begin{align}\label{for:2009}
 \mathcal{P}_1=\pi_{\mathfrak{V}^1}(f_3\circ a^{-1}),\quad&\mathcal{P}_2=\pi_{\mathfrak{U}^2}(f_3\circ a^{-1})\pi_{\mathfrak{W}}(f_2\circ a^{-1})\pi_{\mathfrak{V}^2}(f_3\circ a^{-1})\notag\\
 &\mathcal{P}_3=\pi_{\mathfrak{U}^1}(f_3\circ a^{-1});
\end{align}
and define $\eta_i$, $i=1,2$ as
, $i=1,2$ as
\begin{align}
\eta_1&=\mathcal{P}_1\eta, \quad \eta_2=\mathcal{P}_2\eta_1=\mathcal{P}_2\mathcal{P}_1\eta.\label{for:85}
\end{align}
It is clear that
\begin{align}\label{for:2010}
  \Theta=\mathcal{P}_3\eta_2=\mathcal{P}_3\mathcal{P}_2\mathcal{P}_1\eta.
\end{align}
Then we define
\begin{align}\label{for:115}
 \mathfrak{R}_{\textbf{u}_i,j}=\mathfrak{p}_{\textbf{u}_i}-(\textbf{u}_i+\textrm{ad}_{\textbf{u}_i})\eta_j,\quad i=1,2,\,\,j=1,2.
\end{align}
From \eqref{for:363} of \hyperlink{o.3}{Property (P) in case II}, by using \eqref{ob:3} of Lemma \ref{le:2} we have
\begin{align}
  \mathcal{P}_1\textbf{u}_1&=\textbf{u}_1\mathcal{P}_1;\qquad\text{and}\label{for:134}\\
  \mathcal{P}_2\textbf{u}_2&=\textbf{u}_2\mathcal{P}_2;\qquad\text{and}\label{for:135}\\
  \mathcal{P}_3\textbf{u}_1&=\textbf{u}_1\mathcal{P}_3.\label{for:161}
\end{align}
At the end of part, we define a new function, whose estimates will be frequently used later:
\begin{align}\label{for:243}
 \Psi(l)&=\norm{\mathfrak{p}_{\textbf{u}_1}-\pi_{\mathfrak{V}^1}(f_3\circ a^{-1})\mathfrak{p}_{\textbf{u}_1}}_{l}+\norm{\mathfrak{p}_{\textbf{u}_2}-\pi_{\mathfrak{U}^2}(f_3\circ a^{-1})\mathfrak{p}_{\textbf{u}_2}}_l\notag\\
 &+
 \norm{\mathfrak{p}_{\textbf{u}_2}-\pi_{\mathfrak{W}}(f_2\circ a^{-1})\mathfrak{p}_{\textbf{u}_2}}_l
 +
 \norm{\mathfrak{p}_{\textbf{u}_2}-\pi_{\mathfrak{V}^2}(f_3\circ a^{-1})\mathfrak{p}_{\textbf{u}_2}}_l\notag\\
 &+\norm{\mathfrak{p}_{\textbf{u}_1}-\pi_{\mathfrak{U}^1}(f_3\circ a^{-1})\mathfrak{p}_{\textbf{u}_1}}_{l}.
\end{align}
Then by \eqref{for:146} of Corollary \ref{cor:1}, we have
\begin{align}\label{for:148}
  \Psi(\varrho)\leq C\big(a^{-s}\norm{\xi}_s\big)^{\text{\tiny$1-\frac{\varrho}{\ell}$}}\norm{\xi}^{\text{\tiny$\frac{\varrho}{\ell}$}}_{\ell}
\end{align}
for any $\ell\geq\varrho$.

\smallskip
\textbf{\emph{Step $4$}: Estimates for $\mathfrak{R}_{\textbf{u}_1,1}$. } By applying the operator $\mathcal{P}_1=\pi_{\mathfrak{V}^1}(f_3\circ a^{-1})$ to both sides of \eqref{for:104} we have
\begin{align}\label{for:74}
\mathcal{P}_1\mathfrak{p}_{\textbf{u}_1}&=
\mathcal{P}_1\big((\textbf{u}_1+\textrm{ad}_{\textbf{u}_1})\eta\big)+\mathcal{P}_1\mathcal{E}_{\textbf{u}_1}
\overset{\text{\tiny$(1)$}}{=}(\textbf{u}_1+\textrm{ad}_{\textbf{u}_1})(\mathcal{P}_1\eta)+\mathcal{P}_1\mathcal{E}_{\textbf{u}_1}\notag\\
&=(\textbf{u}_1+\textrm{ad}_{\textbf{u}_1})\eta_1+\mathcal{P}_1\mathcal{E}_{\textbf{u}_1}.
\end{align}
Here in $(1)$ we use \eqref{for:134}.

It follows that
\begin{align*}
\mathfrak{R}_{\textbf{u}_1,1}&=\mathfrak{p}_{\textbf{u}_1}-(\textbf{u}_1+\textrm{ad}_{\textbf{u}_1})\eta_1\\
 &=
 \big(\mathfrak{p}_{\textbf{u}_1}-\mathcal{P}_1\mathfrak{p}_{\textbf{u}_1}\big)
 +\big(\mathcal{P}_1\mathfrak{p}_{\textbf{u}_1}-(\textbf{u}_1+\textrm{ad}_{\textbf{u}_1})\eta_1\big)\notag\\
 &\overset{\text{\tiny$(1)$}}{=}\big(\mathfrak{p}_{\textbf{u}_1}-\mathcal{P}_1\mathfrak{p}_{\textbf{u}_1}\big)
 +\mathcal{P}_1\mathcal{E}_{\textbf{u}_1}.\notag
\end{align*}
Here in $(1)$ we use notation \eqref{for:74}.

Then we have
\begin{align}\label{for:120}
 \norm{\mathfrak{R}_{\textbf{u}_1,1}}_{\mathcal{J},l}
 &\leq \norm{\mathfrak{p}_{\textbf{u}_1}-\mathcal{P}_1\mathfrak{p}_{\textbf{u}_1}}_{l}
 +\norm{\mathcal{P}_1\mathcal{E}_{\textbf{u}_1}}_{\mathcal{J},l}\notag\\
 &\overset{\text{\tiny$(1)$}}{\leq}\Psi(l)+C_{l}\norm{\mathcal{E}_{\textbf{u}_1}}_{\mathcal{J},l}
 \overset{\text{\tiny$(2)$}}{\leq} \Psi(l)+C_l\norm{ \mathfrak{w}_{\textbf{u}_1,U}}_{l+\sigma_0}.
\end{align}
for any $l\geq0$. Here in $(1)$ we use \eqref{for:90} of Corollary \ref{cor:10}; in
$(2)$ we use \eqref{for:152}.

\smallskip
\textbf{\emph{Step $5$}: Estimates for $\mathfrak{R}_{\textbf{u}_2,1}$. } Next, a natural thing that comes to mind is: we apply $\mathcal{P}_2$ on each side
of the equation
\begin{align}\label{for:96}
 \mathfrak{R}_{\textbf{u}_1,1}=\mathfrak{p}_{\textbf{u}_1}-(\textbf{u}_1+\textrm{ad}_{\textbf{u}_1})\eta_1
\end{align}
over $\textbf{u}_1$. Thus, we could construct a new almost twisted coboundary which could be solved by $\eta_2=\mathcal{P}_2\eta_1$.    However,
it seems not hopeful as $\mathcal{P}_2$ does not commute with $\textbf{u}_1$. Instead, \eqref{for:135} suggests that we use the following equation over $\textbf{u}_2$:
\begin{align}\label{for:108}
 \mathfrak{R}_{\textbf{u}_2,1}&=\mathfrak{p}_{\textbf{u}_2}-(\textbf{u}_2+\textrm{ad}_{\textbf{u}_2})\eta_1.
\end{align}
To do so, we need to estimate $\mathfrak{R}_{\textbf{u}_2,1}$ at first. We use again the higher rank trick as what we did in step 4 of Section \ref{sec:19}.
We recall the twisted cocycle equation over $\textbf{u}_1$ and $\textbf{u}_2$ from \eqref{for:95}:
\begin{align}
 (\textbf{u}_1+\textrm{ad}_{\textbf{u}_1})\mathfrak{p}_{\textbf{u}_2}-(\textbf{u}_2+\textrm{ad}_{\textbf{u}_2})\mathfrak{p}_{\textbf{u}_1}= \mathfrak{w}_{\textbf{u}_1,\textbf{u}_2}.\label{for:150}
\end{align}
We substitute the expressions for $\mathfrak{p}_{\textbf{u}_i}$, $i=1,2$ in \eqref{for:96} and  \eqref{for:108} respectively  into \eqref{for:150}. Then we have
\begin{align}
 &(\textbf{u}_1+\textrm{ad}_{\textbf{u}_1})\big(\mathfrak{R}_{\textbf{u}_2,1}+(\textbf{u}_2+\textrm{ad}_{\textbf{u}_2})\eta_1\big)\notag\\
 &-(\textbf{u}_2+\textrm{ad}_{\textbf{u}_2})\big(\mathfrak{R}_{\textbf{u}_1,1}+(\textbf{u}_1+\textrm{ad}_{\textbf{u}_1})\eta_1\big)= \mathfrak{w}_{\textbf{u}_1,\textbf{u}_2}\notag\\
 \overset{\text{(1)}}{\Longrightarrow}&(\textbf{u}_1+\textrm{ad}_{\textbf{u}_1})\mathfrak{R}_{\textbf{u}_2,1}= (\textbf{u}_2+\textrm{ad}_{\textbf{u}_2})\mathfrak{R}_{\textbf{u}_1,1}+\mathfrak{w}_{\textbf{u}_1,\textbf{u}_2}.\label{for:116}
\end{align}
Here in $(1)$ we use $[\textbf{u}_1,\textbf{u}_2]=0$.

Then we apply Corollary \ref{cor:6} to \eqref{for:116} to estimate $\mathfrak{R}_{\textbf{u}_2,1}$. Since $G_{\textbf{u}_2}\subset C(G_{\textbf{u}_1})$ (see \eqref{for:364} of \hyperlink{o.3}{Property (P) in case II}), let $H=G_{\textbf{u}_2}$ and $u=\textbf{u}_1$. Then $\mathcal{J}=\{H,G_{u}\}$. Hence we have
\begin{align}\label{for:130}
 \norm{\mathfrak{R}_{\textbf{u}_2,1}}_{\mathcal{J},l}&\leq C_{l}\|(\textbf{u}_2+\textrm{ad}_{\textbf{u}_2})\mathfrak{R}_{\textbf{u}_1,1}
 +\mathfrak{w}_{\textbf{u}_1,\textbf{u}_2}\|_{\mathcal{J},l+\text{\small$\frac{5}{2}$}\dim\mathfrak{g}}\notag\\
 &\leq C_{l}\|\mathfrak{R}_{\textbf{u}_1,1}\|_{\mathcal{J},l+\text{\small$\frac{5}{2}$}\dim\mathfrak{g}+1}
 +C_l\norm{\mathfrak{w}_{\textbf{u}_1,\textbf{u}_2}}_{l+\text{\small$\frac{5}{2}$}\dim\mathfrak{g}}\notag\\
 &\overset{\text{\tiny$(1)$}}{\leq} C_l\Psi(l+\text{\tiny$\frac{5}{2}$}\dim\mathfrak{g}+1)+C_l\norm{ \mathfrak{w}}_{l+\text{\small$\frac{5}{2}$}\dim\mathfrak{g}+\sigma_0}
 \end{align}
for any $l\geq0$. Here in $(1)$ we use \eqref{for:120}.

\smallskip
\textbf{\emph{Step $5$}: Estimates for $\mathfrak{R}_{\textbf{u}_2,2}$.} Now we use the equation \ref{for:108} to get $\eta_2$, a new approximate solution.
By applying the operator
\begin{align*}
\mathcal{P}_2=\pi_{\mathfrak{U}^2}(f_3\circ a^{-1})\pi_{\mathfrak{W}}(f_2\circ a^{-1})\pi_{\mathfrak{V}^2}(f_3\circ a^{-1})
\end{align*}
on both sides of \eqref{for:108} we have
\begin{align}\label{for:155}
\mathcal{P}_2\mathfrak{R}_{\textbf{u}_2,1}&=\mathcal{P}_2\mathfrak{p}_{\textbf{u}_2}-\mathcal{P}_2\big((\textbf{u}_2+\textrm{ad}_{\textbf{u}_2})\eta_1\big)\overset{\text{\tiny$(1)$}}{=}\mathcal{P}_2\mathfrak{p}_{\textbf{u}_2}
-(\textbf{u}_2+\textrm{ad}_{\textbf{u}_2})(\mathcal{P}_2\eta_1)\notag\\
 &=\mathcal{P}_2\mathfrak{p}_{\textbf{u}_2}-(\textbf{u}_2+\textrm{ad}_{\textbf{u}_2})\eta_2.
\end{align}
Here in $(1)$ we use \eqref{for:135}.

It follows that
\begin{align*}
 \mathfrak{R}_{\textbf{u}_2,2}&=\mathfrak{p}_{\textbf{u}_2}-(\textbf{u}_2+\textrm{ad}_{\textbf{u}_2})\eta_2\\
 &=
 \big(\mathfrak{p}_{\textbf{u}_2}-\mathcal{P}_2\mathfrak{p}_{\textbf{u}_2}\big)+\big(\mathcal{P}_2\mathfrak{p}_{\textbf{u}_2}-(\textbf{u}_2+\textrm{ad}_{\textbf{u}_2})\eta_2\big)\\
 &\overset{\text{\tiny$(1)$}}{=}\big(\mathfrak{p}_{\textbf{u}_2}-\mathcal{P}_2\mathfrak{p}_{\textbf{u}_2}\big)+\mathcal{P}_2\mathfrak{R}_{\textbf{u}_2,1}.
\end{align*}
Here in $(1)$ we use \eqref{for:155}.

To estimate $\norm{\mathfrak{R}_{\textbf{u}_2,2}}_{\mathcal{J},l}$, we need to bound
\begin{align*}
 &\norm{\mathfrak{p}_{\textbf{u}_2}-\mathcal{P}_2\mathfrak{p}_{\textbf{u}_2}}_{l}\quad\text{and}\quad\norm{\mathcal{P}_2\mathfrak{R}_{\textbf{u}_2,1}}_{\mathcal{J},l}
\end{align*}
respectively. For the former, similar to \eqref{for:43}, firstly, we rewrite it as
\begin{align*}
\mathfrak{p}_{\textbf{u}_2}-\mathcal{P}_2\mathfrak{p}_{\textbf{u}_2}= &\mathfrak{p}_{\textbf{u}_2}-\pi_{\mathfrak{U}^2}(f_3\circ a^{-1})\pi_{\mathfrak{W}}(f_2\circ a^{-1})\pi_{\mathfrak{V}^2}(f_3\circ a^{-1})\mathfrak{p}_{\textbf{u}_2}\\
 &= \mathfrak{p}_{\textbf{u}_2}-\pi_{\mathfrak{U}^2}(f_3\circ a^{-1})\mathfrak{p}_{\textbf{u}_2}\\
 &+\pi_{\mathfrak{U}^2}(f_3\circ a^{-1})\big(\mathfrak{p}_{\textbf{u}_2}-\pi_{\mathfrak{W}}(f_2\circ a^{-1})\mathfrak{p}_{\textbf{u}_2}\big)\\
 &+\pi_{\mathfrak{U}^2}(f_3\circ a^{-1})\pi_{\mathfrak{W}}(f_2\circ a^{-1})\big(\mathfrak{p}_{\textbf{u}_2}-\pi_{\mathfrak{V}^2}(f_3\circ a^{-1})\mathfrak{p}_{\textbf{u}_2}\big).
\end{align*}
It follows that
\begin{align*}
 \norm{\mathfrak{p}_{\textbf{u}_2}-\mathcal{P}_2\mathfrak{p}_{\textbf{u}_2}}_{l}&\leq
 \norm{\mathfrak{p}_{\textbf{u}_2}-\pi_{\mathfrak{U}^2}(f_3\circ a^{-1})\mathfrak{p}_{\textbf{u}_2}}_l\\
 &+\norm{\pi_{\mathfrak{U}^2}(f_3\circ a^{-1})\big(\mathfrak{p}_{\textbf{u}_2}-\pi_{\mathfrak{W}}(f_2\circ a^{-1})\mathfrak{p}_{\textbf{u}_2}\big)}_l\\
 &+\norm{\pi_{\mathfrak{U}^2}(f_3\circ a^{-1})\pi_{\mathfrak{W}}(f_2\circ a^{-1})\big(\mathfrak{p}_{\textbf{u}_2}-\pi_{\mathfrak{V}^2}(f_3\circ a^{-1})\mathfrak{p}_{\textbf{u}_2}\big)}_l\\
 &\overset{(1)}{\leq}\norm{\mathfrak{p}_{\textbf{u}_2}-\pi_{\mathfrak{U}^2}(f_3\circ a^{-1})\mathfrak{p}_{\textbf{u}_2}}_l
 +C_l
 \norm{\mathfrak{p}_{\textbf{u}_2}-\pi_{\mathfrak{W}}(f_2\circ a^{-1})\mathfrak{p}_{\textbf{u}_2}}_l\\
 &+C_l
 \norm{\mathfrak{p}_{\textbf{u}_2}-\pi_{\mathfrak{V}^2}(f_3\circ a^{-1})\mathfrak{p}_{\textbf{u}_2}}_l\\
 &\overset{(2)}{\leq} C_{l}\Psi(l).
\end{align*}
for any $l\geq0$.  Here in $(1)$ we use \eqref{for:7} of Corollary \ref{cor:1}; in $(2)$ we recall the definition of $\Psi$, see \eqref{for:243}.

For the latter, by \eqref{for:90} of Corollary \ref{cor:10} we have
\begin{align*}
  \norm{\mathcal{P}_2\mathfrak{R}_{\textbf{u}_2,1}}_{\mathcal{J},l}&\leq C_{l,1}\norm{\mathfrak{R}_{\textbf{u}_2,1}}_{\mathcal{J},l}\\
  &\overset{\text{\tiny$(1)$}}{\leq}C_l\Psi(l+\text{\tiny$\frac{5}{2}$}\dim\mathfrak{g}+1)+C_l\norm{ \mathfrak{w}}_{l+\text{\small$\frac{5}{2}$}\dim\mathfrak{g}+\sigma_0}
\end{align*}
for any $l\geq0$. Here in $(1)$ we use \eqref{for:130}.

Hence we have
\begin{align}\label{for:138}
 \norm{\mathfrak{R}_{\textbf{u}_2,2}}&_{\mathcal{J},l}\leq \norm{\mathfrak{p}_{\textbf{u}_2}-\mathcal{P}_2\mathfrak{p}_{\textbf{u}_2}}_{l}+\norm{\mathcal{P}_2\mathfrak{R}_{\textbf{u}_2,1}}_{\mathcal{J},l}\notag\\
 &\leq C_l\Psi(l+\text{\tiny$\frac{5}{2}$}\dim\mathfrak{g}+1)+C_l\norm{ \mathfrak{w}}_{l+\text{\small$\frac{5}{2}$}\dim\mathfrak{g}+\sigma_0}
\end{align}
for any $l\geq0$.

\smallskip
\textbf{\emph{Step $6$}: Estimates for $\mathfrak{R}_{\textbf{u}_1,2}$.}  Now it is clear what we should do next. We will use the equation
\begin{align}\label{for:107}
\mathfrak{R}_{\textbf{u}_1,2}&=\mathfrak{p}_{\textbf{u}_1}-(\textbf{u}_1+\textrm{ad}_{\textbf{u}_1})\eta_2
\end{align}
instead  of
\begin{align}\label{for:114}
 \mathfrak{R}_{\textbf{u}_2,2}&=\mathfrak{p}_{\textbf{u}_2}-(\textbf{u}_2+\textrm{ad}_{\textbf{u}_2})\eta_2
\end{align}
to apply $\mathcal{P}_3$ to get $\Theta$. Before that, we still need to estimate $\mathfrak{R}_{\textbf{u}_1,2}$. We use again the twisted cocycle equation \eqref{for:150}.
We substitute the expressions for $\mathfrak{p}_{\textbf{u}_i}$, $i=1,2$ in \eqref{for:107} and \eqref{for:114} respectively  into \eqref{for:150}. Then similar to \eqref{for:116} we have
\begin{align}
 &(\textbf{u}_1+\textrm{ad}_{\textbf{u}_1})\big(\mathfrak{R}_{\textbf{u}_2,2}+(\textbf{u}_2+\textrm{ad}_{\textbf{u}_2})\eta_2\big)\notag\\
 &-(\textbf{u}_2+\textrm{ad}_{\textbf{u}_2})\big(\mathfrak{R}_{\textbf{u}_1,2}+(\textbf{u}_1+\textrm{ad}_{\textbf{u}_1})\eta_2\big)= \mathfrak{w}_{\textbf{u}_1,\textbf{u}_2}\notag\\
 \overset{\text{(1)}}{\Longrightarrow}&(\textbf{u}_2+\textrm{ad}_{\textbf{u}_2})\mathfrak{R}_{\textbf{u}_1,2}
 =(\textbf{u}_1+\textrm{ad}_{\textbf{u}_1})\mathfrak{R}_{\textbf{u}_2,2}-\mathfrak{w}_{\textbf{u}_1,\textbf{u}_2}.\label{for:131}
\end{align}
Here in $(1)$ we use $[\textbf{u}_1,\textbf{u}_2]=0$.

Again we apply Corollary \ref{cor:6} to \eqref{for:131} to estimate $\mathfrak{R}_{\textbf{u}_1,2}$. Since  $G_{\textbf{u}_1}\subset C(G_{\textbf{u}_2})$ (see \eqref{for:364} of \hyperlink{o.3}{Property (P) in case II}), let $H=G_{\textbf{u}_1}$ and $u=\textbf{u}_2$. Then $\mathcal{J}=\{H,G_{u}\}$. Hence we have
\begin{align}\label{for:139}
 \norm{\mathfrak{R}_{\textbf{u}_1,2}}&\leq C\|(\textbf{u}_1+\textrm{ad}_{\textbf{u}_1})\mathfrak{R}_{\textbf{u}_2,2}
 -\mathfrak{w}_{\textbf{u}_1,\textbf{u}_2}\|_{\mathcal{J},\text{\small$\frac{5}{2}$}\dim\mathfrak{g}}\notag\\
 &\leq C\|\mathfrak{R}_{\textbf{u}_2,2}\|_{\mathcal{J},\text{\small$\frac{5}{2}$}\dim\mathfrak{g}+1}
 +C\norm{\mathfrak{w}_{\textbf{u}_1,\textbf{u}_2}}_{\mathcal{J},\text{\small$\frac{5}{2}$}\dim\mathfrak{g}}\notag\\
 &\overset{\text{\tiny$(1)$}}{\leq} C\Psi(\text{\tiny$2$}(\text{\tiny$\frac{5}{2}$}\dim\mathfrak{g}+1))+C\norm{ \mathfrak{w}}_{2(\text{\small$\frac{5}{2}$}\dim\mathfrak{g})+\sigma_0}\notag\\
 &\overset{\text{\tiny$(2)$}}{\leq}C\Psi(\varrho)+C\norm{ \mathfrak{w}}_{\varrho}.
 \end{align}
Here in $(1)$ we use \eqref{for:138}; in $(2)$ we use the definition of $\varrho$ in  \eqref{for:99}.

\smallskip
\emph{ Note}.  We only estimate $L^2$ norm of $\mathfrak{R}_{\textbf{u}_1,2}$, instead of the higher order norms on $\mathcal{J}$ like former ones. The reason is
this is the last time we use the higher rank trick to switch between equations over $\textbf{u}_1$ and $\textbf{u}_2$.

\smallskip
\textbf{\emph{Step $7$}: Estimates for $\mathcal{R}_{\textbf{u}_1}$.} Now we use the equation \ref{for:107} to get $\Theta$.
By applying the operator $\mathcal{P}_3=\pi_{\mathfrak{U}^1}(f_3\circ a^{-1})$ on both sides of \ref{for:107} we have
\begin{align}\label{for:163}
\mathcal{P}_3\mathfrak{R}_{\textbf{u}_1,2}&=\mathcal{P}_3\mathfrak{p}_{\textbf{u}_1}
-\mathcal{P}_3\big((\textbf{u}_1+\textrm{ad}_{\textbf{u}_1})\eta_2\big)\overset{\text{\tiny$(1)$}}{=}\mathcal{P}_3\mathfrak{p}_{\textbf{u}_1}
-(\textbf{u}_1+\textrm{ad}_{\textbf{u}_1})(\mathcal{P}_3\eta_2)\notag\\
&=\mathcal{P}_3\mathfrak{p}_{\textbf{u}_1}
-(\textbf{u}_1+\textrm{ad}_{\textbf{u}_1})\Theta.
\end{align}
Here in $(1)$ we use \eqref{for:161}.

It follows that
\begin{align*}
 \mathcal{R}_{\textbf{u}_1}&=\mathfrak{p}_{\textbf{u}_1}-(\textbf{u}_1+\textrm{ad}_{\textbf{u}_1})\Theta\\
 &=
 \big(\mathfrak{p}_{\textbf{u}_1}-\mathcal{P}_3\mathfrak{p}_{\textbf{u}_1}\big)+\big(\mathcal{P}_3\mathfrak{p}_{\textbf{u}_1}-(\textbf{u}_1+\textrm{ad}_{\textbf{u}_1})\Theta\big)\\
 &\overset{\text{\tiny$(1)$}}{=}\big(\mathfrak{p}_{\textbf{u}_1}-\mathcal{P}_3\mathfrak{p}_{\textbf{u}_1}\big)+\mathcal{P}_3\mathfrak{R}_{\textbf{u}_1,2}.
\end{align*}
Here in $(1)$ we use \eqref{for:163}.

Hence we have
\begin{align}\label{for:143}
 \norm{\mathcal{R}_{\textbf{u}_1}}&\leq \norm{\mathfrak{p}_{\textbf{u}_1}-\mathcal{P}_3\mathfrak{p}_{\textbf{u}_1}}+\norm{\mathcal{P}_3\mathfrak{R}_{\textbf{u}_1,2}}\notag\\
 &\overset{\text{\tiny$(1)$}}{\leq} \Psi(0)+C\norm{\mathfrak{R}_{\textbf{u}_1,2}}\overset{\text{\tiny$(2)$}}{\leq}C\Psi(\varrho)+C\norm{ \mathfrak{w}}_{\varrho}.
\end{align}
Here in $(1)$ we use \eqref{for:90} of Corollary \ref{cor:10}; in $(2)$ we use \eqref{for:139}.

\smallskip
\textbf{\emph{Step $8$}: Estimates for $\norm{\mathcal{R}_v}_{\beta}$, $v\in E$.} We follow exactly the same way as step 4 in Section \ref{sec:19}.
We recall the twisted cocycle equation over $v$ and $\textbf{u}_1$ from \eqref{for:95}:
\begin{align*}
 (v+\textrm{ad}_{v})\mathfrak{p}_{\textbf{u}_1}-(\textbf{u}_1+\textrm{ad}_{\textbf{u}_1})\mathfrak{p}_{v}= \mathfrak{w}_{v,\textbf{u}_1}.
\end{align*}
We substitute  the expressions for $\mathfrak{p}_{\textbf{u}_1}$ and $\mathfrak{p}_{v}$
from \eqref{for:103} respectively into the above equation. Similar to \eqref{for:132}, we get
\begin{align*}
 (\textbf{u}_1+\textrm{ad}_{\textbf{u}_1})\mathcal{R}_v= -\mathfrak{w}_{v,\textbf{u}_1}+(v+\textrm{ad}_{v})\mathcal{R}_{\textbf{u}_1},\qquad \forall\,v\in E.
\end{align*}
It follows from Theorem \ref{th:2}   that
\begin{align}\label{for:33}
 \norm{\mathcal{R}_v}_\beta&\leq C\norm{-\mathfrak{w}_{v,\textbf{u}_1}+(v+\textrm{ad}_{v})\mathcal{R}_{\textbf{u}_1}}_{\lambda\beta+\lambda_1}\notag\\
 &\overset{\text{\tiny$(0)$}}{\leq} C\norm{\mathfrak{w}}_{\varrho}+\norm{\mathcal{R}_{\textbf{u}_1}}_{\varrho}\notag\\
 &\overset{\text{\tiny$(1)$}}{\leq} C\norm{\mathfrak{w}}_{\varrho}+C_{\ell}
\norm{\mathcal{R}_{\textbf{u}_1}}^{1-\frac{\varrho}{\ell}}
(\norm{\mathcal{R}_{\textbf{u}_1}}_{\ell})^{\frac{\varrho}{\ell}}\notag\\
 &\overset{\text{\tiny$(2)$}}{\leq} C\norm{\mathfrak{w}}_{\varrho}+C_{\ell}(\Psi(\varrho)
 +\norm{ \mathfrak{w}}_{\varrho})^{1-\frac{\varrho}{\ell}}( \norm{\mathfrak{p}}_{\ell+\sigma_0+1}+a^{\ell+1} \norm{\mathfrak{p}}_{\sigma_0})^{\frac{\varrho}{\ell}}\notag\\
 &\overset{\text{\tiny$(0)$}}{\leq} C\norm{\mathfrak{w}}_{\varrho}+C_{\ell}(\Psi(\varrho)
 +\norm{ \mathfrak{w}}_{\varrho})^{1-\frac{\varrho}{\ell}}( \norm{\mathfrak{p}}_{\ell+\varrho}+a^{\ell+1} \norm{\mathfrak{p}}_{\varrho})^{\frac{\varrho}{\ell}}\notag\\
 &\leq C\norm{\mathfrak{w}}_{\varrho}+C_{\ell}a^{2\varrho}(\Psi(\varrho)
 +\norm{ \mathfrak{w}}_{\varrho})^{1-\frac{\varrho}{\ell}}( \norm{\mathfrak{p}}_{\ell+\varrho})^{\frac{\varrho}{\ell}}\notag\\
 &\overset{\text{\tiny$(3)$}}{\leq}C\norm{\mathfrak{w}}_{\varrho}+C_{\ell}a^{2\varrho}(\Psi(\varrho))^{1-\frac{\varrho}{\ell}}( \norm{\mathfrak{p}}_{\ell+\varrho})^{\frac{\varrho}{\ell}}
 +C_{\ell}a^{2\varrho}(\norm{ \mathfrak{w}}_{\varrho})^{1-\frac{\varrho}{\ell}}( \norm{\mathfrak{p}}_{\ell+\varrho})^{\frac{\varrho}{\ell}}\notag\\
 &\overset{\text{\tiny$(4)$}}{\leq}C\norm{\mathfrak{w}}_{\varrho}+C_{\ell}a^{2\varrho}\big((a^{-s}\norm{\mathfrak{p}}_s)^{\text{\tiny$1-\frac{\varrho}{\ell}$}}
 \norm{\mathfrak{p}}^{\text{\tiny$\frac{\varrho}{\ell}$}}_{\ell} \big)^{1-\frac{\varrho}{\ell}}( \norm{\mathfrak{p}}_{\ell+\varrho})^{\frac{\varrho}{\ell}}\notag\\
 &+C_{\ell}a^{2\varrho}(\norm{ \mathfrak{w}}_{\varrho})^{1-\frac{\varrho}{\ell}}( \norm{\mathfrak{p}}_{\ell+\varrho})^{\frac{\varrho}{\ell}}\notag\\
 &\leq C\norm{\mathfrak{w}}_{\varrho}+C_{\ell}a^{2\varrho}(a^{-s}\norm{\mathfrak{p}}_s)^{(\text{\tiny$1-\frac{\varrho}{\ell}$})^2}
 (\norm{\mathfrak{p}}_{\ell+\varrho})^{\frac{\varrho}{\ell}(2-\frac{\varrho}{\ell})} \notag\\
 &+C_{\ell}a^{2\varrho}(\norm{ \mathfrak{w}}_{\varrho})^{1-\frac{\varrho}{\ell}}( \norm{\mathfrak{p}}_{\ell+\varrho})^{\frac{\varrho}{\ell}}.
\end{align}
Here in $(0)$ we use the definition of $\varrho$ in  \eqref{for:99}; in $(1)$ we use interpolation inequalities; and in $(2)$ we use \eqref{for:143} to estimate
$\norm{\mathcal{R}_{\textbf{u}_1}}$ and use \eqref{for:112} to estimate $\norm{\mathcal{R}_{\textbf{u}_1}}_{\ell}$; in $(3)$ we use \eqref{for:217}; in $(4)$ we use \eqref{for:148}.

Finally, \eqref{for:157} follows from \eqref{for:76} and \eqref{for:33}. Hence we finish the proof.

\section{Proof of Theorem \ref{th:13} and Theorem \ref{th:11}}\label{sec:8}

\subsection{Iterative step and the error estimate}\label{sec:32} In this part we show that given a perturbation
of the action $\alpha_A$ satisfying a certain set of conditions, one constructs a conjugacy
such that the new action satisfies another set of conditions. We recall notations in
Section \ref{sec:25}. Suppose $\alpha_A$ is generated by vectors fields $E=\{E_1,E_2,\cdots\}$ as described in \eqref{for:286} of Section \ref{sec:47}.

\begin{proposition}\label{po:5} There exists $0<\bar{c}<1$ such that the following holds: for any perturbation $\tilde{\alpha}_A$ of $\alpha_A$ generated by $C^\infty$ vector
fields $\tilde{E}=E+\mathfrak{p}$,
where $\norm{\mathfrak{p}}_{C^\varrho}\leq \bar{c}$ (see \eqref{for:99}), and for any $a,\,b>1$,  there is a linear map $\mathcal{T}$ on $\text{Lie}(A)$, $g\in \GG$ and $\mathfrak{h}\in \text{Vect}^\infty(\mathcal{X})$ such that
for
\begin{align*}
 h&=\Psi_g^{-1}\circ\exp(\mathfrak{h})\quad\text{and }\quad\tilde{E}^{(1)}=\mathcal{T}(h_*\tilde{E})=E+\mathfrak{p}^{(1)}
\end{align*}
where $\Psi_g$ denotes the diffeomorphism on $\mathcal{X}$ induced by the left translation of $g$ on $\GG$, we have:
\begin{enumerate}
  \item \label{for:180}  for any $r\geq \varrho$
\begin{align*}
 \norm{\mathfrak{h}}_{C^r}&\leq
 C_{r}(a^{r}b^\varrho\norm{\mathfrak{p}}_{C^{\varrho}}+b^\varrho\norm{\mathfrak{p}}_{C^r});
\end{align*}

  \item \label{for:182} $\norm{\mathcal{T}-I}+\norm{g-I}\leq C\norm{\mathfrak{p}}_{C^0}\leq C\bar{c}$; and
  \begin{align*}
   d(h,I)_{C^r}\leq C_r(a^{r}b^\varrho\norm{\mathfrak{p}}_{C^{\varrho}}+b^\varrho\norm{\mathfrak{p}}_{C^r})
  \end{align*}
for any $r\geq\varrho$;  we suppose $a,\,b$ are well chosen such that
 \begin{align*}
   C(a^{\varrho}b^\varrho\norm{\mathfrak{p}}_{C^{\varrho}}+b^\varrho\norm{\mathfrak{p}}_{C^{\varrho}})< \bar{c},
 \end{align*}
then $h$ is invertible;

\item\label{for:168}  the estimate for $\norm{\mathfrak{p}^{(1)}}_{C^0}$ holds:
       \begin{align}\label{for:144}
\norm{\mathfrak{p}^{(1)}}&_{C^0}\leq Ca^{2\varrho}b^{2\varrho}\norm{\mathfrak{p}}^2_{C^{\varrho+1}}
+C_{\ell}a^{2\varrho}(\norm{\mathfrak{p}}_{C^{\varrho+1}})^{2(1-\frac{\varrho}{\ell})}(\norm{\mathfrak{p}}_{C^{\ell+\varrho}})^{\frac{\varrho}{\ell}}\notag\\
 &+C_{\ell}a^{2\varrho}(a^{\text{\tiny$-s$}}
 \norm{\mathfrak{p}}_{C^s})^{1-\frac{\varrho}{\ell}}(\norm{\mathfrak{p}}_{C^{\ell+\varrho}})^{\frac{\varrho}{\ell}}\notag\\
 &+C_{\ell}a^{2\varrho}(a^{-s}\norm{\mathfrak{p}}_{C^s})^{(\text{\tiny$1-\frac{\varrho}{\ell}$})^2}
 (\norm{\mathfrak{p}}_{C^{\ell+\varrho}})^{\frac{\varrho}{\ell}(2-\frac{\varrho}{\ell})}\notag\\
 &+C_\ell b^{\text{\tiny$-\ell+1$}}a^{\ell+\varrho}\norm{\mathfrak{p}}_{C^\varrho}+C_\ell b^{\text{\tiny$-\ell+1$}}\norm{\mathfrak{p}}_{C^{\ell+\varrho}}
\end{align}
for any $s\geq\ell>\varrho$;

 \item  \label{for:181} for any $r\geq\varrho$ we have:
 \begin{align*}
  \norm{\mathfrak{p}^{(1)}}_{C^r}\leq C_{r}(a^{r}b^\varrho \norm{\mathfrak{p}}_{C^{\varrho}}+b^\varrho\norm{\mathfrak{p}}_{C^r}+1).
 \end{align*}
 \end{enumerate}

\emph{Note}. The constants in \eqref{for:144} are independent of $s$.
\end{proposition}

From Theorem \ref{cor:7}, we see that there is a fixed loss of regularity when solving the almost cocycle equations.
To overcome this fixed loss of
regularity at each step of the iteration process, it is standard (see  \cite{Zehnder})
to use the smoothing operators $\mathfrak{s}_{b}$. Consequently, one more parameter $b$ is introduced in comparison to Theorem \ref{cor:7}.

\begin{proof} \emph{Construction and estimates for $\mathcal{T}$ and $g$}: We point out that $\text{Ave}(\mathfrak{p})$ determines the coordinate change $\mathcal{T}$ and the inner automorphism of the vector fields. Let $E'=\text{Ave}(\mathfrak{p})+E$.
\begin{align}\label{for:164}
 \norm{\mathcal{M}(E')}&=\Big\|\mathcal{M}\big(\text{Ave}(\mathfrak{p})+E\big)\Big\|\notag\\
 &\leq \big\|\mathcal{M}\big(\text{Ave}(\mathfrak{p})\big)\big\|+\big\|\mathcal{M}(E)\big\|
 +\big\|M\big(\text{Ave}(\mathfrak{p})\big)\big\|\notag\\
 &\overset{\text{\tiny$(1)$}}{\leq} C\norm{\mathfrak{p}}_{C^{0}}\norm{\mathfrak{p}}_{C^{1}}.
\end{align}
Here in $(1)$ we use \eqref{for:113}, $\mathcal{M}(E)=0$ and Lemma \ref{le:21}.

Hence we have
\begin{align}\label{for:136}
  \norm{E'-E}&+\norm{\mathcal{M}(E')}\overset{\text{\tiny$(1)$}}{\leq} \norm{\text{Ave}(\mathfrak{p})}+C\norm{\mathfrak{p}}_{C^{0}}\norm{\mathfrak{p}}_{C^{1}}\notag\\
  &\leq \norm{\mathfrak{p}}_{C^{0}}+C\norm{\mathfrak{p}}_{C^{0}}\norm{\mathfrak{p}}_{C^{1}} \notag\\
  &\leq\bar{c}+C\bar{c}^2\overset{\text{\tiny$(2)$}}{\leq}\delta
\end{align}
Here in $(1)$ we use \eqref{for:164}; $(2)$ holds if $\bar{c}$ is sufficiently small ($\delta$ is defined in Definition \ref{de:1}).

\eqref{for:136} shows that  we can apply  Proposition \ref{po:2} to the $\norm{\text{Ave}(\mathfrak{p})}$-perturbation $E'$. Then
there exists a linear map $\mathcal{T}$ on $\text{Lie}(A)$ and $g\in \GG$ such that:
\begin{align}\label{for:165}
\norm{\mathcal{T}-I}+\norm{g-I}\leq C\norm{\text{Ave}(\mathfrak{p})}\leq C\norm{\mathfrak{p}}_{C^0}
\end{align}
and
\begin{align}\label{for:121}
 \norm{\mathcal{T}E'-\text{Ad}&_{g}E}\leq C\big\|\mathcal{M}(E')\big\|+C\norm{\text{Ave}(\mathfrak{p})}^2\overset{\text{\tiny$(1)$}}{\leq} C\norm{\mathfrak{p}}_{C^{0}}\norm{\mathfrak{p}}_{C^{1}}.
\end{align}
Here in $(1)$ we use \eqref{for:164}.

\medskip

\emph{Construction and estimates for $\mathfrak{h}$ and $h$}: Set $\mathfrak{p}_v^o=\mathfrak{p}_v-\text{Ave}(\mathfrak{p}_v)$, for any $v\in E$  and
 \begin{align}
  \mathcal{L}_{v}\mathfrak{p}^o_{u}&-\mathcal{L}_{u}\mathfrak{p}^o_{v}= \mathfrak{w}_{u,v},\qquad v,\,u\in E.\label{for:94}
 \end{align}
Let $\norm{\mathfrak{p}^o}_{C^r}=\max_{v\in E}\{\norm{\mathfrak{p}^o_{v}}_{C^r}\}$ and $\norm{\mathfrak{w}}_{C^r}=\max_{u,\,v\in E}\{\norm{\mathfrak{w}_{v,u}}_{C^r}\}$.

Next, we will apply Theorem \ref{cor:7} to the almost twisted cocycle \eqref{for:94}. Before that, we need to estimate $\norm{\mathfrak{p}^o}_{C^r}$ and $\norm{\mathfrak{w}}_{C^r}$. It is clear that
\begin{align}\label{for:137}
 \norm{\mathfrak{p}^o}_{C^{r}}\leq 2\norm{\mathfrak{p}}_{C^{r}},\qquad \forall\,r\geq0;
\end{align}
and for any $r\geq0$ we have:
\begin{align}\label{for:119}
  \norm{\mathfrak{w}}_{C^r}&=\big\|M(\mathfrak{p}^o)\big\|_{C^r}\leq C_r\big(\big\|M(\mathfrak{p})\big\|_{C^r}
  +C\Big\|M\big(\text{Ave}(\mathfrak{p})\big)\Big\|\notag\\
  &\overset{\text{\tiny$(1)$}}{\leq}C_r\norm{\mathfrak{p}}_{C^{0}}\norm{\mathfrak{p}}_{C^{r+1}}.
 \end{align}
Here in $(1)$ we use Lemma \ref{le:21}.

From \eqref{for:94} by Theorem \ref{cor:7} we see that for any $a>1$ there exist  $\mathfrak{h}^o,\,\mathcal{R}_v\in \text{Vect}^\infty(\mathcal{X})$, $v\in E$
such that
\begin{align}\label{for:167}
 \mathfrak{p}^o_v=\mathcal{L}_{v}\mathfrak{h}^o+\mathcal{R}_v,\qquad v\in E
\end{align}
with estimates: for any $r\geq0$ any $v\in E$
\begin{align}\label{for:118}
 \norm{\mathfrak{h}^o, \mathcal{R}_v}_{C^r}&\leq C_{r}(a^{r+\varrho}\norm{\mathfrak{p}^o}_{C^\varrho}+\norm{\mathfrak{p}^o}_{C^{r+\varrho}})\notag\\
 &\overset{\text{\tiny$(1)$}}{\leq} C_{r}(a^{r+\varrho}\norm{\mathfrak{p}}_{C^\varrho}+\norm{\mathfrak{p}}_{C^{r+\varrho}});
\end{align}
and for any $s\geq\ell>\varrho$
\begin{align}\label{for:124}
 \norm{\mathcal{R}_v}_{C^0}&\leq C\norm{\mathfrak{w}}_{C^\varrho}
+C_{\ell}a^{2\varrho}\norm{\mathfrak{w}}_{C^\varrho}^{1-\frac{\varrho}{\ell}}(\norm{\mathfrak{p}^o}_{C^{\ell+\varrho}})^{\frac{\varrho}{\ell}}\notag\\
 &+C_{\ell}a^{2\varrho}(a^{\text{\tiny$-s$}}
 \norm{\mathfrak{p}^o}_{C^s})^{1-\frac{\varrho}{\ell}}(\norm{\mathfrak{p}^o}_{C^{\ell+\varrho}})^{\frac{\varrho}{\ell}}\notag\\
 &+C_{\ell}a^{2\varrho}(a^{-s}\norm{\mathfrak{p}^o}_{C^s})^{(\text{\tiny$1-\frac{\varrho}{\ell}$})^2}
 (\norm{\mathfrak{p}^o}_{C^{\ell+\varrho}})^{\frac{\varrho}{\ell}(2-\frac{\varrho}{\ell})}\notag\\
&\overset{\text{\tiny$(1,2)$}}{\leq} C\norm{\mathfrak{p}}_{C^{\varrho+1}}^2
+C_{\ell}a^{2\varrho}(\norm{\mathfrak{p}}_{C^{\varrho+1}})^{2(1-\frac{\varrho}{\ell})}(\norm{\mathfrak{p}}_{C^{\ell+\varrho}})^{\frac{\varrho}{\ell}}\notag\\
 &+C_{\ell}a^{2\varrho}(a^{\text{\tiny$-s$}}
 \norm{\mathfrak{p}}_{C^s})^{1-\frac{\varrho}{\ell}}(\norm{\mathfrak{p}}_{C^{\ell+\varrho}})^{\frac{\varrho}{\ell}}\notag\\
 &+C_{\ell}a^{2\varrho}(a^{-s}\norm{\mathfrak{p}}_{C^s})^{(\text{\tiny$1-\frac{\varrho}{\ell}$})^2}
 (\norm{\mathfrak{p}}_{C^{\ell+\varrho}})^{\frac{\varrho}{\ell}(2-\frac{\varrho}{\ell})}.
 \end{align}
 Here in $(1)$ we use \eqref{for:137}; in $(2)$ we use \eqref{for:119}. Let
\begin{align}\label{for:166}
 \mathfrak{h}=\mathfrak{s}_{b}\mathfrak{h}^o
\end{align}
where $\mathfrak{s}_{b}$ is as defined in Section \ref{sec:26}. Hence we have
\begin{align}\label{for:160}
 \norm{\mathfrak{h}}_{C^r}\overset{\text{\tiny$(1)$}}{\leq}  C_rb^\varrho\norm{\mathfrak{h}^o}_{C^{r-\varrho}}\overset{\text{\tiny$(2)$}}{\leq} C_{r}b^\varrho(a^{r}\norm{\mathfrak{p}}_{C^{\varrho}}+\norm{\mathfrak{p}}_{C^r}),\quad
 \forall\, r\geq \varrho.
\end{align}
Here in $(1)$ we use  \eqref{for:111}; in $(2)$ we use \eqref{for:118}.

Then we have
\begin{align*}
  d(h,I)_{C^r} &\leq C\norm{g-I}+C_r\norm{\mathfrak{h}}_{C^r}
  \overset{\text{\tiny$(1)$}}{\leq} C_{r}b^\varrho(a^{r}\norm{\mathfrak{p}}_{C^{\varrho}}+\norm{\mathfrak{p}}_{C^r})
\end{align*}
for any $r\geq \varrho$.  Here in $(1)$ we use \eqref{for:165}, \eqref{for:160} and $b>1$.

\medskip

\emph{Estimate for $\norm{\mathfrak{p}^{(1)}}_{C^0}$}: We set $h_1=\exp(\mathfrak{h})$. We suppose $a,\,b$ are well chosen such that
 \begin{align*}
  d(h,I)_{C^\varrho}\leq Cb^\varrho(a^{\varrho}\norm{\mathfrak{p}}_{C^{\varrho}}+\norm{\mathfrak{p}}_{C^\varrho})< \bar{c}.
 \end{align*}
Then $h$ is inventible as we assume that $\bar{c}$ is sufficiently small. For $r\geq 0$, we have
\begin{align}\label{for:122}
 \norm{\mathfrak{p}^{(1)}}_{C^r}&=\Big\|\mathcal{T}\big(\text{Ad}_{g^{-1}}(h_1)_*\tilde{E}\big)-E\Big\|_{C^r}
 \overset{\text{\tiny$(1)$}}{\leq}C\Big\|\mathcal{T}\big((h_1)_*\tilde{E}\big)-\text{Ad}_{g}E\Big\|_{C^r}\notag\\
 &\leq C\Big\|\mathcal{T}\big((h_1)_*\tilde{E}\big)-\mathcal{T}E'\Big\|_{C^r}+C\|\mathcal{T}E'-\text{Ad}_{g}E\|\notag\\
&\overset{\text{\tiny$(1,2)$}}{\leq} C_1\norm{(h_1)_*\tilde{E}-E'}_{C^r}+C_1\norm{\mathfrak{p}}_{C^{0}}\norm{\mathfrak{p}}_{C^{1}}.
\end{align}
Here in $(1)$ we use \eqref{for:165}; in $(2)$ we use \eqref{for:121}.

\eqref{for:122} shows that to estimate $\norm{\mathfrak{p}^{(1)}}_{C^0}$, it suffices to estimate the
$C^0$ norm of
\begin{align*}
 W=(h_1)_*\tilde{E}-E'=(h_1)_*(E+\mathfrak{p})
 -(\text{Ave}(\mathfrak{p})+E).
\end{align*}
Then for each $v\in E$, we have
\begin{align*}
 W_v=\mathcal{R}_v+\mathcal{E}_v+W_{L,v},
\end{align*}
where
\begin{align*}
 W_{L,v}=(h_1)_*(v+\mathfrak{p}_v)-(v+\mathfrak{p}_v)-[\mathfrak{h},v+\mathfrak{p}_v]
\end{align*}
is the error from linearization; and
\begin{align*}
 \mathcal{E}_v=[\mathfrak{h}-\mathfrak{h}^o,v+\mathfrak{p}_v]+[\mathfrak{h}^o,\mathfrak{p}_v]
\end{align*}
is the error coming from solving the linearized equation only approximately.

We have
\begin{align}\label{for:123}
  \norm{W_{L,v}}_{C^0}&\leq C\norm{\mathfrak{h}}_{C^2}^2\norm{E+\mathfrak{p}}_{C^2}
\overset{\text{\tiny$(1)$}}{\leq} C_1\norm{\mathfrak{h}}_{C^2}^2\notag\\
   &\overset{\text{\tiny$(2)$}}{\leq} C_2\big(a^{\varrho}b^\varrho\norm{\mathfrak{p}}_{C^{\varrho}}+b^\varrho\norm{\mathfrak{p}}_{C^\varrho}\big)^2
   \leq 4C_2 a^{2\varrho}b^{2\varrho}\norm{\mathfrak{p}}^2_{C^{\varrho}}.
 \end{align}
 Here in $(1)$ we use the assumption $\norm{\mathfrak{p}}_{C^{\varrho}}\leq\bar{c}$;
 in $(2)$ we use \eqref{for:160}.

 Further, we have
  \begin{align}\label{for:125}
  \norm{\mathcal{E}_v}_{C^0}=&\norm{[\mathfrak{h}-\mathfrak{h}^o,v+\mathfrak{p}_v]}_{C^0}+\norm{[\mathfrak{h}^o,\mathfrak{p}_v]}_{C^0}\notag\\
  &\overset{\text{\tiny$(1)$}}{\leq} \norm{\mathfrak{h}-\mathfrak{h}^o}_{C^1}(\norm{\mathfrak{p}}_{C^1}+C)+C\norm{\mathfrak{h}^o}_{C^1}\norm{\mathfrak{p}}_{C^1}\notag\\
  &\overset{\text{\tiny$(2)$}}{\leq} C_\ell b^{\text{\tiny$-\ell+1$}}\norm{\mathfrak{h}^o}_{C^\ell}+C\norm{\mathfrak{h}^o}_{C^1}\norm{\mathfrak{p}}_{C^1}\notag\\
 &\overset{\text{\tiny$(3)$}}{\leq} C_\ell b^{\text{\tiny$-\ell+1$}}(a^{\ell+\varrho}\norm{\mathfrak{p}}_{C^\varrho}+\norm{\mathfrak{p}}_{C^{\ell+\varrho}})
 +Ca^{1+\varrho}\norm{\mathfrak{p}}_{C^{\varrho}}\norm{\mathfrak{p}}_{C^1}\notag\\
 &+C\norm{\mathfrak{p}}_{C^{\varrho+1}}\norm{\mathfrak{p}}_{C^1}\notag\\
 &\leq  C_\ell b^{\text{\tiny$-\ell+1$}}a^{\ell+\varrho}\norm{\mathfrak{p}}_{C^\varrho}+C_\ell b^{\text{\tiny$-\ell+1$}}\norm{\mathfrak{p}}_{C^{\ell+\varrho}}
 +2Ca^{1+\varrho}\norm{\mathfrak{p}}^2_{C^{\varrho+1}}
\end{align}
for any $\ell>\varrho$. Here in $(1)$ we use \eqref{for:113}; in $(2)$  we use \eqref{for:117} and the assumption $\norm{\mathfrak{p}}_{C^{\varrho}}\leq\bar{c}$;
in $(3)$ we use \eqref{for:118}.

Then as a direct consequence of \eqref{for:124}, \eqref{for:122}, \eqref{for:123}, \eqref{for:125} we have:
\begin{align*}
  \norm{\mathfrak{p}^{(1)}}_{C^0}&\leq \max_{v\in E}\big(\norm{\mathcal{R}_v}_{C^0}+\norm{W_{L,v}}_{C^0}+\norm{\mathcal{E}_v}_{C^0}\big)+C\norm{\mathfrak{p}}_{C^{0}}\norm{\mathfrak{p}}_{C^{1}},
 \end{align*}
which gives \eqref{for:144}.

\medskip

\emph{Estimate for $\norm{\mathfrak{p}^{(1)}}_{C^r}$, $r\geq1$}: In this part,   we only need to have a ``linear" bound
with respect to the corresponding norm of the old error $\norm{\mathfrak{p}}_{C^r}$. From \eqref{for:122} we have
\begin{align*}
 \norm{\mathfrak{p}^{(1)}}_{C^r}& \leq C\norm{(h_1)_*\tilde{E}}_{C^r}+\norm{E'}+C\norm{\mathfrak{p}}_{C^{0}}\norm{\mathfrak{p}}_{C^{1}}\\
 &\overset{\text{\tiny$(1)$}}{\leq} C\norm{(h_1)_*\tilde{E}}_{C^r}+C\leq C_r(\norm{h_1}_{C^r}+\norm{\mathfrak{p}}_{C^r}+1)\\
 &\overset{\text{\tiny$(2)$}}{\leq} C_{r}( a^{r}b^\varrho\norm{\mathfrak{p}}_{C^{\varrho}}+b^\varrho\norm{\mathfrak{p}}_{C^r}+1)
\end{align*}
for any $r\geq\varrho$. Here in $(1)$ we use the assumption that $\norm{\mathfrak{p}}_{C^\varrho}$ is sufficiently small; in $(2)$ we use \eqref{for:160}. Hence we get the \eqref{for:181}.
\end{proof}
The estimate of $\norm{\mathfrak{p}^{(1)}}_{C^0}$ from Proposition \ref{po:5} is simplified in the following corollary under some additional assumptions that will be all met during the iterative step.

\begin{corollary}\label{le:14} If $a\leq b^{\frac{1}{2}}$ and $\norm{\mathfrak{p}}_{C^{\varrho+1}}<1$, then
\begin{enumerate}
  \item\label{for:153}  if $s=\ell+\varrho$, \eqref{for:144} of Proposition \ref{po:5} can be simplified as
  \begin{align*}
   \norm{\mathfrak{p}^{(1)}}&_{C^0}\leq C_{\ell}b^{3\varrho}(\norm{\mathfrak{p}}_{C^{\varrho+1}})^{2(1-\frac{\varrho}{\ell})}\big((\norm{\mathfrak{p}}_{C^{\ell+\varrho}})^{\frac{\varrho}{\ell}}+1\big)\notag\\
 &+4C_{\ell}b^{-\frac{\ell}{2}+2\varrho}
 \norm{\mathfrak{p}}_{C^{\ell+\varrho}};
  \end{align*}
  \item\label{for:15} if $a^{-s}\norm{\mathfrak{p}}_{C^s}<1$ and $\norm{\mathfrak{p}}_{C^{\ell+\varrho}}<y$ with $y>1$, \eqref{for:144} can be simplified as
  \begin{align*}
\norm{\mathfrak{p}^{(1)}}&_{C^0}\leq C_{\ell}b^{3\varrho}(\norm{\mathfrak{p}}_{C^{\varrho+1}})^{2(1-\frac{\varrho}{\ell})}(y^{\frac{\varrho}{\ell}}+1)\notag\\
 &+2C_{\ell}b^{\varrho}(a^{-s}\norm{\mathfrak{p}}_{C^s})^{(\text{\tiny$1-\frac{\varrho}{\ell}$})^2}
 y^{\frac{\varrho}{\ell}(2-\frac{\varrho}{\ell})}\notag\\
 &+2C_\ell b^{\text{\tiny$-\frac{\ell}{2}+2\varrho$}}y.
\end{align*}
\end{enumerate}
\end{corollary}
\begin{proof}
\eqref{for:153}: Let $s=\ell+\varrho$ in \eqref{for:144}. Then we have
\begin{align*}
\norm{\mathfrak{p}^{(1)}}&_{C^0}\overset{\text{\tiny$(1)$}}{\leq} Ca^{6\varrho}\norm{\mathfrak{p}}_{C^{\varrho+1}}^2
+C_{\ell}a^{2\varrho}(\norm{\mathfrak{p}}_{C^{\varrho+1}})^{2(1-\frac{\varrho}{\ell})}(\norm{\mathfrak{p}}_{C^{\ell+\varrho}})^{\frac{\varrho}{\ell}}\notag\\
 &+C_{\ell}a^{2\varrho}(a^{\text{\tiny$-\ell-\varrho$}})^{1-\frac{\varrho}{\ell}}
 \norm{\mathfrak{p}}_{C^{\ell+\varrho}}\notag\\
 &+C_{\ell}a^{2\varrho}(a^{\text{\tiny$-\ell-\varrho$}}
 )^{(1-\frac{\varrho}{\ell})^2}
 \norm{\mathfrak{p}}_{C^{\ell+\varrho}}\notag\\
  &+C_\ell a^{\text{\tiny$-\ell+2+\varrho$}}\norm{\mathfrak{p}}_{C^\varrho}+C_\ell a^{\text{\tiny$-2\ell+2$}}\norm{\mathfrak{p}}_{C^{\ell+\varrho}}\notag\\
  &\overset{\text{\tiny$(2)$}}{\leq} C_{\ell}a^{6\varrho}(\norm{\mathfrak{p}}_{C^{\varrho+1}})^{2(1-\frac{\varrho}{\ell})}\big((\norm{\mathfrak{p}}_{C^{\ell+\varrho}})^{\frac{\varrho}{\ell}}+1\big)\notag\\
 &+4C_{\ell}a^{-\ell+4\varrho}
 \norm{\mathfrak{p}}_{C^{\ell+\varrho}}\\
 &\overset{\text{\tiny$(3)$}}{\leq} C_{\ell}b^{3\varrho}(\norm{\mathfrak{p}}_{C^{\varrho+1}})^{2(1-\frac{\varrho}{\ell})}\big((\norm{\mathfrak{p}}_{C^{\ell+\varrho}})^{\frac{\varrho}{\ell}}+1\big)\notag\\
 &+4C_{\ell}b^{-\frac{\ell}{2}+2\varrho}
 \norm{\mathfrak{p}}_{C^{\ell+\varrho}}
\end{align*}
Here in $(1)$ we use $1=(1-\frac{\varrho}{\ell})^2+\frac{\varrho}{\ell}(2-\frac{\varrho}{\ell})$;
in $(2)$ we use $a>1$, $\varrho\geq2$ and $\norm{\mathfrak{p}}_{C^{\varrho+1}}<1$; in $(3)$ we use $a\leq b^{\frac{1}{2}}$.

\smallskip
\eqref{for:15}: By \eqref{for:144} we have
 \begin{align*}
\norm{\mathfrak{p}^{(1)}}&_{C^0}\overset{(1)}{\leq} Cb^{3\varrho}(\norm{\mathfrak{p}}_{C^{\varrho+1}})^{2(1-\frac{\varrho}{\ell})}
+C_{\ell}b^{\varrho}(\norm{\mathfrak{p}}_{C^{\varrho+1}})^{2(1-\frac{\varrho}{\ell})}y^{\frac{\varrho}{\ell}}\notag\\
 &+C_{\ell}b^{\varrho}(a^{\text{\tiny$-s$}}
 \norm{\mathfrak{p}}_{C^s})^{(1-\frac{\varrho}{\ell})^2}y^{\frac{\varrho}{\ell}}\notag\\
 &+C_{\ell}b^{\varrho}(a^{-s}\norm{\mathfrak{p}}_{C^s})^{(\text{\tiny$1-\frac{\varrho}{\ell}$})^2}
 y^{\frac{\varrho}{\ell}(2-\frac{\varrho}{\ell})}\notag\\
 &+C_\ell b^{\text{\tiny$-\frac{\ell}{2}+1+\frac{\varrho}{2}$}}\norm{\mathfrak{p}}_{C^\varrho}+C_\ell b^{\text{\tiny$-\ell+1$}}\norm{\mathfrak{p}}_{C^{\ell+\varrho}}\notag\\
 &\overset{(2)}{\leq} C_{\ell}b^{3\varrho}(\norm{\mathfrak{p}}_{C^{\varrho+1}})^{2(1-\frac{\varrho}{\ell})}(y^{\frac{\varrho}{\ell}}+1)\notag\\
 &+2C_{\ell}b^{\varrho}(a^{-s}\norm{\mathfrak{p}}_{C^s})^{(\text{\tiny$1-\frac{\varrho}{\ell}$})^2}
 y^{\frac{\varrho}{\ell}(2-\frac{\varrho}{\ell})}\notag\\
 &+2C_\ell b^{\text{\tiny$-\frac{\ell}{2}+2\varrho$}}y.
\end{align*}
Here in $(1)$ we use $a\leq b^{\frac{1}{2}}$, $\norm{\mathfrak{p}}_{C^{\varrho+1}}<1$ and $a^{-s}\norm{\mathfrak{p}}_{C^s}<1$; in $(2)$
we use $b>1$, $y>1$ and $\ell>\varrho>2$.
\end{proof}

\subsection{Setting up the iterative process }\label{sec:6} We consider the action $\alpha_A$ as described  in Theorem \ref{th:13} or \ref{th:11}. Recall notations in Section \ref{sec:32}.  Assuming $\tilde{\alpha}_A$ is a perturbation of $\alpha_A$ generated by $C^\infty$ vector
fields $\tilde{E}=E+\mathfrak{p}$,
where $\mathfrak{p}=\{\mathfrak{p}_1,\,\mathfrak{p}_2,\cdots \}$ are all small in some $C^\ell$ norm ($\ell$ is fixed and is determined from \eqref{for:170} to \eqref{for:190}).

In the following, we establish an iterative scheme and show the convergence of the process to a $C^\infty$
conjugacy between the initial perturbation $\tilde{\alpha}_A$ and $\alpha_A$ up to a coordinate change. To set up the iterative process we first pick up $1<\gamma<2$.  Then there exists $\kappa$ such that
\begin{align*}
 2\gamma<\kappa<\gamma^2+1.
\end{align*}
Fix
\begin{align*}
 0<\eta<\min\{\frac{1}{2},\,1-\text{\small$\frac{\gamma}{2}$},\,\gamma^2-(\kappa-1),\,\kappa-2\gamma,\,\gamma-1\}.
\end{align*}
Let $\ell >\varrho$ (see \eqref{for:99} of Section \ref{sec:47} for definition of $\varrho$) be sufficiently large such that
\begin{gather}
1-\text{\small$\frac{\varrho(1+\gamma)}{\ell}$}>\text{\small$\frac{1}{2}$}+\eta,\quad
1-\text{\tiny$\text{\small$\frac{3\kappa\varrho}{\ell}$}$}-\text{\small$\frac{(1+\gamma)\varrho}{\ell}$}> \text{\small$\frac{\gamma}{2}$}+\eta\label{for:170}\\
\text{\small$\frac{(3\kappa+1+\gamma)\varrho}{\ell}$}+\kappa-1<\gamma^2-\eta\label{for:171}\\
1-\text{\small$\frac{(\varrho+1)(1+\gamma)}{\ell}$}>\text{\small$\frac{1}{2}$}+\eta\label{for:175}\\
-\text{\small$\frac{6\kappa\varrho}{\ell}$}+2\big(1-\text{\small$\frac{(1+\gamma)(1+\varrho)}{\ell}$}\big)\big(1-\text{\tiny$\frac{\varrho}{\text{\tiny$\ell$}}$}\big)
-\frac{\gamma\varrho}{\ell}>\gamma+\eta \label{for:178}\\
\kappa-\gamma
-\text{\small$\text{\small$\frac{4\kappa\varrho}{\text{\small$\ell$}}$}$}>\gamma+\eta\label{for:173}\\
\text{\small$\frac{2\kappa\varrho}{\ell}$}+2(\gamma-1)+2\gamma<2\gamma^2 \label{for:158}\\
-\text{\small$\frac{2\varrho(1+\gamma)}{\ell}$}+2\gamma-1>0\label{for:213}\\
-\text{\small$\frac{2\kappa\varrho}{\ell}$}
+(2\gamma-1-\text{\small$\frac{2\varrho(1+\gamma)}{\ell}$})(\text{\tiny$1-\frac{\varrho}{\ell}$})^2-\gamma\frac{\varrho}{\ell}(2-\frac{\varrho}{\ell})
>\gamma+\eta\label{for:190}.
\end{gather}
Next, we show that the choice of $\ell$ satisfying all these constraints is possible. From \eqref{for:170} to \eqref{for:190} by letting $\ell\to\infty$ these inequalities become
\begin{gather*}
 (10.17)\to\big(1>\text{\small$\frac{1}{2}$}+\eta,\quad 1>\text{\small$\frac{\gamma}{2}$}+\eta\big),\quad (10.18)\to\big(\kappa-1<\gamma^2-\eta \big)\\
 (10.19)\to\big(1>\frac{1}{2}+\eta \big) \quad  (10.20)\to\big(2>\gamma+\eta \big),\quad (10.21)\to\big( \kappa-\gamma>\gamma+\eta \big),\\
  (10.22)\to\big( 2(\gamma-1)+2\gamma<2\gamma^2 \big),\quad (10.23)\to\big( 2\gamma-1>0 \big),\\
  (10.24)\to\big( 2\gamma-1>\gamma+\eta \big).
\end{gather*}
All the above  inequalities hold either automatically or as a direct consequence of assumptions. Thus \eqref{for:170} to \eqref{for:190} hold if we choose $\ell$ big enough.

We fix an increasing sequence $\beta_n\to \infty$ with $\beta_1>2\ell$. We construct $\mathfrak{p}^{(n)}$, $h_n$ and $\mathcal{T}_{n}$ inductively as follows. Set
\begin{align*}
 \mathfrak{p}^{(0)}=\mathfrak{p},\quad h_{0}=I, \quad\mathcal{T}_{0}=I, \quad\text{and}\quad\epsilon_n=\epsilon^{\text{\tiny$\gamma^n$}}
\end{align*}
where $0<\epsilon^{\text{\tiny$\frac{1}{2}$}}<\bar{c}$ is sufficiently small so that the following holds
\begin{align*}
\norm{\mathfrak{p}^{(0)}}_{C^0}\leq \epsilon_0=\epsilon,\quad \norm{\mathfrak{p}^{(0)}}_{C^{\ell+\varrho}}\leq \epsilon_0^{-\gamma}, \quad d(h_0,I)_{C^1}<\epsilon_0^{\text{\tiny$\frac{1}{2}$}},
\quad \norm{\mathcal{T}_{0}-I}<\epsilon_0^{\text{\tiny$\frac{1}{2}$}}.
\end{align*}
Suppose  inductively that $\tilde{E}^{(n)}=E+\mathfrak{p}^{(n)}$ and
\begin{align}\label{for:149}
 \norm{\mathfrak{p}^{(n)}}_{C^0}\leq \epsilon_n, \quad \norm{\mathfrak{p}^{(n)}}_{C^{\ell+\varrho}}\leq \epsilon^{-\gamma}_n,\notag\\
 \norm{\mathfrak{p}^{(n)}}_{C^{\beta_m}}<K_m^n\epsilon^{\text{\tiny$-2\gamma$}}_n(\norm{\mathfrak{p}^{(m-1)}}_{C^{\beta_m}}+1)
 \end{align}
for any $m\leq n$; and $K_m$ is a constant dependent only on $m$.

By interpolation inequalities we have: for any $0\leq r\leq \ell+\varrho$
\begin{align}\label{for:183}
 \norm{\mathfrak{p}^{(n)}}_{C^r}&\leq C_{\ell} (\norm{\mathfrak{p}^{(n)}}_{C^0})^{\text{\small$\frac{\ell+\varrho-r}{\ell+\varrho}$}}(\norm{\mathfrak{p}^{(n)}}_{C^{\ell+\varrho}})^{\text{\small$\frac{r}{\ell+\varrho}$}}\leq C_{\ell}\epsilon_n^{1-\text{\small$\frac{(1+\gamma)r}{\ell+\varrho}$}}\notag\\
 &<C_{\ell}\epsilon_n^{1-\text{\small$\frac{(1+\gamma)r}{\ell}$}}.
\end{align}
\subsection{Convergence } In this subsection, by induction we prove that all the bounds in \eqref{for:149} are valid
for any $n\in\NN$.
\begin{proposition}\label{po:3} Suppose $n\geq0$ and all the bounds in \eqref{for:149} hold for $n$.
Then there is a linear map $\mathcal{T}_{n+1}$ on $\text{Lie}(A)$, $g_{n+1}\in \GG$ and $\mathfrak{h}_{n+1}\in \text{Vect}^\infty(\mathcal{X})$ such that
for
\begin{gather*}
 h_{n+1}=\Psi_{g_{n+1}}^{-1}\exp(\mathfrak{h}_{n+1})\qquad\text{and }\\
 \tilde{E}^{(n+1)}=\mathcal{T}_{n+1}\big((h_{n+1})_*\tilde{E}^{(n)}\big)=E+\mathfrak{p}^{(n+1)},
\end{gather*}
we have:
\begin{enumerate}
  \item\label{for:176} $\norm{\mathcal{T}_{n+1}-I}+\norm{g_{n+1}-I}\leq \epsilon_n^{\text{\tiny$\frac{1}{2}$}}$;

  \smallskip
  \item\label{for:177} $\max\{\norm{\mathfrak{h}_{n+1}}_{C^1}, d(h_{n+1},I)_{C^1}\}\leq \epsilon_{n+1}^{\text{\tiny$\frac{1}{2}$}}$;

  \smallskip
  \item\label{for:187} $\norm{\mathfrak{p}^{(n+1)}}_{C^{\ell+\varrho}}\leq \epsilon^{-\gamma}_{n+1}$;

  \smallskip

  \item\label{for:188} $\norm{\mathfrak{p}^{(n+1)}}_{C^0}\leq \epsilon_{n+1}$;

  \smallskip

  \item\label{for:189} for any $m\leq n+1$ we have
  \begin{align*}
   &\max\{\norm{\mathfrak{p}^{(n+1)}}_{C^{\beta_m}},\,d(h_{n+1},I)_{C^{\beta_m}}\}\\
   &<K_m^{n+1}\epsilon^{\text{\tiny$-2\gamma$}}_{n+1}(\norm{\mathfrak{p}^{(m-1)}}_{C^{\beta_m}}+1)
  \end{align*}
  where $K_m$ is a constant dependent only on $m$;

  \smallskip

  \item \label{for:193} for any $m\leq n+1$ we have
  \begin{align*}
 d(h_{n+1},I)_{C^{\frac{\beta_m}{9}}}\leq C_m(\norm{\mathfrak{p}^{(m-1)}}_{C^{\beta_m}}+1)^{\frac{1}{9}}K_m^{\frac{(n+1)}{9}}\epsilon_{n+1}^{\frac{4-2\gamma}{9}}.
\end{align*}

\end{enumerate}
\end{proposition}

\subsubsection{Proof strategy}\label{re:5} The proof is based on Proposition \ref{po:5}.
 First, we briefly explain how the constants $\gamma,\,a,\,b,\,s$ are chosen to ensure the induction works for $C^0$ and $C^\ell$ norms.
$\gamma$ determines the size of the new error, which is at best to be quadratically small for $C^0$ norm. So we let $1<\gamma<2$.

 Let $s=\ell+\varrho$ and $b=a^2$.
By  \eqref{for:181} of Proposition \ref{po:5},  the main part for $\norm{\mathfrak{p}^{(n+1)}}_{C^{\ell}}$ is $b^{\frac{\ell}{2}}\norm{\mathfrak{p}^{(n)}}_{C^{\varrho}}$.
Then we should have
\begin{align}\label{for:265}
 b^{\frac{\ell}{2}}\norm{\mathfrak{p}^{(n)}}_{C^{\varrho}}\overset{\text{\small$(*)$}}{\leq}b^{\frac{\ell}{2}}\epsilon_n^{1-\text{\small$\frac{\varrho(1+\gamma)}{\ell}$}}
 <\epsilon_{n+1}^{-\gamma}=\epsilon_{n}^{-\gamma^2}
\end{align}
Here in $(*)$ we use \eqref{for:183}.
By \eqref{for:153}  of Corollary \ref{le:14}, the main part for $\norm{\mathfrak{p}^{(n+1)}}_{C^{0}}$ is
 $b^{-\frac{\ell}{2}}\norm{\mathfrak{p}^{(n)}}_{C^{\ell+\varrho}}=b^{-\frac{\ell}{2}}\epsilon^{-\gamma}_n$. Then we should have
\begin{align}\label{for:266}
 b^{-\frac{\ell}{2}}\epsilon^{-\gamma}_n<\epsilon_{n+1}=\epsilon_{n}^\gamma.
\end{align}
\eqref{for:265} and \eqref{for:266} give
\begin{align*}
  \epsilon^{-2\gamma}_n <b^{\frac{\ell}{2}}<\epsilon_{n}^{-\gamma^2-1+\text{\small$\frac{\varrho(1+\gamma)}{\ell}$}}.
\end{align*}
We note that for sufficiently large $\ell$,
\begin{align*}
  \epsilon^{-2\gamma}_n <\epsilon_{n}^{-\gamma^2-1+\text{\small$\frac{\varrho(1+\gamma)}{\ell}$}}
\end{align*}
 holds  if $\gamma>1$. Hence, we choose
 \begin{align*}
 2\gamma<\kappa<\gamma^2+1\quad\text{and}\quad b=\epsilon^{-\text{\small$\frac{2\kappa}{\ell}$}}_n.
 \end{align*}
 Second, we briefly explain how to choose $a$ and $s$ (while keeping $b$ fixed) to obtain a desired estimate of $\norm{\mathfrak{p}^{(n+1)}}_{C^{m}}$ for any $m>\ell$. By  \eqref{for:181} of Proposition \ref{po:5},  the main part for $\norm{\mathfrak{p}^{(n+1)}}_{C^{m}}$ is $a^{m}\norm{\mathfrak{p}^{(n)}}_{C^{\varrho}}$, which diverges as $m\to\infty$.
To overcome this difficulty, we introduce a new parameter $\tau_m$ (see \eqref{for:184}), which compares the increasing speed of $\norm{\mathfrak{p}^{(n)}}_{C^{\beta_m}}$ and $a^{\beta_m}\norm{\mathfrak{p}^{(n)}}_{C^{\varrho}}$, thus enabling us to choose the right $a$ and $s$ to carry out the inductive procedure.

\noindent \textbf{Case 1:} $\tau_m>b^{\frac{1}{2}}$. \quad In this case,  $a^{\beta_m}\norm{\mathfrak{p}^{(n)}}_{C^{\varrho}}$ grows slower than $\norm{\mathfrak{p}^{(n)}}_{C^{\beta_m}}$. Then we just let
 \[
a = b^{\frac{1}{2}} \quad \text{and} \quad s = \ell + \varrho.
\]
 \textbf{Case 2:} $\tau_m\leq b^{\frac{1}{2}}$. \quad To prevent $a^{\beta_m}\|\mathfrak{p}^{(n)}\|_{C^{\varrho}}$ from exceeding $\|\mathfrak{p}^{(n+1)}\|_{C^{\beta_m}}$, we choose
   \begin{align*}
    a=\tau_m.
   \end{align*}
    However, estimating $\|\mathfrak{p}^{(n+1)}\|_{C^{0}}$ by taking $s = \ell + \varrho$ might be risky,   as
$a^{-(\ell+\varrho)}\norm{\mathfrak{p}^{{(n)}}}_{C^{\ell+\varrho}}$ may not be small.  A key observation is if $\tau_m\leq b^{\frac{1}{2}}$, then there is $1\leq p\leq m$ such that
$a^{-\beta_p}\norm{\mathfrak{p}^{{(n)}}}_{C^{\beta_p}}$ is sufficiently small (see \eqref{for:159}). Then we let
\begin{align*}
  s=\beta_p.
\end{align*}
Note that in this situation $s\gg \ell$. Another important point is that the constants in \eqref{for:151} are independent of $s$. This ensures that the induction still work for the $C^0$ norm.
This is how the new scheme work and where the parameter $s$ and $a$ play the crucial role  (with
$b$ kept fixed).

\subsubsection{Proof of Proposition \ref{po:3}} Let $b=\epsilon^{-\text{\small$\frac{2\kappa}{\ell}$}}_n$. Set
\begin{align}\label{for:184}
 \tau_m=\big(\norm{\mathfrak{p}^{(n)}}_{C^{\beta_m}}\norm{\mathfrak{p}^{(n)}}_{C^\varrho}^{-1}\epsilon^{\text{\tiny$-2(\gamma-1)$}}_n\big)^{\frac{1}{\beta_m}},\quad 1\leq m\leq n
\end{align}
and $\tau=\min_{1\leq i\leq n}\{\tau_i\}$. We point out that $\tau$ and $b$ are both dependent on $n$.

The below two lemmas establish Proposition \ref{po:3}  in the case of $\tau>b^{\frac{1}{2}}$ (see Lemma \ref{le:4}) and $\tau\leq b^{\frac{1}{2}}$ (see Lemma \ref{le:12}) respectively.

\begin{lemma}\label{le:4} Proposition \ref{po:3} holds if $\tau>b^{\frac{1}{2}}$.
\end{lemma}
\begin{proof}
By \eqref{for:183} we have
\begin{align}
 \norm{\mathfrak{p}^{(n)}}_{C^\varrho}\leq C_{\ell}\epsilon_n^{1-\text{\small$\frac{\varrho(1+\gamma)}{\ell}$}}<
 C_\ell\epsilon_n^{\text{\tiny$\frac{1}{2}$}+\eta}\overset{\text{\small$(*)$}}{<} \epsilon_n^{\text{\tiny$\frac{1}{2}$}}<\bar{c},
\end{align}
which allows us to apply Proposition \ref{po:5} to obtain  the new iterates $\mathfrak{p}^{(n+1)}$, $h_{n+1}$. Here in $(*)$ we use \eqref{for:170}.

Set $a=b^{\frac{1}{2}}=\epsilon^{-\text{\small$\frac{\kappa}{\ell}$}}_n$ and $s=\ell+\varrho$.

\smallskip

\eqref{for:176}:  By \eqref{for:182} of Proposition \ref{po:5} we have
\begin{align}
 \norm{\mathcal{T}_{n+1}-I}+\norm{g_{n+1}-I}&\leq C\norm{\mathfrak{p}^{(n)}}_{C^0} \leq \epsilon_n^{\text{\tiny$\frac{1}{2}$}} \label{for:169}.
\end{align}

\eqref{for:177}: By \eqref{for:180} and \eqref{for:182} of Proposition \ref{po:5} we have
\begin{align}\label{for:81}
 \max\{\norm{\mathfrak{h}&_{n+1}}_{C^1}, d(h_{n+1},I)_{C^1}\}\leq C a^{\text{\tiny$\varrho$}}b^{\text{\tiny$\varrho$}}\norm{\mathfrak{p}^{(n)}}_{C^{\varrho}}
 \overset{\text{\small$(\heartsuit)$}}{\leq} C b^{\text{\tiny$\frac{3\varrho}{2}$}}\norm{\mathfrak{p}^{(n)}}_{C^{\varrho}}\notag\\
 &\overset{\text{\small$(*)$}}{\leq} C_\ell\epsilon^{-\text{\small$\frac{3\kappa\varrho}{\ell}$}}_n\epsilon_n^{1-\text{\small$\frac{(1+\gamma)\varrho}{\ell}$}}
 \overset{\text{\small$(**)$}}{\leq} C_\ell\epsilon_n^{\text{\tiny$\frac{\gamma}{2}+\eta$}}<\epsilon_n^{\text{\tiny$\frac{\gamma}{2}$}}=\epsilon_{n+1}^{\text{\tiny$\frac{1}{2}$}}.
\end{align}
Here in $(*)$ we use \eqref{for:183}; in $(**)$ we use \eqref{for:170}.

\smallskip

\emph{Note}. Inequality  $\heartsuit$ still holds if we assume $a\leq b^{\frac{1}{2}}$, which will be used in the proof of Lemma \ref{le:12}.

\smallskip

\eqref{for:187}:  By \eqref{for:181} of Proposition \ref{po:5} we have
\begin{align}\label{for:174}
  \norm{\mathfrak{p}^{(n+1)}}_{C^{\ell+\varrho}}&\leq C_{\ell}(a^{\ell+\varrho}b^\varrho \norm{\mathfrak{p}}_{C^{\varrho}}+b^\varrho\norm{\mathfrak{p}^{(n)}}_{C^\ell}+1)\notag\\
  &\overset{\text{\small$(\heartsuit)$}}{\leq} C_{\ell}(b^{\frac{\ell+3\varrho}{2}} \norm{\mathfrak{p}}_{C^{\varrho}}+b^\varrho\norm{\mathfrak{p}^{(n)}}_{C^\ell}+1)\notag\\
  &\overset{\text{\small$(*)$}}{\leq} C_{\ell}(\epsilon^{-\text{\small$\frac{\kappa(\ell+3\varrho)}{\ell}$}}_n
 \epsilon_n^{1-\text{\small$\frac{(1+\gamma)\varrho}{\ell}$}}+
\epsilon^{-\text{\small$\frac{2\kappa\varrho}{\ell}$}}_n\epsilon^{-\gamma}_n+1)\notag\\
&\overset{\text{\small$(\diamond)$}}{<} C_{\ell}(2
\epsilon^{-\text{\small$\frac{(3\kappa+1+\gamma)\varrho}{\ell}$}}_n\epsilon^{1-\kappa}_n+1)
\overset{\text{\small$(**)$}}{<}4C_\ell\epsilon_n^{\text{\tiny$-\gamma^2$}+\eta}\notag\\
&<\epsilon_n^{\text{\tiny$-\gamma^2$}}=(\epsilon_{n+1})^{-\gamma}.
 \end{align}
Here in $(*)$ we use \eqref{for:183}; in $(\diamond)$ we note that $\gamma<\kappa-1$; in $(**)$ we use \eqref{for:171}.

\smallskip

\emph{Note}. Inequality  $\heartsuit$ still holds if we assume $a\leq b^{\frac{1}{2}}$, which will be used in the proof of Lemma \ref{le:12}.

\smallskip

\eqref{for:188}: We note that
\begin{align}\label{for:151}
 \norm{\mathfrak{p}^{(n)}}_{C^{\varrho+1}}\overset{\text{\tiny$(*)$}}{<}C_{\ell}\epsilon_n^{1-\text{\small$\frac{(1+\gamma)(\varrho+1)}{\ell}$}}
 \overset{\text{\tiny$(\diamondsuit)$}}{<} C_{\ell}\epsilon_n^{\text{\tiny$\frac{1}{2}$}+\eta}<\epsilon_n^{\text{\tiny$\frac{1}{2}$}}<1.
\end{align}
Here in $(*)$  we use \eqref{for:183} to estimate $\norm{\mathfrak{p}^{(n)}}_{C^{\varrho+1}}$; in $\text{\tiny$(\diamondsuit)$}$ we use \eqref{for:175}.

It follows from \eqref{for:153} of Corollary \ref{le:14} that
\begin{align}\label{for:191}
\norm{\mathfrak{p}^{(n+1)}}_{C^0}
 &\overset{\text{\small$(*)$}}{\leq} 2C_{\ell}\epsilon^{-\text{\small$\frac{6\kappa\varrho}{\ell}$}}_n
 \big(\epsilon_n^{1-\text{\small$\frac{(1+\gamma)(\varrho+1)}{\ell}$}}\big)^{2(1-\frac{\varrho}{\ell})}( \epsilon^{-\gamma}_n)^{\frac{\varrho}{\ell}}\notag\\
 &+4C_{\ell}\epsilon^{-\text{\small$\frac{\kappa(-\ell+4\varrho)}{\ell}$}}_n\epsilon^{-\gamma}_n\notag\\
 &\overset{\text{\small$(**)$}}{\leq} 6C_\ell\epsilon_n^{\gamma+\eta}<\epsilon_n^{\text{\tiny$\gamma$}}=\epsilon_{n+1}
 \end{align}
 Here in $(*)$ we use \eqref{for:183} to estimate $\norm{\mathfrak{p}^{(n)}}_{C^{\varrho+1}}$; in $(**)$ we use \eqref{for:178} and \eqref{for:173}.

\smallskip

\eqref{for:189}:  By \eqref{for:182} and \eqref{for:181} of Proposition \ref{po:5}, for any $1\leq m\leq n+1$ we have
\begin{align}\label{for:172}
 &\max\{\norm{\mathfrak{p}^{(n+1)}}_{C^{\beta_m}},\,d(h_{n+1},I)_{C^{\beta_m}}\}\notag\\
 &\leq C_m(a^{\beta_m}b^\varrho\norm{\mathfrak{p}^{(n)}}_{C^{\varrho}}+b^\varrho\norm{\mathfrak{p}^{(n)}}_{C^{\beta_m}}+1).
 \end{align}
 For $m=n+1$, let $K_{n+1}=2C_{n+1}a^{\beta_{n+1}}b^\varrho$, then
 \begin{align*}
 &\max\{\norm{\mathfrak{p}^{(n+1)}}_{C^{\beta_{n+1}}},\,d(h_{n+1},I)_{C^{\beta_{n+1}}}\}\leq K_{n+1}(\norm{\mathfrak{p}^{(n)}}_{C^{\beta_{n+1}}}+1).
 \end{align*}
If  $1\leq m\leq n$, using \eqref{for:172} we have
 \begin{align}\label{for:192}
 &\max\{\norm{\mathfrak{p}^{(n+1)}}_{C^{\beta_m}},\,d(h_{n+1},I)_{C^{\beta_m}}\}\notag\\
 &\overset{\text{\tiny$(o)$}}{\leq} C_{m}
(\tau_m^{\beta_m}\epsilon^{-\text{\small$\frac{2\kappa\varrho}{\ell}$}}_n\norm{\mathfrak{p}^{(n)}}_{C^{\varrho}}+1)+C_{m}
\epsilon^{-\text{\small$\frac{2\kappa\varrho}{\ell}$}}_n\norm{\mathfrak{p}^{(n)}}_{C^{\beta_m}}\notag\\
&\overset{\text{\tiny$(\diamond)$}}{=} C_{m}
(\epsilon^{-\text{\small$\frac{2\kappa\varrho}{\ell}$}}_n\norm{\mathfrak{p}^{(n)}}_{C^{\beta_m}}\epsilon^{\text{\tiny$-2(\gamma-1)$}}_n+1)+C_{m}
\epsilon^{-\text{\small$\frac{2\kappa\varrho}{\ell}$}}_n\norm{\mathfrak{p}^{(n)}}_{C^{\beta_m}}\notag\\
&\leq 2C_{m}\epsilon^{-\text{\small$\frac{2\kappa\varrho}{\ell}$}}_n\epsilon^{\text{\tiny$-2(\gamma-1)$}}_n (\norm{\mathfrak{p}^{(n)}}_{C^{\beta_m}}+1)\notag\\
&\overset{\text{\tiny$(*)$}}{\leq} 4C_{m}\epsilon^{-\text{\small$\frac{2\kappa\varrho}{\ell}$}}_n\epsilon^{\text{\tiny$-2(\gamma-1)$}}_n K_m^n\epsilon^{\text{\tiny$-2\gamma$}}_n(\norm{\mathfrak{p}^{(m-1)}}_{C^{\beta_m}}+1)\notag\\
& \overset{\text{\tiny$(**)$}}{\leq} 4C_{m}K_m^n\epsilon^{\text{\tiny$-2\gamma$}}_{n+1}(\norm{\mathfrak{p}^{(m-1)}}_{C^{\beta_m}}+1)\notag\\
&=K_m^{n+1}\epsilon^{\text{\tiny$-2\gamma$}}_{n+1}(\norm{\mathfrak{p}^{(m-1)}}_{C^{\beta_m}}+1).
\end{align}
Here in $(o)$ we use the fact $\tau_m\geq\tau>a$; in $(\diamond)$ we use \eqref{for:184}; in $(*)$ we use induction assumption; in $(**)$ we use \eqref{for:158}.

\smallskip
\eqref{for:193}:  By interpolation inequalities, for $m\leq n+1$ we have
\begin{align}\label{for:194}
 d(h_{n+1},I)_{C^{\frac{\beta_m}{9}}}&\leq C_md(h_{n+1},I)_{C^{0}}^{\frac{8}{9}}d(h_{n+1},I)_{C^{\beta_m}}^{\frac{1}{9}}\notag\\
 &\overset{\text{\tiny$(*)$}}{\leq} C_m(\norm{\mathfrak{p}^{(m-1)}}_{C^{\beta_m}}+1)^{\frac{1}{9}}K_m^{\frac{(n+1)}{9}}\epsilon_{n+1}^{\frac{4-2\gamma}{9}}.
\end{align}
Here in $(*)$ we use \eqref{for:81} and \eqref{for:192}.

\end{proof}

\begin{lemma}\label{le:12} Proposition \ref{po:3} holds if $\tau\leq b^{\frac{1}{2}}$.
\end{lemma}
\begin{proof}
Choose $1\leq p\leq n$ such that $\tau_p=\tau$.  Set $a=\tau$ and $s=\beta_p$.

\smallskip
\eqref{for:176}, \eqref{for:177} and \eqref{for:187}: Since $\tau=a\leq b^{\frac{1}{2}}= \epsilon^{-\text{\small$\frac{\kappa}{\ell}$}}_n$, the estimates for $g_{n+1}$ and  $\mathcal{T}_{n+1}$ (see \eqref{for:169}), $\norm{\mathfrak{h}_{n+1}}_{C^{1}}$ and $\norm{h_{n+1}-I}_{C^1}$ (see \eqref{for:81}), $\norm{\mathfrak{p}^{(n+1)}}_{C^{\ell+\varrho}}$ (see \eqref{for:174}) still hold.

\smallskip
\eqref{for:188}: We use \eqref{for:15} of Corollary \ref{le:14}. Next, we estimate $a^{-s}\norm{\mathfrak{p}^{(n)}}_{C^s}$ which is an essential component for the estimate:
\begin{align}\label{for:159}
 a^{-s}\norm{\mathfrak{p}^{(n)}}_{C^s}&=\tau_p^{-\beta_p}\norm{\mathfrak{p}^{(n)}}_{C^{\beta_p}}
 \overset{\text{\tiny$(*)$}}{=}\norm{\mathfrak{p}^{(n)}}_{C^\varrho}\epsilon^{\text{\tiny$2(\gamma-1)$}}_n\notag\\
 &\overset{\text{\tiny$(**)$}}{\leq} C_{\ell}\epsilon_n^{1-\text{\small$\frac{\varrho(1+\gamma)}{\ell}$}}\epsilon^{\text{\tiny$2(\gamma-1)$}}_n
 =C_{\ell}\epsilon_n^{-\text{\small$\frac{\varrho(1+\gamma)}{\ell}$}}\epsilon^{\text{\tiny$2\gamma-1$}}_n\notag\\
 &\overset{\text{\tiny$(\diamond)$}}{<}\epsilon_n^{-\text{\small$\frac{2\varrho(1+\gamma)}{\ell}$}}\epsilon^{\text{\tiny$2\gamma-1$}}_n\overset{\text{\tiny$(***)$}}{<}1.
\end{align}
here in $(*)$ we use  \eqref{for:184} and in $(**)$ we use \eqref{for:183}; in $(\diamond)$ we use
$C_\ell\epsilon_n^{\text{\small$\frac{\varrho(1+\gamma)}{\ell}$}}<1$ if $\epsilon$ is chosen sufficiently small (note that $\epsilon$ is chosen after $\ell$ is chosen); in
$(***)$ we use \eqref{for:213}.

It follows from \eqref{for:15} of Corollary \ref{le:14} and \eqref{for:159} that
\begin{align*}
 \norm{\mathfrak{p}^{(n+1)}}_{C^{0}}&\leq 2C_{\ell}\epsilon^{-\text{\small$\frac{6\kappa\varrho}{\ell}$}}_n
 \big(\epsilon_n^{1-\text{\small$\frac{(1+\gamma)(\varrho+1)}{\ell}$}}\big)^{2(1-\frac{\varrho}{\ell})}( \epsilon^{-\gamma}_n)^{\frac{\varrho}{\ell}}\notag\\
&+2C_{\ell}\epsilon^{-\text{\small$\frac{2\kappa\varrho}{\ell}$}}_n(\epsilon_n^{-\text{\small$\frac{2\varrho(1+\gamma)}{\ell}$}}
 \epsilon^{\text{\tiny$2\gamma-1$}}_n)^{(\text{\tiny$1-\frac{\varrho}{\ell}$})^2}
 (\epsilon_n^{-\gamma})^{\frac{\varrho}{\ell}(2-\frac{\varrho}{\ell})}\notag\\
 &+2C_\ell \epsilon_n^{-\text{\small$\frac{2\kappa}{\ell}$}(\text{\tiny$-\frac{\ell}{2}+2\varrho$})}\epsilon^{-\gamma}_n\\
 &\overset{\text{\tiny$(*)$}}{\leq} 2C_\ell\epsilon_n^{\gamma+\eta}\\
 &+2C_{\ell}\epsilon_n^{-\text{\small$\frac{2\kappa\varrho}{\ell}$}}
 (\epsilon_n^{-\text{\small$\frac{2\varrho(1+\gamma)}{\ell}$}}\epsilon^{\text{\tiny$2\gamma-1$}}_n)^{(\text{\tiny$1-\frac{\varrho}{\ell}$})^2}
 (\epsilon_n^{-\gamma})^{\frac{\varrho}{\ell}(2-\frac{\varrho}{\ell})}\\
 &+2C_\ell\epsilon_n^{\gamma+\eta}\\
 &\overset{\text{\tiny$(\diamond)$}}{\leq} 2C_\ell\epsilon_n^{\gamma+\eta}+2C_\ell\epsilon_n^{\gamma+\eta}
 +2C_\ell\epsilon_n^{\gamma+\eta}
 <\epsilon_n^{\text{\tiny$\gamma$}}=\epsilon_{n+1}.
\end{align*}
Here in $(*)$ we use \eqref{for:178}, \eqref{for:173} and $2-\frac{\varrho}{\ell}>1$; in $(\diamond)$ we use \eqref{for:190}

\smallskip
\eqref{for:189}: By \eqref{for:182}, \eqref{for:181} for any $m\leq n+1$ we have
\begin{align*}
 &\max\{\norm{\mathfrak{p}^{(n+1)}}_{C^{\beta_m}},\,\norm{h_n-I}_{C^{\beta_m}}\}\notag\\
 &\leq C_m(a^{\beta_m}b^\varrho\norm{\mathfrak{p}^{(n)}}_{C^{\varrho}}+b^\varrho\norm{\mathfrak{p}^{(n)}}_{C^{\beta_m}}+1)\\
 &\overset{\text{\tiny$(*)$}}{\leq} C_{m}
(\epsilon^{-\text{\small$\frac{2\kappa\varrho}{\ell}$}}_n\tau_m^{\beta_m}\norm{\mathfrak{p}^{(n)}}_{C^{\varrho}}+1)+C_{m}
\epsilon^{-\text{\small$\frac{2\kappa\varrho}{\ell}$}}_n\norm{\mathfrak{p}^{(n)}}_{C^{\beta_m}}.
 \end{align*}
Here in $(*)$ we use $\tau_m\geq\tau=a$. By the same arguments as in \eqref{for:192}, we still get
\begin{align*}
 &\max\{\norm{\mathfrak{p}^{(n+1)}}_{C^{\beta_m}},\,d(h_{n+1},I)_{C^{\beta_m}}\}\\
 &\leq K_m^{n+1}\epsilon^{\text{\tiny$-2\gamma$}}_{n+1}(\norm{\mathfrak{p}^{(m-1)}}_{C^{\beta_m}}+1).
 \end{align*}

 \smallskip
\eqref{for:193}: Once \eqref{for:177} and \eqref{for:189} are proved, the result follows exactly the same way as \eqref{for:194} is obtained.

\smallskip

Thus we complete the proof of the lemma.

\end{proof}

\subsection{Proof of Theorem \ref{th:13} and Theorem \ref{th:11}}
Proposition \ref{po:3} shows that  we can obtain an infinite sequence $\mathfrak{p}^{(n)}$ inductively. Set
\begin{align*}
 H_n=h_n\circ \cdots \circ h_0\quad\text{and}\quad \iota_n=\mathcal{T}_n\circ \cdots \circ \mathcal{T}_0.
\end{align*}
Then \eqref{for:177} of Proposition \ref{po:3} shows that $H_n$ converges in $C^1$ topology to a $C^1$ conjugacy $h$ between $\tilde{\alpha}_A$ and $\alpha_A$; moreover, \eqref{for:193} of Proposition \ref{po:3} shows that
 the convergence of the sequence $H_n$ holds in $C^{\frac{\beta_m}{9}}$ for any $m\in\NN$. Hence we see that
 $h$ is of class $C^{\infty}$. \eqref{for:176} of Proposition \ref{po:3} shows that $\iota_n$ converges to an invertible   linear map $\iota$ of $\text{Lie}(A)$.  The convergence step shows that:
\begin{align*}
  h\circ\tilde{\alpha}_A\big(\exp(t(\iota E_i)),h^{-1}x\big)=\alpha_{A}(\exp(tE_i),x).
\end{align*}
for all $x\in \mathcal{X}$, $t\in\RR$, $1\leq i\leq d$.

Let $\mathfrak{i}$ be the group isomorphism of $A$ induced by $\iota$.  We also have
\begin{align*}
  h\circ\tilde{\alpha}_A(\mathfrak{i}(\textbf{a}),h^{-1}x)=\alpha_{A}(\textbf{a},x),\quad \text{for all }\textbf{a}\in A,\,x\in \mathcal{X}.
\end{align*}
This completes the proof of Theorem \ref{th:11} and Theorem \ref{th:13}.

\section{Proof of Corollaries to  Theorem \ref{th:13} and Theorem \ref{th:11}}

\subsection{Proof of Corollary \ref{cor:8}}\label{sec:41} It suffices to check the conditions of Theorem \ref{th:11}.   By Proposition \ref{po:2},
$\alpha_A$ is geometrically stable. It is harmless to assume that a basis of $\text{Lie}(A)$ is : $\{\mathfrak{u}_{i,j}:i\in 2\NN-1,j\in 2\NN\}$,
see \eqref{for:22} of Section \ref{sec:47}, the description of $\text{Lie}(A)$. Let $v_1=\mathfrak{u}_{1,2}$, $v_2=\mathfrak{u}_{3,4}$ and $v_3=\mathfrak{u}_{5,6}$. It is clear that they embed in a subalgebra isomorphic to $\mathfrak{sl}(2,\RR)\times \mathfrak{sl}(2,\RR)\times \mathfrak{sl}(2,\RR)$. Thus we complete the proof.

\subsection{Proof of Corollary \ref{cor:2}} It suffices to check the conditions of Theorem \ref{th:13}. Since any maximal abelian subgroup in $SL(n,\RR)$, $n\geq 4$
is unipotent \cite{Malcev}, $A$ is unipotent as each $A_i$, $1\leq i\leq k$ is unipotent. Then $\alpha_A$ is parabolic.   By Proposition \ref{po:2},
$\alpha_A$ is geometrically stable. It is harmless to assume that a basis of $\text{Lie}(A_1)$ is : $\{\mathfrak{u}_{i,j}:i\in 2\NN-1,j\in 2\NN\}$,
see \eqref{for:22} of Section \ref{sec:47}, the description of $\text{Lie}(A)$. Let $\textbf{v}=\mathfrak{u}_{1,2}$ and fix $\textbf{u}\in \text{Lie}(A_2)$.
It is clear that $\textbf{v}$ and $\textbf{u}$ satisfy the assumption in Theorem \ref{th:13}. Thus we complete the proof.

\subsection{Proof of Corollary \ref{cor:13}} We check the conditions of Theorem \ref{th:11}. By Proposition \ref{po:2} $\alpha_A$ is geometrically stable. As $n\geq 7$, from  arguments in Section \ref{sec:41} we see that there are elements $v_i\in A_1$, $1\leq i\leq3$ such that they embed in a subalgebra isomorphic to $\mathfrak{sl}(2,\RR)\times \mathfrak{sl}(2,\RR)\times \mathfrak{sl}(2,\RR)$. This completes the proof.

\subsection{Proof of Corollary \ref{cor:11}} From Theorem \ref{th:8} we see that there is a desired splitting for $\alpha_A$. Consequently, weak local rigidity follows immediately
from a standard argument, see \cite{Damjanovic2}, \cite{DT}.

\appendix
\section{Proof of Theorem \ref{th:2}}\label{sec:12}

Below, we recall  a conclusion from \cite{ramirez2009cocycles} for cohomological equations over unipotent flows.
\begin{lemma} (Theorem B' of \cite{ramirez2009cocycles})
Suppose $v\in \mathfrak{G}^1$ is nilpotent. Then there exist a set of vectors $\{u_i: 1\leq i\leq k\}$ in $\mathfrak{G}^1$ whose commutators span $\mathfrak{G}$
such that: for any unitary representation $(\pi,\mathcal{H})$ of $G$,  if the restriction of $\pi$ to each
simple factor of $G$ has a spectral gap, then for any $f\in \mathcal{H}^\infty$ satisfying the cohomological equation $vf=g$, we have
\begin{align*}
    \norm{u_i^mf}\leq C_m\norm{g}_{m+2}, \qquad 1\leq i\leq k,\,\,m\geq0.
\end{align*}

\end{lemma}
The next result provides global estimates for the solution of the regular representations, which is a direct consequence of  the above lemma, Theorem \ref{th:10} and Theorem \ref{th:5}:
\begin{theorem}\label{th:1}
Suppose $v\in \mathfrak{G}^1$ is nilpotent. If $\Gamma$ is a cocompact irreducible lattice and $\mathcal{H}=L^2_0(G/\Gamma)$, then there are constants $s_1>0$ and $s_2\geq2$ dependent only on $G$ and $\Gamma$ such that if $f\in \mathcal{H}^\infty$ satisfying the cohomological equation $vf=g$, we have
\begin{align*}
 \norm{f}_t\leq C_{t}\norm{g}_{s_2t+s_1},\qquad t\geq0.
\end{align*}

\end{theorem}
Now we proceed to the proof of Theorem \ref{th:2}. Choose a basis in which $\textrm{ad}_{v}$ has its Jordan normal form. Let $J_v=(z_{i,j})$ be an $m\times m$ matrix which consists of blocks of $\textrm{ad}_{v}$; i.e., $z_{i,i}=0$, and  $z_{i,i+1}=*_i\in \{0,\,1\}$ for
all $i=1,\cdots,m-1$.  Let the coordinate functions of $\mathfrak{u}$ and $\mathfrak{v}$ be $\mathfrak{u}_i$ and $\mathfrak{v}_i$, $1\leq i\leq m$ respectively.

Then the $m$-th equation of \eqref{for:20} is $v\mathfrak{u}_m=\mathfrak{v}_m$. Then the estimates
\begin{align}\label{for:202}
 \norm{\mathfrak{u}_m}_{t}\leq C_t\norm{\mathfrak{v}_m}_{s_2t+s_1},\qquad t\geq0
\end{align}
follow from Theorem \ref{th:1}.  The $(m-1)$-th equation in \eqref{for:12} is
\begin{align*}
 v\mathfrak{u}_{m-1}+*_{m-1}\mathfrak{u}_{m}=\mathfrak{v}_{m-1}.
\end{align*}
Then we obtain $v\mathfrak{u}_{m-1}=\mathfrak{v}_{m-1}-*_{m-1}\mathfrak{u}_{m}$. By Theorem \ref{th:1} and \eqref{for:202}, the following estimates hold:
\begin{align*}
 \norm{\mathfrak{u}_{m-1}}_{t}\leq C_t\norm{\mathfrak{v}_{m-1}-*_{m-1}\mathfrak{u}_{m}}_{s_2t+s_1}\leq C_t\norm{\mathfrak{v}}_{s_2^2t+(s_2+1)s_1}.
\end{align*}
Set $p_{0}(s_2,s_1)=s_1$. We can obtain a sequence $p_1(s_2,s_1),\cdots,p_{m-1}(s_2,s_1)$ using a recursive rule:
\begin{align}\label{for:51}
p_{i+1}(s_2,s_1)=s_2\cdot p_{i}(s_2,s_1)+s_1.
\end{align}
Inductively, we can show that for $1\leq k\leq m-1$ we have
\begin{align*}
 \norm{\mathfrak{u}_{k}}_{t}\leq C_t\norm{\mathfrak{v}}_{s_2^{m-k+1}t+p_{m-k}(s_2,s_1)},\qquad t\geq0.
\end{align*}
Hence we finish the proof on $J_v$.  If repeated for all Jordan blocks we get the result. It is clear that $\lambda_1$ is the maximum of $p_{m-1}(s_2,s_1)$ where $p_{m-1}$ ranges over all Jordan blocks and $\lambda\leq\dim\mathfrak{G}$. Hence we finish the proof.



\section{Proof for extended representations}\label{sec:23}
\subsection{Proof of Corollary \ref{cor:6}} \label{sec:31}
 Choose a basis in which $\textrm{ad}_{u}$ has its Jordan normal form.  We use $J_u=(w_{i,j})$ to denote an $m\times m$ matrix which consists of blocks of $\textrm{ad}_{u}$; i.e., let $w_{i,i}=0$ for all $i=1,\cdots,m$ (we note that $u$ is nilpotent) and $w_{i,i+1}=*_i\in \{0,\,1\}$ for
all $i=1,\cdots,m-1$. The $m$-th equation of \eqref{for:17} is
\begin{align}\label{for:91}
 u\Theta_m=\Omega_m;
\end{align}
and the $k$-th equation, $1\leq k\leq m-1$ in \eqref{for:17} is
\begin{align}\label{for:92}
 u\Theta_{k}+*_{k}\Theta_{k+1}=\Omega_{k},
\end{align}
where $\Theta_{k}$ and $\Omega_{k}$ are coordinate functions of $\Theta$ and $\Omega$ respectively.

From \eqref{for:91} by \eqref{for:140} of Lemma \ref{le:1} we have
\begin{align*}
 \norm{\Theta_m}_{\{H,G_u\},t}\leq C_{t} \norm{\Omega_m}_{\{H,G_u\},t+\frac{5}{2}}
\end{align*}
for any $0\leq t\leq s-\text{\small$\frac{5}{2}$}$.

We proceed by induction. Fix $1\leq k\leq m-1$.  Suppose for any $k+1\leq j\leq m$ we have
\begin{align}\label{for:208}
 \norm{\Theta_{j}}_{\{H,G_u\},t}\leq C_t\norm{\Omega}_{\{H,G_u\},t+(m-j+1)\frac{5}{2}}
\end{align}
for any $0\leq t\leq s-(m-j+1)\text{\small$\frac{5}{2}$}$.

From \eqref{for:92} we have
\begin{align*}
 \norm{\Theta_k}_{\{H,G_u\},t}\overset{\text{\tiny$(1)$}}{\leq} C_{t} \norm{\Omega_{k}-*_{k}\Theta_{k+1}}_{\{H,G_u\},t+\frac{5}{2}}\overset{\text{\tiny$(2)$}}{\leq} C_t\norm{\Omega}_{\{H,G_u\},t+(m-k+1)\frac{5}{2}}
\end{align*}
for any $0\leq t\leq s-(m-k+1)\text{\small$\frac{5}{2}$}$. Here in $(1)$ we use \eqref{for:140} of Lemma \ref{le:1}; in $(2)$ we use \eqref{for:208}.

Then we proved the case of $k$ and thus finish the proof on the block $J_u$.  By repeating the above arguments for all Jordan blocks we get the result by noting that the size of each block is less than $\dim\mathfrak{g}$.

\subsection{Proof of Corollary \ref{le:6}} \label{sec:35}
Choose a basis for $\text{ad}_u$ and fix an $m\times m$ matrix $J_{u}$ as described in Section \ref{sec:31}.
Let $J_v=(z_{k,j})$ be the corresponding blocks of $\textrm{ad}_v$: $z_{i,i}=0$
for all $1\leq i\leq m$ and $z_{k,j}=0$ for all $m\geq k>j\geq1$.
Since $\textrm{ad}_{u}$ and $\textrm{ad}_v$ commute, we have:
\begin{align}\label{for:78}
 *_jz_{k,j}=*_kz_{k+1,j+1}
\end{align}
for all $1\leq k\leq m-1$ and $k+1\leq j \leq m-1$.

 \eqref{for:18} splits into $m$ equations. For the $m$-th equation we have
\begin{align}\label{for:27}
v\Omega_m-u\Psi_m=\mathfrak{w}_m;
\end{align}
and for every $k=1,\cdots,m-1$ we have the following equation:
\begin{align}\label{for:39}
\big(v\Omega_k+\sum_{\text{\tiny$k+1\leq j\leq m$}}z_{k,j}\Omega_j\big)-(u\Psi_k+*_k\Psi_{k+1})=\mathfrak{w}_k.
\end{align}
Since $\mathcal{D}^l(\Omega_m)=\Omega_m$, by \eqref{for:11} of Lemma \ref{le:1} we see that the equation
\begin{align}\label{for:24}
  u\eta_m=\Omega_m
\end{align}
has a solution $\eta_m\in \mathcal{H}$ satisfying $\mathcal{D}^l(\eta_m)=\eta_m$ with estimates
      \begin{align*}
\norm{Y^j\eta_m}_{G_u,t}\leq C_{j,t} \max_{0\leq i\leq j}\{\norm{Y^i\Omega_m}_{G_u,t+\frac{3}{2}}\}
\end{align*}
for any $t\leq \sigma$, $j\geq0$, where $Y$ stands for $X_u,\,u$ or $Y\in \mathcal{C}(\mathfrak{g}_u)$.

From \eqref{for:27} and \eqref{for:24}, noting that $[v,u]=0$, we have
\begin{align*}
 u\mathcal{R}_m=-\mathfrak{w}_m.
\end{align*}
where $\mathcal{R}_m=\Psi_m-v\eta_m$.

Since $\mathcal{D}^l(\mathfrak{w}_m)=\mathfrak{w}_m$, it follows from \eqref{for:11} of Lemma \ref{le:1} that
\begin{align*}
\norm{Y^j\mathcal{R}_m}_{G_u,t}\leq C_{j,t} \max_{0\leq i\leq j}\{\norm{Y^i\mathfrak{w}_m}_{G_u,t+\frac{3}{2}}\},
\end{align*}
for any $t\leq \sigma$ and $j\geq0$, where $Y$ stands for $X_u,\,u$ or $Y\in \mathcal{C}(\mathfrak{g}_u)$.

Now we proceed by induction. Fix $k$ between $1$ and $m-1$ and assume that for all $k+1\leq r\leq m$ we already have the the following
\begin{align}\label{for:38}
 \Psi_r&=v\mathfrak{\eta}_r+\sum_{\text{\tiny$r+1\leq l\leq m$}}z_{r,l}\mathfrak{\eta}_l+\mathcal{R}_r, \notag\\
 \Omega_r&=u\eta_r+*_r\eta_{r+1},
\end{align}
where $\eta_r$ satisfies $\mathcal{D}^l(\eta_r)=\eta_r$,  with the estimates: for any $t\leq \sigma-\text{\small$\frac{3(m-r+1)}{2}$}$ and $j\geq0$
\begin{align}\label{for:42}
\norm{Y^j\eta_r}_{G_u,t}&\leq C_{j,t} \max_{\substack{j\leq i\leq m,\\ 0\leq p\leq j}}\{\norm{Y^p\Omega_i}_{G_u,t+\frac{3(m-r+1)}{2}}\}; \quad\text{and}\\
\norm{Y^j\mathcal{R}_r}_{G_u,t}&\leq C_{j,t} \max_{\substack{j\leq i\leq m,\\ 0\leq p\leq j}}\{\norm{Y^p\mathfrak{w}_i}_{G_u,t+\frac{3(m-r+1)}{2}}\}
\end{align}
where $Y$ stands for $X_u,\,u$ or $Y\in \mathcal{C}(\mathfrak{g}_u)$.

We substitute  the expressions for $\Omega_r$ and $v_r$ for all $k+1\leq r\leq m$ from \eqref{for:38}  into \eqref{for:39}. Then we have
\begin{align*}
&v(\Omega_k-*_{k}\eta_{k+1})-u(\Psi_k-\sum_{\text{\tiny$k+1\leq r\leq m$}}z_{k,r}\eta_r)+\mathcal{R}=\mathfrak{w}_k+*_k\mathcal{R}_{k+1}.
\end{align*}
where
\begin{align*}
 \mathcal{R}=\sum_{r=k+1}^m*_rz_{k,r}\eta_{r+1}-*_k\sum_{r=k+2}^mz_{k+1,r}\eta_r.
\end{align*}
From \eqref{for:78} we see that $\mathcal{R}=0$. Hence we have
\begin{align}\label{for:25}
&v(\Omega_k-*_{k}\eta_{k+1})-u(\Psi_k-\sum_{\text{\tiny$k+1\leq j\leq m$}}z_{k,j}\eta_j)=\mathfrak{w}_k+*_k\mathcal{R}_{k+1}.
\end{align}
By \eqref{for:11} of Lemma \ref{le:1} the equation
\begin{align}\label{for:30}
 u\eta_{k}=\Omega_k-*_{k}\eta_{k+1}
\end{align}
has a solution $\eta_{k}\in\mathcal{H}$ satisfying $\mathcal{D}^l(\eta_k)=\eta_k$ with estimates: for any $j\geq0$
\begin{align*}
\norm{Y^j\eta_k}_{G_u,t}&\leq C_{j,t} \max_{0\leq p\leq j}\{\norm{Y^p(\Omega_k-*_{k}\eta_{k+1})}_{G_u,t+\frac{3}{2}}\}\\
&\overset{\text{(1)}}{\leq}
C_{j,t} \max_{\substack{k\leq i\leq m,\\0\leq p\leq j}}\{\norm{Y^p\Omega_i}_{G_u,t+\frac{3(m-k+1)}{2}}\}
\end{align*}
for any $t\leq \sigma-\text{\small$\frac{3(m-k+1)}{2}$}$, where $Y$ stands for $X_u,\,u$ or $Y\in \mathcal{C}(\mathfrak{g}_u)$. Here in $(1)$ we use
\eqref{for:42}.

From \eqref{for:25} and \eqref{for:30} we have
\begin{align*}
&u\mathcal{R}_k
=-(\mathfrak{w}_k+*_k\mathcal{R}_{k+1}).
\end{align*}
where
\begin{align*}
 \mathcal{R}_k=\Psi_k-v\eta_{k}-\sum_{\text{\tiny$k+1\leq j\leq m$}}z_{k,j}\eta_j.
\end{align*}
It follows from \eqref{for:11} of Lemma \ref{le:1} that: for any $j\geq0$
\begin{align*}
\norm{Y^j\mathcal{R}_k}_{G_u,t}&\leq C_{j,t} \max_{0\leq p\leq j}\{\norm{Y^p\big(\mathfrak{w}_k+*_k\mathcal{R}_{k+1}\big)}_{G_u,t+\frac{3}{2}}\}\\
&\overset{\text{(2)}}{\leq}
C_{j,t} \max_{\substack{k\leq i\leq m,\\0\leq p\leq j}}\{\norm{Y^p\mathfrak{w}_i}_{G_u,t+\frac{3(m-k+1)}{2}}\}
\end{align*}
for any $t\leq \sigma-\text{\small$\frac{3(m-k+1)}{2}$}$, where $Y$ stands for $X_u,\,u$ or $Y\in \mathcal{C}(\mathfrak{g}_u)$.

Then we proved the case of $k$ and thus finish the proof on the block $J_u$.  By repeating the above arguments for all Jordan blocks we get the result. It is clear that the size of each block is less than $\dim\mathfrak{g}$. Set
\begin{align*}
 \eta=(\eta_1,\cdots,\eta_{\dim\mathfrak{g}})\quad\text{and}\quad\mathcal{R}=(\mathcal{R}_1,\cdots,\mathcal{R}_{\dim\mathfrak{g}}).
\end{align*}
Then we see that $\eta$ and $\mathcal{R}$ satisfy equation \ref{for:83} with estimates:
for $j\geq0$
\begin{align*}
\norm{Y^j\eta}&\leq C_{j} \max_{0\leq p\leq j}\{\norm{Y^p\Omega}_{G_u,\text{\tiny$\frac{3}{2}$}\dim\mathfrak{g}}\}; \quad\text{and}\\
\norm{Y^j\mathcal{R}}&\leq C_{j} \max_{0\leq p\leq j}\{\norm{Y^p\mathfrak{w}}_{G_u,\text{\tiny$\frac{3}{2}$}\dim\mathfrak{g}}\}
\end{align*}
where $Y$ stands for $X_u,\,u$ or $Y\in \mathcal{C}(\mathfrak{g}_u)$. Hence \eqref{for:88} and \eqref{for:89}
follow from the above estimates and Theorem \ref{th:4}. Then we finish the proof.

\subsection{Proof of Corollary \ref{le:11}}\label{sec:36}

 We follow the notations and proof line of that of Lemma \ref{le:6}. \eqref{for:6} splits into $m$ equations (under the basis as described in the proof of Lemma \ref{le:6}). For the $m$-th equation we have \eqref{for:27}. By Lemma \ref{le:9} there exists $\eta_m\in \mathcal{H}_{S_0}^{\infty}$ satisfying $\mathcal{D}^l(\eta_m)=0$ with estimates
\begin{align*}
 \norm{\eta_m}_{S_0,t}\leq C_{t}\norm{\Omega_m}_{S_0,t+6+\text{\tiny$\frac{l}{2}$}}
\end{align*}
for any $t\geq0$,  such that
\begin{align*}
 \Omega_m&=u\eta_m+\mathcal{R}_{1,m},\qquad\text{and}\\
 \Psi_m&=v\eta_m+\mathcal{R}_{2,m}
\end{align*}
with estimates
\begin{align*}
 \norm{\mathcal{R}_{1,m},\,\mathcal{R}_{2,m}}_{L_1,t}\leq C_{t}\norm{\mathfrak{w}_m}_{L_1,t+6+\text{\tiny$\frac{l}{2}$}},\qquad t\geq0.
\end{align*}
Next we still proceed by induction.  We note that for every $k=1,\cdots,m-1$ \eqref{for:39} still holds.  Fix $k$ between $1$ and $m-1$ and assume that for all $j=k+1,\cdots,m$ we already have the the following
\begin{align}\label{for:37}
 \Psi_j&=v\mathfrak{\eta}_j+\sum_{\text{\tiny$j+1\leq l\leq m$}}z_{j,l}\mathfrak{\eta}_l+\mathcal{R}_{2,j}, \notag\\
 \Omega_j&=u\eta_j+*_j\eta_{j+1}+\mathcal{R}_{1,j},
\end{align}
where $\eta_j$ satisfies $\mathcal{D}^l(\eta_j)=0$ with the estimates: for any $t\geq0$
\begin{align}\label{for:48}
\norm{\eta_j}_{S_0,t}&\leq C_{t} \max_{j\leq i\leq m}\{\norm{f_i}_{S_0,t+(6+\frac{l}{2})(m-j+1)}\},
\end{align}
and
\begin{align}\label{for:49}
  \norm{\mathcal{R}_{1,j},\,\mathcal{R}_{2,j}}_{L_1,t}&\leq C_{t} \max_{j\leq i\leq m}\{\norm{h_i}_{L_1,t+(6+\frac{l}{2})(m-j+1)}\}.
\end{align}
We substitute the expressions for $\Omega_j$ and $g_j$ for all $k+1\leq j\leq m$ from \eqref{for:37}  into \eqref{for:39}. Then we have
\begin{align*}
&v(\Omega_k-*_{k}\eta_{k+1})-u(\Psi_k-\sum_{\text{\tiny$k+1\leq j\leq m$}}z_{k,j}\eta_j)+\mathcal{R}\notag\\
&=\mathfrak{w}_k-\sum_{\text{\tiny$k+1\leq j\leq m$}}z_{k,j}\mathcal{R}_{1,j}+*_k\mathcal{R}_{2,k+1}.
\end{align*}
where
\begin{align*}
\mathcal{R}=\sum_{\text{\tiny$k+1\leq j\leq m$}}*_jz_{k,j}\eta_{j+1}-*_k\sum_{\text{\tiny$k+2\leq l\leq m$}}z_{k+1,l}\mathfrak{\eta}_l.
\end{align*}
From \eqref{for:78} we see that $\mathcal{R}=0$. Hence we have
\begin{align*}
&v(\Omega_k-*_{k}\eta_{k+1})-u(\Psi_k-\sum_{\text{\tiny$k+1\leq j\leq m$}}z_{k,j}\eta_j)\notag\\
&=\mathfrak{w}_k-\sum_{\text{\tiny$k+1\leq j\leq m$}}z_{k,j}\mathcal{R}_{1,j}+*_k\mathcal{R}_{2,k+1}.
\end{align*}
By Lemma \ref{le:9} there exists $\eta_k\in \mathcal{H}$ satisfying $\mathcal{D}^l(\eta_k)=0$ with estimates
\begin{align*}
 \norm{\eta_k}_{S_0,t}&\leq C_{t}\norm{\Omega_k-*_{k}\eta_{k+1}}_{S_0,t+6+\text{\tiny$\frac{l}{2}$}}\\
 &\overset{\text{(1)}}{\leq}
C_{t} \max_{k\leq i\leq m}\{\norm{\Omega_i}_{S_0,t+(6+\text{\tiny$\frac{l}{2}$})(m-k+1)}\}
\end{align*}
for any $t\geq0$, such that
\begin{align*}
 \Omega_k-*_{k}\eta_{k+1}&=u\eta_k+\mathcal{R}_{1,k},\qquad\text{and}\\
 \Psi_k-\sum_{\text{\tiny$k+1\leq j\leq m$}}&z_{k,j}\eta_j=v\eta_k+\mathcal{R}_{2,k}
\end{align*}
with estimates
\begin{align*}
 \norm{\mathcal{R}_{1,k},\,\mathcal{R}_{2,k}}_{L_1,t}&\leq C_{t}\norm{\mathfrak{w}_k-\sum_{\text{\tiny$k+1\leq j\leq m$}}z_{k,j}\mathcal{R}_{1,j}+*_k\mathcal{R}_{2,k+1}}_{L_1,t+6+\text{\tiny$\frac{l}{2}$}}\\
 &\overset{\text{(2)}}{\leq}C_{t} \max_{k\leq i\leq m}\{\norm{\mathfrak{w}_i}_{L_1,t+(6+\frac{l}{2})(m-k+1)}\}
\end{align*}
for any $t\geq0$. Here in $(1)$ we use \eqref{for:48}; in $(2)$ we use \eqref{for:49}.

Then we proved the case of $k$ and thus finish the proof on the block $J_u$.  By repeating the above arguments for all Jordan blocks we get the result. It is clear that the size of each block is less than $\dim\mathfrak{g}$.  Hence we finish the proof.

\section{Proof of Proposition \ref{po:2}}\label{sec:7}

\subsection{$\alpha_A$ in Corollary \ref{cor:8}}\label{sec:40} We show that: let $A$ be a maximal abelian subgroup of $\GG$ of type $A_n$, $n\geq3$, then $A$ is  geometrically stable. In fact, the proofs for different types of $\GG$ are very similar. After minor modifications the proof for type $A_n$ still works for other types of groups.

For any $v=(v_{i,j})\in \mathfrak{sl}(n,\RR)$, let $\norm{v}=\max|v_{i,j}|$.  We use $\mathfrak{u}_{i,j}$ to denote the $n\times n$ matrix with all entries $0$ except the $(i,j)$ entry to be $1$.

We say that $(i,j)$ is a pair if $i\neq j$. We say that a pair $(i,j)$ is \emph{good} if $i$ is odd and $j$ is even.  By Remark \ref{re:7} we can assume that $A$ is spanned by $\mathfrak{u}_{i,j}$, where $(i,j)$ is good. Suppose $E'=\{u'_{i,j}\}$ is a $c$-perturbation of $E$ satisfying $c+\norm{\mathcal{M}(E')}<\delta$. We write $u'_{i,j}=\mathfrak{u}_{i,j}+\mathfrak{o}_{i,j}$; where $\mathfrak{o}_{i,j}=(\mathfrak{o}_{i,j,k,l})$ is a $n\times n$ matrix.

\textbf{\emph{Step $1$}}: Suppose $v_i\in \mathfrak{sl}(n,\RR)$ with $\norm{v_i}\leq \norm{E-E'}$, $1\leq i\leq m$. Let $v=\exp(v_m)\cdots\exp(v_1)$.  Then
\begin{gather*}
 \norm{\text{Ad}_{v}E'-E'-\sum_{i=1}^m\text{ad}_{v_i}E}\leq C_m\norm{E'-E}^2; \text{ and}\\
 \mathcal{M}(E')=\mathcal{M}\big(\text{Ad}_{v}E'\big).
\end{gather*}
Hence we have
\begin{gather*}
  \|\mathcal{M}\big(E'+\sum_{i=1}^m\text{ad}_{v_i}E\big)\|\leq \norm{\mathcal{M}(E')}+C_m\norm{E'-E}^2; \text{ and}\\
  \norm{E'+\sum_{i=1}^m\text{ad}_{v_i}E-E}\leq C_m\norm{E-E'},\quad \norm{v-I}\leq C_m\norm{E-E'}.
\end{gather*}
The above discussion shows that: $(*)$ it is harmless to replace $E'$ by $E'+\sum_{i=1}^m\text{ad}_{v_i}E$.

\smallskip
\textbf{\emph{Step $2$}}: In this part, we show that it is harmless to assume that the following hold for $E'$:

\smallskip
$(*')$ $\mathfrak{o}_{1,2,p,m}=0$ if $m=2$ or $p=1$ and $(p,m)$ is a pair; and
$\mathfrak{o}_{1,2,2,2}=\mathfrak{o}_{1,2,3,3}$.

\smallskip
$(*'')$ $\mathfrak{o}_{1,i,1,m}=0$ if $i\geq 2$, $m\geq 2$ and $(1,i)$ is good.
\smallskip

We note that the image of $\text{ad}_{\mathfrak{u}_{1,2}}$ is spanned by
$\{\mathfrak{u}_{1,2},\cdots,\mathfrak{u}_{1,n},\mathfrak{u}_{3,2},\cdots,$ $ \mathfrak{u}_{n,2}$, $\mathfrak{u}_{1,1}-\mathfrak{u}_{2,2}\}$. Then by $(*)$ we can assume $(*')$ holds.

$(*')$ shows that $(*'')$ holds for $i=2$. Then we argue by induction. Suppose $\mathfrak{o}_{1,i,1,m}=0$, if $m\geq 2$ for all $2\leq i\leq k$ where $2\leq k\leq n-2$. We note that:
$[\mathfrak{u}_{1,k+2},\mathfrak{u}_{k+2,m}]=\mathfrak{u}_{1,m}$ and $[\mathfrak{u}_{k+2,m},\mathfrak{u}_{1,p}]=0$ if $2\leq m\neq k+2$ and $p<k+1$; moreover,
$[\mathfrak{u}_{1,k+2},X_{k+1}]=\mathfrak{u}_{1,k+2}$ and $[X_{k+1},\mathfrak{u}_{1,p}]=0$ if $p<k+1$, where $X_{j}=\frac{1}{j+1}\diag(1,\cdots,\underset{j}{1},\underset{j+1}{-j},0,\cdots,0)\in\mathfrak{sl}(n,\RR)$.
Then by $(*)$ we can assume that: $\mathfrak{o}_{1,k+2,1,m}=0$, if $m\geq 2$.


\smallskip

 \smallskip
\textbf{\emph{Step $3$}}: For good pairs $(i,j)$ and $(k,l)$, we have
\begin{align*}
 (**)\quad \big\|[\mathfrak{u}_{i,j},\mathfrak{o}_{k,l}]-[\mathfrak{u}_{k,l},\mathfrak{o}_{i,j}]\big\|\leq \norm{\mathcal{M}(E')}+C\norm{E-E'}^2.
\end{align*}
Fix a good pair $(i,j)$. Next, we will obtain useful information from $(**)$ by choosing different good pairs $(k,l)$.

1. Choose $k\neq i$ and $l\neq j$. By checking the coefficient of $\mathfrak{u}_{k,l}$ in $(**)$, we have
$\mathfrak{o}_{i,j,k,k}\equiv\mathfrak{o}_{i,j,l,l}$, meaning
\begin{align*}
 |\mathfrak{o}_{i,j,k,k}-\mathfrak{o}_{i,j,l,l}|\leq \norm{\mathcal{M}(E')}+C\norm{E-E'}^2.
\end{align*}
2. We note that for any pair $(p,m)$ not good with $p\neq i$ and $m\neq j$, there is a good pair $(k,l)$ such that $[\mathfrak{u}_{k,l},\mathfrak{u}_{p,m}]$ is not the in image of $\text{ad}_{\mathfrak{u}_{i,j}}$. This shows that $\mathfrak{o}_{i,j,p,m}\equiv0$.

3. Choose $k\neq i$ and $l\neq j$.  By checking the coefficient of $\mathfrak{u}_{i,l}$ in $(**)$, we conclude that
$\mathfrak{o}_{i,j,i,k}\equiv-\mathfrak{o}_{k,l,j,l}$.

4.  Choose $k\neq i$ and let $l=j$. By checking the coefficient of $\mathfrak{u}_{i,j}$ in $(**)$, we conclude that $\mathfrak{o}_{i,j,i,k}\equiv\mathfrak{o}_{k,j,i,i}-\mathfrak{o}_{k,j,j,j}$.

5.  3 and 4 give: $-\mathfrak{o}_{k,l,j,l}\equiv\mathfrak{o}_{k,j,i,i}-\mathfrak{o}_{k,j,j,j}$ if $k\neq i$, $l\neq j$.

6. Choose $l\neq j$ and let $k=i$. By checking the coefficient of $\mathfrak{u}_{i,j}$ in $(**)$, we conclude that $\mathfrak{o}_{i,j,l,j}\equiv-\mathfrak{o}_{i,l,i,i}+\mathfrak{o}_{i,l,j,j}$.


\smallskip
\textbf{\emph{Step $4$}}: In this part,  we show that $(\spadesuit)$: $\mathfrak{o}_{1,l}\equiv0$ mod $\text{Lie}(A)$, $l\geq2$, meaning
$\mathfrak{o}_{1,l}-\mathfrak{c}\in\text{Lie}(A)$, where $\norm{\mathfrak{c}}\leq \norm{\mathcal{M}(E')}+C\norm{E-E'}^2$. We emphasize that we still use $(i,j)$ and $(k,l)$ to denote good pairs as in \textbf{\emph{Step $3$}}. We note that
\begin{align}\label{for:219}
 -\mathfrak{o}_{1,l,j,l}\overset{\text{\tiny$(a)$}}{\equiv}\mathfrak{o}_{1,j,i,i}-\mathfrak{o}_{1,j,j,j}\overset{\text{\tiny$(a)$}}{\equiv}\mathfrak{o}_{1,2,j,2}\overset{\text{\tiny$(b)$}}{=}0, \quad\text{if }j\neq 2,l.
\end{align}
Here $(a)$ is from 5; $(b)$ is from $(*')$. We also have
\begin{align}\label{for:220}
 -\mathfrak{o}_{1,l,2,l}\overset{\text{\tiny$(a)$}}{\equiv}\mathfrak{o}_{1,2,i,i}-\mathfrak{o}_{1,2,2,2}\overset{\text{\tiny$(b)$}}{=}0,\quad i\neq 1,\,\,l\neq2.
\end{align}
Here in $(a)$ we use 5; in $(b)$ we use $(*')$ and 1. \eqref{for:219} and \eqref{for:220} imply that
\begin{align}\label{for:221}
 \mathfrak{o}_{1,l,j,l}\equiv0,\quad l\neq j.
\end{align}
By \eqref{for:221}, 2 and $(*'')$ we have: for $l\geq2$, $\mathfrak{o}_{1,l}\equiv\sum_m\mathfrak{o}_{1,l,m,m}\mathfrak{u}_{m,m}$ mod $\text{Lie}(A)$. Then it suffices to show that
$\mathfrak{o}_{1,l,m,m}\equiv0$ for any $m$. We note that if $1\neq i,\, l\neq j$
\begin{align*}
 0&\overset{\text{\tiny$(a)$}}{\equiv}-\mathfrak{o}_{1,j,l,j}\overset{\text{\tiny$(b)$}}{\equiv}\mathfrak{o}_{1,l,1,1}-\mathfrak{o}_{1,l,j,j},\quad\text{and}\\
 0&\overset{\text{\tiny$(a)$}}{\equiv}-\mathfrak{o}_{1,j,l,j}\overset{\text{\tiny$(c)$}}{\equiv}\mathfrak{o}_{1,l,i,i}-\mathfrak{o}_{1,l,l,l}.
\end{align*}
Here $(a)$ is from \eqref{for:221};
 $(b)$ is from 6; $(c)$ is from 5. Moreover, by $1$ we have $\mathfrak{o}_{1,l,i,i}\equiv\mathfrak{o}_{1,l,j,j}$ if $1\neq i,\, l\neq j$.
 Hence $\mathfrak{o}_{1,l,i,i}=\mathfrak{o}_{1,l,j,j}$ for any $i,j$. Since $\mathfrak{o}_{1,l}\in
 \mathfrak{sl}(n,\RR)$, $\mathfrak{o}_{1,l,m,m}\equiv0$ for any $m$. Then we finish the proof for $(\spadesuit)$.

\smallskip
\textbf{\emph{Step $5$}}:  In this part we show that $\mathfrak{o}_{i,j}\equiv\sum_l\mathfrak{o}_{i,j,m,m}\mathfrak{u}_{m,m}$ mod $\text{Lie}(A)$. we have: $k\neq 1$, $l\neq j$
\begin{align}\label{for:222}
0\overset{\text{\tiny$(a)$}}{\equiv}-\mathfrak{o}_{1,j,1,k}\overset{\text{\tiny$(b)$}}{\equiv}\mathfrak{o}_{k,l,j,l}.
\end{align}
Here $(a)$ is from $(\spadesuit)$ and $(b)$ is form 3. \eqref{for:222} implies
\begin{align}\label{for:223}
  \mathfrak{o}_{i,j,l,j}\equiv0,\quad l\neq j
\end{align}
by letting $k=i$ and switching $j,l$. From \eqref{for:222} we have: $k\neq i$
\begin{align}\label{for:224}
 0\equiv\mathfrak{o}_{k,l,j,l}\overset{\text{\tiny$(b)$}}{\equiv}-\mathfrak{o}_{i,j,i,k}.
\end{align}
Here $(b)$ is form 3. By \eqref{for:223}, \eqref{for:224}, and 2 we get the result.

\smallskip
\textbf{\emph{Step $6$}}: In this part, we show that $\mathfrak{o}_{i,l,m,m}=0$ for any $m$. We have
\begin{gather*}
0\overset{\text{\tiny$(a)$}}{\equiv}-\mathfrak{o}_{i,j,l,j}\overset{\text{\tiny$(b)$}}{\equiv}\mathfrak{o}_{i,l,i,i}-\mathfrak{o}_{i,l,j,j},\quad l\neq j\\
0\overset{\text{\tiny$(c)$}}{\equiv}\mathfrak{o}_{i,j,i,k}\overset{\text{\tiny$(d)$}}{\equiv}\mathfrak{o}_{k,j,i,i}-\mathfrak{o}_{k,j,j,j},\quad k\neq i.
\end{gather*}
Here $(a)$ is from \eqref{for:223}; $(b)$ is from 6; $(c)$ is from \eqref{for:224}; $(d)$ is from 4. In $(d)$ switching $k,i$ and letting $j=l$ we have
$\mathfrak{o}_{i,l,k,k}-\mathfrak{o}_{i,l,l,l}$, if $k\neq i$. By 1 we have $\mathfrak{o}_{i,l,k,k}\equiv\mathfrak{o}_{i,l,j,j}$ if $k\neq i$, $l\neq j$.  Hence we see that $\mathfrak{o}_{i,l,k,k}\equiv \mathfrak{o}_{i,l,j,j}$ for any $k,j$. Since $\mathfrak{o}_{i,l}\in
 \mathfrak{sl}(n,\RR)$, $\mathfrak{o}_{i,l,m,m}\equiv0$ for any $m$.

\smallskip
\textbf{\emph{Step $7$}}: \textbf{\emph{Step $5$}} and \textbf{\emph{Step $6$}} imply that $\mathfrak{o}_{i,j}\equiv0$ mod $\text{Lie}(A)$. Hence we finish the proof.

\subsection{$\alpha_A$ in Corollary \ref{cor:2}} We fix a basis $e_{i,1},e_{i,2}\cdots $ for each $\text{Lie}(A_i)$. Suppose $E'=\{e_{i,j}'\}$ is a $c$-perturbation of $E$ satisfying $c+\norm{\mathcal{M}(E')}<\delta$. We write $e'_{i,j}=e_{i,j}+\mathfrak{o}_{i,j}$. We denote by $p_i$ the projection from $\text{Lie}(\GG)$ to the $i$-th factor. As each $\alpha_A|_{A_i}$ is geometrically stable inside $\GG_i$ (see Section \ref{sec:40}), it is harmless to assume that
\begin{align}\label{for:209}
 p_i(\mathfrak{o}_{i,j})=0,\qquad \text{for each }i,\,j.
\end{align}
Next, we show that
\begin{align}\label{for:211}
 \norm{p_k(\mathfrak{o}_{i,j})/\text{Lie}(A_k)}\leq C\norm{\mathcal{M}(E')},\qquad\text{if }k\neq i.
\end{align}
For each $k$ we see that the map $q_k:\,\mathfrak{g}_k/\text{Lie}(A_k)\to q_k\big(\mathfrak{g}_k/\text{Lie}(A_k) \big)\subset\mathfrak{g}_k^{\dim \text{Lie}(A_k)}$ with the assignment: $Y\to ([Y,e_{k,1}],\,[Y,e_{k,2}]\cdots )$ is both well defined ($\text{Lie}(A_k)$ is abelian) and injective ($\text{Lie}(A_k)$ is maximal). Thus we have
\begin{align}\label{for:210}
 \norm{q_k^{-1}}\leq C,\qquad \forall\,k.
\end{align}
Note that
\begin{align*}
p_k([e'_{i,j},e'_{k,l}])=[p_k(e'_{i,j}),\,p_k(e'_{k,l})]\overset{\text{\tiny$(1)$}}{=}[p_k(\mathfrak{o}_{i,j}) ,\,e_{k,l}]
\end{align*}
for all $j,\,l$ if $i\neq k$. Here in $(1)$ we use \eqref{for:209}. This means
\begin{align*}
 \norm{[p_k(\mathfrak{o}_{i,j}) ,\,e_{k,l}]}\leq \norm{[e'_{i,j},e'_{k,l}]}\leq \norm{\mathcal{M}(E')}
\end{align*}
for all $j,\,l$ if $i\neq k$. This and \eqref{for:210} give \eqref{for:211}.

The result follows from \eqref{for:209} and \eqref{for:211} immediately.

\subsection{$\alpha_A$ in Corollary \ref{cor:13}} We denote by $\mathfrak{g}_1$ the Lie algebra of the $SL(n-1,\RR)$ subgroup containing $A_1$. It is easy to check that
 \begin{align}\label{for:212}
  F_1=\{Y\in \mathfrak{g}: [\textbf{x},Y]=0\}=\{\RR\textbf{x}\}\times\mathfrak{g}_1.
 \end{align}
We note that the space
 \begin{align}\label{for:226}
 F=\{X\in \mathfrak{g}: X \text{ is semisimple and commutes with }
 \text{Lie}(A_1)\}
 \end{align}
 is one dimensional.

 We fix a basis $e_{1},e_{2}\cdots $ for $\text{Lie}(A_1)$. Suppose $E'=\{\textbf{x}',e_{1}',e_2',\cdots\}$ is a $c$-perturbation of $E$ satisfying $c+\norm{\mathcal{M}(E')}<\delta$. There is $g\in \GG$ with $\norm{g-I}\leq Cc$ such that for $\text{Ad}_g(\textbf{x}')$
we have a decomposition
\begin{align*}
 \text{Ad}_g(\textbf{x}')=\mathfrak{s}+\mathfrak{k}+\mathfrak{n}
\end{align*}
for $3$ commuting elements,   where $\mathfrak{s}$ is a diagonal matrix, $\mathfrak{k}$ is compact and $\mathfrak{n}$ is nilpotent satisfying
\begin{align*}
 \norm{\mathfrak{s}-\textbf{x}}+\norm{\mathfrak{k}}+\norm{\mathfrak{n}}\leq Cc.
\end{align*}
Thus we have
\begin{align}\label{for:257}
\norm{\mathfrak{k}}+\norm{\mathfrak{n}}\leq Cc.
\end{align}
From \eqref{for:212} we see that if  $\delta$ is sufficiently small then we have
\begin{align*}
  \{Y\in \mathfrak{g}: [\mathfrak{s},Y]=0\}\subseteq\{\RR\textbf{x}\}\times\mathfrak{g}_1
 \end{align*}
This implies that
\begin{align}\label{for:259}
 \text{$\mathfrak{k}$ and $\mathfrak{n}$ are both in $\mathfrak{g}_1$}.
\end{align}
Thus
\begin{align*}
  \{Y\in \mathfrak{g}: [\text{Ad}_g(\textbf{x}'),Y]=0\}\subseteq\{\RR\textbf{x}\}\times\mathfrak{g}_1.
 \end{align*}
We consider the map $q: \mathfrak{g}\to \mathfrak{g}$ with the assignment: $q(z)=[z,\text{Ad}_g(\textbf{x}')]$. Then we have
\begin{align*}
 \norm{q^{-1}}\leq C, \qquad\text{where }q^{-1}:q(\mathfrak{g})\to \mathfrak{g}/\ker(q).
\end{align*}
Since
\begin{align}\label{for:225}
 \big\|[\text{Ad}_g(\textbf{x}'),\,\text{Ad}_g(e_k')]\big\|\leq \norm{\mathcal{M}(\text{Ad}_gE')}\leq C\norm{\mathcal{M}(E')}
\end{align}
for any $k$, we see that there are $\mathfrak{o}_k\in \mathfrak{g}$ with
\begin{align*}
 \norm{\mathfrak{o}_k}\leq C\norm{\mathcal{M}(E')}
\end{align*}
 such that $e_k''=\text{Ad}_g(e_k')-\mathfrak{o}_k\in \ker(q)\subseteq \{\RR\textbf{x}\}\times\mathfrak{g}_1$ for any $k$.

 As any maximal abelian subgroup in $SL(n-1,\RR)$ is geometrically stable (see Section \ref{sec:40}), we can assume that
 \begin{align}\label{for:228}
  e_k''=e_k+d_k\textbf{x},\quad \forall\, k
 \end{align}
 where $d_k\in\RR$ and
 \begin{align}\label{for:258}
  \norm{d_k}\leq Cc\qquad \forall\, k.
 \end{align}
\eqref{for:228} together with \eqref{for:225} give
 \begin{align}
 \big\|[\text{Ad}_g(\textbf{x}'),\,e_k+d_k\textbf{x}]\big\|&\leq C\norm{\mathcal{M}(E')}\notag\\
 \overset{(1)}{\Rightarrow} \big\|[\text{Ad}_g(\textbf{x}'),\,e_k]\big\|&\leq C\norm{\mathcal{M}(E')}+C_1c^2\notag\\
 \overset{(2)}{\Rightarrow} \big\|[g_0,\,e_k]\big\|&\leq C\norm{\mathcal{M}(E')}+C_1c^2\label{for:260}
\end{align}
 for any $k$. Here in $(1)$ we note that
 \begin{align*}
 \big\|[\text{Ad}_g(\textbf{x}'),\,d_k\textbf{x}]\big\|=\big\|[\mathfrak{k}+\mathfrak{n},\,d_k\textbf{x}]\big\|\overset{(a)}{\Rightarrow} C_1c^2.
 \end{align*}
Here in $(a)$ we use \eqref{for:257} and \eqref{for:258}; in $(2)$ we note that there is some $l\in\RR$ such that $\mathfrak{s}-l\textbf{x}\in \mathfrak{g}_1$, which gives
\begin{align}\label{for:271}
 \text{Ad}_g(\textbf{x}')=l\textbf{x}+\mathfrak{s}-l\textbf{x}+\mathfrak{k}+\mathfrak{n}=l\textbf{x}+g_0
\end{align}
where $g_0\in \mathfrak{g}_1$ (we recall \eqref{for:259}).

We define a map $p:  \mathfrak{g}/\text{Lie}(A_1)\to p(\mathfrak{g}/\text{Lie}(A_1))\subseteq\mathfrak{g}^{\dim \text{Lie}(A_1)}$ with the assignment: $Y\to ([Y,e_{1}],\,[Y,e_{2}]\cdots )$ is both well defined ($\text{Lie}(A_1)$ is abelian) and injective ($\text{Lie}(A_1)$ is maximal). Thus we have
\begin{align*}
 \norm{p^{-1}}\leq C.
\end{align*}
This together with \eqref{for:260} imply that
 \begin{align*}
 \norm{ g_0}\leq C_2\norm{\mathcal{M}(E')}+C_2c^2.
 \end{align*}
This means that it is harmless to assume that $\text{Ad}_g(\textbf{x}')=l\textbf{x}$ (see \eqref{for:271}). This together with \eqref{for:228} gives  $e_k''\in \text{Lie}(A_1)$ for any $k$. This completes the proof.

\bibliographystyle{ieeetr}

\bibliography{mathBib}

\end{document}